\documentclass[a4paper,oneside,british,12pt]{amsbook}

\pdfoutput=1

\usepackage{theshead}

\usepackage{babel}
\usepackage[T1]{fontenc}
\usepackage[utf8]{inputenc}
\usepackage[varg]{pxfonts}

\usepackage[left=3cm, right=3cm, top=2.5cm, bottom=2.5cm]{geometry}
\usepackage{needspace}
\usepackage{setspace}
\doublespace

\usepackage{aliascnt}
\usepackage{varioref}
\labelformat{equation}{\textup{(#1)}}
\labelformat{enumi}{\textup{(#1)}}

\usepackage{amssymb,mathrsfs}
\usepackage{subfigure}

\usepackage[all,arc]{xy}

\usepackage{ifpdf}
\ifpdf
\IfFileExists{xypdf}{
\xyoption{pdf}
}{}
\usepackage{microtype}
\fi

\usepackage[pdfusetitle, unicode=true, bookmarks=true, hidelinks, pdfborder={0 0 0}, colorlinks=false, breaklinks=true, linktoc=all]{hyperref}

\title{Operads and moduli spaces}

\author{Christopher Braun}
\degreeclass{Doctor of Philosophy\\at the University of Leicester}
\date{April 2012}
\address{Department of Mathematics\\
University of Leicester}

\theoremstyle{theorem}
\newtheorem{theorem}{Theorem}[section]
\newcommand{\newautoreftheorem}[2]{
\newaliascnt{#1}{theorem}\newtheorem{#1}[#1]{#2}\aliascntresetthe{#1}%
\expandafter\def\csname #1autorefname\endcsname{#2}}

\newautoreftheorem{proposition}{Proposition}
\newautoreftheorem{lemma}{Lemma}
\newautoreftheorem{corollary}{Corollary}
\newtheorem*{theorem*}{Theorem}

\theoremstyle{definition}
\newautoreftheorem{definition}{Definition}
\newautoreftheorem{remark}{Remark}
\newautoreftheorem{example}{Example}
\newautoreftheorem{examples}{Examples}
\newautoreftheorem{notation}{Notation}
\newautoreftheorem{convention}{Convention}
\newtheorem*{notation*}{Notation}

\numberwithin{section}{chapter}
\numberwithin{equation}{section} 
\numberwithin{figure}{section}   

\newcommand{\cstart}{\xybox{*\xycircle(1,3){-}}}
\newcommand{\cstop}{\xybox{\ellipse(1,3){.}\ellipse(1,3)_,=:a(180){-}}}

\newcommand{\handle}{\xybox{(4,0.9)="A";"A"-(8,0.5)**\crv{"A"-(4,3.5)},(3,0.3)="A";"A"-(6,0.4)**\crv{"A"+(-4.3,1.5)}}}
\newcommand{\puncture}{\xybox{*\xycircle(2,2){-}}}
\newcommand{\crosscap}{\xybox{{\ellipse(2.5,2){.}}*\xycircle(0,2){-}*\xycircle(2.5,0){-}}}
\newcommand{\cobord}[8][(10,0)]{\xybox{@=#6@@{@i@={#6}@@{*{\handle}}@i}
@=#7@@{@i@={#7}@@{*{\puncture}}@i}@=#8@@{@i@={#8}@@{*{\crosscap}}@i},
(0,0)="d"@=#2@@{@i-(0,3)="B1"="T1"@={#2}
@@{;"T1";*{#4}-(0,3)="b"**\crv{"T1"+"d"&"b"+"d"}+(0,6)="T1",(5,0)="d"}@i},
(0,0)="d"@=#3@@{@i-(0,3)="B2"="T2"@={#3}
@@{;"T2";*{#5}-(0,3)="b"**\crv{"T2"-"d"&"b"-"d"}+(0,6)="T2",(5,0)="d"}@i},
#1="S"@=@(@=#2@@{@i@(@=@(@=#3@@{@i,
"T1";"T2"**\crv{"T1"+"S"&"T2"-"S"},"B1";"B2"**\crv{"B1"+"S"&"B2"-"S"}
@)}@i@)@@{@i,"T1";"B1"**\crv{"T1"+"S"&"B1"+"S"}@)}@)}
@i@)@@{@i@=#3@@{@i,"T2";"B2"**\crv{"T2"-"S"&"B2"-"S"}}@)}}}
\newcommand{\basiccob}[5][(10,0)]{\cobord[#1]{#2}{#3}{#4}{#5}{@i}{@i}{@i}}

\newcommand{\cylinder}[2]{\basiccob{(0,0)}{(20,0)}{#1}{#2}}
\newcommand{\shortcylinder}[2]{\basiccob{(0,0)}{(10,0)}{#1}{#2}}
\newcommand{\twist}[2]{\xybox{*{\basiccob[(5,0)]{(0,0)}{(20,12)}{#1}{#2}}*{\basiccob[(5,0)]{(0,12)}{(20,0)}{#1}{#2}}}}
\newcommand{\pants}[2]{\basiccob{(0,0),(0,12)}{(20,6)}{#1}{#2}}
\newcommand{\pair}[1]{\basiccob[(15,0)]{(0,0),(0,12)}{@i}{#1}{}}
\newcommand{\copants}[2]{\basiccob{(0,6)}{(20,0),(20,12)}{#1}{#2}}
\newcommand{\copair}[1]{\basiccob[(15,0)]{@i}{(20,0),(20,12)}{}{#1}}
\newcommand{\birth}[1]{\basiccob{@i}{(0,0)}{}{#1}}
\newcommand{\death}[1]{\basiccob{(0,0)}{@i}{#1}{}}
\newcommand{\rp}[1]{\xybox{*{\birth{}},(10,0)*{\shortcylinder{}{#1}},(7,0)*{\crosscap}}}
\newcommand{\flip}[2]{\xybox{(0,0)*{#1},(10,0)*{#2},(0,3);(10,-3)**\crv{(3,3)&(7,-3)},(0,-3);(10,3)**\crv{(3,-3)&(7,3)}}}

\newcommand{\co}{\colon\thinspace}
\newcommand{\twistshriek}{{\ensuremath{\overset{\boldsymbol{\wr}}{.}}}}
\newcommand{\dual}{\mathbf{D}}
\newcommand{\cobar}{\mathbf{B}}
\newcommand{\feyn}{\mathbf{F}}
\newcommand{\mob}{\mathrm{M}}
\newcommand{\modc}[1]{\overline{#1}}
\newcommand{\modn}[1]{\underline{#1}}
\newcommand{\cyc}[1]{\mathrm{Cyc}(#1)}
\newcommand{\unimod}{\mathsf{K}}

\newcommand{\ctft}{\mathrm{TFT}}
\newcommand{\otft}{\mathrm{OTFT}}
\newcommand{\cktft}{\mathrm{pKTFT}}
\newcommand{\oktft}{{\mathrm{OKTFT}}}
\newcommand{\com}{\mathcal{C}om}
\newcommand{\ass}{\mathcal{A}ss}
\newcommand{\lie}{\mathcal{L}ie}

\newcommand{\mcom}{\mob\com}
\newcommand{\mass}{\mob\ass}
\newcommand{\klie}{\unimod\lie}
\newcommand{\dass}{\dual\ass}
\newcommand{\dmass}{\dual\mass}
\newcommand{\modcom}{\modc{\com}}
\newcommand{\modass}{\modc{\ass}}
\newcommand{\modmcom}{\modc{\mcom}}
\newcommand{\modmass}{\modc{\mass}}
\newcommand{\moddass}{\modc{\dass}}
\newcommand{\moddmass}{\modc{\dmass}}
\newcommand{\no}{\mathcal{N}}
\newcommand{\ko}{\mathcal{K}}
\newcommand{\mo}{\mathcal{M}^\mathbb{R}}
\newcommand{\nb}{\overline{\no}}
\newcommand{\kb}{\overline{\ko}}
\newcommand{\mb}{\overline{\mo}}
\newcommand{\dr}{D^\mathbb{R}}
\newcommand{\Mat}{\mathrm{Mat}}
\newcommand{\CC}{\mathrm{CC}}
\newcommand{\HCC}{\mathrm{HC}}
\newcommand{\CD}{\mathrm{CD}}
\newcommand{\HCD}{\mathrm{HD}}
\newcommand{\CE}{\mathrm{CE}}
\newcommand{\HCE}{\mathrm{HCE}}
\newcommand{\hoch}{\mathrm{CH}}
\newcommand{\Hhoch}{\mathrm{HH}}
\newcommand{\MC}{\mathrm{MC}}
\newcommand{\MCmoduli}{\mathscr{MC}}
\newcommand{\Def}{\mathrm{Def}}
\newcommand{\Der}{\mathrm{Der}}
\newcommand{\cycl}{\mathrm{cycl}}

\newcommand{\id}{\mathrm{id}}
\newcommand{\degree}[1]{\bar{#1}}
\DeclareMathOperator{\tr}{tr}
\DeclareMathOperator{\ad}{ad}
\DeclareMathOperator{\im}{Im}
\DeclareMathOperator*{\colim}{colim}
\DeclareMathOperator{\vertices}{Vert}
\DeclareMathOperator{\halfedges}{Half}
\DeclareMathOperator{\edges}{Edge}
\DeclareMathOperator{\flags}{Flag}
\DeclareMathOperator{\legs}{Leg}
\DeclareMathOperator{\inputs}{In}

\DeclareMathOperator{\ob}{Ob}
\DeclareMathOperator{\iso}{Iso}
\DeclareMathOperator{\Hom}{Hom}
\DeclareMathOperator{\End}{End}
\DeclareMathOperator{\ctsHom}{Hom_{\mathrm{cts}}}
\DeclareMathOperator{\IHom}{\underline{Hom}}
\DeclareMathOperator{\Det}{Det}

\newcommand{\cob}{\mathbf{2Cob}}
\newcommand{\cobo}{\mathbf{2Cob}^{\mathrm{o}}}
\newcommand{\cobcl}{\mathbf{2Cob}^{\mathrm{cl}}}
\newcommand{\kcob}{\mathbf{2KCob}}
\newcommand{\kcobo}{\mathbf{2KCob}^{\mathrm{o}}}
\newcommand{\kcobcl}{\mathbf{2KCob}^{\mathrm{cl}}}
\newcommand{\vect}{\mathbf{Vect}_k}
\newcommand{\dgvect}{\mathbf{dgVect}_k}
\newcommand{\fdgvect}{\mathscr{F}\dgvect}
\newcommand{\topol}{\mathbf{Top}}
\newcommand{\dgop}{\mathbf{dgOp}}

\newcommand{\klein}{\mathbf{Klein}}
\newcommand{\dklein}{\mathbf{dKlein}}
\newcommand{\symriem}{\mathbf{SymRiem}}
\newcommand{\dsymriem}{\mathbf{dSymRiem}}
\newcommand{\nklein}{\mathbf{nKlein}}
\newcommand{\dnklein}{\mathbf{dnKlein}}
\newcommand{\nsymriem}{\mathbf{nSymRiem}}
\newcommand{\dnsymriem}{\mathbf{dnSymRiem}}

\begin{document}

\frontmatter

\begin{abstract}
This thesis is concerned with the application of operadic methods, particularly modular operads, to questions arising in the study of moduli spaces of surfaces as well as applications to the study of homotopy algebras and new constructions of `quantum invariants' of manifolds inspired by ideas originating from physics.

We consider the extension of classical $2$--dimensional topological quantum field theories to Klein topological quantum field theories which allow unorientable surfaces. We generalise open topological conformal field theories to open Klein topological conformal field theories and consider various related moduli spaces, in particular deducing a M\"obius graph decomposition of the moduli spaces of Klein surfaces, analogous to the ribbon graph decomposition of the moduli spaces of Riemann surfaces.

We also begin a study, in generality, of quantum homotopy algebras, which arise as `higher genus' versions of classical homotopy algebras. In particular we study the problem of quantum lifting. We consider applications to understanding invariants of manifolds arising in the quantisation of Chern--Simons field theory.
\end{abstract}

\maketitle

\chapter*{Acknowledgements}
I shall not attempt to list here all the people I want to thank since such a list would be far too long and impersonal. However there are some people I do wish to mention here.

First and foremost I wish to thank my supervisor, Andrey Lazarev. His knowledge, inspiration and good company made my years as a PhD student thoroughly enjoyable. I cannot imagine a better supervisor and without his good advice, teaching and patience I would have been lost.

I am also grateful to my friends and family, who have always been a source of support, advice and much happiness. Most of all, I want to thank my parents, Jan and Mike Braun. I don't think I will ever be able to fully appreciate everything they've done for me. I humbly dedicate this thesis to them.

\tableofcontents

\mainmatter
\def\chapterautorefname{Chapter}
\def\sectionautorefname{Section}

\chapter{Introduction}
This thesis is concerned with the application of operadic methods, particularly modular operads, to questions arising in the study of moduli spaces of surfaces as well as applications to the study of homotopy algebras and new constructions of `quantum invariants' of manifolds inspired by ideas originating from physics.

The structure of this thesis can be broadly regarded as naturally divided into two parts. The first part involves the study of moduli spaces of surfaces, in particular Klein surfaces, while the second part concerns the study of homotopy algebras and quantum homotopy algebras. In this introduction we will briefly outline these areas and summarise the main results in each.

\begin{notation*}
Throughout this thesis $k$ will be used to denote a field which, for simplicity and convenience, we will normally assume to be $\mathbb{Q}$ unless stated otherwise. Many of the definitions and results should of course work over more general fields, however this is certainly not a topic we wish to concern ourselves with here.
\end{notation*}

\section{Moduli spaces of Klein surfaces and related operads}
One property of the original axiomatic definition by Atiyah \cite{atiyah} of a topological quantum field theory (TFT) is that all the manifolds considered are oriented. Alexeevski and Natanzon \cite{alexeevskinatanzon} considered a generalisation to manifolds that are not oriented (or even necessarily orientable) in dimension $2$. An unoriented TFT in this sense is then called a Klein topological quantum field theory (KTFT).

It is well known that $2$--dimensional closed TFTs are equivalent to commutative Frobenius algebras and open TFTs are equivalent to symmetric (but not necessarily commutative) Frobenius algebras, for example see Moore \cite{moore2} and Segal \cite{segal}. Theorems of this flavour identifying the algebraic structures of KTFTs have also been shown. In the language of modular operads, developed by Getzler and Kapranov \cite{getzlerkapranov}, these results for oriented TFTs say that the modular operads governing closed and open TFTs are $\modcom$ and $\modass$ which are the modular closures (the smallest modular operad containing a cyclic operad) of $\com$ and $\ass$, which govern commutative and associative algebras.

It is also possible to generalise TFTs by adding extra structure to our manifolds such as a complex structure which gives the notion of a topological conformal field theory (TCFT). We can also find topological modular operads governing TCFTs constructed from moduli spaces of Riemann surfaces.

The ribbon graph decomposition of moduli space is an orbi-cell complex homeomorphic to $\mathcal{M}_{g,n}\times \mathbb{R}^n_{>0}$ with cells labelled by ribbon graphs, introduced in Harer \cite{harer} and Penner \cite{penner}. Ribbon graphs arise from the modular closure of the $A_\infty$ operad (cf Kontsevich \cite{kontsevich}). Indeed the cellular chain complex of the operad given by gluing stable holomorphic discs with marked points on the boundary is equivalent to the $A_\infty$ operad and can be thought of as the genus $0$ part of the operad governing open TCFTs. It was shown by Kevin Costello \cite{costello1,costello2} that this gives a dual point of view on the ribbon graph decomposition of moduli space: The operad governing open TCFTs is homotopy equivalent to the modular closure of the suboperad of conformal discs and so this gives a quasi-isomorphism on the chain complex level to the modular closure of the $A_\infty$ operad. The moduli spaces underlying the open TCFT operad are those of stable Riemann surfaces with boundary and marked points on the boundary. In particular this yields new proofs of ribbon graph complexes computing the homology of these moduli spaces.

We wish to consider the corresponding theory for KTFTs. Alexeevski and Natanzon \cite{alexeevskinatanzon} considered open--closed KTFTs and Turaev and Turner \cite{turaevturner} considered just closed KTFTs. We will concentrate mainly on the open version in order to parallel the theory outlined above. We begin by recasting the definitions for KTFTs in terms of modular operads. We show that the open KTFT operad is given by the modular closure of the cyclic operad $\mass$ which is the operad governing associative algebras with involution. The corresponding notion of a ribbon graph, a M\"obius graph, is also developed to identify $\mass$ and the various operads obtained from it. On the other hand the closed KTFT operad is not the modular closure of a cyclic operad.

We then generalise to the Klein analogue of open TCFTs (open KTCFTs). The correct notion here of an `unoriented Riemann surface' is a Klein surface, where we allow transition functions between charts to be anti-analytic. Alling and Greenleaf \cite{allinggreenleaf} developed some of the classical theory of Klein surfaces and showed that Klein surfaces are equivalent to smooth projective real algebraic curves. We find appropriate partial compactifications of moduli spaces of Klein surfaces which form the modular operad governing open KTCFTs. We also consider other different (although more common) partial compactifications giving rise to a quite different modular operad. The underlying moduli spaces of this latter operad are spaces of `admissible' stable symmetric Riemann surfaces (which are open subspaces of the usual compactifications containing all stable symmetric surfaces).

By following the methods of Costello \cite{costello1,costello2} we can obtain graph decompositions of these moduli spaces. Precisely this means we find orbi-cell complexes homotopy equivalent to these spaces with each orbi-cell labelled by a type of graph. As a consequence we see that open KTCFTs are governed by the modular closure of the operad governing $A_\infty$--algebras with involution and we obtain a M\"obius graph complex computing the homology of the moduli spaces of smooth Klein surfaces. We also obtain a different graph complex computing the homology of the other partial compactifications.

\subsection{Outline and main results}
\autoref{chap:tqfts} and \autoref{chap:operads} provide the necessary background and notation. In \autoref{chap:tqfts} definitions of topological quantum field theories and their Klein analogues are briefly introduced in terms of symmetric monoidal categories of cobordisms. The known results concerning the structure of KTFTs are stated and we provide some pictures that hopefully shed light on how these results arise. In \autoref{chap:operads} the definitions from the theory of modular operads that we use is recalled and the cobar construction is outlined. We include a slight generalisation of modular operads: extended modular operads (which is very similar to the generalisation of Chuang and Lazarev \cite{chuanglazarev}). For the reader familiar with modular operads this section will likely be of little interest apart from making clear the notation used here.

\autoref{chap:mobius} introduces the open KTFT modular operad denoted $\oktft$. M\"obius trees and graphs are discussed in detail and the operad $\mass$ is defined in terms of M\"obius trees. This is the operad governing associative algebras with an involutive anti-automorphism. We then show $\oktft\cong\modmass$ thereby providing a generators and relations description of $\oktft$ in terms of M\"obius graphs. We show $\mass$ is its own quadratic dual, is Koszul and identify the dual dg operad $\dmass$ (governing $A_\infty$--algebras with an involution) and its modular closure. Finally we generalise our construction and discuss the closed KTFT operad, showing that only part of the closed KTFT operad is the modular closure of an operad $\mcom$.

In \autoref{chap:moduli} we generalise to open KTCFTs. We discuss the necessary definitions and theory of Klein surfaces and nodal Klein surfaces. A subtlety arises when considering nodal surfaces and we find there are two different natural notions of a node. We provide some clarity on this difference by establishing some equivalences of categories: we show that one sort of nodal Klein surface is equivalent to a certain sort of symmetric nodal Riemann surface with boundary and the other is equivalent to a certain sort of symmetric nodal Riemann surface without boundary. We obtain moduli spaces $\kb_{g,u,h,n}$ of stable nodal Klein surfaces with $g$ handles, $u$ crosscaps, $h$ boundary components and $n$ oriented marked points using one definition of a node. We also obtain quite different moduli spaces $\mb_{\tilde{g},n}$ of `admissible' stable symmetric Riemann surfaces without boundary of genus $\tilde{g}$ and $n$ fixed marked points using the other definition. The spaces $\kb_{g,u,h,n}$ are homotopy equivalent to their interiors which are the spaces $\ko_{g,u,h,n}$ of smooth Klein surfaces with oriented marked points. The spaces $\mb_{\tilde{g},n}$ are partial compactifications of the spaces $\mo_{\tilde{g},n}$ of smooth symmetric Riemann surfaces, which are the same as the spaces of smooth Klein surfaces with unoriented marked points. Let $D_{g,u,h,n}\subset\kb_{g,u,h,n}$ be the locus of surfaces such that each irreducible part is a disc. Let $\dr_{\tilde{g},n}$ be the corresponding subspace of $\mb_{\tilde{g},n}$. We obtain topological modular operads $\kb$ and $\mb$ by gluing at marked points. The operad $\kb$ gives the correct generalisation governing open KTCFTs. We then show the inclusions of the suboperads arising from the spaces $D_{g,u,h,n}$ and $\dr_{\tilde{g},n}$ are homotopy equivalences.

\begin{theorem*}
\Needspace*{3\baselineskip}\mbox{}
\begin{itemize}
\item The inclusion $D\hookrightarrow\kb$ is a homotopy equivalence of extended topological modular operads.
\item The inclusion $\dr\hookrightarrow\mb$ is a homotopy equivalence of extended topological modular operads.
\end{itemize}
\end{theorem*}

Applying an appropriate chain complex functor $C_*$ from topological spaces to dg vector spaces over $\mathbb{Q}$ we obtain dg modular operads and the above result translates to:
\begin{theorem*}
There are quasi-isomorphism of extended dg modular operads over $\mathbb{Q}$
\begin{gather*}
C_*(D)\simeq C_*(\kb)\\
C_*(D)/(a=1)\simeq C_*(\mb)
\end{gather*}
where $a\in C_*(D)((0,2))\cong\mathbb{Q}[\mathbb{Z}_2]$ is the involution.
\end{theorem*}

The spaces $D_{g,u,h,n}$ decompose into orbi-cells labelled by M\"obius graphs and so we can identify the cellular chain complexes $C_*(D)$ in terms of the operad $\mass$ so that $C_*(D)\cong\moddmass$. Therefore we see that an open KTCFT is a Frobenius $A_\infty$--algebra with involution and we also obtain M\"obius graph complexes computing the homology of the moduli spaces of smooth Klein surfaces as well as different graph complexes (arising from $\moddmass/(a=1)$) computing the homology of the partial compactifications given by $\mb$. Unlike $H_\bullet(\kb)$, the genus $0$ part of $H_\bullet(\mb)$ has non-trivial components in higher degrees. The gluings for the operad $\mb$ can be thought of as `closed string' gluings similar to those for the Deligne--Mumford operad.

We finish \autoref{chap:moduli} by unwrapping our main theorems to give concrete and elementary descriptions of the different graph complexes and explain the isomorphisms of homology without reference to operads.

Finally, in \autoref{chap:dihedral} we examine in more detail algebras over $\dmass$ and $\moddmass$, which we christen involutive $A_\infty$--algebras and cyclic involutive $A_\infty$--algebras. We construct the natural cohomology theories associated to these algebras, which govern their deformations, and relate this to dihedral cohomology theory \cite{lodaybook,lodaydihedral}.

\section{Quantum homotopy algebras and Chern--Simons field theory}
Let $\mathcal{P}$ be a cyclic Koszul operad. A quantum homotopy $\mathcal{P}$--algebra (or sometimes a loop homotopy $\mathcal{P}$--algebra \cite{markl}) is a `higher genus' generalisation of a cyclic homotopy $\mathcal{P}$--algebra. Just as a homotopy $\mathcal{P}$--algebra is governed by a certain operad obtained naturally from $\mathcal{P}$ so a quantum homotopy $\mathcal{P}$--algebra is governed by a certain \emph{modular} operad obtained from $\mathcal{P}$. It is a very natural question to ask when a cyclic homotopy $\mathcal{P}$--algebra structure is itself naturally a part of a richer quantum homotopy $\mathcal{P}$--algebra structure, or indeed if a cyclic homotopy $\mathcal{P}$--algebra can be lifted to a quantum homotopy $\mathcal{P}$--algebra at all. In \autoref{chap:quantum} we will begin a study of quantum homotopy $\mathcal{P}$--algebras and quantum lifts.

A good example of an application of the question of quantum lifting is given for the case $\mathcal{P}=\ass$. Just as Kontsevich \cite{kontsevich} showed that cyclic $A_\infty$--algebras give rise to homology classes in the moduli space of Riemann surfaces it is also known \cite{barannikov, hamilton1} that quantum $A_\infty$--algebras give rise to homology classes in a certain compactification of this moduli space. Given a class in the (one point compactification of) the moduli space of surfaces one can ask if it lifts to this compactified moduli space. In \cite{hamilton2} Hamilton reinterprets this question as a problem of lifting cyclic $A_\infty$--algebras to quantum $A_\infty$--algebras and provides an explicit obstruction theory. One of the main aims of \autoref{chap:mc} and \autoref{chap:quantum} is to generalise this obstruction theory for general $\mathcal{P}$.

In a rather different direction a rather intriguing manifestation of these ideas appears also in \cite{costello4} which concerns the quantisation of Chern--Simons field theory via an infinite dimensional Batalin--Vilkovisky formalism. Given a closed oriented manifold $M$ and a cyclic Lie algebra $\mathfrak{g}$ the cohomology of the cyclic Lie algebra $\Omega^{\bullet}(M)\otimes\mathfrak{g}$ has the structure of a cyclic $L_\infty$--algebra\footnote{Technically this is an \emph{odd} cyclic $L_\infty$--algebra structure, but this will be dealt with later once we have established a general framework.} by transferring the Lie algebra structure via the techniques of homological perturbation theory or the theory of minimal models. It is shown in \cite{costello4} that this structure is then in turn part of a richer structure, a quantum $L_\infty$--algebra, although we will show an additional assumption not mentioned in \cite{costello4} is necessary in the case when $M$ is even dimensional. This structure arises from the quantisation of Chern--Simons field theory. Cattaneo and Mn\"ev \cite{cattaneomnev} studied this construction in more detail by first modelling $\Omega^{\bullet}(M)$ as a \emph{finite dimensional} Frobenius algebra. Our second main aim is to provide the foundations for understanding this circle of ideas from the perspective of modular operads and quantum lifts of homotopy algebras in the hope that it will be more conceptual and will shed light on the nature of this quantum $L_\infty$--algebra structure.

To this end, in \autoref{chap:mc} we recall the theory of Maurer--Cartan elements in differential graded Lie algebras and the construction of the Maurer--Cartan moduli set. We extend this theory to \emph{curved} Lie algebras and we then develop a theory of lifting Maurer--Cartan elements in curved Lie algebras, which provides a general theory for tackling a large class of deformation problems. In particular this provides the necessary theory needed to generalise Hamilton's obstruction theory.

In \autoref{chap:quantum} we begin by reviewing the theory of hyperoperads, which is necessary in order to treat quantum homotopy algebras in maximum generality. We also extend the relevant definitions of Koszul duality for cyclic operads in terms of hyperoperads in order to provide a more elegant and unified presentation. Hyperoperads are somewhat technical and underused, but a full review of the technical issues is not the primary concern of this thesis so the reader unfamiliar with hyperoperads should consult the original paper of Getzler--Kapranov \cite{getzlerkapranov} for a more complete reference.

We then proceed to review and develop the theory of Maurer--Cartan elements in a modular operad and prove some of the key general fundamental results which will be of particular use for us. Then we introduce the definitions of quantum and semi-quantum homotopy $\mathcal{P}$--algebras and examine the lifting problem, relating it to the theory developed in \autoref{chap:mc}. We also examine the structure of tensor products of quantum homotopy $\mathcal{P}$--algebras with $\mathcal{P}^!$--algebras proving what can be understood as a very general version of the well known observation that the tensor product of a commutative algebra with a Lie algebra has the natural structure of a Lie algebra.

We then consider the problem of lifting a cyclic homotopy $\mathcal{P}$--algebra to a quantum homotopy $\mathcal{P}$--algebra in the case that we start with a strict algebra. This admits a particularly simple and pleasant solution and it turns out that the important property required for a strict algebra to lift is \emph{unimodularity}, suitably generalised for algebras over an arbitrary quadratic operad. We also unwrap the explicit obstruction theory governing quantum lifting of $L_\infty$--algebras in the style Hamilton did for $A_\infty$--algebras.

Finally we consider how our theory applies to the work of Costello \cite{costello4} on Chern--Simons field theory and attempt to provide some conceptual insight.

\chapter{Introduction to topological quantum field theories}\label{chap:tqfts}
This chapter provides a brief introduction to the basic axiomatic definition of topological quantum field theories due to Atiyah \cite{atiyah}. In this formulation a closed $n$--dimensional topological quantum field theory is a rule associating to each closed oriented manifold $\Sigma$ of dimension $n-1$ a vector space and to each oriented $n$--manifold with boundary $\Sigma$ a vector in this vector space. One of the clearest formulations of this idea is in categorical terms. First one defines a symmetric monoidal category whose objects are $n-1$--manifolds and whose morphisms are cobordisms, with composition given by gluing along boundaries. Then a topological quantum field theory is a symmetric monoidal functor to vector spaces. An open--closed $n$--dimensional topological quantum field theory can be defined similarly, except the $n-1$--manifold $\Sigma$ is not required to be closed.

The main purpose of this chapter is to explain how to generalise these definitions to obtain a Klein topological quantum field theory, where the $n-1$--manifold $\Sigma$ is not required to be oriented. This was first done for closed Klein topological quantum field theories by Turaev--Turner \cite{turaevturner} and for open--closed theories by Alexeevski--Natanzon \cite{alexeevskinatanzon}.

The main results of this chapter are now mostly regarded as classical. In dimension $2$, by using the classification of surfaces, it is possible to describe explicitly the algebraic structures which determine and are determined by a topological quantum field theory. For example, closed $2$--dimensional topological quantum field theories turn out to be equivalent to commutative Frobenius algebras, in other words a commutative algebra $A$ with a non-degenerate symmetric bilinear form (non-degenerate meaning that it induces an isomorphism $A\cong A^*$) which is invariant with respect to the multiplication: for all $a,b,c\in A$ then $\langle ab, c \rangle = \langle a, bc \rangle$. A gentle introduction and detailed proof of this result can be found in the book by Kock \cite{kock}. Other sorts of $2$--dimensional quantum field theories turn out to be equivalent to similar Frobenius type structures.

Standard examples of Frobenius algebras include matrix algebras and group algebras. A more interesting example is the cohomology of a closed oriented manifold, which is in a natural way a graded Frobenius algebra. Frobenius algebras have shown up in a variety of topological contexts and these results make more precise the idea that Frobenius algebras are essentially topological structures.

Since we will be working in dimension $2$ we restrict our definitions to dimension $2$. It is possible to give definitions in arbitrary dimension easily for closed field theories, but for open field theories it is necessary to mention manifolds with faces in order to glue cobordisms properly. This is a technical concern not of interest to us here. We will first briefly recall the details of oriented topological quantum field theories and then define a Klein topological quantum field theory and recall well known results about dimension $2$ topological field theories and their unoriented analogues.

\section{Oriented topological field theories}
We begin by recalling the classical definitions.

\begin{definition}
We define the category $\cob$ as follows:
\begin{itemize}
\item Objects of $\cob$ are compact oriented $1$--manifolds (disjoint unions of circles and intervals).
\item Morphisms between a pair of objects $\Sigma_0$ and $\Sigma_1$, are oriented cobordisms from $\Sigma_0$ to $\Sigma_1$ up to diffeomorphism. That is a compact, oriented $2$--manifold $M$ together with orientation preserving diffeomorphisms $\Sigma_0\simeq\partial M_{\mathrm{in}}\subset\partial M$ and $\Sigma_1\simeq\overline{\partial M}_{\mathrm{out}}\subset\partial M$ (where $\overline{\partial M}_{\mathrm{out}}$ means $\partial M_{\mathrm{out}}$ with the opposite orientation) with $\partial M_{\mathrm{in}}\cap\partial M_{\mathrm{out}}=\emptyset$. We call $\partial M_{\mathrm{in}}$, $\partial M_{\mathrm{out}}$ and $\partial M_{\mathrm{free}}=\partial M\setminus(\partial M_{\mathrm{in}}\cup\partial M_{\mathrm{out}})$ the \emph{in boundary}, the \emph{out boundary} and the \emph{free boundary} respectively. We say two cobordisms $M$ and $M'$ are diffeomorphic if there is an orientation preserving diffeomorphism $\psi\co M\stackrel{\sim}{\rightarrow}M'$ where the following commutes:
\[\xymatrix{
 & M\ar[dd]_\psi^\simeq\\
\Sigma_0\ar[ur]\ar[dr] & &\Sigma_1\ar[ul]\ar[dl]\\
 & M'
}\]
\item Composition is given by gluing cobordisms together. As mentioned above, care must be taken to ensure that gluing is well defined up to diffeomorphism. In dimension $2$ we know that smooth structure depends only on the topological structure of our manifold so we will avoid discussing the technical issues. Gluing is associative and the identity morphism from $\Sigma$ to itself is given by the cylinder $\Sigma\times I$.
\end{itemize}
\end{definition}

It can be shown that the category $\cob$ is a symmetric monoidal category with the tensor product operation given by disjoint union of manifolds.

\begin{definition}
An \emph{open--closed topological field theory} is a symmetric monoidal functor $\cob\rightarrow\vect$, where $\vect$ is the category of vector spaces over the field $k$.
\end{definition}

We can now consider open and closed theories separately by restricting to the appropriate subcategory.

\begin{definition}
\Needspace*{3\baselineskip}\mbox{}
\begin{itemize}
\item The category $\cobcl$ is the (symmetric monoidal) subcategory of $\cob$ with objects closed oriented $1$--manifolds (disjoint unions of circles) and morphisms with empty free boundary. A \emph{closed topological field theory} is a symmetric monoidal functor to $\vect$.
\item The category $\cobo$ is the full (symmetric monoidal) subcategory of $\cob$ with those objects which are not in $\cobcl$ (disjoint unions of intervals). An \emph{open topological field theory} is a symmetric monoidal functor to $\vect$.
\end{itemize}
\end{definition}

We then have the following classical results:

\begin{proposition}\label{prop:ctft}
Closed topological field theories of dimension $2$ are equivalent to commutative Frobenius algebras (see for example the book by Kock \cite{kock}).
\end{proposition}

\begin{proposition}\label{prop:otft}
Open topological field theories of dimension $2$ are equivalent to symmetric Frobenius algebras (in other words not necessarily commutative but the bilinear form is symmetric, see Moore \cite{moore2}, Segal \cite{segal} or Chuang and Lazarev \cite{chuanglazarev}).
\end{proposition}

\begin{proposition}\label{prop:octft}
Open--closed topological field theories of dimension $2$ are equivalent to `knowledgeable Frobenius algebras' (see Lauda and Pfeiffer \cite{laudapfeiffer} for definitions and proof or also Lazaroiu \cite{lazaroiu} and Moore \cite{moore1}).
\end{proposition}

\section{Klein topological field theories}
To extend to the unorientable case we suppress all mentions of orientations. This leads to the following definition:

\begin{definition}
We define the category $\kcob$ as follows:
\begin{itemize}
\item Objects of $\kcob$ are compact $1$--manifolds (disjoint unions of circles and intervals).
\item Morphisms between a pair of objects $\Sigma_0$ and $\Sigma_1$, are (not necessarily orientable) cobordisms from $\Sigma_0$ to $\Sigma_1$ up to diffeomorphism. That is a compact $2$--manifold $M$ together with diffeomorphisms $\Sigma_0\simeq\partial M_{\mathrm{in}}\subset\partial M$ and $\Sigma_1\simeq\partial M_{\mathrm{out}}\subset\partial M$ with $\partial M_{\mathrm{in}}\cap\partial M_{\mathrm{out}}=\emptyset$. We say two cobordisms $M$ and $M'$ are diffeomorphic if there is a diffeomorphism $\psi\co M\stackrel{\sim}{\rightarrow}M'$ where the following commutes:
\[\xymatrix{
 & M\ar[dd]_\psi^\simeq\\
\Sigma_0\ar[ur]\ar[dr] & &\Sigma_1\ar[ul]\ar[dl]\\
 & M'
}\]
\item Composition is given by gluing cobordisms together. The identity morphism from $\Sigma$ to itself is given by the cylinder $\Sigma\times I$.
\end{itemize}
\end{definition}

As in the orientable case $\kcob$ is a symmetric monoidal category by disjoint union of manifolds.

It is convenient to identify $\cob$ and $\kcob$ with their skeletons. Recall that since all oriented circles are isomorphic (since $S^1$ is diffeomorphic to itself with the opposite orientation) the skeleton of $\cob$ is the full subcategory with objects disjoint unions of copies of a single oriented $S^1$ (so the set of objects can be identified with the natural numbers). Similarly the skeleton of $\kcob$ is the full subcategory with objects disjoint unions of copies of a single unoriented $S^1$ (so again the set of objects can be identified with the natural numbers). In this way we can think of $\cob$ as a subcategory of $\kcob$ by forgetting orientations. Note that even if the underlying manifold $M$ of a cobordism in $\kcob$ is orientable the cobordism itself is not necessarily in $\cob$, since it may not be possible to choose an orientation of $M$ such that the embeddings $\Sigma_0\hookrightarrow\partial M\hookleftarrow\Sigma_1$ are orientation preserving. Consider for example:
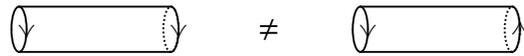
\[
\begin{xy}
*{\cylinder{\cstart}{\cstop}},
(-9,-1)*\dir2{<},(11,-1)*\dir2{<}
\end{xy}
\qquad\neq\qquad
\begin{xy}
*{\cylinder{\cstart}{\cstop}},
(-9,-1)*\dir2{<},(11,1)*\dir2{>}
\end{xy}
\]
The cobordisms above are both morphisms from $S^1$ to itself (where the arrows denote the directions of the embeddings of $S^1$). However while the cobordism on the left is the identity morphism, the cobordism on the right is in $\kcob$ but not in $\cob$.

\begin{definition}
An \emph{open--closed Klein topological field theory} is a symmetric monoidal functor $\kcob\rightarrow\vect$.
\end{definition}

\begin{definition}
\Needspace*{3\baselineskip}\mbox{}
\begin{itemize}
\item The category $\kcobcl$ is the (symmetric monoidal) subcategory of $\kcob$ with objects closed $1$--manifolds without boundary (disjoint unions of unoriented circles) and morphisms with empty free boundary. A \emph{closed Klein topological field theory} is a symmetric monoidal functor to $\vect$.
\item The category $\kcobo$ is the full (symmetric monoidal) subcategory of $\kcob$ with those objects which are not in $\kcobcl$ (disjoint unions of intervals). An \emph{open Klein topological field theory} is a symmetric monoidal functor to $\vect$.
\end{itemize}
\end{definition}

We then have analogues of \autoref{prop:ctft}, \autoref{prop:otft} and \autoref{prop:octft}.

\begin{proposition}\label{prop:ocktft}
Open--closed Klein topological field theories of dimension $2$ are equivalent to `structure algebras' (see Alexeevski and Natanzon \cite{alexeevskinatanzon} for a definition and a proof).
\end{proposition}

In particular we can immediately deduce from the above result proved in \cite{alexeevskinatanzon}, by setting the open part of a structure algebra to $0$, the result for closed KTFTs. It is also proved separately by Turaev and Turner \cite{turaevturner}.

\begin{proposition}\label{prop:cktft}
Closed Klein topological field theories of dimension $2$ are equivalent to the following structures.
\begin{itemize}
\item A commutative Frobenius algebra $A$ with an involutive anti-automorphism\footnote{Since $A$ is commutative an anti-automorphism is of course just an automorphism. Here however it is best thought of as an anti-automorphism on an algebra that just happens to be commutative for comparison with open KTFTs.} $x\mapsto x^*$ preserving the pairing. That is, $(x^*)^*=x$, $(xy)^*=y^*x^*$ and $\langle x^*,y^*\rangle=\langle x,y \rangle$.
\item There is an element $U\in A$ such that $(aU)^*=aU$ for any $a\in A$ and $U^2=\sum\alpha_i\beta_i^*$, where the copairing $\Delta\co k\rightarrow A\otimes A$ is given by $\Delta(1)=\sum\alpha_i\otimes\beta_i$.
\end{itemize}
\end{proposition}

We will not reproduce a proof of \autoref{prop:cktft}, however we will now briefly recall with pictures where each part of the structure comes from. In pictures of cobordisms we denote a crosscap attached to a surface by a dotted circle with a cross. So for example the following is an unorientable cobordism with an underlying surface made with $1$ handle, $1$ crosscap and $5$ holes:
\[
\begin{xy}
*{\cobord{(0,6),(0,18)}{(20,0),(20,12),(20,24)}{\cstart}{\cstop}{(10,8)}{@i}{(10,16)}},
(-9,5)*\dir2{<},(11,1)*\dir2{>},
(-9,-7)*\dir2{<},(11,11)*\dir2{<},
(11,-13)*\dir2{<}
\end{xy}
\]
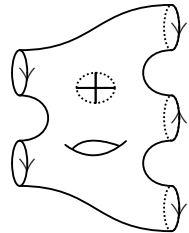
\autoref{fig:2cob} shows the generators of the orientable part of $\kcobcl$.

\begin{figure}[ht!]
\centering
\begin{gather*}
\begin{xy}
*{\cylinder{\cstart}{\cstop}},
(-9,-1)*\dir2{<},(11,-1)*\dir2{<}
\end{xy}\qquad
\begin{xy}
*{\pants{\cstart}{\cstop}},
(-9,5)*\dir2{<},(11,-1)*\dir2{<},(-9,-7)*\dir2{<}
\end{xy}\qquad
\begin{xy}
*{\copants{\cstart}{\cstop}},
(-9,-1)*\dir2{<},(11,5)*\dir2{<},(11,-7)*\dir2{<}
\end{xy}
\\
\begin{xy}
*{\birth{\cstop}},
(5.5,-1)*\dir2{<}
\end{xy}\qquad
\begin{xy}
*{\death{\cstart}},
(-3.5,-1)*\dir2{<}
\end{xy}\qquad
\begin{xy}
*{\twist{\cstart}{\cstop}},
(-9,-7)*\dir2{<},(-9,5)*\dir2{<},(11,5)*\dir2{<},(11,-7)*\dir2{<}
\end{xy}
\end{gather*}
\caption{Generators of $\cobcl$ (considered as a subcategory of $\kcobcl$)}
\label{fig:2cob}
\end{figure}

\begin{figure}[ht!]
\centering
\[
\begin{xy}
*{\cylinder{\cstart}{\cstop}},
(-9,-1)*\dir2{<},(11,1)*\dir2{>}
\end{xy}\qquad
\begin{xy}
*{\rp{\cstop}},
(10.25,-1)*\dir2{<}
\end{xy}
\]
\caption{Additional generators of $\kcobcl$ not in $\cobcl$}
\label{fig:2kcob}
\end{figure}

By moving crosscaps and flipping orientations of boundaries we can decompose any cobordism into an orientable cobordism composed with copies of the two cobordisms in \autoref{fig:2kcob}. For example we can decompose our previous example as:
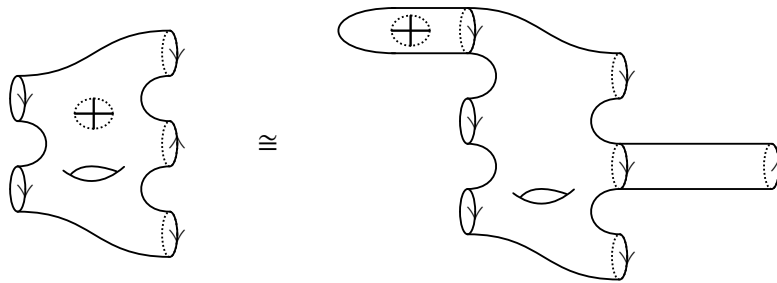
\[
\begin{xy}
*{\cobord{(0,6),(0,18)}{(20,0),(20,12),(20,24)}{\cstart}{\cstop}{(10,8)}{@i}{(10,16)}},
(-9,5)*\dir2{<},(11,1)*\dir2{>},
(-9,-7)*\dir2{<},(11,11)*\dir2{<},
(11,-13)*\dir2{<}
\end{xy}\qquad \cong \quad
\begin{xy}
*{\cobord{(0,6),(0,18),(0,30)}{(20,0),(20,12),(20,24)}{}{\cstop}{(10,8)}{@i}{@i}},
(-9.5,2)*\dir2{<},(10.5,-4)*\dir2{<},
(-9.5,-10)*\dir2{<},(10.5,8)*\dir2{<},
(10.5,-16)*\dir2{<},(-9.5,14)*\dir2{<},
(-19.75,15)*{\rp{\cstop}},
(20,-3)*{\cylinder{}{\cstop}},(30.5,-2)*\dir2{>},
(-10.5,-9)*{\cstart},
(-10.5,3)*{\cstart}
\end{xy}
\]

This shows us that the cobordisms in \autoref{fig:2cob} and \autoref{fig:2kcob} together generate $\kcobcl$. In particular we see that a closed KTFT is given by a commutative Frobenius algebra $A$ together with a linear map corresponding to the cobordism on the left in \autoref{fig:2kcob} which is clearly an involution and an element $U\in A$ given by the image of $1\in k$ under the map corresponding to the cobordism on the right. That the involution is an anti-automorphism corresponds to the relation
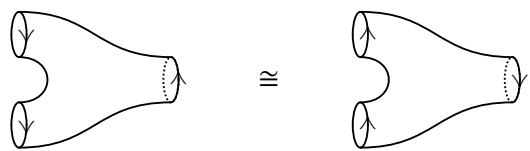
\[
\begin{xy}
*{\pants{\cstart}{\cstop}},
(-9,5)*\dir2{<},(11,1)*\dir2{>},(-9,-7)*\dir2{<}
\end{xy}\qquad\cong\qquad
\begin{xy}
*{\pants{\cstart}{\cstop}},
(-9,7)*\dir2{>},(11,-1)*\dir2{<},(-9,-5)*\dir2{>}
\end{xy}
\]
which can be seen by reflecting the cobordism in a suitable horizontal plane.

The relation $U^2=\sum\alpha_i\beta_i^*$ arises from the fact that $2$ crosscaps are diffeomorphic to a Klein bottle with a hole which can be decomposed into orientable surfaces:
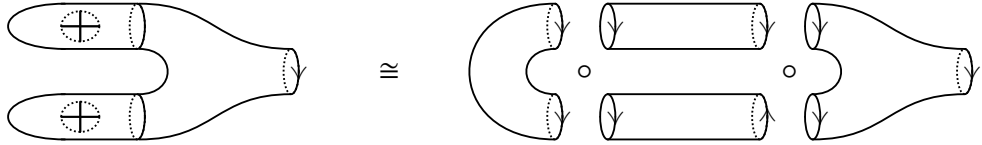
\[
\begin{xy}
*{\pants{}{\cstop}},
(10.5,-1)*\dir2{<},
(-20,6)*{\rp{\cstop}},
(-20,-6)*{\rp{\cstop}}
\end{xy}\qquad\cong\quad
\begin{xy}
*{\copair{\cstop}},
(8,5)*\dir2{<},(8,-7)*\dir2{<}
\end{xy}\circ
\begin{xy}
(0,6)*{\cylinder{\cstart}{\cstop}},
(0,-6)*{\cylinder{\cstart}{\cstop}},
(-9,5)*\dir2{<},(11,5)*\dir2{<},(-9,-7)*\dir2{<},(11,-5)*\dir2{>}
\end{xy}\circ
\begin{xy}
*{\pants{\cstart}{\cstop}},
(-9,5)*\dir2{<},(11,-1)*\dir2{<},(-9,-7)*\dir2{<}
\end{xy}
\]

Finally the relation $(aU)^*=aU$ can be seen by considering a M\"obius strip (which is equivalent to a crosscap) with a hole:

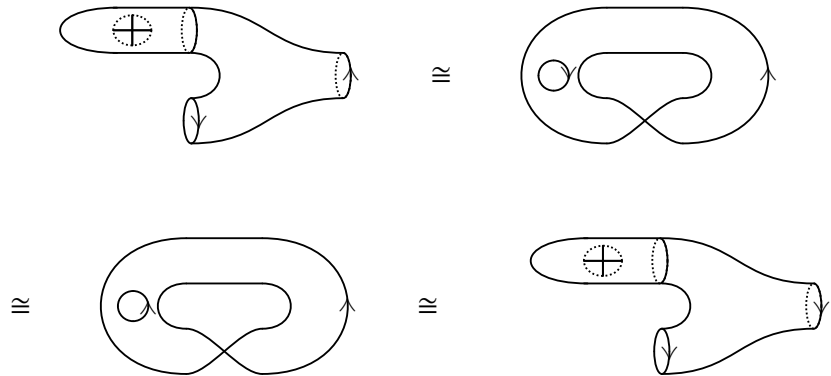
\begin{gather*}
\begin{xy}
*{\pants{}{\cstop}},
(10.5,1)*\dir2{>},
(-20,6)*{\rp{\cstop}},
(-9.5,-7)*\dir2{<},
(-10.5,-6)*{\cstart}
\end{xy}\qquad\cong\quad
\begin{xy}
(-12.5,0)*{\copair{}},
(0,6)*{\shortcylinder{}{}},
(0,-6)*{\flip{}{}},
(12.5,0)*{\pair{}},
(-10,-0.5)*\dir2{<},
(16.2,1)*\dir2{>},
(-12,0)*{\puncture}
\end{xy}\\\\
\cong\quad
\begin{xy}
(-12.5,0)*{\copair{}},
(0,6)*{\shortcylinder{}{}},
(0,-6)*{\flip{}{}},
(12.5,0)*{\pair{}},
(-10,0.5)*\dir2{>},
(16.2,1)*\dir2{>},
(-12,0)*{\puncture}
\end{xy}
\quad\cong\qquad
\begin{xy}
*{\pants{}{\cstop}},
(10.5,-1)*\dir2{<},
(-20,6)*{\rp{\cstop}},
(-9.5,-7)*\dir2{<},
(-10.5,-6)*{\cstart}
\end{xy}
\end{gather*}
Here the second diffeomorphism can be seen by pushing the left hole once around the M\"obius strip (so its orientation changes when it passes through the twist).

It is not too difficult to convince oneself that these relations generate all relations and hence give a sufficient set of relations.

Open KTFTs are however our main object of study. We will prove the following result for open KTFTs later as a corollary of our approach using operads.

\begin{proposition}\label{prop:oktft}
Open Klein topological field theories of dimension $2$ are equivalent to symmetric Frobenius algebras together with an involutive anti-automorphism $x\mapsto x^*$ preserving the pairing.
\end{proposition}

\begin{examples}
\Needspace*{3\baselineskip}\mbox{}
\begin{itemize}
\item Any matrix algebra over a field is a symmetric Frobenius algebra with $\langle A,B \rangle = \operatorname{tr} AB$. With an involution given by the transpose we obtain an open KTFT.
\item Let $G$ be a finite group. Then the group algebra $\mathbb{C}[G]$ is a symmetric Frobenius algebra with bilinear form $\langle a,b \rangle$ given by the coefficient of the identity element in $ab$. Define an involution as the linear extension of $g\mapsto g^{-1}$. This is an open KTFT.
\item If $G$ is abelian, then $\mathbb{C}[G]$ forms a closed KTFT with \mbox{$U=\frac{1}{\sqrt{|G|}}\sum g^2$}.
\end{itemize}
\end{examples}

\chapter{Preliminaries on operads}\label{chap:operads}

In this chapter we shall review the main technical machinery we will use as well as fixing our notation and conventions. We will make extensive use of the language of operads, cyclic operads and modular operads. This chapter is not intended as a complete review of operads but rather a convenient reference as well as a chance to make clear our choice of notation. We will however make some slight alterations to the usual definitions, for example a very mild generalisation of modular operads, extended modular operads. We will mainly be concerned with algebraic operads and as such the relevant references are Ginzburg--Kapranov \cite{ginzburgkapranov} which concerns Koszul duality for operads and Getzler--Kapranov \cite{getzlerkapranov} for modular operads.

\section{Some categories}

Denote by $\topol$ the symmetric monoidal category of topological spaces with the usual product of spaces.

Denote by $\vect$ the symmetric monoidal category of vector spaces over $k$ with the usual tensor product. We denote the set of morphisms from $V$ to $W$ by $\Hom(V,W)$. The category $\vect$ is a symmetric monoidal closed category with internal $\Hom$ given by $\IHom(V,W)=\Hom(V,W)$.

Denote by $\dgvect$ the symmetric monoidal category of differential cohomologically $\mathbb{Z}$--graded $k$--linear vector spaces with symmetry isomorphism $s\co V\otimes W \rightarrow W\otimes V$ given by $s(v\otimes w) = (-1)^{\degree{v}\degree{w}} w \otimes v$. Here $\degree{v}$ and $\degree{w}$ are the degrees of the homogeneous elements $v$ and $w$. We denote the set of morphisms from $V$ to $W$, which are linear maps preserving the grading and the differentials, by $\Hom(V,W)$. This has the structure of a vector space.

Let $\Sigma k$ be the one dimensional vector space concentrated in degree $-1$ with zero differential and let $\Sigma^{-1} k$ be the one dimensional vector space concentrated in degree $1$ with zero differential. We define the suspension functor by $V \mapsto \Sigma V = \Sigma k \otimes V$ and the desuspension functor by $V \mapsto \Sigma^{-1} V = \Sigma^{-1} k \otimes V$. By $\Sigma^n V$ we mean the $n$--fold suspension/desuspension of $V$. Observe that there are natural isomorphisms $\Sigma^n V \otimes \Sigma^m W \cong \Sigma^{n+m} (V \otimes W)$ induced by the maps \[\id^{\otimes n}\otimes s \otimes \id_W \co (\Sigma k)^{\otimes n} \otimes V \otimes (\Sigma k)^{\otimes m} \otimes W \longrightarrow (\Sigma k)^{\otimes n} \otimes (\Sigma k)^{\otimes m} \otimes V \otimes W.\]

Let $V,W$ be differential graded vector spaces. Denote by $\IHom(V,W)^n$ the vector space of homogeneous linear maps of degree $n$ (maps of graded vector spaces $f\co V\rightarrow \Sigma^n W$ preserving the grading but not necessarily the differential). Denote by $\IHom(V,W)=\bigoplus_n\Sigma^{-n}\IHom(V,W)^n$. This is a differential graded vector space with differential given by $dm = d_W\circ m - (-1)^{\degree{m}}m\circ d_V$ where $d_V$ and $d_W$ are the differentials on $V$ and $W$ respectively and $m$ is a homogeneous map of degree $\degree{m}$. Observe that $m$ is a chain map if both $\degree{m}=0$ and $dm=0$. In fact $\dgvect$ is then a symmetric monoidal closed category with internal $\Hom$ given by $\IHom$.

If one wishes to work with homologically graded objects then set $\Sigma k$ to be concentrated in degree $1$ and $\Sigma^{-1} k$ concentrated in degree $-1$. Then the definitions of the suspension and desuspension are the same as above. Similarly we can also work with supergraded objects, in which case $\Sigma=\Sigma^{-1}$. Everything presented will work in any of these gradings unless otherwise specified. When we wish to explicitly emphasise that we are working in the supergraded setting we will call this functor \emph{parity reversion} and denote it by $\Pi$.

From \autoref{chap:dihedral} onwards we will also on occasion consider formal vector spaces. Precisely, a differential formal $\mathbb{Z}$--graded $k$--linear vector space $V$ is an inverse limit of finite dimensional $\mathbb{Z}$--graded $k$--linear vector spaces $V_i$, so that $V=\lim_{\leftarrow} V_i$, equipped with the inverse limit topology and a continuous differential. Morphisms between such spaces are required to be continuous linear maps preserving the grading and the differentials. We denote the set of morphisms from $V$ to $W$ by $\ctsHom(V,W)$. This has the structure of a vector space.

The \emph{completed tensor product} of two formal spaces $V=\lim_{\leftarrow} V_i$ and $W=\lim_{\leftarrow}W_j$ is the formal space $V \otimes W=\lim_{\leftarrow} V_i\otimes W_j$. Denote by $\fdgvect$ the symmetric monoidal category of differential formal $\mathbb{Z}$--graded $k$--linear vector spaces with symmetry isomorphism $s\co V \otimes W\rightarrow W \otimes V$ given by $s(v \otimes w) = (-1)^{\degree{v}\degree{w}} w \otimes v$. We have suspension and desuspension functors defined in the same way as before.

In particular given $V\in\dgvect$ then $V$ is the direct limit of its finite dimensional subspaces so $V=\lim_{\rightarrow}V_i$. We write $V^*\in\fdgvect$ for the space $V^*=\lim_{\leftarrow}\IHom(V_i, k)$. Similarly given $V\in\fdgvect$ with $V=\lim_{\leftarrow}V_i$ we write $V^*\in\dgvect$ for the space $V^*=\lim_{\rightarrow}\IHom(V_i,k)$. With this convention we have $V^{**}\cong V$ and $(V\otimes W)^*\cong V^*\otimes W^*$ for any pair $V$ and $W$ both in either $\dgvect$ or $\fdgvect$. In fact the functor $V\mapsto V^*$ is an anti-equivalence of symmetric monoidal categories. In particular an algebra in the category $\fdgvect$ is the same as a coalgebra in $\dgvect$.

Given $V=\lim_\leftarrow V_i\in\fdgvect$ and $W\in\dgvect$ define $V\otimes W$ to be the space $\lim_\leftarrow V_i\otimes W$, which in general is neither formal nor discrete. However, note that in this case $(V\otimes W)^0$ is the space of linear maps from $V^*$ to $W$ preserving the grading.

For clarity we will write $\Sigma V^*$ to mean $\Sigma (V^*)$. There is a natural isomorphism $(\Sigma V)^* \cong \Sigma^{-1} V^*$.

Note that in general we will not require algebras to be unital unless stated. By an augmented associative or commutative algebra it is meant a unital algebra $A$ equipped with an algebra map $A\rightarrow k$. The augmentation ideal $A_+$ is the kernel of this map.

Given $V\in\dgvect$ we write $TV$ for the free augmented differential graded associative algebra generated by $V$ given explicitly by the tensor algebra
\[
TV = \bigoplus_{n=0}^\infty V^{\otimes n} = k \oplus V \oplus (V\otimes V) \oplus \dots
\]
and we write $SV$ for the free augmented differential graded commutative algebra generated by $V$ given explicitly by the symmetric algebra
\[
SV = \bigoplus_{n=0}^{\infty} (V^{\otimes n})_{S_n}
\]
where the coinvariants are taken with respect to the action of $S_n$ permuting the factors of $V^{\otimes n}$ via the symmetry isomorphism $s$.

Similarly given $V\in\fdgvect$ we write $\widehat{T}V$ for the free formal augmented differential graded associative algebra generated by $V$ given explicitly by the completed tensor algebra
\[
\widehat{T}V = \prod_{n=0}^{\infty} V^{\otimes n} = k \times (V\otimes V) \times \dots
\]
and we write $\widehat{S}V$ for the free formal augmented differential graded commutative algebra generated by $V$ given explicitly by the completed symmetric algebra
\[
\widehat{S}V = \prod_{n=0}^{\infty} (V^{\otimes n})_{S_n}.
\]

These algebras all have an alternative natural grading determined by the order of the tensors and when we wish to refer to this grading we will refer to the order of an element rather than the degree. Furthermore we write $T_{\geq n}V$ (or similar notation for the other algebras) for the subalgebra of elements of order $n$ and above. This gives a decreasing filtration of these algebras.

\section{Trees, graphs, operads and modular operads}
In this section we will outline the notation we will use and recall for convenience the definitions of (modular) operads with some minor modifications. For full details see Ginzburg and Kapranov \cite{ginzburgkapranov} and Getzler and Kapranov \cite{getzlerkapranov}.

We need the notions of graphs and trees. A graph is what we expect but we allow graphs with external half edges (legs). Precisely a graph can be defined as follows:

\begin{definition}
A graph $G$ consists of the following data:
\begin{itemize}
\item Finite sets $\vertices(G)$ and $\halfedges(G)$ with a map $\lambda\co\halfedges(G)\rightarrow\vertices(G)$
\item An involution $\sigma\co\halfedges(G)\rightarrow\halfedges(G)$
\end{itemize}
The set $\vertices(G)$ is the set of \emph{vertices} of $G$ and $\halfedges(G)$ is the set of \emph{half edges} of $G$. A half edge $a$ is connected to a vertex $v$ if $\lambda(a)=v$. We denote the set of half edges connected to $v$ by $\flags(v)$ and we write $n(v)$ for the cardinality of $\flags(v)$ (the \emph{valence} of $v$). Two half edges $a\neq b$ form an \emph{edge} if $\sigma(a)=b$. The set $\edges(G)$ is the set of unordered pairs of half edges forming an edge. We call half edges that are fixed by $\sigma$ the \emph{legs} of $G$ and denote the set of legs as $\legs(G)$.
\end{definition}

\begin{definition}
An isomorphism of graphs $f\co G\rightarrow G'$ consists of bijections $f_1\co\vertices(G)\rightarrow\vertices(G')$ and  $f_2\co\halfedges(G)\rightarrow\halfedges(G')$ satisfying $\lambda\circ f_2 = f_1\circ\lambda$ and $\sigma\circ f_2 = f_2\circ\sigma$.
\end{definition}

Given a graph $G$ we can associate a finite $1$--dimensional cell complex $|G|$ in the obvious way with $0$--cells corresponding to vertices and the ends of legs and $1$--cells corresponding to edges and legs. We say $G$ is \emph{connected} if $|G|$ is connected.

\begin{definition}
By a \emph{tree} we mean a connected graph $T$ with at least $2$ legs such that $\dim{H_1(|T|)}=0$ (equivalently $|T|$ is contractible).
\end{definition}

\begin{definition}
A \emph{labelled graph} is a connected non-empty graph $G$ together with a labelling of the $n$ legs of $G$ by the set $\{1,\dots,n\}$ and a map $g\co \vertices(G)\rightarrow \mathbb{Z}_{\geq 0}$. We call the value $g(v)$ the \emph{genus} of $v$. The genus of a labelled graph $G$ is defined by the formula:
\[g(G) = \dim{H_1(|G|)}+\sum_{v\in\vertices(G)}g(v)\]
Clearly this is the number of loops in the graph obtained by gluing $g(v)$ loops to each vertex $v$ of the underlying graph and contracting all internal edges that are not loops. A vertex of a labelled graph is called \emph{stable} if $2g(v)+n(v)>2$. A labelled graph is called stable if all its vertices are stable. An extended stable graph is defined in the same way except a vertex is called extended stable if $2g(v)+n(v) \geq 2$. An isomorphism of labelled graphs is an isomorphism of graphs preserving the label of each leg and the genus of each vertex.
\end{definition}

\begin{definition}
 By a \emph{labelled tree} we mean a tree $T$ with $n+1\geq 2$ legs with a labelling of the legs by the set $\{1,\dots,n+1\}$. Given such a labelled tree we call the leg labelled by $n+1$ the output or root of $T$ and the other legs the inputs of $T$, denoted $\inputs(T)$. This induces  a direction on the tree where each half edge is directed towards the output and given a vertex $v$ we write $\inputs(v)\subset\flags(v)$ for the set of $n(v)-1$ incoming half edges at $v$. Note that $n(v)\geq 2$ for all vertices $v$. We call $v$ \emph{reduced} if $n(v)>2$. We call a labelled tree reduced if all its vertices are reduced. An isomorphism of labelled trees is an isomorphism of trees that preserves the labelling. 

We denote by $\edges^+(T)=\edges(T)\cup\inputs(T)$ the set of internal edges together with the inputs of $T$.
\end{definition}

\begin{remark}
Note that a labelled tree is equivalent to an extended stable graph of genus $0$ (by assigning a genus of $0$ to each vertex). Reduced trees can then be thought of as stable graphs of genus $0$. We use the term `reduced' as opposed to `stable' here to emphasise the fact that we do not consider the vertices as having a genus.
\end{remark}

Given a labelled graph $G$ we denote by $G/e$ the labelled graph obtained by contracting the internal edge $e$. The genus of each of the vertices of $G/e$ is defined in the natural way, so that the overall genus of the graph remains constant. More precisely, if we contract an edge $e$ connected to two different vertices $v_1$ and $v_2$ into a single vertex $v$ then we set $g(v)=g(v_1)+g(v_2)$. If we contract an edge $e$ connected to a single vertex $v$ (so $e$ is a loop) then the genus of $v$ increases by one.

Observe that if we contract multiple edges it does not matter (up to isomorphism) in which order we contract them. We write $\Gamma((g,n))$ for the category of extended stable graphs of genus $g$ with $n$ legs with morphisms generated by isomorphisms of labelled graphs and edge contractions. For a labelled tree $T$ we define $T/e$ similarly. We denote by $T((n))$ the category of trees with $n$ legs. By an $n$--tree we mean a tree with $n$ inputs (equivalently $n+1$ legs). We denote by $T(n)$ the category of $n$--trees. Note that $T(n)$ is isomorphic to $T((n+1))$ and $\Gamma((0,n+1))$.

We can glue graphs with legs. If $G'$ has $n>0$ legs and $G$ has $m>0$ legs then we write $G\circ_i G'$ for the graph obtained by gluing the leg of $G'$ labelled by $n$ to the leg of $G$ labelled by $i$. For trees this corresponds to gluing the output of one tree to the $i$--th input of the other.

Fix $\mathcal{C}$ to be one of the symmetric monoidal categories $\vect$, $\dgvect$ or $\topol$.

\begin{definition}
\Needspace*{3\baselineskip}\mbox{}
\begin{itemize}
\item An \emph{$S$--module} is a collection $V=\{V(n) : n\geq 1\}$ with $V(n)\in\ob\mathcal{C}$ equipped with a left action of $S_n$ (the symmetric group on $n$ elements) on $V(n)$.
\item A \emph{cyclic $S$--module} is a collection $U=\{U((n)) : n\geq 2\}$ with $U((n))\in\ob\mathcal{C}$ equipped with a left action of $S_n$ on $U((n))$.
\item An \emph{extended stable $S$--module}\footnote{This differs slightly from the definition in \cite{chuanglazarev} since we also allow the pair $(g,n)=(1,0)$. This makes very little difference in practice however.} is a collection $W=\{W((g,n)) : n,g \geq 0\}$ with $W((g,n))\in\ob\mathcal{C}$ equipped with a left action of $S_n$ on $W((g,n))$ and where $W((g,n))=0$ whenever $2g+n \leq 1$. We call an extended stable $S$--module a \emph{stable $S$--module} if $W((g,n))=0$ whenever $2g+n \leq 2$.
\end{itemize}
A morphism of (cyclic/extended stable) $S$--modules is given by a collection of $S_n$--equivariant morphisms.
\end{definition}

\begin{remark}\label{rem:cyclic}
Note that a cyclic $S$--module can also be defined as an $S$--module $V$ with an action of $S_{n+1}$ extending the action of $S_n$ on $V(n)$. This can be seen by setting $V((n))=V(n-1)$. Similarly given a cyclic $S$--module $U$, by restricting to the action of $S_{n}\subset S_{n+1}$ on $U(n)=U((n+1))$ we see that a cyclic $S$--module has an underlying $S$--module.
\end{remark}

Given an $S$--module $V$ and a finite set $I$ with $n$ elements we define
\[
V(I)=\left(\bigoplus_{f\in\iso([n],I)}V(n)\right)_{S_n}
\]
the coinvariants with respect to the simultaneous action of $S_n$ on $\iso([n],I)$ and $V(n)$ (where $[n]=\{1,\dots,n\}$). Similarly given a cyclic $S$--module $U$ we define
\[
U((I))=\left(\bigoplus_{f\in\iso([n],I)}U((n))\right)_{S_n}
\]
and given an extended stable $S$--module $W$ we define:
\[
W((g,I))=\left(\bigoplus_{f\in\iso([n],I)}W((g,n))\right)_{S_n}
\]

\begin{remark}
For simplicity we have used direct sums above since we shall normally be working in the category of (differential graded) vector spaces. More generally one should use coproducts so, for example, in the case that $V$ is an $S$--module in $\topol$ direct sums in the above definitions are replaced by disjoint unions.
\end{remark}

If $T$ is a labelled tree and $V$ is an $S$--module then we define the space of $V$--decorations on $T$ as:
\[
V(T) = \bigotimes_{v\in\vertices(T)}V(\inputs(v))
\]
Similarly for $U$ a cyclic $S$--module the space of $U$--decorations on $T$ is
\[
U((T)) = \bigotimes_{v\in\vertices(T)}U((\flags(v)))
\]
and for $W$ an extended stable module and $G$ an extended stable graph we define the space of $W$--decorations on $G$ as:
\[
W((G)) = \bigotimes_{v\in\vertices(G)}W((g(v),\flags(v)))
\]

Given an isomorphism of labelled graphs $G\rightarrow G'$ or labelled trees $T\rightarrow T'$ there are induced isomorphisms on the corresponding spaces of decorations. 

Note that if $W$ is a stable $S$--module, then $W((G))=0$ unless $G$ is also stable.

\begin{definition}
We define an endofunctor $\mathbb{O}$ on the category of $S$--modules by the formula:
\[
\mathbb{O}V(n)=\operatorname*{colim}_{T\in\iso T(n)}V(T)
\]

We define an endofunctor $\mathbb{C}$ on the category of cyclic $S$--modules by the formula:
\[
\mathbb{C}U((n))=\operatorname*{colim}_{T\in\iso T((n))}U((T))
\]

We define an endofunctor $\mathbb{M}$ on the category of extended stable $S$--modules by the formula:
\[
\mathbb{M}W((g,n))=\operatorname*{colim}_{G\in\iso\Gamma((g,n))}W((G))
\]

Each of these endofunctors can be given the structure of a monad (triple) in the natural way as shown by Getzler and Kapranov \cite{getzlerkapranov}. We call an algebra over these monads an operad, a cyclic operad and an extended modular operad respectively.  A modular operad is an extended modular operad whose underlying $S$--module is stable.
\end{definition}

\begin{convention}
We use the term `extended modular operad' to bring our definitions closer to \cite{getzlerkapranov,chuanglazarev}. However we are not concerned with the distinction between a modular operad and an extended modular operad. Therefore we will from now on use the term `modular operad' to mean extended modular operad unless explicitly stated otherwise.
\end{convention}

\begin{remark}\label{rem:classicaloperaddefs}
We can unpack these somewhat technical definitions to gain more concrete descriptions closer to the classical definition of operads.
\begin{itemize}
\item An operad is an $S$--module $\mathcal{P}$ together with composition maps $\circ_i\co\mathcal{P}(n)\otimes\mathcal{P}(m)\rightarrow\mathcal{P}(n+m-1)$ for $n,m\geq 1$, $1\leq i\leq n$. These maps must satisfy equivariance and associativity conditions.
\item A cyclic operad is a cyclic $S$--module $\mathcal{Q}$ together with composition maps $\circ_i\co\mathcal{Q}((n))\otimes\mathcal{Q}((m))\rightarrow\mathcal{Q}((n+m-2))$ for $n,m\geq 2$, $1\leq i\leq n$. These maps must satisfy equivariance and associativity conditions.
\item A modular operad is an extended stable $S$--module $\mathcal{O}$ together with composition maps $\circ_i\co\mathcal{O}((g,n))\otimes\mathcal{O}((g',m))\rightarrow\mathcal{O}((g+g',n+m-2))$ for  $n,m\geq 1$, $1\leq i \leq n$ and contraction maps $\xi_{ij}\co\mathcal{O}((g,n))\rightarrow\mathcal{O}((g+1,n-2))$ for $n\geq 2$, $1\leq i\neq j \leq n$. These maps must satisfy equivariance and associativity conditions.
\end{itemize}
\end{remark}

We can understand the associativity and equivariance conditions mentioned in \autoref{rem:classicaloperaddefs} in a simple way using trees and graphs as in \cite{getzlerkapranov,ginzburgkapranov}. Given a tree $T$ with a vertex $v$ with $n(v)=n$ and an $S$--module $V$ we observe that choosing a particular direct summand representing $V(\inputs(v))$ is equivalent to choosing a labelling of $\inputs(v)$ by the set $[n-1]$. Similarly given a cyclic $S$--module $U$ choosing a particular direct summand representing $U((\flags(v)))$ is equivalent to choosing a labelling of $\flags(v)$ by $[n]$. Given an extended stable graph $G$ with a vertex $v$ and an extended stable module $W$ choosing a particular direct summand representing $W((g(v),\flags(v)))$ is equivalent to choosing a labelling of $\flags(v)$ by $[n]$.

By choosing appropriate labellings of $\inputs(v)$ or $\flags(v)$ at two vertices connected by the edge $e$ in a tree $T$ we can use the composition maps of an operad $\mathcal{P}$ or a cyclic operad $\mathcal{Q}$ to define a map $\mathcal{P}(T)\rightarrow\mathcal{P}(T/e)$ or $\mathcal{Q}((T))\rightarrow\mathcal{Q}((T/e))$ in the obvious way by considering $\circ_i$ as gluing the output (labelled by $n(v)$ in the cyclic case) at one vertex to the $i$--th input/leg (the leg labelled by $i$) at the other vertex. The equivariance condition simply says that this is well defined regardless of the particular labellings (choice of direct summands) we choose. The associativity condition corresponds to these maps assembling to a well defined functor on $T(n)$ or $T((n))$. Precisely this simply means that no matter in which order we contract the edges of a tree $T$, the induced map on $\mathcal{P}(T)$ or $\mathcal{Q}((T))$ is the same.

In the case of modular operads the same applies but since we are using graphs the edge $e$ could be a loop at a vertex $v$ and then we must use the contraction maps, considering $\xi_{ij}$ as gluing together the half edges making up $e$ labelled by $i$ and $j$ to define a map $\mathcal{O}((G))\rightarrow\mathcal{O}((G/e))$.

\begin{definition}
A \emph{unital} operad is an operad $\mathcal{P}$ with an element $1\in\mathcal{P}(1)$ such that $1\circ_1 a=a=a\circ_i 1$ for any $a\in\mathcal{P}(n)$ with $n\geq 1$ and $1\leq i\leq n$.
\end{definition}

For completeness we note the following lemma/alternative definition which follows from considering \autoref{rem:cyclic}. This allows us to ask whether an operad can be given the additional structure of a cyclic operad.

\begin{lemma}\label{lem:cyclic}
A cyclic operad is a cyclic $S$--module $\mathcal{Q}$ whose underlying $S$--module has the structure of an operad such that $(a\circ_m b)^* = b^*\circ_1 a^*$ for any $a\in\mathcal{Q}(m), b\in\mathcal{Q}(n)$ where $c^*$ is the result of applying the cycle $(1\quad 2\,\dots\, n+1)\in S_{n+1}$  to $c\in \mathcal{Q}(n)=\mathcal{Q}((n+1))$
\end{lemma}

There is clearly a functor from cyclic operads to operads. Given a modular operad $\mathcal{O}$, the genus $0$ part consisting of the spaces $\mathcal{O}((0,n))$ forms a cyclic operad. This gives a functor from modular operads to cyclic operads. If $\mathcal{Q}$ is a cyclic operad then the \emph{modular closure}\footnote{This is also sometimes called the modular envelope and denoted $\mathbf{Mod}(\mathcal{Q})$ as in \cite{costello1}.} $\modc{\mathcal{Q}}$ is the left adjoint functor to this functor and the \emph{na\"ive closure} $\underline{\mathcal{Q}}$ is the right adjoint.

The modular closure is obtained from $\mathcal{Q}$ by freely adjoining the contraction maps and imposing only those relations necessary for associativity and equivariance to still hold. The na\"ive closure is obtained by setting all contraction maps to zero.

\begin{definition}
Let $\mathcal{C}$ be one of $\vect$ or $\dgvect$ and let $V\in\ob\mathcal{C}$. The \emph{endomorphism operad} of $V$, denoted $\End[V]$, is defined as having underlying $S$--module given by setting $\End[V](n)=\IHom(V^{\otimes n},V)$ with the natural action of $S_n$. Composition maps are given by composing morphisms in the obvious way.

Now assume we have a symmetric non-degenerate bilinear form $\langle -, - \rangle\co V\otimes V\rightarrow k$. We define the \emph{endomorphism cyclic operad} of $V$ as having underlying cyclic $S$--module $\mathcal{E}[V]((n))=V^{\otimes n}$ with the natural action of $S_n$. If $a\in V^{\otimes n}$ and $b\in V^{\otimes m}$ then $a\circ_i b\in V^{\otimes(n+m-2)}$ is defined by contracting $a\otimes b$ with the bilinear form, applied to the $i$--th factor of $a$ and the $m$--th factor of $b$. Using the isomorphism $V^{\otimes(n+1)}\cong\IHom(V^{\otimes n},V)$ we see the underlying operad of the endomorphism cyclic operad is just the endomorphism operad. We define the \emph{endomorphism modular operad} as having underlying $S$--module $\mathcal{E}[V]((g,n))=V^{\otimes n}$ with composition maps defined as for the endomorphism cyclic operad and for $a\in\mathcal{E}[V]((g,n))$ we define $\xi_{ij}(a)\in\mathcal{E}[V]((g+1,n-2))$ by contracting the $i$--th factor and the $j$--th factor of $a$ using the bilinear form.
\end{definition}

\begin{definition}
Given an operad $\mathcal{P}$ in $\vect$ or $\dgvect$ an algebra over $\mathcal{P}$ is a vector space/differential graded vector space $V$ together with a morphism of operads $\mathcal{P}\rightarrow\End[V]$. Similarly an algebra over a cyclic/modular operad $\mathcal{O}$ is a vector space/differential graded vector space $V$ with a symmetric non-degenerate bilinear form $B$, together with a morphism of cyclic/modular operads $\mathcal{O}\rightarrow\mathcal{E}[V]$.
\end{definition}

Clearly algebras over various types of operads can be given by a collection of maps in $\IHom(V^{\otimes n},V)$ satisfying certain conditions.

\begin{remark}
We will refer to operads in the category $\dgvect$ as dg operads and operads in the category $\topol$ as topological operads. Obviously by considering a vector space as concentrated in degree $0$ with the zero differential we can consider $\vect$ as a subcategory of $\dgvect$ and hence an operad in $\vect$ can be considered as a dg operad.
\end{remark}

\section{Koszul duality for operads}
We now restrict ourselves to operads and recall the theory of Koszul duality from \cite{ginzburgkapranov}. Let $k$ be a field and let $K$ be an associative unital $k$--algebra. All our operads in this section are required to be unital.

\begin{definition}
A \emph{$K$--collection} is a collection $E=\{E(n) : n\geq 2\}$ of $k$--vector spaces equipped with the following structures:
\begin{itemize}
\item A left $S_n$ action on $E(n)$ for each $n\geq 2$
\item A $(K,K^{\otimes n})$--bimodule structure on $E(n)$ that is compatible with the $S_n$ action. This means for any $\sigma\in S_n$, $\mu,\lambda_i\in K$ and $a\in E(n)$ we have $\sigma(\mu a) = \mu\sigma(a)$ and
\[
\sigma(a(\lambda_1\otimes\ldots\otimes\lambda_n)) = \sigma(a)(\lambda_{\sigma(1)}\otimes\ldots\otimes\lambda_{\sigma(n)})
\]
\end{itemize}
\end{definition}

By setting $E(1)=K$, a $K$--collection should be thought of as an $S$--module $E$ together with composition maps $\circ_i\co E(n)\otimes E(1)\rightarrow E(n)$ and $\circ_1\co E(1)\otimes E(n)\rightarrow E(n)$ satisfying associativity and equivariance conditions. A morphism of $K$--collections is then a morphism of the underlying $S$--modules that preserve these composition maps.

Given a reduced tree $T$ and a $K$--collection $E$ we define
\[
E(T)=\bigotimes_{v\in \vertices(T)}E(\inputs(v))
\]
where the tensor product is taken over $K$ using the $(K,K^{\otimes\inputs(v)})$--bimodule structure on each $E(\inputs(v))$. Given an isomorphism of trees $T\rightarrow T'$ we have an induced isomorphism $E(T)\rightarrow E(T')$.

Clearly if $\mathcal{P}$ is an operad with $\mathcal{P}(1)=K$ then $\{\mathcal{P}(n) : n\geq 2\}$ is a $K$--collection. Given a $K$--collection $E$ we can form the free operad $F(E)$ consisting of $E$--decorated reduced trees with composition given by gluing trees. More precisely, denoting the category of reduced $n$--trees by $\widehat{T}(n)$, we set
\[
F(E)(n) = \operatorname*{colim}_{T\in\iso \widehat{T}(n)}E(T)
\]
and compositions are induced by the natural maps
\[\circ_i\co E(T)\otimes E(T')\rightarrow E(T)\otimes_K E(T')\cong E(T\circ_i T')\]
where the tensor product over $K$ is using the right $K$--module structure on $E(T)$ corresponding to the $i$--th input.

Let $K$ be semisimple and let $E$ be a finite dimensional $K$--collection with $E(n)=0$ for $n>2$. We will denote the $(K,K^{\otimes 2})$--bimodule also by $E$. Let $R\subset F(E)(3)$ be an $S_3$--stable $(K,K^{\otimes 3})$--sub-bimodule. Let $(R)$ be the ideal in $F(E)$ generated by $R$. We define an operad $\mathcal{P}(K,E,R)=F(E)/(R)$. An operad of type $\mathcal{P}(K,E,R)$ is called a \emph{quadratic operad}.

\begin{definition}
Given a $(K,K^{\otimes n})$--bimodule $E$ with a compatible $S_n$ action we denote by $E^{*}=\Hom_K(E,K)$ the space of (left) $K$--linear maps. This has the natural structure of a $(K^{\mathrm{op}},(K^{\mathrm{op}})^{\otimes n})$--bimodule with the transposed action of $S_n$. We can also equip it with the transposed action of $S_n$ twisted by the sign representation in which case we denote it $E^{\vee}=\Hom_K(E,K)\otimes\mathrm{sgn}_n$.
\end{definition}

\begin{definition}
Given a quadratic operad $\mathcal{P}(K,E,R)$ we can form a $K^{\mathrm{op}}$--collection from $E^{\vee}$. Observe that $F(E^{\vee})(3)=F(E(3))^{\vee}$. Let $R^{\perp}\subset F(E^{\vee})(3)$ be the orthogonal complement of $R$, which is an $S_3$--stable $(K,K^{\otimes 3})$--sub-bimodule. We define the dual quadratic operad $\mathcal{P}^!$ to be
\[
\mathcal{P}^!=\mathcal{P}(K^{\mathrm{op}},E^{\vee},R^{\perp})
\]
\end{definition}

We next briefly recall the definitions and results on the cobar construction and the dual dg operad, full details of which can be found in Ginzburg and Kapranov \cite{ginzburgkapranov}. Recall that for a dg operad $\mathcal{P}$ the cobar construction is the operad $F(\mathcal{P}^*[-\mathrm{1}])$ with differential coming from the internal differential and the unique differential dual to the composition of $\mathcal{P}$. Here we give the construction explicitly.

Let $V$ be a finite dimensional vector space. We denote by $\Det(V)$ the top exterior power of $V$. Given a tree $T$ we set $\det(T)=\Det(k^{\edges(T)})$\label{subsec:detT} and $\Det(T)=\Det(k^{\edges^+(T)})$. We denote by $|T|$ the number of internal edges of $T$.

Let $\mathcal{P}$ be a dg operad with $\mathcal{P}(n)$ finite dimensional and $K=\mathcal{P}(1)$ a semisimple unital $k$--algebra concentrated in degree $0$. We call such a dg operad \emph{admissible} and denote the category of admissible dg operads by $\dgop(K)$. For $n\geq 2$ we construct complexes $C'(\mathcal{P})(n)^s = 0 $ for $s\leq 0$ and
\[
C'(\mathcal{P})(n)^s = \bigoplus_{\substack{n\text{--trees }T\\|T|=s-1}}\mathcal{P}(T)^*\otimes\det(T)
\]
where the direct sums are over isomorphism classes of reduced trees and $\mathcal{P}(T)$ is defined by considering the underlying dg $K$--collection of $\mathcal{P}$ (and so tensor products are taken over $K$).

To define the differential $\delta$ recall if $T_2=T_1/e$ is obtained by contracting an internal edge $e$, we have a composition map $\gamma_{T_1,T_2}\co\mathcal{P}(T_1)\rightarrow\mathcal{P}(T_2)$. We define $\delta$ on the direct summands by maps
\[
\delta_{T'}\co\mathcal{P}(T')^*\otimes\det(T')\rightarrow\bigoplus_{\substack{n\text{--trees }T\\|T|=i+1}}\mathcal{P}(T)^*\otimes\det(T)
\]
for $T'$ an $n$--tree with $|T'|=i$, with
\[
\delta_{T'}=\bigoplus_{\substack{(T,e)\\T'=T/e}}(\gamma_{T,T'})^*\otimes l_e
\]
and $l_e\co\det(T')\rightarrow\det(T)$ is defined by:
\[
l_e(f_1\wedge\ldots\wedge f_i)=e\wedge f_1\wedge\ldots\wedge f_i
\]

Since $\mathcal{P}$ is a dg operad each term of the complex just defined has an internal differential $d$. This is compatible with $\delta$ and we write $C(\mathcal{P})(n)^{\bullet}$ for the total complex of the double complex. These complexes together form a dg $K^{\mathrm{op}}$--collection $C(\mathcal{P})$. 

\begin{definition}
It can be shown (by comparing to the operad $F(\mathcal{P^*[-\mathrm{1}]})$) that $C(\mathcal{P})$ has a natural structure of a dg operad. We call this operad the \emph{cobar construction} of $\mathcal{P}$.
\end{definition}

Let $T$ and $T'$ be $n$--trees and $m$--trees respectively with $|T|=p$ and $|T'|=q$. Composition can be obtained explicitly using the maps $\circ_i\co(\mathcal{P}(T)\otimes\det(T))\otimes(\mathcal{P}(T')\otimes\det(T'))\rightarrow\mathcal{P}(T\circ_i T')\otimes\det(T\circ_i T')$ given by\label{subsec:cobar}
\[
\begin{split}
(a_1\otimes\ldots\otimes a_{p+1})\otimes(e_1\wedge\ldots\wedge e_p)\circ_i(b_1\otimes\ldots\otimes b_{q+1})(f_1\wedge\ldots\wedge f_q)
\\=
(a_1\otimes\ldots\otimes a_{p+1}\otimes b_1\otimes\ldots\otimes b_{q+1})\otimes(e_1\wedge\ldots\wedge e_p\wedge f_1\wedge\ldots\wedge f_q\wedge e)
\end{split}
\]
where $e$ is the new internal edge formed from gluing the root of $T'$ to the $i$--th input of $T$.

\begin{definition}\label{def:dualdg}
The \emph{dual dg operad} $\dual\mathcal{P}$ is defined as
\[
\dual\mathcal{P}=C(\mathcal{P})\otimes\Lambda
\]
where $\Lambda$ is the determinant operad with $\Lambda(n)=k$ concentrated in degree $1-n$ carrying the sign representation of $S_n$.
\end{definition}

Further, from the definitions, it follows that $\mathcal{P}\mapsto \dual\mathcal{P}$ extends to a contravariant functor $\dual\co\dgop(K)\rightarrow\dgop(K^{\mathrm{op}})$ which takes quasi-isomorphisms to quasi-isomorphisms.

\begin{remark}\label{rem:dualdg}
$\dual\mathcal{P}$ can also be obtained from the cobar construction by shifting the grading by $1-n$, twisting by the sign representation and introducing a sign $(-1)^{(m-1)i-1}$ to the composition $\circ_i\co\dual\mathcal{P}(n)\otimes\dual\mathcal{P}(m)\rightarrow\dual\mathcal{P}(n+m-1)$. If $\mathcal{P}$ is an admissible dg operad concentrated in degree $0$ then the highest non-zero term of $\dual\mathcal{P}$ is in degree $0$ and is given by:
\[
\dual\mathcal{P}(n)^0=\bigoplus_{\substack{n\text{--trees }T\\|T|=n-2}}\mathcal{P}(T)^*\otimes\Det(T)
\]
\end{remark}

To justify the notion of duality we have the following shown by Ginzburg and Kapranov \cite{ginzburgkapranov}:

\begin{theorem}
Let $\mathcal{P}$ be an admissible dg operad. Then there is a canonical quasi-isomorphism $\dual\dual\mathcal{P}\rightarrow\mathcal{P}$.
\end{theorem}

Finally we briefly recall the definition of a Koszul operad. Let $\mathcal{P}=\mathcal{P}(K,E,R)$ be a quadratic operad. As in \autoref{rem:dualdg} for every $n$ we have
\[
\dual\mathcal{P}(n)^0=
\bigoplus_{\substack{\text{binary}\\n\text{--trees }T}}E^{*}(T)\otimes\Det(T)=F(E)(n)^{\vee}=F(E^{\vee})(n)
\]
and so we have a morphism of dg operads $\gamma_{\mathcal{P}}\co\dual\mathcal{P}\rightarrow\mathcal{P}^!$ given in degree $0$ by taking the quotient of $\dual\mathcal{P}(n)^0$ by the relations in $R^{\perp}$. In fact this induces an isomorphism $H^0(\dual\mathcal{P}(n))\rightarrow\mathcal{P}^!(n)$. 

\begin{definition}
We call $\mathcal{P}$ \emph{Koszul} if $\gamma_{\mathcal{P}}$ is a quasi-isomorphism. In other words each $\dual\mathcal{P}(n)$ is exact everywhere but the right end.
\end{definition}

\begin{definition}
If $\mathcal{P}$ is Koszul then a homotopy $\mathcal{P}$--algebra\footnote{More generally, a homotopy $\mathcal{P}$--algebra is an algebra over a cofibrant replacement for $\mathcal{P}$. That $\mathcal{P}$ is Koszul means that $\dual(\mathcal{P}^!)$ is such a cofibrant replacement. For simplicity we take this to be the definition, so as to avoid the need to discuss in any detail the model category structure on $\dgop(K)$.} is an algebra over $\dual(\mathcal{P}^!)$.
\end{definition}

\chapter{M\"obius graphs and unoriented surfaces}\label{chap:mobius}
In this chapter we will introduce and examine in detail the modular operad governing open Klein topological field theories. We shall introduce the important combinatorial notion of a M\"obius graph and show that these play the same role as ribbon graphs do for oriented topological field theories.

The first main result of this chapter is to identify the modular operad governing open Klein topological field theories as the modular closure of the quadratic operad $\mass$, which provides a generators and relations description of this modular operad. Consequently it follows that KTFTs are equivalent to symmetric Frobenius algebras with an involution thus proving \autoref{prop:oktft}.

We will then embark on a detailed examination of the properties of $\mass$. Since $\mass$ governs associative algebras with an involution it is natural to expect it to have the same nice properties as $\ass$ and indeed this will be seen to be true. In particular we'll show it is its own Koszul dual. We'll also spend some time unwrapping the structure of $\dual\mass$ and the modular closure $\modc{\dual\mass}$ in terms of M\"obius graph complexes in preparation for the results of \autoref{chap:moduli}.

\section{M\"obius trees and the operad $\mass$}

\begin{definition}
A \emph{planar tree} is a labelled tree with a cyclic ordering of the half edges at each vertex. An isomorphism of planar trees is an isomorphism of labelled trees that preserves the cyclic ordering at each vertex.
\end{definition}

\begin{definition}
A \emph{M\"obius tree} is a planar tree $T$ with a colouring of the half edges by two colours. Here a `colouring by two colours' means a map $c\co \halfedges(T)\rightarrow\{0,1\}$. An isomorphism of M\"obius trees is an isomorphism of labelled trees such that at each vertex $v$ either
\begin{enumerate}
\item the map preserves the cyclic ordering at $v$ and the colouring of the half edges connected to $v$
\item the map reverses the cyclic ordering at $v$ and reverses the colouring of the half edges connected to $v$ (we refer to this as \emph{reflection at $v$}).
\end{enumerate}
\end{definition}

\begin{remark}
A planar tree can be drawn in the plane with the cyclic ordering at each vertex given by the clockwise ordering. When drawing labelled trees we shall place the output leg unlabelled at the bottom (so the induced direction on the tree is downwards). For example the following two trees are isomorphic as labelled trees, but not as planar trees:
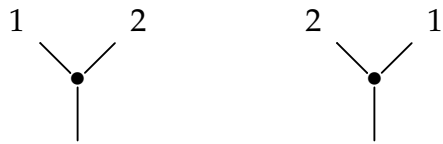
\[
\xygraph{!{<0cm,0cm>;<0.8cm,0cm>:<0cm,0.8cm>::}
1="1"&&2="2" \\
&*{\bullet}="v"\\
&="0"
"v"-"1" "v"-"2" "v"-"0"
}\qquad\qquad
\xygraph{!{<0cm,0cm>;<0.8cm,0cm>:<0cm,0.8cm>::}
2="2"&&1="1" \\
&*{\bullet}="v"\\
&="0"
"v"-"1" "v"-"2" "v"-"0"
}
\]

When drawing M\"obius trees, we shall draw the half edges coloured by $0$ as straight lines and the half edges coloured by $1$ as dotted lines. For example:
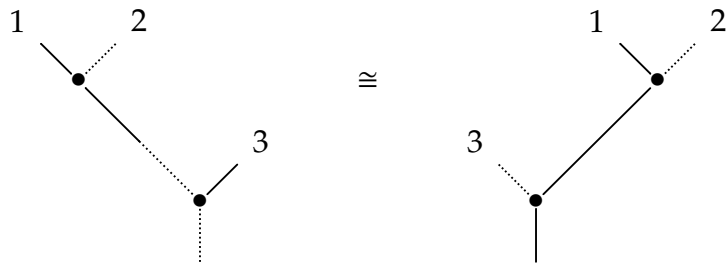
\[
\xygraph{!{<0cm,0cm>;<0.8cm,0cm>:<0cm,0.8cm>::}
1="1"&&2="2"\\
&*{\bullet}="v1"\\
&&="v1v2"&&3="3"\\
&&&*{\bullet}="v2"\\
&&&="0"
"1"-"v1" "2"-@{.}"v1" "v1"-"v1v2" "v1v2"-@{.}"v2" "v2"-"3" "v2"-@{.}"0"
}\qquad\cong\qquad
\xygraph{!{<0cm,0cm>;<0.8cm,0cm>:<0cm,0.8cm>::}
&&1="1"&&2="2"\\
&&&*{\bullet}="v1"\\
3="3"\\
&*{\bullet}="v2"\\
&="0"
"1"-"v1" "2"-@{.}"v1" "v1"-"v2" "v2"-@{.}"3" "v2"-"0"
}
\]
\end{remark}

If $T$ is a planar tree then we can define edge contraction by equipping the vertex in $T/e$ that the edge $e$ contracts to with the obvious cyclic ordering coming from the two cyclic orderings at the vertices either end of $e$, for example:

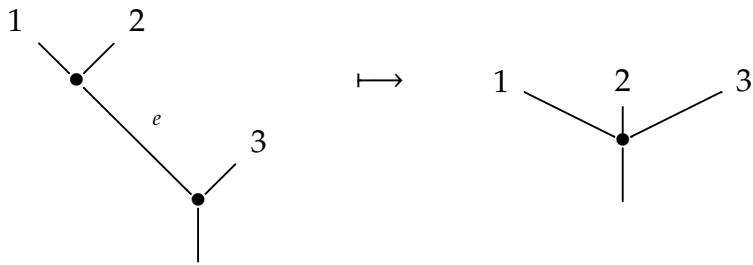
\[\xygraph{!{<0cm,0cm>;<0.8cm,0cm>:<0cm,0.8cm>::}
1="1"&&2="2"\\
&*{\bullet}="v1"\\
&&&&3="3"\\
&&&*{\bullet}="v2"\\
&&&="0"
"1"-"v1" "2"-"v1" "v1"-^e"v2" "v2"-"3" "v2"-"0"
}\qquad\longmapsto\qquad
\xygraph{!{<0cm,0cm>;<0.8cm,0cm>:<0cm,0.8cm>::}
\\1="1"&&2="2"&&3="3"\\
&&*{\bullet}="v"\\
&&="0"
"1"-"v" "2"-"v" "v"-"3" "v"-"0"
}
\]

If $T$ is a M\"obius tree and $e$ is an internal edge where both the half edges of $e$ are coloured the same then we define the tree $T/e$ as for planar trees, with the obvious colouring on $T/e$. If $f\co T'\rightarrow T$ is an isomorphism with $f(e)$ also an edge where both half edges are coloured the same we note that $T/e \cong T'/f(e)$. Therefore this is a well defined operation on isomorphism classes of M\"obius trees. Furthermore for any internal edge of $T$ we can find a tree in the same isomorphism class of $T$ such that this edge is made up of two similarly coloured half edges (by considering, if necessary, a tree with one of the vertices adjacent to the edge reflected). Therefore we have an edge contraction operation on isomorphism classes of M\"obius trees defined for any edge. For example:

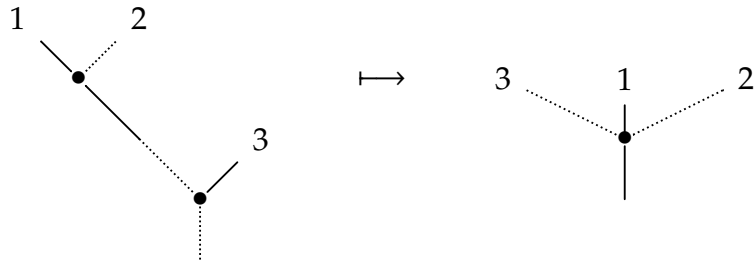
\[\xygraph{!{<0cm,0cm>;<0.8cm,0cm>:<0cm,0.8cm>::}
1="1"&&2="2"\\
&*{\bullet}="v1"\\
&&="v1v2"&&3="3"\\
&&&*{\bullet}="v2"\\
&&&="0"
"1"-"v1" "2"-@{.}"v1" "v1"-"v1v2" "v1v2"-@{.}"v2" "v2"-"3" "v2"-@{.}"0"
}\qquad\longmapsto\qquad
\xygraph{!{<0cm,0cm>;<0.8cm,0cm>:<0cm,0.8cm>::}
\\3="1"&&1="2"&&2="3"\\
&&*{\bullet}="v"\\
&&="0"
"1"-@{.}"v" "2"-"v" "v"-@{.}"3" "v"-"0"
}
\]

Finally we observe that $(T/e)/e'\cong(T/e')/e$ so it does not matter in which order we contract edges.

Now recall that the associative operad $\ass$ can be defined as consisting of the vector spaces generated by planar corollas (isomorphism classes of planar trees with $1$ vertex) where composition is given by gluing such corollas and contracting internal edges. This leads us to the following definition:

\begin{definition}
The operad $\mass$ is defined as follows:
\begin{itemize}
\item $\mass(n)$ is the vector space generated by M\"obius corollas (isomorphism classes of M\"obius trees with $1$ vertex) with $n$ inputs, where $S_n$ acts by relabelling the inputs.
\item Composition maps are given by gluing corollas and contracting the internal edges. These maps satisfy associativity since, as mentioned previously, it does not matter in which order we contract internal edges.
\end{itemize}
\end{definition}

\begin{remark}
It is easy to see that as for $\ass$ the operad $\mass$ can be given the structure of a cyclic operad in the obvious way.
\end{remark}

It is important to note that with these definitions planar trees are $\ass$--decorated trees and M\"obius trees are $\mass$--decorated trees (where we are considering decorations by the underlying $S$--modules), since at each vertex $v$ decorated by a M\"obius corolla (an element of $\mass$) one obtains a cyclic ordering at $v$ from the ordering of the inputs of the M\"obius corolla and a colouring of the half edges of $v$ from the colouring of the corolla. Therefore given a labelled tree $T$, the space of $\mass$--decorations on $T$ is generated by the set of M\"obius trees up to isomorphism whose underlying labelled tree is $T$.

\begin{remark}\label{rem:assinmass}
$\ass$ is a suboperad of $\mass$. In fact $\mass$ is obtained from the operad generated by adjoining an involutive operation to $\ass$ by taking the quotient by the ideal generated by the reflection relation for the binary operation:
\[\xygraph{!{<0cm,0cm>;<0.8cm,0cm>:<0cm,0.8cm>::}
1="1"&&2="2" \\
&*{\bullet}="v"\\
&="0"\\
&*{\circ}="a"\\
&="00"
"v"-"1" "v"-"2" "v"-"0" "0"-"a" "a"-"00"
}\qquad=\qquad
\xygraph{!{<0cm,0cm>;<0.8cm,0cm>:<0cm,0.8cm>::}
2="2"&&1="1" \\
*{\circ}="a1"&&*{\circ}="a2"\\
="22"&&="11"\\
&*{\bullet}="v"\\
&="0"
"v"-"11" "v"-"22" "v"-"0" "11"-"a2" "22"-"a1" "a1"-"2" "a2"-"1"
}\]
Here
$\xygraph{!{;<0.5cm,0cm>:::}
="l"&*{\circ}="a"&="r"
"l"-"a" "a"-"r"}$
denotes the involution.
\end{remark}

\begin{proposition}
$\mass$ is the operad governing (non-unital) associative algebras with an involutive anti-automorphism.
\end{proposition}

\begin{proof}
This follows immediately from \autoref{rem:assinmass}.
\end{proof}

\section{The open KTFT modular operad and M\"obius graphs}
We now define the modular operad governing Klein topological quantum field theories.

\begin{definition}
We define the $k$--linear extended modular operad $\oktft$ (open Klein topological field theory) as follows:
\begin{itemize}
\item For $n,g \geq 0$ and $2g+n \geq 2$ the vector space $\oktft((g,n))$ is generated by diffeomorphism classes of surfaces with $m$ handles, $u$ crosscaps and $h$ boundary components with $2m+h+u-1=g$ and with $n$ disjoint copies of the unit interval embedded in the boundary labelled by $\{1,\ldots,n\}$, with an action of $S_n$ permuting the labels.
\item Composition and contraction is given by gluing along intervals.
\end{itemize}
\end{definition}

\begin{remark}\label{rem:nounit}
Since the connected sum of $3$ crosscaps is diffeomorphic to the connected sum of $1$ handle and $1$ crosscap, the value of $2m+h+u-1$ is well defined regardless of how we choose to represent the topological type of the surface. We note that all classes of surfaces feature in $\oktft$ except the disc with no marked points and the disc with one marked point due to the condition $2g+n\geq 2$.
\end{remark}

We recall that in the case of oriented topological field theories the corresponding modular operads are the modular closures of their genus $0$ part which in turn are identified with the commutative and associative operads: $\ctft\cong\modcom$ and $\otft\cong\modass$ (this formulation in terms of modular operads can be seen in Chuang and Lazarev \cite[Theorem 2.7]{chuanglazarev}). In particular the genus $0$ cyclic part contains all the relations. These results are modular operad versions of \autoref{prop:ctft} and \autoref{prop:otft} identifying $2$--dimensional TFTs as Frobenius algebras. 

The same is true for $\oktft$, giving us the desired simple algebraic description of $\oktft$.

\begin{theorem}\label{thm:oktftmass}
$\oktft\cong\modmass$.
\end{theorem}

Before proving \autoref{thm:oktftmass} we will identify the operad $\modmass$ in terms of graphs. Therefore we need to extend our definitions of M\"obius trees to graphs.

\begin{definition}
A \emph{ribbon graph} is a graph with all vertices having valence at least $2$ equipped with a cyclic ordering of the half edges at each vertex and a labelling of the legs. An isomorphism of ribbon graphs is an isomorphism of graphs that preserves the cyclic ordering at each vertex and the labelling of the legs.
\end{definition}

\begin{definition}
A \emph{M\"obius graph} is a ribbon graph with a colouring of the half edges by two colours. An isomorphism of M\"obius graphs is an isomorphism of graphs preserving the labelling of the legs such that at each vertex $v$ either
\begin{enumerate}
\item the map preserves the cyclic ordering at $v$ and the colouring of the half edges at $v$
\item the map reverses the cyclic ordering at $v$ and reverses the colouring of the half edges connected to $v$ (again we refer to this as reflection at $v$).
\end{enumerate}
\end{definition}

\begin{remark}
Obviously our notions of planar and M\"obius trees correspond to ribbon and M\"obius graphs with no loops. Once again we can draw these graphs in the plane (although possibly with some edges intersecting of course) with the cyclic ordering at each vertex given by the clockwise ordering:
\[
\xygraph{!{<0cm,0cm>;<1.2cm,0cm>:<0cm,1.2cm>::}
!{(0,0)}*{\bullet}="v1"
!{(2,0)}*{\bullet}="v2"
!{(3,0.5)}*+{2}="2"
!{(3,-0.5)}*+{1}="1"
!{(0,-1)}*+{3}="3"
!{(-1,0)}*{}="loop"
!{(1,0.4)}*{}="u"
!{(1,-0.4)}*{}="d"
"v1" -@`{c+(0,-0.5),p+(0,-0.5)}@{.} "loop"
"loop" -@`{c+(0,0.5),p+(0,0.5)} "v1"
"v1"-"3"
"v1" -@`{c+(0.25,0.3),p+(-0.25,0)} "u"
"u" -@`{c+(0.25,0),p+(-0.25,0.3)} "v2"
"v1" -@`{c+(0.25,-0.3),p+(-0.25,0)} "d"
"d" -@`{c+(0.25,0),p+(-0.25,-0.3)}@{.} "v2"
"v2"-@{.} "2"
"v2"-"1"
}
\]
\end{remark}

\begin{remark}\label{rem:decgraphs}
Ribbon and M\"obius graphs are $\underline{\ass}$--decorated graphs and $\underline{\mass}$--decorated graphs respectively.
\end{remark}

If $G$ is a ribbon graph and $e$ is an internal edge of $G$ which is not a loop we can define edge contraction by equipping $G/e$ with the obvious cyclic ordering coming from $G$ as for trees.

If $G$ is a M\"obius graph and $e$ is an internal edge of $G$ that is not a loop, where both the half edges of $e$ are the same colour then we define $G/e$ as we did for trees. Further we observe that as for M\"obius trees this is well defined on isomorphism classes and can be extended to all internal edges except loops regardless of colour.

Let $G$ be a M\"obius or ribbon graph. Given two internal edges $e$ and $e'$ of $G$ that are not loops we have $(G/e)/e'\cong (G/e')/e$ provided both sides are defined. However if $e$ and $e'$ are connected to the same vertices contracting one will make the other into a loop. As a result we do not obtain a well defined operation on graphs by repeatedly contracting edges until we have only one vertex, which we did for trees. See \autoref{fig:graphcontract} for an example.

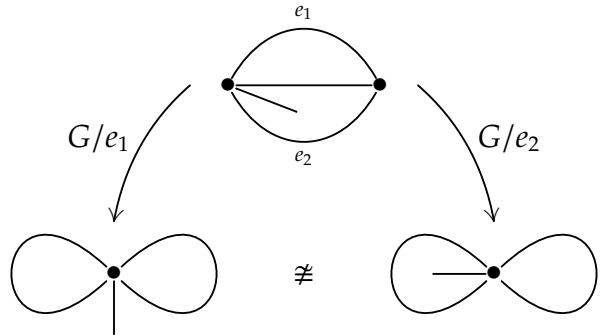
\begin{figure}[ht!]
\centering
\[
\begin{xy}
(0,10)*\xybox{
(0,0)*\xybox{
\xygraph{!{<0cm,0cm>;<1cm,0cm>:<0cm,1cm>::}
!{(0,0)}*{\bullet}="v1"
!{(2,0)}*{\bullet}="v2"
!{(0.9,-0.35)}*{}="o"
"v1" -@`{(0.5,1),(1.5,1)} "v2" ^{e_1}
"v1" -@`{(0.5,-1),(1.5,-1)} "v2" _{e_2}
"v1" - "v2"
"v1" - "o"
}};
(-15,-25)*\xybox{
\xygraph{!{<0cm,0cm>;<1cm,0cm>:<0cm,1cm>::}
!{(0,0)}*{\bullet}="v1"
!{(0,-0.8)}*{}="o"
"v1" -@`{(-1.8,-1.8),(-1.8,+1.8)} "v1" ^{}
"v1" -@`{(+1.8,-1.8),(+1.8,+1.8)} "v1"
"v1" - "o"
}};
(35,-25)*\xybox{
\xygraph{!{<0cm,0cm>;<1cm,0cm>:<0cm,1cm>::}
!{(0,0)}*{\bullet}="v1"
!{(-0.8,0)}*{}="o"
"v1" -@`{(-1.8,-1.8),(-1.8,+1.8)} "v1"
"v1" -@`{(+1.8,-1.8),(+1.8,+1.8)} "v1"
"v1" - "o"
}};
(10,-25)*{\ncong};
(-5,0);(-15,-18)
**\crv{(-13,-7)};?(1)*\dir2{>};
(25,0);(35,-18)
**\crv{(33,-7)}?(1)*\dir2{>};
(-17,-7)*{G/e_1};
(37,-7)*{G/e_2}
}
\end{xy}
\]
\caption{Contracting all edges that are not loops is not well defined for ribbon graphs.}
\label{fig:graphcontract}
\end{figure}

Consequently we define a relation $\approx$ on (isomorphism classes of) M\"obius or ribbon graphs where $G\approx G'$ whenever one is obtained from the other by an edge contraction so that $G=G'/e$ or $G'=G/e$. The transitive closure of this is then an equivalence relation we will also denote by $\approx$. All elements of an equivalence class have the same genus and the same number of legs. There is also at least one graph with one vertex in each class. Observe that the space of corollas is obtained from the space of trees modulo this relation. As a result $\ass$ and $\mass$ could be defined as the operads of planar and M\"obius trees modulo $\approx$ with composition given by gluing trees.

We can now describe the operad $\modmass$. Recall that $\modass$ is the modular operad given by ribbon graphs up to the relation $\approx$ with composition and contraction given by the gluing legs of graphs. This is true since the modular closure of $\ass$ is generated by freely adding contractions and applying just those relations necessary to ensure that associativity and equivariance holds. More explicitly, we can first identify the space of planar corollas with contractions added in freely as ribbon graphs with $1$ vertex with loops directed and ordered. The equivariance condition means that we must forget the directions and order of the loops. Composition is given by gluing such objects and contracting internal edges that are not loops. The associativity condition requires that it does not matter in what order we contract internal edges. The relation $\approx$ (induced on ribbon graphs with $1$ vertex) is precisely the minimal relation required to ensure this is true. For example the bottom two graphs in \autoref{fig:graphcontract} are equivalent under $\approx$ but not isomorphic. It is clear that the same argument holds true for $\modmass$.

\begin{lemma}
The extended modular operad $\modmass$ can be described as follows:
\begin{itemize}
\item If $2g+n\geq 2$ then $\modmass((g,n))$ is the vector space generated by isomorphism classes of M\"obius graphs with $n$ legs and genus $g$ modulo the relation $\approx$.
\item Composition and contraction are given by gluing legs of graphs.
\qed
\end{itemize}
\end{lemma}

We next describe the main construction arising in the proof of \autoref{thm:oktftmass}. Let $G$ be a ribbon graph. The ribbon structure of $G$ allows one to replace each edge with a thin oriented strip and each vertex with an oriented disc using the cyclic ordering to glue the strips to discs in an orientation preserving manner. As such we obtain an oriented surface with boundary well defined up to diffeomorphism. Further we can identify the legs as labelled copies of the interval embedded in the boundary in an orientation preserving manner.

We can generalise this to a similar construction for M\"obius graphs. We replace each vertex $v$ with an oriented disc and we replace each edge $e$ with an oriented strip. We then use the cyclic ordering to glue the strips to discs. If the edge $e$ is connected to the vertex $v$ by a half edge coloured by $0$ we glue the strip corresponding to $e$ to the disc corresponding to $v$ such that their orientations are compatible. However if the half edge is coloured by $1$ we glue such that the orientations are not compatible. We identify the legs as labelled copies of the interval embedded compatibly with the disc's orientation if the leg is coloured by $0$ and incompatibly otherwise. We finally forget all the orientations on each part of our surface. This yields a surface that is not necessarily orientable. These constructions coincide for those M\"obius graphs that are just ribbon graphs (that is, graphs all of the same colour).

We should verify this construction is well defined up to diffeomorphism. However this is clear since applying the reflection relation at a vertex $v$ corresponds to constructing a surface identical everywhere except at the disc corresponding to $v$ which has been reflected (see \autoref{fig:discreflection}). Reflection of the disc is a smooth (orientation reversing) map so the construction yields a diffeomorphic surface.

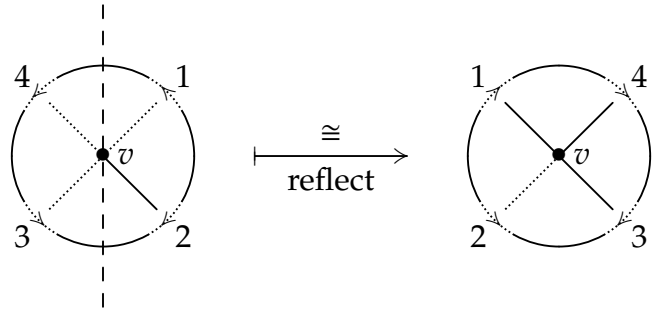
\begin{figure}[ht!]
\centering
\[
\begin{xy}
(0,0)*\xybox{*\xybox{\ellipse(12,12):a(60),:a(120){-}},
*\xybox{\ellipse(12,12):a(150),:a(210){-}},
*\xybox{\ellipse(12,12):a(240),:a(300){-}},
*\xybox{\ellipse(12,12):a(330),:a(30){-}},
*\xybox{\ellipse(12,12):a(30),:a(60){.}},
*\xybox{\ellipse(12,12):a(120),:a(150){.}},
*\xybox{\ellipse(12,12):a(210),:a(240){.}},
*\xybox{\ellipse(12,12):a(300),:a(330){.}},
*{\bullet},(3,0)*{v},
(0,0);a(45):(0,0);(10,0)**@{.}+(5,0)*{1},
(0,0);(-10,0)**@{.}-(5,0)*{3},
(-1,12)*\dir2{>},(-1,-12)*\dir2{>},
(0,0);(0,1):(0,0);(10,0)**@{.}+(5,0)*{4},
(0,0);(-10,0)**@{-}-(5,0)*{2},
(-1,12)*\dir2{>},(1,-12)*\dir2{<}},
(0,20);(0,-20)**@{--},
(60,0)*\xybox{*\xybox{\ellipse(12,12):a(60),:a(120){-}},
*\xybox{\ellipse(12,12):a(150),:a(210){-}},
*\xybox{\ellipse(12,12):a(240),:a(300){-}},
*\xybox{\ellipse(12,12):a(330),:a(30){-}},
*\xybox{\ellipse(12,12):a(30),:a(60){.}},
*\xybox{\ellipse(12,12):a(120),:a(150){.}},
*\xybox{\ellipse(12,12):a(210),:a(240){.}},
*\xybox{\ellipse(12,12):a(300),:a(330){.}},
*{\bullet},(3,0)*{v},
(0,0);a(45):(0,0);(10,0)**@{-}+(5,0)*{4},
(0,0);(-10,0)**@{.}-(5,0)*{2},
(1,12)*\dir2{<},(-1,-12)*\dir2{>},
(0,0);(0,1):(0,0);(10,0)**@{-}+(5,0)*{1},
(0,0);(-10,0)**@{-}-(5,0)*{3},
(-1,12)*\dir2{>},(-1,-12)*\dir2{>}},
(20,0);(40,0)**@{-}?(0)*\dir2{|}?(1)*\dir2{>}?(0.5)-(0,3)*{\textrm{reflect}}+(0,6)*{\cong}
\end{xy}
\]
\caption{The reflection relation at a vertex $v$ corresponds to reflection of the disc associated to $v$.}
\label{fig:discreflection}
\end{figure}

Since contracting an edge corresponds to contracting a strip this construction is in fact well defined on equivalence classes of $\approx$. \autoref{fig:basicsurfaces} shows the basic graphs corresponding to a handle, a crosscap and a boundary component (annulus). From this we can see that if a M\"obius graph has genus $g$ and the corresponding surface consists of $m$ handles, $u$ crosscaps and $h$ boundary components then $2m+h+u-1 = g$. This means that by this construction we obtain maps of the underlying vector spaces $\modmass((g,n))\rightarrow\oktft((g,n))$.

\begin{figure}[ht!]
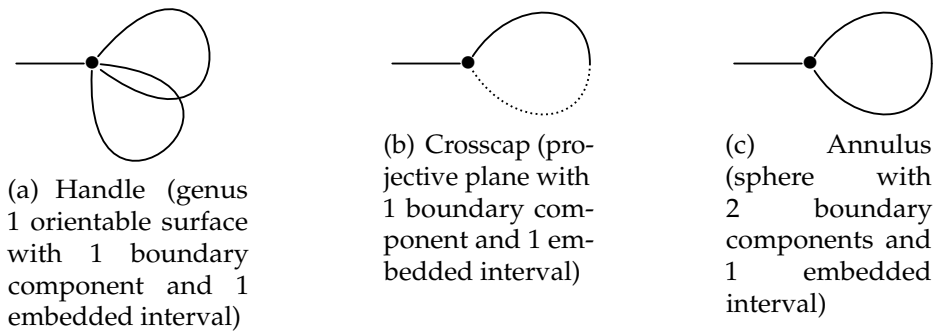

\centering
\subfigure[Handle (genus $1$ orientable surface with $1$ boundary component and $1$ embedded interval)]{
\xygraph{!{<0cm,0cm>;<1cm,0cm>:<0cm,1cm>::}
!{(0,0)}*{\bullet}="v1"
!{(-1,0)}*{}="o"
"v1" -@`{(+2.8,-0.2),(-0.2,-2.8)} "v1"
"v1" -@`{(+2.3,-1.8),(+1.8,+2.3)} "v1"
"v1" - "o"}}
\qquad\qquad
\subfigure[Crosscap (projective plane with $1$ boundary component and $1$ embedded interval)]{
\xygraph{!{<0cm,0cm>;<1cm,0cm>:<0cm,1cm>::}
!{(0,0)}*{\bullet}="v1"
!{(-1,0)}*{}="o"
!{(1.6,0)}*{}="loop"
"v1" -@`{c+(0.5,-0.9),p+(0,-0.9)}@{.} "loop"
"loop" -@`{c+(0,0.9),p+(0.5,0.9)} "v1"
"v1"-"o"}}
\qquad\qquad
\subfigure[Annulus (sphere with $2$ boundary components and $1$ embedded interval)]{
\xygraph{!{<0cm,0cm>;<1cm,0cm>:<0cm,1cm>::}
!{(0,0)}*{\bullet}="v1"
!{(-1,0)}*{}="o"
!{(1.6,0)}*{}="loop"
"v1" -@`{c+(0.5,-0.9),p+(0,-0.9)} "loop"
"loop" -@`{c+(0,0.9),p+(0.5,0.9)} "v1"
"v1"-"o"}}
\caption{M\"obius graphs corresponding to basic surfaces}
\label{fig:basicsurfaces}
\end{figure}

It is also clear that these constructions are compatible with operadic gluings so we obtain maps $\modass\rightarrow\otft$ and $\modmass\rightarrow\oktft$. As shown by Chuang and Lazarev \cite{chuanglazarev} the former is an isomorphism. We can now prove \autoref{thm:oktftmass} by showing the latter is too.

\begin{proof}[Proof of \autoref{thm:oktftmass}]
It is sufficient to show the map $\modmass\rightarrow\oktft$ described above is an isomorphism of the underlying $S$--modules. The surjectivity of this map follows from the classification of unoriented topological surfaces with boundary and \autoref{fig:basicsurfaces}, which shows how to build a surface of any topological type. To see that it is injective it is necessary to show that any two graphs with the same topological type are equivalent\footnote{This is analogous to proving the sufficiency of a set of relations on the generators of $\kcob$ if we were proving \autoref{prop:oktft} without mention of operads.} under the relation $\approx$. We first note two graphs that are equivalent as shown by the following diagram:
\begin{equation}\label{eqn:rel1}
\begin{split}
\xygraph{!{<0cm,0cm>;<1cm,0cm>:<0cm,1cm>::}
!{(0,0)}*{\bullet}="v1"
!{(-1,+0.5)}*+{1}="1"
!{(-1,-0.5)}*+{2}="2"
!{(1.3,0)}*{}="loop"
"v1" -@`{c+(0.5,-0.7),p+(0,-0.7)} "loop"
"loop" -@`{c+(0,0.7),p+(0.5,0.7)}@{.} "v1"
"v1"-"1" "v1"-"2"}
\quad\approx\quad
\xygraph{!{<0cm,0cm>;<1cm,0cm>:<0cm,1cm>::}
!{(0,0)}*{\bullet}="v1"
!{(1.5,0)}*{\bullet}="v2"
!{(-1,0)}*+{2}="2"
!{(2.5,0)}*+{1}="1"
!{(0.75,0.4)}*{}="u"
!{(0.75,-0.4)}*{}="d"
"v1" -@`{c+(0.25,0.3),p+(-0.25,0)} "u"
"u" -@`{c+(0.25,0),p+(-0.25,0.3)} "v2"
"v1" -@`{c+(0.25,-0.3),p+(-0.25,0)} "d"
"d" -@`{c+(0.25,0),p+(-0.25,-0.3)}@{.} "v2"
"v2"- "1"
"v1"-"2"}
\\
\cong\quad
\xygraph{!{<0cm,0cm>;<1cm,0cm>:<0cm,1cm>::}
!{(0,0)}*{\bullet}="v1"
!{(1.5,0)}*{\bullet}="v2"
!{(-1,0)}*+{2}="2"
!{(0.6,0)}*+{1}="1"
!{(0.75,0.4)}*{}="u"
!{(0.75,-0.4)}*{}="d"
"v1" -@`{c+(0.25,0.3),p+(-0.25,0)} "u"
"u" -@`{c+(0.25,0),p+(-0.25,0.3)}@{.} "v2"
"v1" -@`{c+(0.25,-0.3),p+(-0.25,0)} "d"
"d" -@`{c+(0.25,0),p+(-0.25,-0.3)} "v2"
"v2"-@{.} "1"
"v1"-"2"}
\quad\approx\quad
\xygraph{!{<0cm,0cm>;<1cm,0cm>:<0cm,1cm>::}
!{(0,0)}*{\bullet}="v1"
!{(1,0)}*+{1}="1"
!{(-1,0)}*+{2}="2"
!{(1.3,0)}*{}="loop"
"v1" -@`{c+(0.5,-0.7),p+(0,-0.7)}@{.} "loop"
"loop" -@`{c+(0,0.7),p+(0.5,0.7)} "v1"
"v1"-@{.}"1" "v1"-"2"}
\end{split}
\end{equation}

Using this we can prove the relation corresponding to the fact that the connected sum of $3$ crosscaps corresponds to the connected sum of a handle and a crosscap:

\begin{equation}\label{eqn:rel2}
\xygraph{!{<0cm,0cm>;<1cm,0cm>:<0cm,1cm>::}
!{(0,0)}*{\bullet}="v1"
!{(1.3,0)}*{}="loop"
!{(-1.3,0)}*{}="loop2"
!{(0,+1.3)}*{}="loop3"
!{(0,-1)}*{}="o"
"v1" -@`{c+(0.5,-0.5),p+(0,-0.5)} "loop"
"loop" -@`{c+(0,0.5),p+(0.5,0.5)}@{.} "v1"
"v1" -@`{c+(-0.5,-0.5),p+(0,-0.5)} "loop2"
"loop2" -@`{c+(0,0.5),p+(-0.5,0.5)}@{.} "v1"
"v1" -@`{c+(-0.5,0.5),p+(-0.5,0)} "loop3"
"loop3" -@`{c+(0.5,0),p+(0.5,0.5)}@{.} "v1"
"v1"-"o"
}
\quad\approx\quad
\xygraph{!{<0cm,0cm>;<1cm,0cm>:<0cm,1cm>::}
!{(0,0)}*{\bullet}="v1"
!{(0,-1)}*{}="o"
!{(-1.3,0)}*{}="loop2"
"v1" -@`{(+2.8,+0.2),(-0.2,+2.8)} "v1"
"v1" -@`{(+2.3,+1.8),(+1.8,-2.3)} "v1"
"v1" -@`{c+(-0.5,-0.5),p+(0,-0.5)} "loop2"
"loop2" -@`{c+(0,0.5),p+(-0.5,0.5)}@{.} "v1"
"v1" - "o"}
\end{equation}
This can be shown by drawing graphs and repeatedly applying relation \ref{eqn:rel1}. We leave this to the reader.

Given a graph we can contract all internal edges that are not loops. Then we can ensure that all loops which are composed of half edges of the same colour (which we will call untwisted loops) are all coloured by $0$ since a loop coloured by $1$ is equivalent to a loop coloured by $0$ (by expanding the loop into $2$ edges of different colours and contracting the edge of colour $1$). We then apply  relation \ref{eqn:rel1} repeatedly to ensure that all the twisted loops are adjacent and have no half edges or legs on their inside. Finally we use relation \ref{eqn:rel2} repeatedly until there are at most $2$ twisted loops remaining. Therefore any graph is equivalent to a `normal form' consisting of either a ribbon graph with $1$ vertex or the connected sum (by which we mean vertices connected by a single untwisted edge) of a ribbon graph with $1$ vertex and a M\"obius graph with $1$ vertex and at most $2$ twisted loops. If two graphs have the same topological type then in this normal form the M\"obius graph components must be isomorphic. But the ribbon graph components must therefore be of the same topological type and we know that they are equivalent under the relation $\approx$ since $\modass\rightarrow\otft$ is an isomorphism.
\end{proof}

\begin{corollary}
Algebras over the modular operad $\oktft$ are (non-unital) symmetric Frobenius algebras together with an involution preserving the inner product.
\end{corollary}

\begin{proof}
Since $\oktft$ is the modular closure of its genus $0$ cyclic operad then algebras over $\oktft$ are simply algebras over $\mass$ considered as a cyclic operad (this is immediate as in \cite[Proposition 2.4]{chuanglazarev}). Cyclic $\mass$--algebras are just $\mass$--algebras with an invariant inner product which are precisely Frobenius algebras with an involution.
\end{proof}

In the formulation of topological field theories as a symmetric monoidal functor from some cobordism category the only difference is that we have a unit and counit (see \autoref{rem:nounit}). Therefore we have now fulfilled our earlier promise and shown:

\begin{corollary}[\autoref{prop:oktft}]
Open Klein topological field theories of dimension $2$ are equivalent to symmetric Frobenius algebras together with an involutive anti-automorphism preserving the pairing.
\qed
\end{corollary}

\section{Cobar duality for $\mass$}
We will now consider the operad $\mass$ in more detail.

Recall that the free operad generated by the vector space $\ass(2)$ over $k$ is the operad of binary planar trees and that $\ass$ is the quotient of this by the associativity relation. It is therefore a quadratic operad since the associativity relation is a quadratic relation. Further $\ass^!\cong\ass$.

$\mass$ is also quadratic: let $K$ be the semisimple algebra $\mass(1)=\langle 1, a\rangle / (a^2=1)=k[\mathbb{Z}_2]$. By taking the quotient of the free operad generated by the $(K,K^{\otimes 2})$--bimodule $\mass(2)$ by the associativity relation we obtain $\mass$. In fact, as we shall see, all the usual duality properties of $\ass$ hold for $\mass$.

\begin{proposition}\label{prop:massqdual}
$(\mass)^!\cong\mass$.
\end{proposition}

\begin{proof}
As in the case of $\ass$ we can simply give an explicit isomorphism. The only potential difficulty arises from the quadratic dual being twisted by the sign representation, however this turns out not to be an issue. Let $K=\mass(1)=\langle 1, a\rangle / (a^2=1)=k[\mathbb{Z}_2]$ and $E=\mass(2)$.

Let $\psi_1\co K\rightarrow K^{\mathrm{op}}=K$ be the isomorphism with $\psi_1(a)=-a$. We define a map $\psi_2\co E\rightarrow E^{\vee}$ of $k$--linear $S_2$--representations as follows. Let $\sigma\in S_2$ denote the transposition and denote by $m$ the corolla:
\[m = 
\xygraph{!{<0cm,0cm>;<0.8cm,0cm>:<0cm,0.8cm>::}
1="1"&&2="2" \\
&*{\bullet}="v"\\
&="0"\\
"v"-"1" "v"-"2" "v"-"0"}
\]

Let $B=\{m,\sigma m(a \otimes 1), \sigma m(1 \otimes a), m(a\otimes a)\}$. Observe that $B$ is a $K$--linear basis for the left $K$--module $E$. For each $e\in B$ we denote by $e^*\in\Hom_K(E,K)$ the element of the dual basis for $E^{\vee}$. By this we mean $e^*$ is defined on each $e'\in B$ by $e^*(e') = 1$ if $e'=e$ and $e^*(e')=0$ otherwise. For $e\in B$ set $\psi_2(e)=e^*$. Now observe that $B$ also freely generates $E$ as a $k$--linear $S_2$--module, so $\psi_2$ extends to an isomorphism $E\rightarrow E^{\vee}$ of $k$--linear $S_2$--modules. Explicitly this sends an element $f$ of the $K$--linear basis $\sigma B=\{\sigma m, m(a\otimes 1), m(1\otimes a), \sigma m(a\otimes a)\}$ to $-f^*$ (where $f^*$ denotes the element of the dual basis of $\sigma B$).

We claim the map $\Psi = (\psi_1,\psi_2,0,\ldots)$ gives an isomorphism of $K$--collections. By definition it is an isomorphism of $S$--modules. Some straightforward calculations verify that it is also a map of $K$--collections. For example, for $e\in B$ we have $\psi_2(ae)=\psi_2(\sigma e (a\otimes a))=\sigma(\psi_2(e(a\otimes a))) = -a\psi_2(e) = \psi_1(a)\psi_2(e)$.

Therefore $F(E)\cong F(E^{\vee})$ with $\Psi$ extending to an isomorphism of operads. Let $R\subset F(E)(3)$ be the $S_3$--stable sub-bimodule generated by the associativity relation for $m$. It remains to show that $R^{\perp}=\Psi(R)$. Since $\dim(R)=\dim(F(E)(3))/2$ it is sufficient to check the associativity relation for $m$ is in $R^{\perp}$. This is a simple check, which we omit.
\end{proof}

We now describe the cobar construction for $\mass$. To do this we will need to identify the space of decorations on a tree $T$ by the underlying $K$--collection of $\mass$ (recall that in general this is different from the space of decorations on a tree $T$ by the underlying $S$--module). We therefore define the notion of reduced M\"obius and planar trees.

\begin{definition}\label{def:reducedtree}
Given a planar or M\"obius tree with at least one vertex of valence at least $3$ we can associate (possibly several) reduced trees by repeatedly contracting an edge attached to a vertex of valence $2$ until the tree is reduced. We say two reduced M\"obius or planar trees are equivalent if they are obtained from the same tree in this manner. When we refer to a reduced M\"obius or planar tree we will mean an isomorphism class of reduced M\"obius or planar trees up to this equivalence.
\end{definition}

\begin{remark}\label{rem:reducedmobius}
This has no effect for planar trees but for M\"obius trees we have that the following reduced M\"obius trees are the same for example:
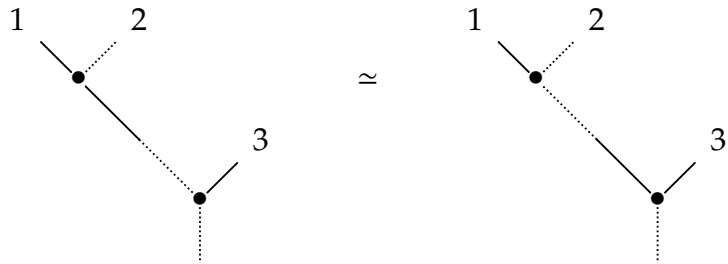
\[
\xygraph{!{<0cm,0cm>;<0.8cm,0cm>:<0cm,0.8cm>::}
1="1"&&2="2"\\
&*{\bullet}="v1"\\
&&="v1v2"&&3="3"\\
&&&*{\bullet}="v2"\\
&&&="0"
"1"-"v1" "2"-@{.}"v1" "v1"-"v1v2" "v1v2"-@{.}"v2" "v2"-"3" "v2"-@{.}"0"
}\qquad\simeq\qquad
\xygraph{!{<0cm,0cm>;<0.8cm,0cm>:<0cm,0.8cm>::}
1="1"&&2="2"\\
&*{\bullet}="v1"\\
&&="v1v2"&&3="3"\\
&&&*{\bullet}="v2"\\
&&&="0"
"1"-"v1" "2"-@{.}"v1" "v1"-@{.}"v1v2" "v1v2"-"v2" "v2"-"3" "v2"-@{.}"0"
}
\]
Also note that edge contraction is still well defined on reduced M\"obius trees.
\end{remark}

Thus defined, the space of decorations on a reduced tree $T$ by the $k$--collection $\ass$ is generated by the set of reduced planar trees whose underlying tree is $T$. The space of decorations on $T$ by the $K$--collection $\mass$ is spanned by the set of reduced M\"obius trees whose underlying tree is $T$.

\begin{definition}\label{def:orientedtree}
The space of oriented planar (or M\"obius) trees is generated by planar (or M\"obius) trees equipped with an ordering of the internal edges subject to the relations arising by requiring that swapping the order of two edges is the same as multiplying by $-1$ so that, for example, the space $\ass(T)\otimes\det(T)$ (see \hyperref[subsec:detT]{Section~\ref*{subsec:detT}}) can then be identified with the space of reduced oriented planar trees whose underlying tree is $T$.
\end{definition}

Now recall the operad $\ass$ is Koszul, $\ass(n)\cong\ass(n)^*$ and $\dass$ is the operad of reduced oriented planar trees (where $S_n$ acts by the sign representation) which governs $A_\infty$--algebras.

By sending a corolla $m$ in $\mass(n)$ to the map $\psi(m)(m)=1$, $\psi(m)(am)=a$, $\psi(m)(m')=0$ for the other corollas $m'$ (similar to the map in the proof of \autoref{prop:massqdual} but without the different signs since $\mass(n)^*$ is not twisted by the sign representation) we obtain an $S_n$--equivariant map $\mass(n)\rightarrow\mass(n)^*$ that is also a map of $(K,K^{\otimes n})$--bimodules. Therefore the underlying spaces $C(\mass)(n)$ are spanned by reduced oriented M\"obius trees with $n$ inputs, graded appropriately by the number of internal edges. Composition corresponds to gluing oriented trees. The differential corresponds to expanding vertices of valence greater than 3, for example:
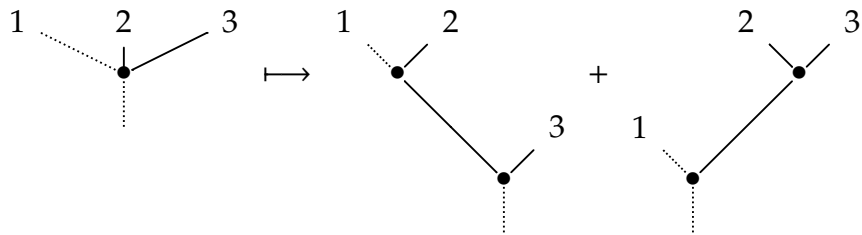
\[
\xygraph{!{<0cm,0cm>;<0.7cm,0cm>:<0cm,0.7cm>::}
1="1"&&2="2"&&3="3"\\
&&*{\bullet}="v"\\
&&="0"
"1"-@{.}"v" "2"-"v" "3"-"v" "v"-@{.}"0"
}
\longmapsto
\xygraph{!{<0cm,0cm>;<0.7cm,0cm>:<0cm,0.7cm>::}
1="1"&&2="2"\\
&*{\bullet}="v1"\\
&&="v1v2"&&3="3"\\
&&&*{\bullet}="v2"\\
&&&="0"
"1"-@{.}"v1" "2"-"v1" "v1"-"v1v2" "v1v2"-"v2" "v2"-"3" "v2"-@{.}"0"
}
+
\xygraph{!{<0cm,0cm>;<0.7cm,0cm>:<0cm,0.7cm>::}
&&2="2"&&3="3"\\
&&&*{\bullet}="v1"\\
1="1"\\
&*{\bullet}="v2"\\
&="0"
"2"-"v1" "3"-"v1" "v1"-"v2" "v2"-@{.}"1" "v2"-@{.}"0"
}
\]
When drawing oriented M\"obius trees like the above we give them the orientation on the edges by ordering the internal edges from left to right, from bottom to top.

\begin{remark}\label{rem:cassincmass}
Observe as in \autoref{rem:assinmass} that $C(\ass)$ is a suboperad of $C(\mass)$, since planar trees are M\"obius trees with straight edges. Indeed once again $C(\mass)$ is generated by adjoining an involution of degree $0$ to $C(\ass)$, this time modulo the reflection relation on all corollas.
\end{remark}

\begin{lemma}\label{lem:koszul}
As dg vector spaces $C(\mass)(n)=\bigoplus^{2^n}C(\ass)(n)$.
\end{lemma}

\begin{proof}
Given a reduced M\"obius tree $T$ we will find a unique reduced tree $T'$ isomorphic to it with the root and all internal edges coloured by $0$ (in other words all straight lines). This is then a tree in $C(\ass)$ with coloured inputs of which there are $2^n$ possibilities. The differentials clearly coincide as $K$ is concentrated in degree $0$.

To find such a tree we apply a sequence of transformations to $T$ that result in an isomorphic tree at each stage. The basic transformations are either using the reflection relation at a vertex or swapping the colourings of an edge (as in, for example, \autoref{rem:reducedmobius}) in a reduced tree. The process is as follows: we first apply the reflection relation if necessary to ensure the root is coloured by $0$. We then apply the edge relations to all the inputs of the bottom vertex to ensure all the half edges connected to the bottom vertex are coloured by $0$. We then repeat this inductively at each of the vertices at the next level until we have transformed the whole tree. The resulting tree $T'$ is unique since there is no choice in this process.
\end{proof}

\begin{corollary}\label{cor:koszul}
$\mass$ is Koszul.
\end{corollary}

\begin{proof}
$\mass$ is Koszul if and only if the complexes $C(\mass)(n)$ are exact everywhere but the right end. This is true by \autoref{lem:koszul} since $\ass$ is Koszul.
\end{proof}

We now consider homotopy $\mass$--algebras. That is to say, algebras over $\dmass$. From \autoref{rem:cassincmass} we have that $\dmass$ is generated by the operations $m_i\in\dass$ for $i\geq 2$, together with an involution of degree $0$, which by convention we will say corresponds to $-a\in K$. The differential on $\dmass$ is the same as that on $\dass$ for the operations $m_i$ so it yields the usual $A_\infty$ conditions. The reflection relation on the $m_i$ however introduces an extra sign since we have now twisted $C(\mass)$ by the sign representation. The sign of the permutation reversing $n$ labels is $(-1)^{n(n-1)/2}$. So we have shown the following:

\begin{proposition}
Algebras over $\dmass$ are $A_\infty$--algebras with an involution such that
\[
m_n(x_1,\dots,x_n)^*=(-1)^{\epsilon}(-1)^{n(n+1)/2-1}m_n(x_n^*,\dots,x_1^*)
\]
where $\epsilon=\sum_{i=1}^{n}\degree{x}_i\left(\sum_{j=i+1}^{n}\degree{x}_j\right)$ arises from permuting the $x_i$ with degrees $\degree{x}_i$.
\qed
\end{proposition}

\begin{remark}
$\dmass$ can be given the structure of a cyclic operad in the obvious way (permuting the labellings of M\"obius trees).
\end{remark}

An important operad for us (which we shall see later controls open Klein topological conformal field theory) is the modular operad $\moddmass$ which we shall now describe explicitly by identifying it as the operad of reduced oriented M\"obius graphs with the expanding differential. This is the analogue of the fact that $\moddass$ is the operad of reduced oriented ribbon graphs with the expanding differential. We need to define these terms of course, which are analogues of \autoref{def:reducedtree} and \autoref{def:orientedtree}.

\begin{definition}
A \emph{reduced} M\"obius or ribbon graph\footnote{As for trees we use the word `reduced' as opposed to `stable' to emphasise that the vertices of these graphs are not equipped with a genus.} is a graph where each vertex has valence at least $3$. Given a graph with at least one vertex of valence at least $3$ we can associate (possibly several) reduced graphs to it by repeatedly contracting an edge attached to a vertex of valence $2$ until the graph is reduced. We say two reduced graphs are equivalent if they are obtained from the same graph in this manner. When we refer to a reduced graph we will mean an isomorphism class of reduced graphs up to this equivalence.
\end{definition}

\begin{remark}\label{rem:reducedmobgraph}
As for trees this equivalence on reduced graphs has no effect for ribbon graphs. However for stable M\"obius graphs we have an additional relation changing the colours on half edges belonging to the same edge as in \autoref{rem:reducedmobius}. If two half edges in an edge are coloured by $0$ then we can replace them by half edges coloured by $1$ and get an equivalent reduced graph. If they are different colours we can swap the colours and get an equivalent reduced graph. It is clear relations of this form are the only ones arising from this equivalence relation.
\end{remark}

Let $G$ be a stable graph of genus $g$ with $n$ legs and $e$ edges. Let $\det(G)=\Det(k^{\edges(G)})\otimes\Det(H_1(|G|))$ be concentrated in degree $e+3-3g-n$.

\begin{proposition}\label{prop:orientedmobgraph}
There are isomorphisms of chain complexes
\[
\moddmass((g,n))\cong\bigoplus_{G\in\iso\Gamma((g,n))}\underline{\mass}((G))\otimes\det(G)
\]
where here $\underline{\mass}((G))=\bigotimes_v\underline{\mass}((g(v),\flags(v)))$ is defined by taking the tensor product over $K$ using the $(K,K^{\otimes \flags(v)})$--bimodule structures. The action of $S_n$ on the right permutes labels of $G$ twisted by the sign. The differential is the natural differential expanding vertices of valence greater than $3$.
\end{proposition}

\begin{proof}
It is easy to convince oneself this is true since both sides are related to the free modular operad generated by $\mass$. We just need to explain how the $\det(G)$ term arises. First we observe that as in \autoref{rem:decgraphs} the space $\underline{\mass}((G))$ can be identified with the space generated by reduced M\"obius graphs. Given a reduced oriented M\"obius tree $T$ with $\omega\in\det(T)$ representing the orientation of $T$ and a contraction $\xi_{ij}$ with $i<j$ we can glue the $i$--th and $j$--th legs of $T$ to obtain a reduced M\"obius graph $G$ with newly formed edge $e$. We direct the edge $e$ such that the $i$--th leg is outgoing and the $j$--th leg incoming. This gives an oriented cycle $c$ in $H_1(|G|)$, by using the canonical direction on the tree $T$. Therefore we map $\xi_{ij}(T)$ to $G\otimes\omega\wedge e\otimes c$. We then extend this map inductively by mapping $\xi_{kl}(G)$ with $k<l$ to the graph $G'$ obtained by gluing the $k$--th and $l$--th legs of $G$, orienting the new edge as before, which gives a new oriented cycle $c'$ so we take the element $\omega'\wedge c'\in\Det(H_1(|G'|))$ given $\omega'\in\Det(H_1(|G|)$. We must check of course that this is a well defined map. In particular we must check it is well defined for the various associativity and equivariance relations. We omit the details, however the main point to observe is that the minus sign arising when we apply the transposition $(ij)\in S_n$ to a reduced oriented M\"obius tree and then contract the $i$--th and $j$--th legs is reflected in the fact that the direction of the edge formed by gluing legs $i$ and $j$ is then reversed so the orientation of the cycle $c$ is reversed and also when we carry out contractions in a different order, we swap the ordering of the cycles, but we also swap the ordering of the new edges. 

It is completely clear that the gradings and the differentials coincide. To see this map is an isomorphism note it is clearly surjective then compare dimensions by observing that both sides are closely related to the free modular operad generated by $\mass$.
\end{proof}

The compositions are of course simply gluing graphs and ordering the edges in the same way as we do for oriented trees (cf \hyperref[subsec:cobar]{Section~\ref*{subsec:cobar}}). Contractions are also obvious and induce the orientation as detailed in the above proof.

We can talk about oriented graphs as we did for trees in \autoref{def:orientedtree}.

\begin{definition}
  The space of oriented ribbon/M\"obius graphs is generated by ribbon/M\"obius graphs $G$ equipped with an ordering of the internal edges and an ordering of a basis of cycles in $H_1(|G|)$ subject to the relations arising by requiring that swapping the order of two edges or of two cycles is the same as multiplying by $-1$. In particular the space $\underline{\mass}((G))\otimes\det(G)$ from above can be identified as the space of reduced oriented M\"obius graphs whose underlying graph is $G$.
\end{definition}

\begin{remark}\label{rem:orientation}
When $k=\mathbb{Q}$ or $k=\mathbb{R}$ an orientation on a graph is equivalent to an ordering of its vertices and directing its edges, up to an even permutation. For example see \cite{conantvogtmann,getzlerkapranov}.
\end{remark}

\section{M\"obiusisation of operads and closed KTFTs}
We briefly outline the general construction for operads that follows from considering the above arguments and we also briefly consider closed KTFTs. As usual we let $k$ be a field and $K$ be the unital associative algebra over $k$ generated by an involution $a$ so $K=\langle 1, a \rangle / (a^2=1)=k[\mathbb{Z}_2]$.

\begin{definition}
Let $\mathcal{P}\in\dgop(k)$ be an admissible dg operad so $\mathcal{P}(1)=k$ is concentrated in degree $0$. The \emph{M\"obiusisation} of $\mathcal{P}$ is an operad $\mob\mathcal{P}\in\dgop(K)$ obtained by freely adjoining an element $a$ to $\mathcal{P}(1)$ in degree $0$ and imposing the relations $da=0$, $a^2=1$ and $am = \tau_n(m)a^{\otimes n}$ (the reflection relation) for all $m\in\mathcal{P}(n)$ where $\tau_n=(1\quad n)(2\quad n-1)(3\quad n-2)\dots\in S_n$ is the permutation reversing $n$ labels. This construction extends to a functor $\mob\co \dgop(k)\rightarrow\dgop(K)$.

Given a unital extended modular operad $\mathcal{O}$ with $\mathcal{O}((0,2))=k$ we define $\mob\mathcal{O}$ in a similar way.
\end{definition}

Note that $\mathcal{P}$ is a suboperad of $\mob\mathcal{P}$. Clearly $\mass$ as defined above is indeed the M\"obiusisation of $\ass$. We have the following properties that generalise those shown for $\mass$ in the previous section:

\begin{theorem}\label{thm:mobiusisation}
Let $\mathcal{P}\in\dgop(k)$.
\begin{enumerate}
\item If $\mathcal{P}$ is quadratic then so is $\mob\mathcal{P}$ and $(\mob\mathcal{P})^!\cong\mob(\mathcal{P}^!)$
\item As dg vector spaces $C(\mob\mathcal{P})(n)=\bigoplus^{2^n}C(\mathcal{P})(n)$
\item $C(\mob\mathcal{P})=\mob C(\mathcal{P})$
\item If $\mathcal{P}$ is Koszul then $\mob\mathcal{P}$ is Koszul
\item If $\mathcal{P}$ is cyclic then $\modc{\mob\mathcal{P}}=\mob\modc{\mathcal{P}}$
\end{enumerate}
\end{theorem}

\begin{proof}[Sketch proof]
\Needspace*{3\baselineskip}\mbox{}
\begin{enumerate}
\item This is a general version of ideas in \autoref{prop:massqdual}. Let $\mathcal{P}=\mathcal{P}(k,E,R)$. Let $E'=K\otimes_k E\otimes _k K^{\otimes 2}$ and $\mob E=E'/I$ where $I$ is generated by the reflection relations $a\otimes m\otimes 1\otimes 1=1 \otimes \sigma(m)\otimes a\otimes a$. Then $\mob R$ is generated by $R\subset F(E)(3)\subset F(\mob E)(3)$ and $\mob\mathcal{P}=\mob\mathcal{P}(K,\mob E,\mob R)$. Given $\psi\in\Hom_k(E,k)$ we define $\psi'\in\Hom_K(E',K)$ by $\psi'(1\otimes m\otimes 1\otimes 1)=\psi(m)$, $\psi'(a\otimes m\otimes 1\otimes 1)=a\psi(m)$, $\psi'(1\otimes m\otimes a\otimes a)=a\psi(\sigma(m))$ and $\psi(1\otimes m\otimes a\otimes 1)=\psi(1\otimes m\otimes 1\otimes a)=0$ for $m\in E$ and $\sigma=(1\quad 2)$. This in turn gives a well defined element of $\Hom_K(\mob E, K)$. This map extends to an isomorphism of $K$--collections $\Psi\co M(E^{\vee})\rightarrow (\mob E)^{\vee}$ and the result follows since $\Psi(\mob(R^{\perp}))=(\mob R)^{\perp}$.
\item This is a general version of \autoref{lem:koszul}. The same simple inductive proof works in the general case. Let $T$ be a tree with $n$ inputs. It's enough to show that $\mob\mathcal{P}(T)^* \cong \mathcal{P}(T)^*\otimes_k K^{\otimes n}$. Write $T$ as $T=T''\circ_i T'$ for $T'$ a corolla with $n'$ inputs. Then by induction on the number of internal edges $\mob\mathcal{P}(T)^* \cong (\mathcal{P}(T'')^*\otimes_k K^{\otimes n-n'+1})\otimes_K (\mathcal{P}(T')^*\otimes_k K^{\otimes n'})\cong \mathcal{P}(T)^*\otimes_k K^{\otimes n}$.
\item This follows from the above result together with a similar inductive argument showing that the reflection relation does indeed hold for any $m\in C(\mathcal{P})(n)\subset C(\mob\mathcal{P})(n)$.
\item This follows from the above results (cf \autoref{cor:koszul}).
\item We observe that both operads are generated by their genus $0$ parts which coincide.
\qedhere
\end{enumerate}
\end{proof}

Finally we briefly consider the situation of closed KTFTs.

\begin{definition}
We define the $k$--linear extended modular operad $\cktft$ (partial closed Klein topological field theory) as follows:
\begin{itemize}
\item For $n,g\geq 0$ and $2g+n\geq 2$ the vector space $\cktft((g,n))$ is generated by diffeomorphism classes of surfaces with $m$ handles and $u$ crosscaps and $n$ boundary components with $m+u/2 = g$ with $n$ copies of the circle embedded into the boundary, labelled by $\{1,\dots,n\}$ with an action of $S_n$ permuting the labels.
\item Composition and contraction is given by gluing boundary components.
\end{itemize}
\end{definition}

\begin{remark}
This is `partial' closed Klein topological field theory in the sense that $u$ must be divisible by $2$. Therefore this operad features just those surfaces obtained as the connected sum of tori and Klein bottles. Since the connected sum of $2$ Klein bottles is diffeomorphic to the sum of $1$ handle and $1$ Klein bottle we see that $m+u/2$ is well defined regardless of how we represent the topological type. 
\end{remark}

Note that a Klein bottle (with boundary) must have genus $1$ in the modular operad sense since it is obtained from self gluing a genus $0$ surface. Therefore in full closed KTFT a crosscap would necessarily have genus $\frac{1}{2}$.

\begin{theorem}
$\cktft\cong\modmcom$.
\end{theorem}

\begin{proof}[Idea of proof]
The operad $\modmcom$ can be described in terms of graphs by forgetting mention of cyclic orderings of half edges at the vertices in our definition of M\"obius graphs. By replacing vertices with spheres with holes and edges by cylinders we obtain surfaces corresponding to such graphs. Then the above proposition follows by a similar argument to the proof of \autoref{thm:oktftmass}.
\end{proof}

It should not be at all surprising that $\modmcom$ does not describe full closed KTFT since \autoref{prop:cktft} tells us closed KTFTs are not just commutative Frobenius algebras with involution but rather have additional structure and additional relations that do not arise from relations in genus $0$, unlike in the open case.

\chapter{Moduli spaces of Klein surfaces}\label{chap:moduli}
In this chapter we will obtain results that are analogues of results concerning the ribbon graph decomposition of moduli spaces of Riemann surfaces with boundary. In particular we will be following methods of Costello \cite{costello1,costello2} which relate the operad $\moddass$ to certain moduli spaces and show $\moddass$ governs open topological conformal field theory. For us the unoriented analogue of a Riemann surface is a Klein surface and M\"obius graphs serve the same role as ribbon graphs.

Klein surfaces are `unoriented Riemann surfaces' (or more correctly Riemann surfaces are oriented Klein surfaces) in the sense that they have a dianalytic structure instead of an analytic structure. Klein surfaces are equivalent to symmetric Riemann surfaces (Riemann surfaces with an antiholomorphic involution) without boundary. In fact it follows Klein surfaces are equivalent to projective real algebraic curves (see Alling and Greenleaf \cite{allinggreenleaf} or Natanzon \cite{natanzon}). Since we wish to use techniques from hyperbolic geometry we will be concerned with the analytic theory.

\section{Informal discussion}
We will first provide an informal discussion outlining the content of this chapter since there are some (slightly tedious) technical issues arising from the need to consider nodal surfaces, which are more subtle for Klein surfaces than for Riemann surfaces and which can hide the more important general picture.

A Klein surface is the natural extension of a Riemann surface allowing unorientable surfaces. Klein surfaces have a dianalytic structure instead of an analytic structure. However, given a Klein surface we can construct a double cover (the complex double) for the surface which is a Riemann surface so that we can use much of the theory of Riemann surfaces to study Klein surfaces. Indeed it is actually the case that Klein surfaces are equivalent to symmetric Riemann surfaces (Riemann surfaces with an antiholomorphic involution) identifying the Klein surface with the quotient. This is the standard way of approaching Klein surfaces, where a Klein surface is then simply a pair $(X,\sigma)$ with $X$ a Riemann surface and $\sigma$ and antiholomorphic involution.

However we will want to consider surfaces with nodes. With our previous comment in mind the obvious way to approach this is to define a nodal Klein surface as a pair $(X,\sigma)$ where $X$ is now a nodal Riemann surface and $\sigma$ is an antiholomorphic involution. This is the standard approach used for example by Sepp\"al\"a \cite{seppala} to construct a compactification of the moduli space of Klein surfaces/symmetric Riemann surfaces. 

There is also another natural way to define a nodal Klein surface without passing to the complex double. A nodal Klein surface is then a surface with some nodal singularities and a dianalytic structure, including at the nodes.

Although one might expect these two notions to coincide they do not. The reason for this is that it is no longer possible to form a unique complex double of a nodal Klein surface in the latter sense. This can be understood by considering a `strangulated' M\"obius strip (see \autoref{fig:strangulatedmobius}). In the second definition the dianalytic structure about the node encodes the twist in the M\"obius strip. If we pass to the complex double of a M\"obius strip, which is a torus, then a node on a strangulated torus does not encode any form of twisting. Indeed if we take the quotient of such a torus by an antiholomorphic involution then there is not a well defined way of giving a dianalytic structure at the node in the quotient.

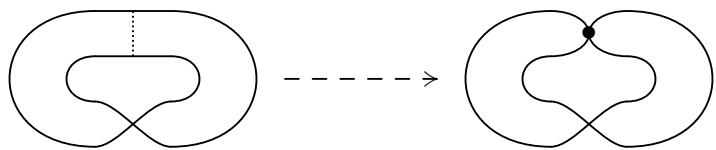
\begin{figure}[ht!]
\centering
\[
\begin{xy}
(-30,0)*\xybox{
(-12.5,0)*{\copair{}},
(0,6)*{\shortcylinder{}{}}="top",
c+U="a","top",c+D;"a"**@{.},
(0,-6)*{\flip{}{}},
(12.5,0)*{\pair{}}},
(30,0)*\xybox{
(-12.5,0)*{\copair{}},
(-5,6)*\xybox{\ellipse(5,3)_,=:a(180){-}},
(5,6)*\xybox{\ellipse(5,3)^,=:a(180){-}},
(0,6)*{\bullet},
(0,-6)*{\flip{}{}},
(12.5,0)*{\pair{}}},
(-10,0);(10,0)**\dir{--};?(1)*\dir2{>}
\end{xy}
\]
\caption{A strangulated M\"obius strip, obtained by contracting the dotted line to a node.}
\label{fig:strangulatedmobius}
\end{figure}

With this in mind it is natural to ask if there is another double cover that we can construct for this second type of nodal surface. The solution is given by the orienting double which is a Riemann surface but possibly with a boundary.

The important difference between the complex double and the orienting double is that the complex double takes the boundary of a Klein surface to the fixed points in the interior of a symmetric Riemann surface and, as mentioned above, if we take the quotient of a symmetric Riemann surface with an interior node fixed by the symmetry then there is not a well defined way of giving a dianalytic structure at the corresponding boundary node in the quotient. However the orienting double takes the boundary of a Klein surface to the boundary of a symmetric Riemann surface and so boundary nodes now correspond to boundary nodes. For example, the quotient of a strangulated torus (with an antiholomorphic involution such that the node in the quotient is a boundary node) could be given the dianalytic structure of either a strangulated M\"obius strip or a strangulated annulus. However the orienting doubles of these are respectively an annulus with two strangulated points (the covering map wraps such an annulus twice around the M\"obius strip) and a disjoint union of two strangulated annuli.

It is also natural to ask if there is another notion of a boundary node on a Klein surface that is equivalent to an interior node fixed by the symmetry on a symmetric Riemann surface. The answer to this is a na\"ive node, which is simply a singularity without a dianalytic structure at the node.

So we can still obtain equivalences for each of these types of nodal surface, although they are different. It is perhaps not entirely necessary to consider these equivalences of categories in too much detail in order to obtain our main results but we give a fairly detailed overview in order to make the subtlety arising from the different types of nodes clearer.

The conclusion of all this is that we obtain two different partial compactifications of the moduli space of Klein surfaces by allowing nodes on the boundary. It turns out the second type of nodal surface is the natural notion for defining Klein topological conformal field theory since it generalises the gluing of intervals discussed in the previous sections. We obtain a topological modular operad which we will denote $\kb$, the operad of Klein surfaces with boundary nodes. This notation reflects the notation $\nb$ used by Costello \cite{costello1,costello2} for the operad of Riemann surfaces with boundary and boundary nodes. The first type of nodal surface gives rise to an operad that is closer in spirit to the Deligne--Mumford operad since we are gluing symmetric Riemann surfaces without boundary at interior marked points. We obtain an operad which we will denote $\mb$, the operad of `admissible' symmetric nodal Riemann surfaces without boundary. This notation reflects the common notation for the space of symmetric Riemann surfaces and the fact that symmetric Riemann surfaces are equivalent to real algebraic curves. Note however that for us $\mb$ is \emph{not} the full space of nodal symmetric Riemann surfaces (stable real algebraic curves) and should not be confused with the full compactification obtained by taking all ways of forming nodes (it is an open subspace of this).

We wish to apply the methods of Costello \cite{costello1,costello2} to find a graph decomposition of both of these operads. The operad $\kb$, being the operad of Klein topological conformal field theory, is similar to $\nb$ and most of the results concerning the latter have a corresponding version for the former. In particular $\moddmass$ (over $\mathbb{Q}$) is a chain model for the homology of $\kb$ (which is homotopy equivalent to the moduli space of smooth Klein surfaces) just as $\moddass$ is for $\nb$. It is from this that we obtain a M\"obius graph decomposition of moduli spaces of Klein surfaces which is a direct analogue to the ribbon graph decomposition of moduli spaces of Riemann surfaces.

In the second case we find that $\moddmass/(a=1)$ (where $a\in\dmass(1)=\mathbb{Q}[\mathbb{Z}_2]$ is the involution) is a chain model for the homology of $\mb$. This gives a `dianalytic ribbon graph' decomposition of the partial compactification $\mb$. This partial compactification is quite different from the other. For example unlike $\moddmass$ the quotient has non-trivial homology in genus $0$.

Later we give a concrete explanation of the graph complexes obtained for each of these spaces and the corresponding isomorphisms on homology without using the language of operads.

We will finish this outline with a few words about the proof of these results. It is important to note that the proof of the results by Costello \cite{costello2} transfers easily to our situation and as such we reference \cite{costello2} heavily. This is not that surprising since we have already mentioned that we can form the orienting double, which is a Riemann surface without boundary, which are the objects considered in \cite{costello1,costello2}. In addition hyperbolic geometry features heavily in the proof and the same methods apply directly to Klein surfaces (which again can be seen by considering an appropriate double).

We begin by reviewing the necessary definitions and theory of Klein surfaces following Alling and Greenleaf \cite{allinggreenleaf} and Liu \cite{liu}, with some modifications.

\section{Klein surfaces and symmetric Riemann surfaces}
Let $D$ be a non-empty open subset of $\mathbb{C}$ and $f\co D\rightarrow\mathbb{C}$ be a smooth map. Recall $f$ is analytic on $D$ if $\frac{\partial f}{\partial \bar{z}}=0$ and anti-analytic if $\frac{\partial f}{\partial z}=0$. We say $f$ is \emph{dianalytic} if its restriction to each component of $D$ is either analytic or anti-analytic. If $A$ and $B$ are any non-empty subsets of $\mathbb{C}^+$ (the upper half plane) we say a function $g\co A\rightarrow B$ is analytic (or anti-analytic) on $A$ if it extends to an analytic (respectively anti-analytic) function $g'\co U\rightarrow\mathbb{C}$ where $U$ is an open neighbourhood of $A$. Once again we call $g$ dianalytic if its restriction to each component of $A$ is either analytic or anti-analytic.

For us a surface is a compact and connected (unless otherwise stated) topological manifold of dimension $2$. Our surfaces can have a boundary. Recall that a smooth structure on a surface is determined by a smooth atlas (an atlas $\mathcal{A}$ such that all the transition functions of $\mathcal{A}$ are smooth) and similarly an analytic structure is given by an atlas such that all transition functions are analytic. A Riemann surface\footnote{When we use the term Riemann surface we are allowing surfaces possibly with non-empty boundary.} is a surface with an analytic structure and morphisms of Riemann surfaces are non-constant analytic maps (maps that are analytic on coordinate charts) that restrict to maps on the boundary. In order to bring our definitions closer to \cite{liu} we refer to Riemann surfaces with non-empty boundary as bordered Riemann surfaces. A Riemann surface is canonically oriented by its analytic structure.

\begin{definition}
An atlas $\mathcal{A}$ on a surface $K$ is dianalytic if all the transition functions of $\mathcal{A}$ are dianalytic. A \emph{dianalytic structure} on $K$ is a maximal dianalytic atlas. A \emph{Klein surface} is a surface equipped with a dianalytic structure.
\end{definition}

An analytic structure can be extended to a dianalytic structure and so a Riemann surface can be viewed as a Klein surface. In doing so we no longer have a canonical orientation. Klein surfaces in general need not be orientable. It is shown in \cite{allinggreenleaf} that every compact surface can carry a dianalytic structure.

\begin{definition}
A morphism between Klein surfaces $K$ and $K'$ is a non-constant continuous map $f\co (K,\partial K)\rightarrow(K',\partial K')$ such that for all $x\in K$ there are charts $(U,\phi)$ and $(V,\psi)$ around $x$ and $f(x)$ respectively and an analytic function $F\co \phi(U)\rightarrow\mathbb{C}$ such that the following diagram commutes:
\[\xymatrix{
U\ar[rr]^f\ar[d]^\phi & & V\ar[d]^\psi \\
\phi(U)\ar[r]^F & \mathbb{C}\ar[r]^\Phi & \mathbb{C}^+
}\]
Here $\Phi(x+iy)=x+i|y|$ and is called the \emph{folding map}. We call $f$ a \emph{dianalytic morphism} if we can choose charts so that $\Phi\circ F$ in the above diagram is dianalytic.
\end{definition}

\begin{remark}
Note that when we consider Riemann surfaces as Klein surfaces morphisms of Riemann surfaces can be thought of as morphisms of Klein surfaces. Note also that morphisms of Klein surfaces are not always dianalytic since we are allowing maps which `fold' along the boundary of $K'$. This is useful since, for example, it means that the category of Klein surfaces is the correct domain for the complex double (see Alling and Greenleaf \cite{allinggreenleaf}) and other quotients and that the category of Klein surfaces is equivalent to the category of symmetric Riemann surfaces without boundary (and then Klein surfaces are real algebraic curves, again see \cite{allinggreenleaf}). If $K,K'$ have no boundary then morphisms between them are dianalytic.
\end{remark}

A morphism $f$ is dianalytic if and only if $f^{-1}(\partial K')=\partial K$. The composition of two dianalytic morphisms is dianalytic.

\begin{definition}
A \emph{symmetric Riemann surface}\footnote{Note again that this is a slightly different definition to that which is normally found elsewhere since we are allowing our Riemann surfaces to have a boundary unless otherwise stated.} $(X,\sigma)$ is a Riemann surface with an antiholomorphic involution $\sigma\co X\rightarrow X$ (which of course restricts to the boundary if our surface is bordered). A morphism $f\co (X,\sigma)\rightarrow(X',\sigma')$ is a morphism of Riemann surfaces such that $f\circ\sigma=\sigma'\circ f$. By convention we allow symmetric Riemann surfaces to be disconnected provided the quotient surface $X/\sigma$ is connected.
\end{definition}

Given a symmetric Riemann surface $(X,\sigma)$ the quotient surface $X/\sigma=K$ has a unique dianalytic structure such that the quotient map $q$ is a morphism of Klein surfaces, see \cite{allinggreenleaf}. Again $q^{-1}(\partial K)=\partial X$ if and only if $q$ is a dianalytic morphism of Klein surfaces. In this case we call $(X,\sigma)$ a \emph{dianalytic symmetric Riemann surface}. 

\begin{definition}
\Needspace*{3\baselineskip}\mbox{}
\begin{itemize}
\item The category $\klein$ has objects Klein surfaces with morphisms as defined above.
\item The category $\dklein$ has objects Klein surfaces with just the dianalytic morphisms.
\item The category $\symriem$ has objects symmetric Riemann surfaces without boundary and morphisms analytic maps as defined above.
\item The category $\dsymriem$ has objects dianalytic symmetric Riemann surfaces (possibly with boundary) and morphisms analytic maps as defined above.
\end{itemize}
\end{definition}

To understand the category of Klein surfaces better we recall the existence of the orienting double of a Klein surface.

\begin{lemma}\label{lem:odouble}
\Needspace*{3\baselineskip}\mbox{}
\begin{itemize}
\item Let $K$ be a Klein surface. Then there exists a Riemann surface $K_\mathbf{O}$ and a morphism $f\co K_\mathbf{O}\rightarrow K$ such that $f^{-1}(\partial K)=\partial K_\mathbf{O}$ (so $f$ is dianalytic) and $K_\mathbf{O}$ is universal with respect to this property. This means if $X$ is a Riemann surface and $h\co X\rightarrow K$ is a morphism with $h^{-1}(\partial K)=\partial X$ then there is a unique analytic morphism $g\co X\rightarrow K_\mathbf{O}$ such that $h=f\circ g$ (more succinctly $K_\mathbf{O}$ is the universal Riemann surface over $K$ in the category $\dklein$). In particular this means $K_\mathbf{O}$ is unique up to unique isomorphism. We call it the \emph{orienting double} of $K$.
\item The map $f\co K_\mathbf{O}\rightarrow K$ is a double cover.
\item $K_\mathbf{O}$ has an antiholomorphic involution $\sigma$ such that $f\circ \sigma=f$.
\item Any double cover $h\co X\rightarrow K$ admitting such an involution and satisfying the property $h^{-1}(\partial K)=\partial X$ is universal with respect to this property (and hence is uniquely isomorphic to $K_\mathbf{O}$ as a double cover).
\item The map $f$ is unramified, $\sigma$ is unique and $K_\mathbf{O}$ is disconnected if and only if $K$ is orientable.
\end{itemize}
\end{lemma}

\begin{proof}
This is just a slight rewording of Alling and Greenleaf \cite[Theorem 1.6.7]{allinggreenleaf}.
\end{proof}
 
If $K$ is a Klein surface then $(K_\mathbf{O},\sigma)$ is a dianalytic symmetric Riemann surface, bordered if and only if $\partial K \neq \emptyset$. Given a dianalytic morphism of Klein surfaces it lifts to a morphism of dianalytic symmetric Riemann surfaces. This defines a functor from  Klein surfaces with dianalytic morphisms to dianalytic symmetric Riemann surfaces. Given a dianalytic symmetric Riemann surface $(X,\sigma')$ then $(X,q)\cong((X/\sigma')_\mathbf{O},f)$. Since $f$ is unramified then maps of dianalytic symmetric Riemann surfaces give dianalytic maps of the underlying Klein surfaces. In particular we can deduce:

\begin{proposition}\label{prop:odoubleequiv}
There is an equivalence of categories $\dklein\rightarrow\dsymriem$ given by taking the orienting double.
\qed
\end{proposition}

Given a Klein surface $K$ we can also construct the \emph{complex double} $K_\mathbf{C}$ of a Klein surface $K$. The complex double $K_\mathbf{C}$ is a symmetric Riemann surface without boundary that is disconnected if and only if $K$ is orientable and has empty boundary. In particular for an orientable surface of genus $g$ with $h$ boundary components it is obtained by taking two copies of the surface with opposite orientations and gluing along the boundary in an orientation preserving manner to give a symmetric Riemann surface without boundary of genus $2g+h-1$, with the antiholomorphic involution switching the two copies. For an unorientable surface with $g$ handles, $u$ crosscaps and $h$ boundary components the complex double is a connected symmetric Riemann surface without boundary of genus $2g+h+u-1$ although the construction in this case is less simple to describe and we refer to Alling and Greenleaf \cite[Theorem 1.6.1]{allinggreenleaf} for full details. In particular, similar to the orienting double, the complex double can in fact be realised as the universal Riemann surface without boundary over $K$ in the category $\klein$. It is then not hard to follow the same process as above and show the well known result:

\begin{proposition}\label{prop:cdoubleequiv}
There is an equivalence of categories $\klein\rightarrow\symriem$ given by taking the complex double.
\qed
\end{proposition}

\begin{remark}
The categories $\klein$ and $\dklein$ have the same objects and will also have the same moduli spaces (which can be identified with those of symmetric Riemann surfaces without boundary by \autoref{prop:cdoubleequiv}). As mentioned in the outline of this section the difference becomes much more noticeable when we consider nodal Klein surfaces, which are then no longer equivalent to nodal symmetric Riemann surfaces without boundary. However nodal Klein surfaces (with just dianalytic morphisms) are still equivalent to nodal symmetric Riemann surfaces possibly with boundary. Therefore we will actually obtain two different partial compactifications of moduli spaces.
\end{remark}

Given a Klein or Riemann surface whose underlying surface has $g$ handles, $0\leq u\leq 2$ crosscaps and $h$ boundary components we define the topological type to be $(g,u,h)$.

\subsection{Klein surfaces and hyperbolic geometry}
Recall that a connected hyperbolic Riemann surface without boundary admits a unique complete hyperbolic metric. If $X$ is a bordered Riemann surface whose complex double $X_\mathbf{C}$ (which is connected without boundary) is hyperbolic, then the antiholomorphic involution on $X_\mathbf{C}$ is an isometry with respect to the unique hyperbolic metric on $X_\mathbf{C}$ and so we can construct a unique (up to conformal isometry) hyperbolic metric on $X$, compatible with the analytic structure, such that the boundary (which corresponds to the fixed points of the involution) is geodesic.

If $K$ is a Klein surface whose complex double is hyperbolic then we can repeat this construction by taking the unique complete hyperbolic metric on $K_\mathbf{C}$. Therefore $K$ has a unique (up to isometry) hyperbolic metric, compatible with the dianalytic structure, such that the boundary is geodesic. Dianalytic morphisms of Klein surfaces correspond to conformal maps on hyperbolic surfaces. Since our surfaces are now unoriented by conformal maps we mean maps which preserve angles (as opposed to oriented angles).

The only Klein surfaces with a non-hyperbolic complex double are those of topological type $(g,u,h)$ with $2g+h+u-2 \leq 0$, since the complex double has topological type $(2g+h+u-1,0,0)$. The only such surfaces with $h> 0$ are the disc, the annulus and the M\"obius strip.

Let $K$ be a Klein surface. Then since $K_\mathbf{O}$ covers $K$ we can pull back the hyperbolic metric on $K$ such that the involution on $K_\mathbf{O}$ is an isometry. Analytic maps between orienting double covers correspond to conformal maps of double covers and the boundary of $K_\mathbf{O}$ is geodesic and this is the same metric inherited from $(K_\mathbf{O})_\mathbf{C}$.

Using the hyperbolic metric on a Klein surface $K$ we can adapt the methods outlined by Costello in \cite{costello1} and elaborated upon in \cite{costello2} to construct a deformation retract on the moduli space of Klein surfaces which we will do below.

\section{Nodal Klein surfaces}
We will need to allow Riemann surfaces and Klein surfaces with certain nodes and marked points.

A singular topological surface $(X,N)$ is a Hausdorff space $X$ with a discrete set $N\subset X$ (the set of singularities) such that $X\setminus N$ is a topological surface. As usual such surfaces will be compact (so $N$ will be finite) and connected and may have boundary unless otherwise stated. The boundary of a singular surface is defined to be the boundary of $X\setminus N$.

\begin{definition}
Let $(X,N)$ be a singular surface. A \emph{boundary node} is a singularity $z\in N$ with a neighbourhood homeomorphic to a neighbourhood of $(0,0)\in\{(x,y)\in(\mathbb{C}^+)^2 : xy=0\}$ such that $z\mapsto (0,0)$. Similarly an \emph{interior node} is a singularity with a neighbourhood homeomorphic to $(0,0)\in\{(x,y)\in\mathbb{C}^2 : xy=0\}$. We set $N_1$ to be the set of interior nodes and $N_2$ to be the set of boundary nodes. If $X$ has only nodal singularities then an atlas on $X$ is given by charts on $X\setminus N$ together with charts at the nodes as described. We call a singular surface with only nodal singularities a \emph{nodal surface}.
\end{definition}

Let $B=\{(x,y)\in(\mathbb{C}^+)^2 : xy=0\}$ and $B^*=B\setminus (0,0)$. Let $I=\{(x,y)\in\mathbb{C}^2 : xy=0\}$ and $I^*=I\setminus (0,0)$. Regarding the smooth curves $B^*$ and $I^*$ as Riemann surfaces with boundary we have a notion of analytic and anti-analytic maps to or from subsets of $B^*$ and $I^*$. We say a map to or from a neighbourhood $U$ of $(0,0)\in B$ or $U'$ of $(0,0)\in I$ is analytic or anti-analytic if it is analytic or anti-analytic when restricted to $U\cap B^*$ or $U'\cap I^*$. Dianalytic maps are again maps which restrict to analytic or anti-analytic maps on each connected component. Note in particular that if $U$ or $U'$ is connected then dianalytic maps on $U$ or $U'$ are either analytic or anti-analytic everywhere (even though $U\cap B^*$ and $U\cap I^*$ are disconnected). We therefore have a notion of a transition function between two charts on a nodal surface being analytic or dianalytic.

\begin{definition}
A \emph{nodal Riemann surface} is a nodal surface $(X,N)$ together with a maximal analytic atlas. A \emph{nodal Klein surface} is a nodal surface $(K,N)$ together with a maximal dianalytic atlas. By an irreducible component of a nodal surface we mean a connected component of the surface obtained by pulling apart all the nodes. Note that this is different from a connected component of $K\setminus N$ since an irreducible component will include points that were formerly nodes.
\end{definition}

We will mostly be concerned with Klein surfaces having only boundary nodes. We will also need a second different notion of a boundary node on a Klein surface.

\begin{definition}
A \emph{na\"ive nodal Klein surface} is a nodal surface $(K,N)$ with only boundary nodes together with a maximal dianalytic atlas on each irreducible component.
\end{definition}

Note that this does differ from the previous notion in the sense that we no longer have a dianalytic structure around the boundary nodes. Indeed, on a neighbourhood of a boundary node there are charts on the intersection with each irreducible component, but not a chart on the whole neighbourhood.

A morphism of nodal Riemann surfaces is a non-constant continuous map that is analytic on the charts (including at nodes). We can also define morphisms easily for na\"ive nodal Klein surfaces: a morphism is given by a non-constant continuous map which takes irreducible components to irreducible components and such that the map induced on each irreducible component is a morphism of smooth Klein surfaces. 

\begin{definition}
A \emph{nodal symmetric Riemann surface} $(X, \sigma)$ is a nodal Riemann surface with an antiholomorphic involution $\sigma\co X\rightarrow X$.
\end{definition}

We will now discuss quotients of nodal Riemann surfaces informally. Given a nodal symmetric Riemann surface $(X,\sigma)$ we can form the quotient $q\co X\rightarrow X/\sigma$. This is a topological nodal surface and each irreducible component has a canonical non-singular dianalytic structure. Let $n\in N_1$ be an interior node. Since $\sigma$ must take nodes to nodes, if $\sigma(n)\neq n$ then $q(n)$ will be an interior node and we can extend the dianalytic structure about $q(n)$ in a unique way such that $q$ is dianalytic on the charts about $n$ and $q(n)$. If $n$ is fixed by $\sigma$ then $q(n)$ will be either an interior point or a boundary node. In the second case there is not a canonical way of choosing a dianalytic structure about $q(n)$ and so $q(n)$ is a na\"ive boundary node. Let $n'\in N_2$ be a boundary node. Similarly $q(n')$ will be either a boundary node or, if $n'$ is fixed by $\sigma$, a boundary point. In either case we can choose a dianalytic structure about $q(n')$ in a canonical way. Since we are interested mainly in Klein surfaces with only boundary nodes we make the following definition:

\begin{definition}
An \emph{admissible symmetric Riemann surface} $(X,\sigma)$ is a nodal symmetric Riemann surface $(X,N)$ such that $q(n)$ is a boundary node (a na\"ive boundary node if $n$ is an interior node) for all nodes $n\in N$.
\end{definition}

It is now not too difficult to work out what a morphism of nodal Klein surfaces should be, allowing folding maps along nodes as just described. Then the above informal discussion can be made precise by saying that there is a unique structure of a nodal Klein surface (possibly with some na\"ive nodes) on the quotient of a nodal Riemann surface such that the quotient map is a morphism of Klein surfaces. We will not give the details however since we only really need to consider dianalytic morphisms here. A dianalytic morphism $f\co K\rightarrow K'$ of nodal Klein surfaces is a non-constant continuous map that is dianalytic on all the charts (including at nodes). In particular such a map induces dianalytic maps on the irreducible components. A dianalytic nodal symmetric Riemann surface is an admissible symmetric Riemann surface such that the quotient map is dianalytic. In particular such surfaces have only boundary nodes.

From now on all our surfaces may be nodal unless otherwise stated. We need nodal surfaces with marked points.

\begin{definition}
A \emph{Klein/Riemann surface with $n$ marked points} $(X,\mathbf{p})$ is a nodal Klein/Riemann surface $(X,N)$ equipped with an ordered $n$--tuple $\mathbf{p}=(p_1,\dots,p_n)$ of distinct points on $X\setminus N$. A morphism $f\co(X,\mathbf{p})\rightarrow(X',\mathbf{p'})$ of surfaces with $n$ marked points is a morphism of the underlying surface such that $f(p_i)=p'_i$.
\end{definition}

\begin{definition}
\Needspace*{3\baselineskip}\mbox{}
\begin{itemize}
\item A \emph{symmetric Riemann surface $X$ with $(m,n)$ marked points} $(X,\sigma,\mathbf{p},\mathbf{p'})$ is a nodal symmetric Riemann surface $(X,\sigma)$ with an ordered $2m$--tuple $\mathbf{p}=(p_1,\dots,p_{2m})$ of distinct points on $X\setminus N$ such that $\sigma(p_i)=p_{m+i}$ for $i=1,\dots,m$ and an ordered $n$--tuple $\mathbf{p'}=(p'_1,\dots,p'_{n})$ of distinct points on $X\setminus N$ such that $\sigma(p'_j)=p'_j$ for $j=1,\dots,n$.
\item A morphism $f\co (X,\sigma,\mathbf{p},\mathbf{p'})\rightarrow(Y,\tau,\mathbf{r},\mathbf{r'})$ is a morphism of the underlying symmetric Riemann surfaces such that $f(p_i)=r_i$ and $f(p'_j)=r'_j$.
\item We say a marked symmetric Riemann surface is \emph{admissible} if the underlying symmetric Riemann surface is and the points $q(p_i)$ and $q(p'_j)$ all lie in the boundary of the quotient. In this case all the $p_i$ must be on the boundary.
\end{itemize}
\end{definition}

Once again we can discuss quotients of marked symmetric Riemann surfaces. The quotient Klein surface is in a natural way a Klein surface with $m+n$ marked points (and if the surface is admissible all the marked points of the Klein surface lie on the boundary). In fact it has more structure: if $p_i$ is a marked point of a symmetric Riemann surface with $\sigma(p_i)=p_{m+i}\neq p_i$ then this gives locally an orientation about $q(p_i)$ induced from the chart about $p_i$. This motivates the following definition:

\begin{definition}
A \emph{Klein surface with $n$ oriented marked points} is a Klein surface with marked points $(K,\mathbf{p})$ equipped with a choice of orientation locally about each marked point (more precisely, a choice of one of the two germs of orientations on orientable neighbourhoods at each marked point).
\end{definition} 

Note finally that dianalytic marked symmetric Riemann surfaces (which are by definition admissible) can only have marked points on the boundary.

We are now ready to define our categories of interest. In particular we are interested in Klein surfaces with nodes and marked points all on the boundary. Equivalently this means we are also interested in admissible symmetric Riemann surfaces.

\begin{definition}
\Needspace*{3\baselineskip}\mbox{}
\begin{itemize}
\item The category $\nklein$ has objects Klein surfaces with only na\"ive boundary nodes and marked points (not oriented) on the boundary with morphisms as defined above for na\"ive nodal surfaces.
\item The category $\dnklein$ has objects Klein surfaces with only boundary nodes (but not na\"ive nodes) and oriented marked points on the boundary with dianalytic morphisms as defined above.
\item The category $\nsymriem$ has objects admissible symmetric Riemann surfaces with marked points and without boundary. Morphisms are analytic maps as defined above.
\item The category $\dnsymriem$ has objects dianalytic symmetric Riemann surfaces (possibly with boundary) with marked points. Morphisms are analytic maps as defined above.
\end{itemize}
\end{definition}

We can extend the notion of an orienting double to objects $(K,\mathbf{p})$ in $\dnklein$ by first constructing the orienting double on each irreducible component and gluing in the canonical way induced by the dianalytic structure at the nodes to obtain a dianalytic symmetric Riemann surface $K_\mathbf{O}$. If $f\co K_\mathbf{O}\rightarrow K$ is the covering map then $f^{-1}(p_i)$ gives two points in $K_\mathbf{O}$. To make $K_\mathbf{O}$ a marked surface we need to order these two points for each $i$. But we can do this using the local orientation about $p_i$ which allows us to canonically choose an $n$--tuple $\mathbf{q}=(q_1,\dots,q_n)$ of distinct points on the boundary of $K_\mathbf{O}$ such that $f(q_i)=p_i$ and $f$ preserves the local orientations about $q_i$ and $p_i$. Then the orienting double of $K$ is defined as $(K_\mathbf{O},\sigma,\mathbf{p'},\mathbf{0})$ where $\mathbf{p'}=(q_1,\dots,q_n,\sigma(q_1),\dots,\sigma(q_n))$ and $\sigma$ is the antiholomorphic involution on $K_\mathbf{O}$. This is an object in $\dnsymriem$.

Given a Klein surface with marked (but not necessarily oriented) points choosing such an $n$--tuple $\mathbf{q}$ of points in $K_\mathbf{O}$ is clearly equivalent to providing local orientations at each $p_i$. Since this data can sometimes be easier to work with we will therefore also denote a Klein surface with $n$ oriented marked points by $(K,\mathbf{p},\mathbf{q})$.

\begin{remark}\label{rem:riemklein}
 A Riemann surface with marked points can be thought of as a Klein surface with oriented marked points using the canonical orientation. If it is in $\dnklein$ then its orienting double is a disjoint union of two copies of itself  and so the canonical orientation means we choose points $q_i$ in the component that maps analytically under the quotient map.
\end{remark}

By showing nodal versions of the properties in \autoref{lem:odouble} it is not too difficult to obtain the following marked nodal analogue of \autoref{prop:odoubleequiv}:

\begin{proposition}\label{prop:odoubleequiv2}
There is an equivalence of categories $\dnklein\rightarrow\dnsymriem$ given by taking the orienting double.
\qed
\end{proposition}

Similarly given an object in $\nklein$ we can construct the complex double by first constructing the complex double on each irreducible component and then gluing to obtain an admissible symmetric Riemann surface $K_\mathbf{C}$ without boundary. If $f\co K_\mathbf{C}\rightarrow K$ is the covering map then $f^{-1}(p_i)$ is a single point in $K_\mathbf{C}$ so we obtain an object in $\nsymriem$. Once again we can show the marked nodal analogue of \autoref{prop:cdoubleequiv}:

\begin{proposition}\label{prop:cdoubleequiv2}
There is an equivalence of categories $\nklein\rightarrow\nsymriem$ given by taking the complex double.
\qed
\end{proposition}

\begin{remark}
We can still define the complex double for surfaces in $\dnklein$ however two Klein surfaces that are not isomorphic may have isomorphic complex doubles.
\end{remark}

Given a boundary node on a Riemann surface $X$ we can replace the node with a narrow oriented strip. We can also replace interior nodes with a narrow oriented cylinder. In this way we can obtain from $X$ a non-singular oriented topological surface.

We define the topological type of a Riemann surface $X$ with $n$ marked points as $(g,0,h,n)$ where $(g,0,h)$ is the topological type of the non-singular oriented surface obtained by the above process.

Given a Klein surface $K$ in $\dnklein$ with $n$ marked points we consider $K_\mathbf{O}$ as a Riemann surface by forgetting the symmetry and let $(\tilde{g},0,\tilde{h},2n)$ be its topological type. Then if $K_\mathbf{O}$ is disconnected the topological type of $K$ is defined as the topological type of one of the connected components of $K_\mathbf{O}$. If $K_\mathbf{O}$ is connected then the topological type of $K$ is defined as $(g,u,h,n)$ where $h=\frac{\tilde{h}}{2}$ and $g$ and $u$ are the unique solutions to $\tilde{g}+1=2g+u$ with $0<u\leq 2$.

For admissible symmetric Riemann surfaces without boundary in $\nsymriem$ we define their topological type as that of the underlying marked Riemann surface obtained by forgetting the symmetry.

The topological type of a symmetric Riemann surface in $\dnsymriem$ is defined as the topological type of its quotient Klein surface in $\dnklein$ and the topological type of a na\"ive nodal Klein surface in $\nklein$ is defined as the topological type of its complex double in $\nsymriem$.

\begin{definition}
A Klein or Riemann surface with $n$ (possibly oriented) marked points is \emph{stable} if it has only finitely many automorphisms.
\end{definition}

A non-singular Klein surface with $n$ (possibly oriented) marked points on the boundary (which we will assume is non-empty) is unstable precisely if it has the topological type of a disc and $n\leq 2$ or if it is an annulus with $n=0$ or a M\"obius strip with $n=0$ (since the orienting double of a M\"obius strip is an annulus). If the Klein surface has singularities (and so is in either $\nklein$ or $\dnklein$) then it is stable if and only if each connected component of its normalisation is. The normalisation is given by pulling apart all the nodes where each node gives two extra boundary marked points. It does not matter how we order these extra marked points.

\section{Moduli spaces of Klein surfaces}
In this section we discuss various moduli spaces and their relationships.

Let $\kb_{g,u,h,n}$ be the moduli space of stable Klein surfaces in $\dnklein$ with topological type $(g,u,h,n)$ and $h\geq 1$. Let $\ko_{g,u,h,n}\subset \kb_{g,u,h,n}$ be the subspace of non-singular Klein surfaces.

Due to \autoref{prop:odoubleequiv2} this can be identified with the moduli space of stable dianalytic symmetric Riemann surfaces in $\dnsymriem$. These moduli spaces are non-empty except for the cases when \[(g,u,h,n)\in\{(0,0,1,0),(0,0,1,1),(0,0,1,2),(0,0,2,0),(0,1,1,0)\}.\] There is an action of the group $\mathbb{Z}_2^{\times n}$ on $\kb_{g,u,h,n}$ given by flipping the orientations of marked points.

Let $\mb_{\tilde{g},n}$ be the moduli space of stable admissible symmetric Riemann surfaces in $\nsymriem$ with topological type $(\tilde{g},0,0,n)$. Let $\mo_{\tilde{g},n}\subset\mb_{\tilde{g},n}$ be the subspace of non-singular Riemann surfaces.

Due to \autoref{prop:cdoubleequiv2} this can be identified with the moduli space of stable Klein surfaces in $\nklein$. These moduli spaces are non-empty except for the cases when $(g,n)\in\{(0,0),(0,1),(0,2),(1,0)\}$. We observe that:
\[\left ( \coprod_{2g+h+u-1=\tilde{g}}\ko_{g,u,h,n}\right )\Biggr/\mathbb{Z}_2^{\times n} \quad\cong\quad \mo_{\tilde{g},n}\]

\begin{remark}
The slight abuse of notation here is potentially misleading. The full compactification of stable symmetric Riemann surfaces (the space of stable real algebraic curves) allowing all nodal Riemann surfaces without boundary, is very often denoted by $\mb$ (for example as in \cite{seppala,liu}). For us however it is an open subspace of this: the subspace of \emph{admissible} surfaces. Here neither $\mb_{\tilde{g},n}$ or $\kb_{g,u,h,n}$ are compact in general.
\end{remark}

Let $\nb_{g,h,n}$ be the moduli space of stable bordered Riemann surfaces with only boundary nodes and marked points on the boundary with topological type $(g,0,h,n)$. Let $\no_{g,h,n}\subset \nb_{g,h,n}$ be the subspace of non-singular bordered Riemann surfaces.

These are the moduli spaces considered by Costello \cite{costello1,costello2}. These spaces are non-empty except for the cases when \[(g,h,n)\in\{(0,1,0),(0,1,1),(0,1,2),(0,2,0)\}.\] By \autoref{rem:riemklein} we have a map $\nb_{g,h,n}\rightarrow\kb_{g,0,h,n}$, injective for $n>0$.

We now outline the construction of these spaces and their topology. We will follow Liu \cite{liu} closely and more details may be found there. The basic idea is to use the symmetric pants decomposition of the complex double to obtain Fenchel--Nielsen coordinates. In fact this will give an orbifold structure.

Recall that any stable Riemann surface $X$ of topological type $(\tilde{g},0,0,n)$ admits a \emph{pants decomposition}. More precisely there are $3\tilde{g} - 3 + n$ disjoint curves $\alpha_i$ on the surface which is obtained from puncturing $X$ at each marked point, with each curve being either a closed geodesic (in the hyperbolic metric) or a node, decomposing $X$ into a disjoint union of $2\tilde{g} - 2 + n$ pairs of pants whose boundary components are either one of the $\alpha_i$ or a puncture corresponding to a marked point. Furthermore if $(X,\sigma)$ is a stable symmetric Riemann surface then there exists a \emph{symmetric pants decomposition} that is invariant with respect to $\sigma$ (see Buser and Sepp\"al\"a \cite{buserseppala}). By this we mean that $\sigma$ induces a permutation on decomposing pairs of pants. 

We call a pants decomposition oriented if the pairs of pants are ordered and the boundary components of each pair of pants are ordered and each have a basepoint and each decomposing curve is oriented. This induces an ordering on the decomposing curves and so defines \emph{Fenchel--Nielsen coordinates} \[(l_1,\dots,l_{3\tilde{g}-3+n},\theta_1,\dots,\theta_{3\tilde{g}-3+n})\] where the $l_i$ are the lengths (in the hyperbolic metric) of the decomposing curves, and the $\theta_i$ are the angles between the basepoints of the two boundary components corresponding to the $i$--th decomposing curve. More precisely the ordering of the pairs of pants determines an ordering of the two basepoints on each decomposing curve and we set $\theta_i=2\pi\frac{\tau_i}{l_i}$ where $\tau_i$ is the distance one travels from the first basepoint to the second basepoint on the $i$--th decomposing curve in the direction that the curve is oriented. Note that we have $l_i\geq 0$ and $0\leq \theta_i < 2\pi$.

In the case that the pants decomposition is symmetric we may assume that the orientation of the pants decomposition has been chosen so that the symmetry permutes basepoints and reverses the orientation of decomposing curves which are not completely fixed by the symmetry. Note that in this case not all the coordinates are independent. In particular if a decomposing curve $\alpha_i$ is mapped to itself by the symmetry, then $\theta_i = 0$ or $\theta_i=\pi$. Similarly if the decomposing curves $\alpha_i$ and $\alpha_j$ are permuted by the symmetry then $l_i=l_j$ and $\theta_i=2\pi-\theta_j$.

\begin{definition}
A \emph{strong deformation} from a stable symmetric Riemann surface $(X',N',\sigma')$ to a stable symmetric Riemann surface $(X,N,\sigma)$ both of topological type $(\tilde{g},0,\tilde{h},n)$ is a continuous map $\kappa\co (X',N',\sigma')\rightarrow (X,N,\sigma)$ such that
\begin{itemize}
\item $\kappa$ takes boundary components to boundary components, interior nodes to interior nodes, boundary nodes to boundary nodes, preserves marked points and $\kappa\circ \sigma' = \sigma \circ\kappa$
\item for each interior node $n$ we have that $\kappa^{-1}(n)$ is either an interior node or an embedded circle in a connected component of $X'\setminus N'$
\item for each boundary node $n$ we have that $\kappa^{-1}(n)$ is either a boundary node or an embedded arc in a connected component of $X'\setminus N'$ with ends in $\partial X'$
\item $\kappa$ restricts to a diffeomorphism $\kappa\co\kappa^{-1}(X\setminus N)\rightarrow X\setminus N$.
\end{itemize}
\end{definition}

Given $X$ and $X'$ stable symmetric Riemann surfaces without boundary with oriented symmetric pants decompositions having decomposing curves $\alpha_i$ and $\alpha_i'$ respectively and a strong deformation $\kappa\co X'\rightarrow X$ we say that $\kappa$ is compatible with the oriented pants decompositions if $\kappa(\alpha_i')=\alpha_i$ and all the orientation data (the ordering of the pants, the ordering of the boundary components, the basepoints of the boundary components and the orientation of each $\alpha_i$) is preserved under $\kappa$. 

Given an oriented symmetric pants decomposition of $X$ with decomposing curves $\alpha_i$ and any strong deformation $\kappa\co X'\rightarrow X$, by choosing as decomposing curves the closed geodesics homotopic to each of the $\kappa^{-1}(\alpha_i)$ we can obtain a symmetric pants decomposition of $X'$ with decomposing curves $\alpha_i'$. Further there exists another strong deformation $\kappa'$ such that $\kappa'(\alpha_i')=\alpha_i$ and also an orientation of the pants decomposition of $X'$ so that it is pulled back from the oriented pants decomposition of $X$ under $\kappa'$. In particular $\kappa'$ is compatible with the oriented pants decompositions.

Recall that the complex structure on a pair of pants is determined, up to equivalence, by the lengths of the three boundary curves in the unique hyperbolic metric where the boundary curves are geodesic. Recall also that gluing pairs of pants along boundary components of common length is determined completely by the angle between two basepoints on the boundary components. Therefore if $X$ and $X'$ have the same Fenchel--Nielsen coordinates with respect to the pants decompositions preserved by $\kappa'$ then they are in fact biholomorphic.

\begin{definition}
A strong deformation from a stable Klein surface $(K',N')$ in $\dnklein$ to a stable Klein surface $(K,N)$ of topological type $(g,u,h,n)$ is a strong deformation between the orienting doubles.
\end{definition}

Note that a strong deformation of stable Klein surfaces induces a strong deformation on the complex doubles of the underlying na\"ive nodal Klein surfaces. Of course the converse is not true.

We are now ready to describe the topology on $\mb_{\tilde{g},n}$ and $\kb_{g,u,h,n}$.

Given a surface $X\in\mb_{\tilde{g},n}$ and an oriented symmetric pants decomposition of $X$ with coordinates $l_i,\theta_j$ denote by $U(X,\epsilon,\delta)$ the set of surfaces $X'$ with an oriented symmetric pants decomposition having coordinates $l_i',\theta_j'$ and admitting a strong deformation $\kappa\co X'\rightarrow X$ compatible with the pants decompositions such that $|l_i'-l_i|<\epsilon$ and $|\theta_j'-\theta_j|<\delta$. The collection $\{U(X,\epsilon,\delta) : X\in\mb_{\tilde{g},n}, \epsilon > 0,\delta > 0 \}$ then generates the topology on $\mb_{\tilde{g},n}$.

Set $z_j=l_j e^{i\theta_j}$ and let $\tilde{U}$ be the fixed locus under the symmetry of $X$, which, up to permutation of coordinates, consists of points of the form \[(z_1,\bar{z}_1,z_2,\bar{z}_2,\dots,\bar{z}_{d_1},x_1,\dots,x_{d_2})\] where $2d_1+d_2=3\tilde{g}-3+n$, $z_i\in \mathbb{C}$ and $x_i\in\mathbb{R}$. This is an open subset of $\mathbb{R}^d$ where $d=3\tilde{g}-3+n$. In particular the open sets $U(X,\epsilon,\delta)$ are homeomorphic to $\tilde{U}/\Gamma$ for an appropriate open subset $\tilde{U}$ of $\mathbb{R}^d$ and $\Gamma$ the automorphism group of $X$. Therefore the space $\mb_{\tilde{g},n}$ is an orbifold.

Similarly given a surface $K\in\kb_{g,u,h,n}$ and an oriented symmetric pants decomposition of $K_\mathbf{C}$ with coordinates $l_i,\theta_j$ denote by $U(K,\epsilon,\delta)$ the set of surfaces $K'$ with an oriented symmetric pants decomposition of $K_\mathbf{C}'$ having coordinates $l_i',\theta_j'$ and admitting a strong deformation $\kappa\co K_\mathbf{O}'\rightarrow K_\mathbf{O}$ compatible with the pants decompositions on the complex doubles such that $|l_i'-l_i|<\epsilon$ and $|\theta_j'-\theta_j|<\delta$. The collection $\{U(K,\epsilon,\delta) : K\in\kb_{g,u,h,n}, \epsilon > 0,\delta > 0 \}$ then generates the topology on $\kb_{g,u,h,n}$.

In this case the open sets are now homeomorphic to $\tilde{U}/\Gamma$ for some open neighbourhood $\tilde{U}$ of $\mathbb{R}^m_{\geq 0}\times \mathbb{R}^{d-m}$ and $\Gamma$ the automorphism group of $K$, where $d=6g+3h+3u+n-6$ and $m$ is the number of nodes of $K$. The value of $d$ is obtained by noting that the complex double of $K$ has topological type $(2g+h+u-1,0,0,n)$, so that the number of Fenchel--Nielsen coordinates of the double is $6(2g+h+u-1)-6+2n$ and again only half of these are independent. However, as discussed previously, the interior nodes of the complex double do not encode the dianalytic structure of the boundary nodes in the quotient, nor are the marked points oriented. As a result there is only one way to smooth a node in $K$ given the dianalytic structure so the coordinates corresponding to (smoothings of) the $m$ nodes lie in $\mathbb{R}^m_{\geq 0}$. In terms of the coordinates on the complex double this corresponds to the fact that a node in the complex double is fixed by the symmetry so as we smooth it the Fenchel--Nielsen coordinate corresponding to the gluing angle is either $0$ or $\pi$. However only one choice is possible if the quotient is to have the correct topological type and orientation of marked points. Therefore the space $\kb_{g,u,h,n}$ is an orbifold with corners. Furthermore an orbifold with corners is homotopy equivalent to its interior which in this case is $\ko_{g,u,h,n}$.

We can also carry out a similar construction for the spaces $\nb_{g,h,n}$. See also \cite{liu,costello1}.

Our discussion can be summarised in the following lemma:

\begin{lemma}
\Needspace*{3\baselineskip}\mbox{}
\begin{itemize}
\item $\nb_{g,h,n}$ is an orbifold with corners of dimension $6g-6+3h+n$. The interior is $\no_{g,h,n}$ and the inclusion $\no_{g,h,n}\hookrightarrow\nb_{g,h,n}$ is then a homotopy equivalence.
\item $\kb_{g,u,h,n}$ is an orbifold with corners of dimension $6g+3h+3u+n-6$. The interior is $\ko_{g,u,h,n}$ and the inclusion $\ko_{g,u,h,n}\hookrightarrow\kb_{g,u,h,n}$  is then a homotopy equivalence.
\item $\mb_{\tilde{g},n}$ is an orbifold of dimension $3\tilde{g}-3+n$.
\end{itemize}
\end{lemma}

\autoref{rem:riemklein} can be taken further. Given a Riemann surface with $n$ marked points together with a colouring of the marked points by $\{0,1\}$ we can map it to a Klein surface with $n$ oriented marked points in $\dnklein$ by choosing the canonical orientation about points coloured by $0$ and the opposite orientation otherwise. Equivalently we choose $q_i$ in the component of the orienting double that maps analytically under the quotient map when $i$ is such that $p_i$ is coloured by $0$ and in the component which maps anti-analytically otherwise.

Two isomorphism classes of Riemann surfaces with $n$ coloured marked points map to the same class of Klein surface precisely when there is an antiholomorphic map between them that reverses the colourings of all the marked points and this map is therefore $2$--to--$1$ for $n>0$. For $n=0$ this map is then $2$--to--$1$ on isomorphism classes except when a Klein surface has isomorphic underlying analytic structures (or equivalently when considering a Riemann surface as a Klein surface its automorphism group is no larger, or equivalently the Riemann surface admits an antiholomorphic automorphism).

In particular for $n>0$ if we restrict to Riemann surfaces where the first marked point is coloured by $0$ then this map is injective.

The following lemma follows from this discussion.

\begin{lemma}\label{lem:riemklein}
There is an isomorphism
\[\left (\bigoplus^{2^{n}}\nb_{g,h,n}\right )\biggr/\sim\quad\longrightarrow\quad \kb_{g,0,h,n}\]
where $\sim$ identifies Riemann surfaces with coloured markings that give rise to the same Klein surfaces with oriented markings. For $n>0$ the left hand side is isomorphic to $\bigoplus^{2^{n-1}}\nb_{g,h,n}$.
\end{lemma}

Let $D_{g,u,h,n}\subset\kb_{g,u,h,n}$ be the subspace consisting of those Klein surfaces whose irreducible components are all discs. Let $D_{g,h,n}\subset\nb_{g,h,n}$ be the corresponding subspace of bordered Riemann surfaces. Let $\dr_{\tilde{g},n}\subset\mb_{\tilde{g},n}$ be the subspace consisting of those admissible Riemann surfaces whose irreducible components are all spheres. Note that when we consider $\mb_{\tilde{g},n}$ as a moduli space of na\"ive nodal Klein surfaces then $\dr_{\tilde{g},n}$ is the subspace consisting of those Klein surfaces whose irreducible components are all discs.

\section{The open KTCFT operad and related operads}

We recall (see \cite{costello1}) that the spaces $\nb_{g,h,n}$ form a modular operad $\nb$ controlling open topological conformal field theory (TCFT). The spaces $D_{g,h,n}$ form a suboperad. Further it was shown by Costello \cite{costello1,costello2} that these spaces are compact orbispaces and admit a decomposition into orbi-cells and if $n>0$ then $D_{g,h,n}$ is an ordinary space instead of an orbispace so we obtain a cell decomposition. Further these orbi-cells are labelled by reduced ribbon graphs.

The collection of spaces $\kb_{g,u,h,n}$ form a topological modular operad $\kb$ by gluing Klein surfaces with oriented marked points such that the orientations are compatible. We can describe the gluing explicitly via the orienting double. Given dianalytic symmetric Riemann surfaces $(X,\sigma,\mathbf{p})$ and $(X',\sigma',\mathbf{p'})$ with $n$ and $m$ marked points respectively we can define an operation gluing along marked points $p_i$ and $q_j$ as follows: we glue the underlying Riemann surfaces at these points and we also glue the points $\sigma(p_i)$ and $\sigma'(q_i)$. Clearly we can use $\sigma$ and $\sigma'$ to define an antiholomorphic involution on the resulting surface which will clearly be dianalytic and will have $n+m-2$ marked points. We also note that the topological type of the resulting surface is the sum of the topological types of $X$ and $X'$. Similarly we can define contractions/self gluings of dianalytic symmetric Riemann surfaces in this way, in which case the resulting topological type either increases the number of boundary components or the number of crosscaps by $1$.

Therefore the space $\kb((\tilde{g},n))$ is the disjoint union of the spaces $\kb_{g,u,h,n}$ with $\tilde{g}=2g+h+u-1$. The group $S_n$ acts by reordering the $n$--tuple of marked points. Composition and contraction is given by gluing the marked point as described above. This gives us the modular operad controlling open Klein topological conformal\footnote{The use of the word `conformal' here is potentially confusing since dianalytic maps correspond to maps preserving angles but not necessarily oriented angles. A conformal map in the presence of the word `Klein' should therefore be understood in this sense.} field theory (KTCFT). The spaces $D_{g,u,h,n}$ form a suboperad which we denote $D$.

We will think of these two operads as extended modular operads by setting $D((0,2))=\kb((0,2))$ to be the discrete group $\mathbb{Z}_2$ which acts on Klein surfaces by switching the orientation of marked points.

Similarly by gluing admissible symmetric Riemann surfaces without boundary at marked points in the natural way the spaces $\mb_{\tilde{g},n}$ form a modular operad $\mb$ and the spaces $\dr_{\tilde{g},n}$ form a suboperad which we denote $\dr$.

\begin{remark}
The gluings of the operad $\mb$ when thought of as gluings of admissible symmetric Riemann surfaces are `closed string' gluings. In this way the operad $\mb$ is closer in spirit to the Deligne--Mumford operad (see, for example, Getzler and Kapranov \cite{getzlerkapranov}).
\end{remark}

\begin{proposition}\label{prop:celldecomp}
The spaces $D_{g,u,h,n}$ admit a decomposition into orbi-cells labelled by reduced M\"obius graphs.
\end{proposition}

\begin{proof}
This is not hard. It is intuitively obvious how we can label a surface by a M\"obius graph but in order to specify the colouring of the edges we need to understand the dianalytic structure about the nodes. It is easiest (although unenlightening) to do this via the orienting double. Let $\mob\Gamma_{g,u,h,n}$ denote the set of reduced M\"obius graphs of topological type $(g,u,h)$ with $n$ legs. Let $(K,\mathbf{p},\mathbf{q})\in D_{g,u,h,n}$. We associate a graph $\gamma(K)\in\mob\Gamma_{g,u,h,n}$ to $K$ as follows: There is one vertex for each irreducible component of $K$, an edge for each node and a leg for each marked point. This yields a graph. We need to specify a ribbon structure and a colouring of the half edges and verify we can do this in a well defined manner. We consider the orienting double $(K_\mathbf{O},f,\sigma)$ and for each irreducible component $A$ of $K$ we choose an irreducible component $\hat{A}$ of $f^{-1}(A)\subset K_\mathbf{O}$. Since $\hat{A}$ is an oriented disc this gives a natural cyclic order on the half edges of $\gamma(K)$. We colour the leg corresponding to $p_i$ by $0$ if $q_i\in\hat{A}$ and by $1$ otherwise. A node $x$ of $K$ lies on the boundary of either $1$ or $2$ irreducible components. If it lies on only $1$ component, $B$ say, then we colour the edge associated to it by $0$ if a preimage of $x$ in $K_\mathbf{O}$ lies only on $\hat{B}$, else we colour the half edges by different colours (by \autoref{rem:reducedmobgraph} it does not matter how we do this). If $x$ lies on the boundary of both $B$ and $C$, then we colour the edge associated to it by $0$ if $\hat{B}$ and $\hat{C}$ intersect at a preimage of $x$ and by different colours otherwise.

This yields a reduced M\"obius graph. We must show it is well defined since we made a choice of irreducible components in $K_\mathbf{O}$. Given an irreducible component $A$ of $K$, if we had chosen the other preimage of $A$ then the cyclic ordering at the corresponding vertex would be reversed and the colouring also reversed. Thus the resulting M\"obius graphs would be isomorphic.

We now show that for any graph $G\in\mob\Gamma_{g,u,h,n}$ the space of surfaces $K\in D_{g,u,h,n}$ with $\gamma(K)=G$ is an orbi-cell. This follows from the topology of the moduli space of discs: let $D$ be a disc with an analytic structure. Then $D$ is holomorphic to the unit disc in the complex plane which has automorphism group $PSL_2(\mathbb{R})$ and so the space of $n\geq 3$ marked points on the unit disc is the configuration space of marked points on $S^1$ modulo $PSL_2(\mathbb{R})$. Further, automorphisms of the unit disc preserve the cyclic ordering of marked points and so this space decomposes into cells labelled by ribbon corollas. As noted in \autoref{lem:riemklein} the moduli space of marked Klein discs can be identified with the moduli space of coloured marked unit discs modulo the action of the anti-analytic map reversing the cyclic ordering of the marked points. The space of coloured marked unit discs decomposes into cells and the action reversing the cyclic ordering freely maps cells to cells and so we have a cell decomposition of the moduli space of marked Klein discs. Clearly each cell is labelled by a different M\"obius corolla. Therefore to each vertex $v$ of a M\"obius graph we associate a cell $X(v)$. Then we can let $X(G)=\prod_v X(v)$. We can identify the orbispace of surfaces with $\gamma(K)=G$ as $X(G)/\operatorname{Aut}(G)$, which is an orbi-cell.
\end{proof}

\begin{remark}\label{rem:cellattach}
The orbi-cell labelled by a graph $G$ is attached to the the cells labelled by the graphs given by all ways of expanding the vertices of $G$ of valence greater than $3$, since two marked points meeting corresponds to bubbling off a disc. This should of course remind us of the differential of the operad $\moddmass$.
\end{remark}

\begin{lemma}
A stable Klein surface with $n>0$ oriented marked points has no non-trivial automorphisms.
\end{lemma}

\begin{proof}
We must show that the orienting double has no non-trivial automorphisms. If the orienting double is disconnected then an automorphism would necessarily restrict to an automorphism of one of the connected components. The result is true for stable bordered Riemann surfaces \cite[Lemma 3.0.11]{costello2} so the orienting double has no non-trivial automorphisms.
\end{proof}

\begin{corollary}
If $n>0$ then $D_{g,u,h,n}$ is an ordinary space and decomposes into a cell complex.
\qed
\end{corollary}

\begin{proposition}\label{prop:dianalyticgraphs}
The spaces $\dr_{\tilde{g},n}$ admit a decomposition into orbi-cells labelled by a certain type of reduced graph (which we call a dianalytic ribbon graph).
\end{proposition}

\begin{proof}
Each orbi-cell will be labelled by a (reduced) ribbon graph where two ribbon graphs are considered equivalent if there is an isomorphism of the underlying graphs that at each vertex either preserves or reverses the cyclic ordering (note that this definition can be thought of as M\"obius graphs without a colouring). We will call such graphs up to this equivalence \emph{dianalytic ribbon graphs}. Given a surface $K\in \dr_{\tilde{g},n}$ we first choose an orientation for each irreducible disc. We can then associate a ribbon graph to it where there is a vertex for each irreducible component of $K$, an edge for each node and a leg for each marked point. Again this is well defined since choosing a different orientation of an irreducible component reverses the cyclic ordering at the vertex associated to that component.

It remains to show that the space of surfaces corresponding to a given graph $G$ is an orbi-cell. This follows from the fact that the moduli space of stable dianalytic discs with $n$ marked points can be identified with the moduli space of marked oriented discs modulo the action of the map reversing the orientation. This action freely maps cells to cells (since there are at least $3$ marked points on a disc so reversing the orientation gives a different cell) so we get a cell decomposition of the moduli space of stable dianalytic discs with cells labelled by dianalytic ribbon corollas. We can associate to each vertex $v$ of $G$ a cell $X(v)$ and once again the orbi-cell of surfaces corresponding to the graph is $X(G)/\operatorname{Aut}(G)$ where $X(G)=\prod_vX(v)$.
\end{proof}

\begin{remark}\label{rem:notgrpquotient}
Note that although each orbi-cell of $\dr_{\tilde{g},n}$ is the quotient of orbi-cells in $\coprod D_{g,u,h,n}$ (where the disjoint union is as usual taken over surfaces such that $2g+u+h-1=\tilde{g}$) by the action of a finite group switching the colourings of half edges, this does not extend to a global action and so the space $\dr_{\tilde{g},n}$ is not obtained as a quotient of $\coprod D_{g,u,h,n}$ by some group action, unlike the space of smooth surfaces $\mo_{\tilde{g},n}$ which, as mentioned, is obtained as the quotient of $\coprod \ko_{g,u,h,n}$ by an action of $\mathbb{Z}_2^{\times n}$.
\end{remark}

We now come to our main result concerning the moduli space of Klein surface with oriented marked points. We have the following main result by Costello \cite{costello1,costello2}:

\begin{proposition}\label{prop:costello}
The inclusion $D_{g,h,n}\hookrightarrow\nb_{g,h,n}$ is a homotopy equivalence.
\end{proposition}

Our main result concerning the moduli space of Klein surfaces will follow from the Klein analogue of this:

\begin{proposition}\label{prop:main1}
The inclusion $D_{g,u,h,n}\hookrightarrow\kb_{g,u,h,n}$ is a homotopy equivalence.
\end{proposition}

\begin{proof}
By \autoref{lem:riemklein} with \autoref{prop:costello} this is clear if $u=0$. We therefore will restrict our attention to $u\neq 0$. Fortunately for us the proof of \autoref{prop:costello} carries over easily to a proof of this. As in that case, to prove \autoref{prop:main1} we first show the following:

\begin{lemma}\label{lem:incl1}
The inclusion $\partial\kb_{g,u,h,0}\hookrightarrow\kb_{g,u,h,0}$ is a homotopy equivalence of orbispaces.
\end{lemma}

\begin{proof}
The key idea of the proof is to construct a deformation retract of $\kb_{g,u,h,0}$ onto its boundary $\partial\kb_{g,u,h,0}$. This is done by using the hyperbolic metric on a Klein surface $K$ to flow the boundary $\partial K$ inwards until $K$ becomes singular. Some of the work involved can be avoided by passing to the orienting double and using the facts for Riemann surfaces proved in \cite{costello2}.

Let $K\in\ko_{g,u,h,0}$ be a Klein surface. Since we are assuming $u\geq 1$ and $(g,u,h)\neq(0,1,1)$ (since $\kb_{0,1,1,0}$ is empty) then the complex double of $K$ is a hyperbolic surface and so there is a unique hyperbolic metric on $K$ such that the boundary is geodesic. By taking the unit inward pointing normal vector field on $\partial K$ and using the geodesic flow on $K$ we can flow $\partial K$ inwards. Let $K_t$ be the surface with boundary obtained by flowing in $\partial K$ a distance $t$. Note that this process lifts to the orienting double $K_\mathbf{O}$ which is connected and corresponds to the process obtained using the hyperbolic metric on $K_\mathbf{O}$. In \cite[Lemma 3.0.8]{costello2} it is shown that this process applied to a Riemann surface eventually yields a singular surface and further that all the singularities are nodes. So by considering the orienting double we see the same is true for the Klein surface $K$. More precisely, for some $T$ we have $K_T\in\partial\kb_{g,u,h,0}$.

Let $S\in\mathbb{R}_{\geq 0}$ be the smallest number such that $K_S\in\partial\kb_{g,u,h,0}$ so that $K_t$ is in the interior $\ko_{g,u,h,0}$ for all $t<S$. We have a map $\Phi\co \ko_{g,u,h,0}\times[0,1]\rightarrow\kb_{g,u,h,0}$ defined by $\Phi(K,x)=K_{Sx}$. We extend this to a map $\Phi'\co \kb_{g,u,h,0}\times[0,1]\rightarrow\kb_{g,u,h,0}$ by setting $\Phi'(K',t)=K'$ for $K'\in\partial\kb_{g,u,h,0}$. To see this extends $\Phi$ continuously we take a sequence $K_i$ of surfaces converging to $K\in\partial\kb_{g,u,h,0}$ and let $x_i$ be any sequence. We must show $\Phi'(K_i,x_i)\rightarrow K$. Observing that, after forgetting the symmetry, $(K_i)_\mathbf{O}\rightarrow K_\mathbf{O}$ and comparing to the proof of \cite[Lemma 3.0.9]{costello2} this is clear. Then $\Phi'$ is a deformation retract of the inclusion as required.
\end{proof}

\begin{lemma}\label{lem:incl2}
The inclusion $\partial\kb_{g,u,h,n}\hookrightarrow\kb_{g,u,h,n}$ is a homotopy equivalence.
\end{lemma}

\begin{proof}
Since the moduli space of a M\"obius strip with an oriented marked point is the same as that of an annulus with a marked point then if $(g,u,h,n)=(0,1,1,1)$ \autoref{lem:incl2} can be seen directly.

There is a map $\kb_{g,u,h,n+1}\rightarrow\kb_{g,u,h,n}$ forgetting the last marked point and contracting any resulting unstable components which is a locally trivial fibration (in the orbispace sense). This follows by considering for some $K\in\kb_{g,u,h,n}$ the space of ways of adding an oriented marked point to $\partial K$. This is the same as the space of ways of adding a single marked point to $\partial K_\mathbf{O}$. Then by the same argument as \cite[Lemma 3.0.5]{costello2} it is clear that this map is a locally trivial fibration. Therefore if $\partial\kb_{g,u,h,n}\hookrightarrow\kb_{g,u,h,n}$ is a homotopy equivalence then so is $\partial\kb_{g,u,h,n+1}\hookrightarrow\kb_{g,u,h,n+1}$ and \autoref{lem:incl2} follows.
\end{proof}

We can now see that \autoref{prop:main1} follows from \autoref{lem:incl2} by an inductive argument like that in \cite[Lemma 3.0.12]{costello2}.
\end{proof}

We can also obtain such a result for the operad $\mb$. We consider $K\in\mb$ as a na\"ive nodal Klein surface and then consider the space of ways of adding a marked point to $\partial K$ by an identical argument to \cite[Lemma 3.0.9]{costello2} to see the map $\mb_{g,n+1}\rightarrow\mb_{g,n}$ forgetting the last marked point and stabilising is a locally trivial fibration in the orbispace sense.

Since $\mb_{\tilde{g},0}$ can be obtained from $\coprod \kb_{g,u,h,0}$ by identifying only points in $\coprod \partial \kb_{g,u,h,0}$ (by forgetting the dianalytic structure at nodes), it follows immediately from \autoref{lem:incl1} that there is a deformation retract of the map $\mb_{\tilde{g},0}\setminus\mo_{\tilde{g},0}\hookrightarrow \mb_{\tilde{g},0}$. Therefore by the same argument as above we can deduce the following:

\begin{proposition}\label{prop:main2}
The inclusion $\dr_{g,n}\hookrightarrow\mb_{g,n}$ is a homotopy equivalence.
\qed
\end{proposition}

We now obtain our main theorem immediately from \autoref{prop:main1} and \autoref{prop:main2}.

\begin{theorem}
\Needspace*{3\baselineskip}\mbox{}
\begin{itemize}
\item The inclusion $D\hookrightarrow\kb$ is a homotopy equivalence of extended topological modular operads.
\item The inclusion $\dr\hookrightarrow\mb$ is a homotopy equivalence of extended topological modular operads.
\qed
\end{itemize}
\end{theorem}

Given an appropriate chain complex $C_*$ (we take coefficients in $\mathbb{Q}$) with a K\"unneth map $C_*(X)\otimes C_*(Y)\rightarrow C_*(X\times Y)$ then $C_*(\kb)$ is an extended dg modular operad. An algebra over this is called an open KTCFT. Since the space of dianalytic ribbon graphs is obtained from the space of M\"obius graphs by forgetting the colouring we see that $C_*(\dr)=C_*(D)/(a=1)$ where $a\in C_*(D)((0,2))\cong\mathbb{Q}[\mathbb{Z}_2]$ is the involution. The above then translates into:

\begin{theorem}\label{thm:main}
There are quasi-isomorphism of extended dg modular operads over $\mathbb{Q}$
\begin{gather*}
C_*(D)\simeq C_*(\kb)\\
C_*(D)/(a=1)\simeq C_*(\mb)
\end{gather*}
where $a\in C_*(D)((0,2))\cong\mathbb{Q}[\mathbb{Z}_2]$ is the involution.
\qed
\end{theorem}

Since the spaces $D_{g,u,h,n}$ are orbi-cell complexes we can give a simple description for the operad $C_*(D)$ over $\mathbb{Q}$ using the cellular chain complex. We now identify this dg operad $C_*(D)$ and so relate our results to the previous sections.

\begin{proposition}\label{prop:cellchainiso}
There is an isomorphism
\[C_*(D)\cong\moddmass\]
(up to homological/cohomological grading).
\end{proposition}

\begin{remark}
By `up to homological/cohomological grading' we mean that $C_*(D)$ is graded homologically whereas $\moddmass$ is graded cohomologically. Given a cohomologically graded complex $V=\bigoplus V^i$ we set $V_{-i}=V^i$ to obtain an equivalent homologically graded complex. We can swap the grading of operads in this way.
\end{remark}

\begin{proof}
This follows from considering \autoref{prop:celldecomp} and \autoref{rem:cellattach} with \autoref{prop:orientedmobgraph}. The space $C_*(D_{g,u,h,n})$ is generated by oriented orbi-cells so a basis is given by reduced M\"obius graphs $G$ of topological type $(g,u,h,n)$ together with an orientation of the corresponding orbi-cell. An orientation can be given by an ordering of the vertices of $G$ and at each vertex an ordering of the set of half edges attached to it. It is clear by considering \autoref{rem:orientation} that orientations of the corresponding orbi-cell are equivalent to orientations on $G$ as defined earlier. Noting then that $C_*(D)$ is the modular closure of its genus $0$ part, it is not too difficult to check that the operad structures coincide.
\end{proof}

This is the Klein version of the fact by Costello \cite{costello1} that $\moddass$ gives a chain model for the homology of $\nb$ which gives the well known ribbon graph complexes computing the homology of the spaces $\no_{g,h,n}\simeq\nb_{g,h,n}$.

In our case this means we obtain M\"obius graph complexes computing the rational homology of the spaces $\ko_{g,u,h,n}\simeq\kb_{g,u,h,n}$, generated by oriented reduced M\"obius graphs with the differential expanding vertices of valence greater than $3$. When $n>0$ this computes the integral homology since $D_{g,u,h,n}$ is then an ordinary cell complex. When $u=0$ and $n>0$ this complex is a sum of ribbon graph complexes as expected. For $u=n=0$ this complex is the ribbon graph complex quotiented by the action of $\mathbb{Z}_2$ reversing the cyclic ordering at every vertex of a graph. For $u\neq 0$ we obtain combinatorially distinct complexes.

We can also obtain graph complexes for the rational homology of the spaces $\mo_{\tilde{g},n}\cong(\coprod\ko_{g,u,h,n})/\mathbb{Z}_2^{\times n}$ by forgetting the colours of the legs of M\"obius graphs. Additionally by forgetting the colours of all the half edges of M\"obius graphs we obtain graph complexes for the spaces $\mb_{\tilde{g},n}$. We will describe all these graph complexes concretely, without reference to operads, in the next section.

We finish this section with some observations. As already stated we have found two different ways of approaching the problem of allowing nodes on Klein surfaces. In the case of surfaces with oriented marked points we obtain a partial compactification that is homotopy equivalent to the space of smooth surfaces. We should note that $H_0(\kb)\cong\modmass$. In the second case the partial compactification is quite different.

Since the spaces $\mo_{\tilde{g},n}$ (which are in general not connected) are obtained from the disjoint union of the spaces $\ko_{g,u,h,n}$ modulo the action of the finite group $\mathbb{Z}_2^{\times n}$ then the non-zero degree rational homology of $\mo_{0,n}$ is trivial (as $\mass$ is Koszul).

There is a map of operads given by the composition $q\co \moddass\hookrightarrow\moddmass\twoheadrightarrow\moddmass/(a=1)$ and this map is surjective. Geometrically this corresponds to the fact that the spaces of $\dr$ are subspaces of the loci of admissible Riemann surfaces in $\mb$ having an orientable quotient (and so up to homotopy we do not need to worry about unorientable surfaces in $\mb$). It is easy to see that $H_0(\mb)\cong H_0(\dr)\cong\modcom$ and so the spaces $\mb_{\tilde{g},n}$ are connected.

More interestingly the spaces $\mb_{0,n}$ have non-trivial rational homology in higher degrees. To see this we first note that if $T\in\dmass/(a=1)$ is a cycle and $d(T')=T$ for some $T'$ then there is a $G\in\dass$ such that $q(G)=T'$ so $T=q(dG)$ and so the cycle $T$ lifts to a cycle $dG$ in $\dass$. Therefore if we find a cycle in $\dmass/(a=1)$ that does not lift to another cycle we know it gives a non-trivial homology class. It is easy to write down an example, see \autoref{fig:nontrivialclass}. This fact also justifies \autoref{rem:notgrpquotient}.

\begin{figure}[ht!]
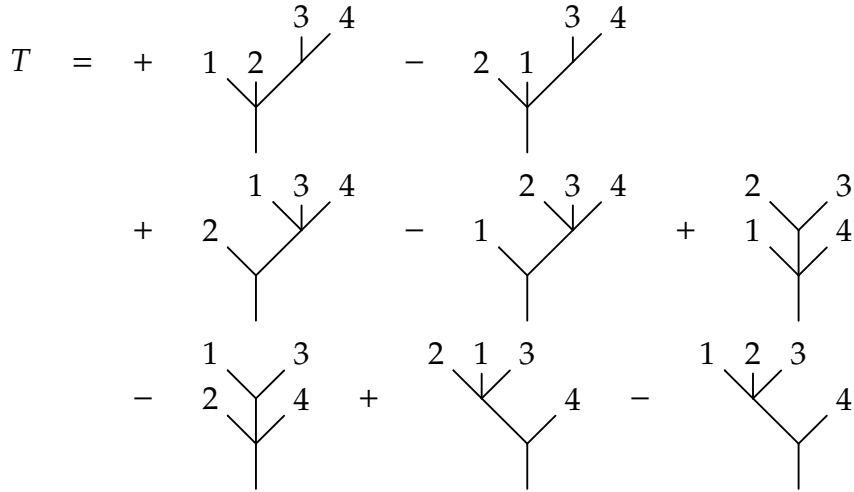

\centering
\begin{align*}
T\quad=\quad &+ \quad
{\xygraph{!{<0cm,0cm>;<0.6cm,0cm>:<0cm,0.6cm>::}
&&3="3"&4="4"\\
1="1"&2="2"&="w" \\
&="v"\\
&="0"
"v"-"1" "v"-"2" "v"-"0" "v"-"w" "w"-"3" "w"-"4"
}}
\quad - \quad
{\xygraph{!{<0cm,0cm>;<0.6cm,0cm>:<0cm,0.6cm>::}
&&3="3"&4="4"\\
2="2"&1="1"&="w" \\
&="v"\\
&="0"
"v"-"1" "v"-"2" "v"-"0" "v"-"w" "w"-"3" "w"-"4"
}}\\
\quad &+ \quad
{\xygraph{!{<0cm,0cm>;<0.6cm,0cm>:<0cm,0.6cm>::}
&1="1"&3="3"&4="4"\\
2="2"&&="w" \\
&="v"\\
&="0"
"v"-"2" "v"-"w" "v"-"0" "w"-"1" "w"-"3" "w"-"4"
}}
\quad - \quad
{\xygraph{!{<0cm,0cm>;<0.6cm,0cm>:<0cm,0.6cm>::}
&2="1"&3="3"&4="4"\\
1="2"&&="w" \\
&="v"\\
&="0"
"v"-"2" "v"-"w" "v"-"0" "w"-"1" "w"-"3" "w"-"4"
}}
\quad + \quad
{\xygraph{!{<0cm,0cm>;<0.6cm,0cm>:<0cm,0.6cm>::}
2="2"&&3="3"\\
1="1"&="w"&4="4" \\
&="v"\\
&="0"
"v"-"1" "v"-"4" "v"-"0" "v"-"w" "w"-"3" "w"-"2"
}}\\
\quad &- \quad
{\xygraph{!{<0cm,0cm>;<0.6cm,0cm>:<0cm,0.6cm>::}
1="2"&&3="3"\\
2="1"&="w"&4="4" \\
&="v"\\
&="0"
"v"-"1" "v"-"4" "v"-"0" "v"-"w" "w"-"3" "w"-"2"
}}
\quad + \quad
{\xygraph{!{<0cm,0cm>;<0.6cm,0cm>:<0cm,0.6cm>::}
2="2"&1="1"&3="3"\\
&="w"&&4="4"\\
&&="v"\\
&&="0"
"v"-"4" "v"-"w" "v"-"0" "w"-"1" "w"-"3" "w"-"2"
}}
\quad - \quad
{\xygraph{!{<0cm,0cm>;<0.6cm,0cm>:<0cm,0.6cm>::}
1="2"&2="1"&3="3"\\
&="w"&&4="4"\\
&&="v"\\
&&="0"
"v"-"4" "v"-"w" "v"-"0" "w"-"1" "w"-"3" "w"-"2"
}}
\end{align*}
\caption{A non-trivial homology class in $H_1(\mb_{0,5})$. When the differential is applied to the first two terms of $T$ two of the resulting trees cancel as elements of $\dmass/(a=1)$ but not as elements of $\dass$ no matter how we lift $T$.}
\label{fig:nontrivialclass}
\end{figure}

\section{Summary of graph complexes}
To make our results explicit we will finish by unwrapping \autoref{thm:main} and \autoref{prop:cellchainiso} and defining the graph complexes in a more explicit and straightforward manner, without reference to operads.

Recall an orientation of a graph $G$ is a choice of orientation of the vector space $\mathbb{Q}^{\edges(G)}\oplus H_1(|G|,\mathbb{Q})$. Denote by $\Gamma_{g,h,n}$ the vector space over $\mathbb{Q}$ generated by oriented reduced ribbon graphs of topological type $(g,h,n)$. That is, equivalence classes of pairs $(G,\mathrm{or})$ with $G$ a reduced ribbon graph of genus $g$ with $h$ boundary components and $n$ legs and $\mathrm{or}$ an orientation of $G$, subject to the relations $(G,-\mathrm{or})=-(G,\mathrm{or})$. 

Denote by $\mob\Gamma_{g,u,h,n}$ the vector space over $\mathbb{Q}$ generated by oriented reduced M\"obius graphs of topological type $(g,u,h,n)$. 

Denote by $\Gamma_{\tilde{g},n}$ the space 
\[\Gamma_{\tilde{g},n}=\bigoplus_{2g+h-1=\tilde{g}}\Gamma_{g,h,n}\]
and denote by $\mob\Gamma_{\tilde{g},n}$ the space:
\[\mob\Gamma_{\tilde{g},n}=\bigoplus_{2g+u+h-1=\tilde{g}}\mob\Gamma_{g,u,h,n}\]
The finite group $\mathbb{Z}_2^{\times n}$ acts on $\mob\Gamma_{\tilde{g},n}$ by switching the colours of legs.

Denote by $\Gamma^\mathbb{R}_{\tilde{g},n}$ the space of oriented reduced dianalytic ribbon graphs (see the proof of \autoref{prop:dianalyticgraphs}) of topological type $(\tilde{g},n)$. Observe that $\Gamma^\mathbb{R}_{\tilde{g},n}=\Gamma_{\tilde{g},n}/I=\mob\Gamma_{\tilde{g},n}/J$ where $I$ is the subspace generated by relations of the form $(G,\mathrm{or})=(H,\mathrm{or}')$ whenever $(G,\mathrm{or})$ is isomorphic to $(H,\mathrm{or}')$ after reversing the cyclic ordering at some of the vertices of $G$ if necessary and $J$ is the subspace generated by relations of the form $(G,\mathrm{or})=(H,\mathrm{or}')$ whenever $(G,\mathrm{or})$ is isomorphic to $(H,\mathrm{or}')$ after changing the colours of any of the half edges of $G$ if necessary. What this means is that $\Gamma^\mathbb{R}_{\tilde{g},n}$ is obtained from $\Gamma_{\tilde{g},n}$ by identifying cyclic orderings at vertices of ribbon graphs with the reverse cyclic orderings or that it is obtained from $\mob\Gamma_{\tilde{g},n}$ by forgetting the colourings of M\"obius graphs.

These spaces are finite dimensional and cohomologically graded by the number of internal edges in a graph. We define a differential on these spaces by
\[d(G,\mathrm{or})=\sum (G',\mathrm{or'})
\]
where the sum is taken over classes $(G',\mathrm{or'})$ arising from all ways of expanding one vertex of $G$ into two vertices each of valence at least $3$ so that $G'/e=G$ where $e$ is the new edge joining the two new vertices. The orientation $\mathrm{or'}$ is the product of the natural orientations on $\mathbb{Q}^{\edges(G')}\supset \mathbb{Q}^{\edges(G')\setminus e}$ and $H_1(|G'|)\cong H_1(|G'/e|)$.

We can then unwrap the main theorems.

\begin{theorem}
\Needspace*{3\baselineskip}\mbox{}
\begin{itemize}
\item There are isomorphisms:
\[
H_\bullet(\no_{g,h,n},\mathbb{Q})\cong H_\bullet(\nb_{g,h,n},\mathbb{Q})\cong H^{6g+3h+n-6-\bullet}(\Gamma_{g,h,n})
\]
Further for $n\geq 1$ such isomorphisms also hold for integral homology. This is the well known ribbon graph decomposition.
\item There are isomorphisms:
\[
H_\bullet(\ko_{g,u,h,n},\mathbb{Q})\cong H_\bullet(\kb_{g,u,h,n},\mathbb{Q})\cong H^{6g+3u+3h+n-6-\bullet}(\mob\Gamma_{g,u,h,n})
\]
Further for $n\geq 1$ such isomorphisms also hold for integral homology.
\item There are isomorphisms:
\[
H_\bullet(\mo_{\tilde{g},n},\mathbb{Q})\cong H^{3\tilde{g}+n-3-\bullet}(\mob\Gamma_{\tilde{g},n})/\mathbb{Z}_2^{\times n}
\]
\item There are isomorphisms:
\[
H_\bullet(\mb_{\tilde{g},n},\mathbb{Q})\cong H^{3\tilde{g}+n-3-\bullet}(\Gamma^\mathbb{R}_{\tilde{g},n})
\]
\end{itemize}
\end{theorem}

\chapter{Dihedral cohomology}\label{chap:dihedral}
In this chapter algebras over the operad $\dual\mass$ and over the modular operad $\moddmass$, which we will call involutive $A_\infty$--algebras and cyclic involutive $A_\infty$--algebras respectively, will be studied in more detail.

Associated to any operad is a corresponding cohomology theory for algebras over (a cofibrant replacement for) that operad. For example, it is well known that for associative algebras, or more generally $A_\infty$--algebras, the corresponding cohomology theory is Hochschild cohomology. We will recall this setup and compare it to the analogous theory for involutive $A_\infty$--algebras. In particular we will relate this to dihedral cohomology theory \cite{lodaybook,lodaydihedral} for involutive associative algebras, obtaining along the way a generalisation of dihedral cohomology to involutive $A_\infty$--algebras. 

It is well known that Hochschild cohomology and cyclic cohomology govern deformations of associative and cyclic associative algebras. This holds true for the corresponding involutive cohomology theories and we shall explain more precisely what this means.

The general construction of M\"obiusisation of an operad and \autoref{thm:mobiusisation} suggest that one could also generalise these ideas to arbitrary operads and so there is a version of dihedral cohomology for other algebras, such as dihedral Chevalley--Eilenberg cohomology or dihedral Harrison cohomology.

The results of this chapter can be seen as an application of the general theory of algebraic operads and a detailed exposition of the general approach is contained in the forthcoming book by Loday and Vallette \cite{lodayvallette}. However, we neither require nor desire heavy use of operad machinery since our aim is to understand in more detail the case for the specific operads $\ass$ and $\mass$.

\section{Involutive Hochschild cohomology of involutive $A_\infty$--algebras}

\begin{definition}
Let $A$ be a $\mathcal{P}$--algebra. A \emph{graded derivation} of $A$ is a graded map $f\co A\rightarrow A$ such that for any $m \in \mathcal{P}(n)$
\[
f(m(a_1, \dots, a_n )) = \sum_{i=1}^n (-1)^{\epsilon} m(a_1, \dots, a_{i-1}, f(a_i), a_{i+1}, \dots, a_n)
\]
where $\epsilon = \degree{f}(\degree{m}+\degree{a}_1+\dots+\degree{a}_{i-1})$. The space spanned by all graded derivations forms a Lie subalgebra of the space $\IHom(A,A)$ of linear maps with the commutator bracket and is denoted by $\Der(A)$.
\end{definition}

An $A_\infty$--algebra is an algebra over $\dass$. The following `operad free' definition in terms of derivations of the free $\ass$--algebra is well known to be equivalent to this, see for example \cite[Proposition 4.2.14]{ginzburgkapranov}.

\begin{definition}
Let $V$ be a graded vector space. An $A_\infty$--algebra structure on $V$ is a derivation $m\co \widehat{T}_{\geq 1}\Sigma^{-1}V^*\rightarrow \widehat{T}_{\geq 1}\Sigma^{-1}V^*$ of degree one such that $m^2 = 0$.
\end{definition}

\begin{remark}
Since we are considering the \emph{completed} tensor algebra, with is a formal vector space, maps and derivations are of course required to be continuous.
\end{remark}

Recall that such a derivation $m$ is determined completely by its restriction to $\Sigma^{-1}V^*$. Denote by $m_n\co\Sigma^{-1}V^*\rightarrow (\Sigma^{-1}V^*)^{\otimes n}$ the order $n$ part of this restriction so that $m = m_1 + m_2 + \dots$ on the subspace $\Sigma^{-1}V^*$. Such a collection of $m_n$ is equivalent to a collection of maps $\hat{m}_n\co V^{\otimes n} \rightarrow V$ of degrees $2-n$ satisfying the usual $A_\infty$--conditions, familiar from the operadic definition.

Note that $m_1^2 = 0$ so $(V,\hat{m}_1)$ is a differential graded vector space. If $m_n = 0$ for $n > 2$ then $\hat{m}_2$ is an associative product on $V$ respecting the differential and so a differential graded associative algebra is a special case of an $A_\infty$--algebra.

\begin{definition}
Let $(V,m)$ and $(W,m')$ be $A_\infty$--algebras. Then an \emph{$\infty$--morphism} of $A_\infty$--algebras is a map $\phi$ of associative algebras $\phi\co\widehat{T}_{\geq 1}\Sigma^{-1}W^*\rightarrow \widehat{T}_{\geq 1}\Sigma^{-1}V^*$ such that $m \circ\phi = m'\circ \phi$.
\end{definition}

\begin{definition}
Let $(V,m)$ be an $A_\infty$--algebra. Then the space of derivations $\Der(\widehat{T}\Sigma^{-1}V^*)$ is a differential graded Lie algebra with bracket the commutator bracket and differential given by $d(\xi)=[m,\xi]$.

The Hochschild cohomology complex of $V$ with coefficients in itself is the differential graded vector space $\hoch^{\bullet}(V,V) = \Sigma^{-1}\Der(\widehat{T}\Sigma^{-1}V^*)$. The cohomology of this will be denoted by $\Hhoch^{\bullet}(V,V)$.
\end{definition}

In the case that $V$ is an associative algebra then note that this coincides with the classical definition of the Hochschild cohomology. This is also the reason for the desuspension in this definition.

\subsection{Involutive $A_\infty$--algebras}
\Needspace*{4\baselineskip}

\begin{definition}
Let $V$ be a graded vector space. An involution on $V$ is a map $v\mapsto v^*$ with $(v^*)^*=v$. An \emph{involutive differential graded associative algebra} is a differential graded associative algebra $A$ with an involution satisfying $(xy)^* = (-1)^{\degree{x} \degree{y}}y^*x^*$ and $d(x)^* = d(x^*)$.
\end{definition}

If $V$ has an involution then $\Sigma V$ and $\Sigma^{-1} V$ do in the obvious way. $V^*$ also has an involution defined by $\phi^*(v) = -\phi(v^*)$ for $\phi\in V^*$ (note the appearance of the minus sign).

Let $W$ be an graded vector space with an involution. Then $W^{\otimes n}$ admits an involution defined by
\[
(w_1\otimes w_2 \otimes \dots \otimes w_n)^* = (-1)^{\epsilon} w_n^*\otimes\dots\otimes w_2^* \otimes w_1^*
\]
where $\epsilon = \sum_{i=1}^n \degree{w}_i\left ( \sum_{j=i+1}^n \degree{w}_j \right )$ arises from permuting the $w_i$ with degrees $\degree{w}_i$. It follows that $\widehat{T}\Sigma^{-1}V^*$ has an involution induced by that on $V$ making it into an involutive graded associative algebra.

\begin{definition}
Let $V$ be a graded vector space with an involution. An \emph{involutive $A_\infty$--algebra structure} on $V$ is a derivation $m\co \widehat{T}_{\geq 1} \Sigma^{-1} V^* \rightarrow \widehat{T}_{\geq 1} \Sigma^{-1} V^*$ of degree one such that $m^2=0$ and $m$ preserves the involution: $m(x^*) = m(x)^*$.
\end{definition}

The requirement $m(x^*) = m(x)^*$ can be unwrapped in terms of the $\hat{m}_n\co V^{\otimes n}\rightarrow V$ and the involution on $V$ to obtain
\[
\hat{m}_n(x_1,\dots , x_n)^* = (-1)^{\epsilon}(-1)^{n(n+1)/2-1} \hat{m}_n(x_n^*,\dots,x_1^*)
\]
where $\epsilon = \sum_{i=1}^n \degree{x}_i\left ( \sum_{j=i+1}^n \degree{x}_j \right )$ arises from permuting the $x_i\in V$ with degrees $\degree{x}_i$. In this way we see that this notion of an involutive $A_\infty$--algebra is indeed equivalent to an algebra over the operad $\dmass$. In particular if $m_n = 0 $ for $n > 2$ then this corresponds to the structure on an involutive differential graded associative algebra, which is thus a special case of an involutive $A_\infty$--algebra.

\begin{definition}
Let $(V,m)$ and $(W,m')$ be involutive $A_\infty$--algebras. Then an \emph{$\infty$--morphism} of involutive $A_\infty$--algebras is a map $\phi\co\widehat{T}_{\geq 1}\Sigma^{-1}W^*\rightarrow \widehat{T}_{\geq 1}\Sigma^{-1}V^*$ of associative algebras such that $m \circ\phi = m'\circ \phi$ and $\phi$ preserves the involution: $\phi(x^*) = \phi(x)^*$.
\end{definition}

\begin{definition}
Let $(V,m)$ be an involutive $A_\infty$--algebra. Then the subspace of derivations $\Der_+(\widehat{T}\Sigma^{-1}V^*)\subset \Der(\widehat{T}\Sigma^{-1}V^*)$ preserving the involution is a differential graded Lie subalgebra with bracket the commutator bracket and differential given by $d(\xi)=[m,\xi]$. 

The \emph{involutive Hochschild cohomology complex} of $V$ with coefficients in itself is the differential graded vector space $\hoch^{\bullet}_+(V,V) = \Sigma^{-1}\Der_+(\widehat{T}\Sigma^{-1}V^*)$. The cohomology of this will be denoted by $\Hhoch^{\bullet}_+(V,V)$.
\end{definition}

\subsection{Decomposition of Hochschild cohomology}
Let $V$ be a graded vector space with an involution. For $\xi\in\Der(\widehat{T}\Sigma^{-1}V^*)$ define $\xi^*$ by $\xi^*(x) = \xi(x^*)^*$. This is then an involution on $\Der(\widehat{T}\Sigma^{-1}V^*)$.

\begin{proposition}
The Lie algebra $\Der(\widehat{T}\Sigma^{-1}V^*)$ is an involutive Lie algebra, by which it is meant that $[\xi,\eta]^* = [\xi^*,\eta^*]$.
\end{proposition}

\begin{proof}
This follows from the observation that $\xi(\eta(x^*))^* = \xi^*(\eta^*(x))$.
\end{proof}

\begin{remark}
If one instead preferred to define the involutive Hochschild cohomology complex as the \emph{quotient} of the usual Hochschild cohomology complex by the action of $\mathbb{Z}_2$ given by $\xi\mapsto \xi^*$ this is of course equivalent to the definition used above by the isomorphism of invariants and coinvariants.
\end{remark}

Note that $\Der_+(\widehat{T}\Sigma^{-1}V^*)$ is the eigenspace of the eigenvalue $+1$ of this involution. So denote by $\Der_-(\widehat{T}\Sigma^{-1}V^*)$ the eigenspace of the eigenvalue $-1$. For $\xi\in\Der(\widehat{T}\Sigma^{-1}V^*)$ denote by $\xi\mapsto\xi^+$ and $\xi\mapsto\xi^-$ the projections onto these eigenspaces given by $\xi^+ = 1/2(\xi+\xi^*)$ and $\xi^- = 1/2(\xi - \xi^*)$. Then $\xi=\xi^+ + \xi^-$ and as graded vector spaces $\Der(\widehat{T}\Sigma^{-1}V^*)=\Der_+(\widehat{T}\Sigma^{-1}V^*)\oplus \Der_-(\widehat{T}\Sigma^{-1}V^*)$. Note that this is not a decomposition of Lie algebras, however
\[
[\Der_-(\widehat{T}\Sigma^{-1}V^*),\Der_-(\widehat{T}\Sigma^{-1}V^*)] \subset \Der_+(\widehat{T}\Sigma^{-1}V^*).
\]
Now let $(V,m)$ be an involutive $A_\infty$--algebra. Then since $[m,\xi]^* = [m,\xi^*]$ it follows that this decomposition is in fact a decomposition of \emph{differential} graded vector spaces.

\begin{definition}
The \emph{skew involutive Hochschild cohomology complex} of an involutive $A_\infty$--algebra $(V,m)$ with coefficients in itself is the differential graded vector space $\hoch^{\bullet}_-(V,V) = \Sigma^{-1}\Der_-(\widehat{T}\Sigma^{-1}V^*)$. The cohomology of this will be denoted by $\Hhoch^{\bullet}_-(V,V)$.
\end{definition}

\begin{theorem}\label{thm:hochschilddecomposition}
For an involutive $A_\infty$--algebra $(V,m)$ the Hochschild cohomology of $V$ decomposes as $\Hhoch^\bullet(V,V)\cong \Hhoch^{\bullet}_+(V,V)\oplus\Hhoch^{\bullet}_-(V,V)$.
\qed
\end{theorem}

For an involutive $A_\infty$--algebra $(V,m)$, general theory tells us that $\Hhoch_+^\bullet(V,V)$ governs involutive $A_\infty$--deformations of $m$ and $\Hhoch^\bullet(V,V)$ governs $A_\infty$--deformations of $m$. Therefore in a certain sense $\Hhoch_-^\bullet(V,V)$ should measure the difference between the involutive and usual deformation theory. 

\section{Dihedral cohomology of cyclic involutive $A_\infty$--algebras}
We begin by recalling the cyclic cohomology of $A_\infty$--algebras and how it relates to cyclic $A_\infty$--algebras. Then we define the dihedral cohomology of involutive $A_\infty$--algebras and consider how it relates to cyclic involutive $A_\infty$--algebras.

\subsection{Cyclic $A_\infty$--algebras}
\Needspace*{4\baselineskip}

\begin{proposition}\label{prop:cyclicdifferential}
Let $(V, m)$ be an $A_\infty$--algebra. Denote by $\CC^{\bullet}(V)$ the graded vector space
\[
\CC^{\bullet}(V) = \Sigma \prod_{i=1}^{\infty}[(\Sigma^{-1}V^*)^{\otimes i}]_{\mathbb{Z}_i}
\]
where $\mathbb{Z}_i$ is the cyclic group of order $i$ acting in the obvious way. Then the derivation $m$ on $\widehat{T}\Sigma^{-1}V^*$ induces a well defined derivation on the quotient $\CC^{\bullet}(V)$.
\end{proposition}

\begin{proof}
Let $M$ be the subspace of $\widehat{T}\Sigma^{-1}V^*$ of convergent sums of elements of the form $ab-(-1)^{\degree{a}\degree{b}}ba$ (in other words the subspace of commutators). Then $\CC^\bullet(V)=\widehat{T}\Sigma^{-1}V^*/M$. A straightforward calculation using the fact that $m$ is a degree one derivation shows that $m(M)\subset M$ as required.
\end{proof}

\begin{definition}
Let $(V,m)$ be an $A_\infty$--algebra. The \emph{cyclic cohomology complex} of $V$ is the differential graded vector space $\CC^{\bullet}(V)$ with differential induced by $m$ as in \autoref{prop:cyclicdifferential}. The cohomology of this will be denoted by $\HCC^{\bullet}(V)$.
\end{definition}

Let $V$ be a graded vector space with a symmetric bilinear form $\langle -, - \rangle \co V\otimes V \rightarrow \Sigma^d k$ of degree $d$. A derivation $m\in\Der(\widehat{T}\Sigma^{-1}V^*)$ is called a \emph{cyclic derivation} if the maps $\hat{m}_n\co V^{\otimes n}\rightarrow V$ satisfy
\begin{equation}\label{eq:cycliccondition}
\langle \hat{m}_n(x_1,\dots , x_n), x_{n+1} \rangle = (-1)^\epsilon(-1)^n \langle \hat{m}_n(x_{n+1},\dots, x_{n-1}), x_n \rangle
\end{equation}
where $\epsilon = \degree{x}_{n+1}\sum_{i=1}^n \degree{x}_{i}$ arises from permuting the $x_i\in V$. The subspace of cyclic derivations will be denoted by $\Der_\cycl(\widehat{T}\Sigma^{-1}V^*)$. It is a Lie subalgebra.

The bilinear form on $V$ is non-degenerate if the map $V\rightarrow \Sigma^d V^*$ given by $v\mapsto \langle v, - \rangle$ is an isomorphism. This yields a degree $-d$ non-degenerate bilinear form on $V^*$ denoted by $\langle -, - \rangle^{-1}\co V^*\otimes V^* \rightarrow \Sigma^{-d}V^*$ which is symmetric if $d$ is even and anti-symmetric if $d$ is odd.

\begin{definition}
Let $V$ be a graded vector space with a symmetric bilinear form. A \emph{cyclic $A_\infty$--algebra structure} on $V$ is a cyclic derivation $m\in\Der^\cycl(\widehat{T}\Sigma^{-1}V^*)$ of degree one such that $m^2=0$.
\end{definition}

\begin{remark}
In the case that the bilinear form on $V$ is \emph{non-degenerate} then the above definition is equivalent to an algebra over the modular operad $\moddass$.
\end{remark}

\begin{remark}
A differential graded associative algebra with a symmetric bilinear form satisfying $\langle ab, c \rangle = \langle a, bc \rangle $ is a special case of a cyclic $A_\infty$--algebra.
\end{remark}

Note that a symmetric bilinear on $V$ of degree $d$ is represented by a degree $d$ element $V^*\otimes V^*$, or equivalently a degree $d+2$ element in $\Sigma^{-1}V^*\otimes \Sigma^{-1}V^*$.

\begin{definition}
Let $(V,m)$ and $(W,m')$ be cyclic $A_\infty$--algebras with degree $d$ symmetric bilinear forms. Then a cyclic \emph{$\infty$--morphism} is a map $\phi$ of $A_\infty$--algebras such that $\phi(\omega) = \omega'$ where $\omega\in\Sigma^{-1}V^*\otimes\Sigma^{-1}V^*$ and $\omega'\in\Sigma^{-1}W^*\otimes \Sigma^{-1}W^*$ are the degree $d+2$ elements representing the bilinear forms.
\end{definition}

\begin{theorem}\label{thm:cycder}
Let $(V,m)$ be a cyclic $A_\infty$--algebra with a \emph{non-degenerate} symmetric bilinear form. Then as complexes $\Sigma^{d+1}\CC^\bullet(V) \cong \Der^\cycl(\widehat{T}\Sigma^{-1}V^*)$ where the right hand side is equipped with the differential given by $d(\xi) = [m,\xi]$.
\end{theorem}

\begin{proof}
Since $\IHom(\Sigma^{-1}V^*, (\Sigma^{-1}V^*)^{\otimes n})\cong \Sigma V \otimes (\Sigma^{-1}V^*)^{\otimes n}$ then composing with the isomorphism $V \rightarrow \Sigma^{d}V^*$ on the first tensor factor gives an isomorphism
\[
f\co\IHom(\Sigma^{-1}V^*, (\Sigma^{-1}V^*)^{\otimes n})\cong \Sigma^{d+2} (\Sigma^{-1}V^*)^{\otimes n+1}.
\]
Given $\xi_n\in \IHom(\Sigma^{-1}V^*, (\Sigma^{-1}V^*)^{\otimes n})$ then $\hat{\xi}_n$ satisfies \autoref{eq:cycliccondition} if and only if $f(\xi_n)$ is an invariant with respect to the cyclic action on $\Sigma^{d+2} (\Sigma^{-1}V^*)^{\otimes n+1}$. By the isomorphism of coinvariants and invariants there is therefore an isomorphism of graded vector spaces $\Der^\cycl(\widehat{T}\Sigma^{-1}V^*)\cong \Sigma^{d+1}\CC^\bullet(V)$. It remains to verify that the differential coincides, in other words that $f([m,\xi]) = m(f(\xi))$, which is a straightforward check left to the reader.
\end{proof}

\begin{remark}
The viewpoint of Kontsevich's formal noncommutative symplectic geometry \cite{kontsevich2} tells us that by regarding the bilinear form on $V$ as a symplectic structure then the underlying space of $\CC^{\bullet}(V)$ can be understood as the space of noncommutative Hamiltonians, and $\Der^{\cycl}(\widehat{T}\Sigma^{-1}V^*)$ can be understood as the space of symplectic vector fields. This gives a rather enlightening view on \autoref{thm:cycder}. For a detailed account of this point of view see \cite{hamiltonlazarev}.
\end{remark}

\begin{remark}\label{rem:cyclie}
It follows that $\Sigma^{d+1}\CC^\bullet(V)$ has the structure of a differential graded Lie algebra. The Lie bracket can be described explicitly on the summands by the formula
\[
[a_1\dots a_n, b_1\dots b_m ] = (-1)^p\sum_{i=1}^{n}\sum_{j=1}^{m}(-1)^{\epsilon}\langle a_i, b_j \rangle^{-1}  a_{i+1}\dots a_1\dots a_{i-1} b_{j+1}\dots b_1\dots b_{j-1}
\]
where $\epsilon$ arises from permuting the $a_i\in\Sigma^{-1}V^*$ and $b_i\in\Sigma^{-1}V^*$ and $p=d\sum_i \degree{a}_i$.
\end{remark}

\autoref{thm:cycder} explains the well known result that cyclic cohomology governs cyclic $A_\infty$--deformations.

\subsection{Cyclic involutive $A_\infty$--algebras and decomposition of cyclic cohomology}
\Needspace*{4\baselineskip}

\begin{definition}\label{def:dihedralactions}
Let $W$ be a graded vector space with an involution. Denote by $D_n$ the dihedral group of order $2n$, presented by $D_n = \langle r, s \mid r^n = s^2 = 1, srs^{-1}=r^{-1} \rangle $. Then there are the following two actions of $D_n$ on $W^{\otimes n}$.
\begin{enumerate}
\item \label{dihedralaction}The \emph{dihedral action} is defined by
\begin{align*}
r(w_1\otimes w_2 \otimes \dots \otimes w_n) &= (-1)^{\epsilon} w_n\otimes w_1 \otimes \dots\otimes w_{n-1}\\
s(w_1\otimes w_2 \otimes \dots \otimes w_n) &= (w_1\otimes w_2 \otimes \dots \otimes w_n)^*
\end{align*}
\item\label{skewdihedralaction} The \emph{skew-dihedral action} is defined by
\begin{align*}
r(w_1\otimes w_2 \otimes \dots \otimes w_n) &= (-1)^{\epsilon} w_n\otimes w_1 \otimes \dots\otimes w_{n-1}\\
s(w_1\otimes w_2 \otimes \dots \otimes w_n) &= -(w_1\otimes w_2 \otimes \dots \otimes w_n)^*
\end{align*}
\end{enumerate}
Here $\epsilon= \degree{w}_{n}\sum_{i=1}^{n-1}\degree{w}_i$ arises, as usual, from permuting the $w_i\in W$.
\end{definition}

\begin{proposition}\label{prop:dihedraldifferential}
Let $(V, m)$ be an involutive $A_\infty$--algebra.
\begin{itemize}
\item Denote by $\CD^{\bullet}_+(V)$ the graded vector space
\[
\CD^{\bullet}_+(V)= \Sigma \prod_{i=1}^{\infty}[ (\Sigma^{-1} V^*)^{\otimes i} ]_{D_{i}}
\]
where $D_i$ is the dihedral group of order $2i$ acting by the dihedral action \ref{dihedralaction} of \autoref{def:dihedralactions}. Then the derivation $m$ on $\widehat{T}\Sigma^{-1}V^*$ induces a well defined derivation on the quotient $\CD^\bullet_+(V)$.
\item Denote by $\CD^{\bullet}_-(V)$ the graded vector space
\[
\CD^\bullet_-(V) = \Sigma \prod_{i=1}^{\infty}[ (\Sigma^{-1} V^*)^{\otimes i} ]_{D_{i}}
\]
where $D_i$ is the dihedral group of order $2i$ acting by the skew-dihedral action \ref{skewdihedralaction} of \autoref{def:dihedralactions}. Then the derivation $m$ on $\widehat{T}\Sigma^{-1}V^*$ induces a well defined derivation on the quotient $\CD^\bullet_-(V)$.
\end{itemize}
\end{proposition}

\begin{proof}
Let $M$ be the subspace of $\widehat{T}\Sigma^{-1}V^*$ of convergent sums of elements of the forms $ab - (-1)^{\degree{a}\degree{b}}ba$ and $a - a^*$. Then $\CD^\bullet_+(V) = \Sigma\widehat{T}\Sigma^{-1}V^*/M$. Since $m$ is involutive $m(a-a^*) = m(a)-m(a)^* \in M$ and together with the proof of \autoref{prop:cyclicdifferential} this means $m(M)\subset M$ and the first part follows. The second part is essentially the same.
\end{proof}

\begin{definition}
Let $(V,m)$ be an involutive $A_\infty$--algebra.
\begin{itemize}
\item The \emph{dihedral cohomology complex} of $V$ is the differential graded vector space $\CD^\bullet_+(V)$ with differential induced by $m$ as in \autoref{prop:dihedraldifferential}. The cohomology of this complex will be denoted by $\HCD^{\bullet}_+(V)$.
\item The \emph{skew-dihedral cohomology complex} of $V$ is the differential graded vector space $\CD^\bullet_-(V)$ with differential induced by $m$ as in \autoref{prop:dihedraldifferential}. The cohomology of this complex will be denoted by $\HCD^\bullet_-(V)$.
\end{itemize}
\end{definition}

\begin{theorem}\label{thm:cyclicdecomposition}
For an involutive $A_\infty$--algebra $(V, m)$ the cyclic cohomology of $V$ decomposes as $\HCC^\bullet(V) \cong \HCD^\bullet_+(V) \oplus \HCD^\bullet_-(V)$.
\end{theorem}

\begin{proof}
Since $D_n = \mathbb{Z}_2\ltimes \mathbb{Z}_n$, the complexes $\CD^\bullet_+(V)$ and $\CD^\bullet_-(V)$ are the quotients of $\CC^\bullet(V)$ by two different actions of $\mathbb{Z}_2$ arising from the involution on $V$. By the isomorphism of coinvariants with invariants these spaces can be identified with the eigenspaces of the eigenvalues $+1$ and $-1$ of the involution on $\CC^\bullet(V)$ and the result follows as for \autoref{thm:hochschilddecomposition}.
\end{proof}

\begin{proposition}
Let $V$ be a graded vector space with an involution and a symmetric bilinear form such that $\langle x^*, y^* \rangle = \langle x, y \rangle$. Then $\Der^\cycl(\widehat{T}\Sigma^{-1} V^*)$ is an involutive Lie subalgebra of $\Der(\widehat{T}\Sigma^{-1} V^*)$.
\end{proposition}

\begin{proof}
Let $m\in \Der^\cycl(\widehat{T}\Sigma^{-1} V^*)$ be a cyclic derivation. Then a straightforward calculation gives
\begin{align*}
\langle \hat{m}^*_n(x_1,\dots , x_n), x_{n+1} \rangle &= (-1)^{\epsilon}(-1)^{n(n+1)/2 - 1} \langle \hat{m}_n(x_n^*,\dots,x_1^*)^*, x_{n+1} \rangle
\\
&= (-1)^{\epsilon}(-1)^{n(n+1)/2 - 1} \langle \hat{m}_n(x_n^*,\dots,x_1^*), x_{n+1}^* \rangle
\\
&=  (-1)^{\epsilon'}(-1)^{n(n+1)/2 - 1}(-1)^n \langle \hat{m}_n(x_{n-1}^*,\dots, x_{n+1}^*), x_n^* \rangle
\\
&= (-1)^{\epsilon'}(-1)^{n(n+1)/2 - 1}(-1)^n \langle \hat{m}_n(x_{n-1}^*,\dots, x_{n+1}^*)^*, x_n \rangle
\\
&= (-1)^{\epsilon''}(-1)^n \langle \hat{m}_n^*(x_{n+1},\dots, x_{n-1}), x_n \rangle
\end{align*}
where $\epsilon$, $\epsilon'$ and $\epsilon''$ arise from the Koszul sign rule permuting the $x_i$, remembering that $\degree{x}_i^* = \degree{x}_i$. Therefore $m^*$ is also a cyclic derivation as required. 
\end{proof}

\begin{corollary}
The space $\Der^{\cycl}_+(\widehat{T}\Sigma^{-1}V^*)=\Der^\cycl(\widehat{T}\Sigma^{-1}V^*)\cap\Der_+(\widehat{T}\Sigma^{-1}V^*)$ of cyclic derivations preserving the involution is a Lie subalgebra and as graded vector spaces $\Der^\cycl(\widehat{T}\Sigma^{-1}V^*) = \Der^{\cycl}_+(\widehat{T}\Sigma^{-1}V^*)\oplus\Der^{\cycl}_-(\widehat{T}\Sigma^{-1}V^*)$.
\qed
\end{corollary}

\begin{definition}
Let $V$ be a graded vector space with an involution and a symmetric bilinear form such that $\langle x^*, y^* \rangle = \langle x, y \rangle$. A \emph{cyclic involutive $A_\infty$--algebra structure} on $V$ is a cyclic involutive derivation $m \in \Der^\cycl_+(\widehat{T} \Sigma^{-1} V^*)$ of degree one such that $m^2 = 0$.
\end{definition}

\begin{remark}
An involutive differential graded associative algebra with a symmetric bilinear form satisfying $\langle a^*, b^* \rangle = \langle a, b \rangle$ and $\langle ab, c \rangle = \langle a, bc \rangle$ is a special case of a cyclic involutive $A_\infty$--algebra.
\end{remark}

Note that a symmetric bilinear on $V$ of degree $d$ is represented by a degree $d$ element $V^*\otimes V^*$, or equivalently a degree $d+2$ element in $\Sigma^{-1}V^*\otimes \Sigma^{-1}V^*$.

\begin{definition}
Let $(V,m)$ and $(W,m')$ be cyclic involutive $A_\infty$--algebras with degree $d$ symmetric bilinear forms. Then a cyclic involutive \emph{$\infty$--morphism} is a map $\phi$ of involutive $A_\infty$--algebras which is also a map of cyclic $A_\infty$--algebras.
\end{definition}

\begin{theorem}
Let $(V, m)$ be a cyclic involutive $A_\infty$--algebra with a \emph{non-degenerate} symmetric bilinear form. Then as complexes
\begin{itemize}
\item $\Sigma^{d+1}\CD^\bullet_+(V) \cong \Der^\cycl_+(\widehat{T} \Sigma^{-1} V^*)$
\item $\Sigma^{d+1}\CD^\bullet_-(V) \cong \Der^\cycl_-(\widehat{T} \Sigma^{-1} V^*)$
\end{itemize}
where $\Der^\cycl_+(\widehat{T} \Sigma^{-1} V^*)$ and $\Der^\cycl_-(\widehat{T} \Sigma^{-1} V^*)$ are each equipped with the differential given by $d(\xi)=[m,\xi]$.
\end{theorem}

\begin{proof}
By \autoref{thm:cycder} $\Sigma^{d+1} \CC^\bullet(V) \cong \Der^\cycl(\widehat{T}\Sigma^{-1}V^*)$ as complexes. Furthermore it is clear this isomorphism also preserves the involution. As in the proof of \autoref{thm:cyclicdecomposition} $\CD^\bullet_+(V)$ and $\CD^\bullet_-(V)$ can be identified with the eigenspaces of the eigenvalues $+1$ and $-1$ of the involution on $\CC^\bullet(V)$, which correspond under this isomorphism to $\Der^\cycl_+(\widehat{T}\Sigma^{-1}V^*)$ and $\Der^\cycl_-(\widehat{T}\Sigma^{-1}V^*)$ respectively.
\end{proof}

\begin{remark}
It follows from \autoref{rem:cyclie} that $\Sigma^{d+1}\CD^\bullet_+(V)$ is a differential graded Lie subalgebra of $\Sigma^{d+1}\CC^\bullet(V)$.
\end{remark}

\section{Deformation theory of involutive $A_\infty$--algebras}
Just as the Hochschild cohomology and cyclic cohomology in a certain sense govern deformations of $A_\infty$--algebras and cyclic $A_\infty$--algebras, so the involutive Hochschild cohomology and dihedral cohomology govern deformations of involutive $A_\infty$--algebras and cyclic involutive $A_\infty$--algebras.

There is much well established general theory concerning deformation functors of homotopy algebras and their representability by differential graded Lie algebras so this section will be very succinct and not at all comprehensive.

To any differential graded Lie algebra $\mathfrak{g}$ and an augmented finite dimensional nilpotent commutative algebra (or more generally an inverse limit of such objects) $R$ with maximal ideal $R_+$, there is a deformation functor $\Def_{\mathfrak{g}}$ associated to $\mathfrak{g}$ which assigns to $R$ the \emph{Maurer--Cartan moduli set} $\MCmoduli(\mathfrak{g}\otimes R_+)$. This functor will be defined in \autoref{chap:mc}.

Furthermore if $A$ is an associative algebra, or more generally an $A_\infty$--algebra and $R$ is an augmented finite dimensional nilpotent commutative algebra (or more generally an inverse limit of such objects) then the classical deformation functor $\Def_A$ assigns to $R$ the set of $\infty$--isomorphism classes of $R$--linear $A_\infty$--algebras $\tilde{A}$ with an isomorphism $\tilde{A}\otimes_R k \rightarrow A$ (here we are considering only free deformations, by which we mean that $\tilde{A}$ is a free $R$--module, or equivalently $\tilde{A}\cong R\otimes A$ as $R$--modules).

Then it is a well known and straightforward result that setting $\mathfrak{g}$ to be the differential graded Lie algebra $\mathfrak{g}=\Sigma\hoch^\bullet(A,A)_{\geq 1} = \Der(\widehat{T}_{\geq 1}\Sigma^{-1}A^*)$ then $\Def_{\mathfrak{g}}\cong \Def_A$, so in this sense the Hochschild cohomology of $A$ governs deformations of $A$. In this same manner, for a cyclic $A_\infty$--algebra $A$ with non-degenerate bilinear form of degree $d$ then $\mathfrak{g}=\Sigma^{d+1}\CC^\bullet(A)_{\geq 1}=\Der^{\cycl}(\widehat{T}_{\geq 1}\Sigma^{-1}V^*)$ governs cyclic deformations of $A$. It is straightforward to see from the definitions and theorems shown above that by simply replacing `Hochschild' with `involutive Hochschild' and `cyclic' with `dihedral' one can obtain analogous results for involutive $A_\infty$--algebras and cyclic involutive $A_\infty$--algebras.

\chapter{Maurer--Cartan elements and lifts}\label{chap:mc}

In this chapter we will first review the construction of the Maurer--Cartan moduli set associated to a differential graded Lie algebra. We extend the theory to curved Lie algebras. We then use this construction and its applications to deformation theory to study a certain class of deformation problems in substantial generality.

\section{Curved Lie algebras and Maurer--Cartan elements}

\begin{definition}
Let $\mathfrak{g}$ be a differential graded Lie algebra. A \emph{Maurer--Cartan element in $\mathfrak{g}$} is a degree one element $\xi\in\mathfrak{g}$ satisfying the Maurer--Cartan equation
\[
d\xi + \frac{1}{2}[\xi, \xi]=0.
\]
We denote the set of Maurer--Cartan elements in $\mathfrak{g}$ by $\MC(\mathfrak{g})$.
\end{definition}

Since a map of $\mathfrak{g}\rightarrow\mathfrak{h}$ takes Maurer--Cartan elements to Maurer--Cartan elements we see that $\MC$ defines a functor on differential graded Lie algebras.

We will also need the notion of a Maurer--Cartan element in a \emph{curved} Lie algebra.

\begin{definition}
A \emph{curved Lie algebra} is a graded Lie algebra $\mathfrak{g}$ with a degree two element $\Omega\in\mathfrak{g}$ called the curvature and a degree one derivation $d$ of $\mathfrak{g}$ called the predifferential such that for all $\eta\in\mathfrak{g}$:
\begin{itemize}
\item $d^2\eta = [\eta, \Omega]$
\item $d\Omega = 0$
\end{itemize}
\end{definition}

\begin{definition}
A morphism of curved Lie algebras $f\co\mathfrak{g}\rightarrow\mathfrak{h}$ is a morphism of graded Lie algebras such that $fd_\mathfrak{g} = d_\mathfrak{h}f$ and $f(\Omega_{\mathfrak{g}}) = \Omega_{\mathfrak{h}}$.
\end{definition}

\begin{definition}
Let $\mathfrak{g}$ be a curved Lie algebra. A \emph{Maurer--Cartan element in $\mathfrak{g}$} is a degree one element $\xi\in\mathfrak{g}$ satisfying the Maurer--Cartan equation
\[
\Omega + d\xi + \frac{1}{2}[\xi,\xi] = 0.
\]
We denote the set of Maurer--Cartan elements in $\mathfrak{g}$ by $\MC(\mathfrak{g})$.
\end{definition}

Note that there is a functor from differential graded Lie algebras to curved Lie algebras by defining the curvature of a differential graded Lie algebra to be zero and the two notions of Maurer--Cartan elements correspond in this way.

\subsection{Twistings}
\Needspace*{4\baselineskip}

\begin{proposition}
 Let $\mathfrak{g}$ be a curved Lie algebra with curvature $\Omega$ and predifferential $d$. Let $\xi\in\mathfrak{g}$ be a degree one element. Then the underlying graded Lie algebra of $\mathfrak{g}$ equipped with the degree two element $\Omega + d\xi + \frac{1}{2}[\xi,\xi]$ and the derivation $d + \ad_\xi$ is a curved Lie algebra, which we denote $\mathfrak{g}^\xi$.
\end{proposition}

\begin{proof}
This is a straightforward check.
\end{proof}

We write $\Omega^\xi = d\xi + \frac{1}{2}[\xi,\xi]$ and call it the curvature of $\xi$. We also write $d^\xi = \ad_\xi$. Note that in general for any two elements $\xi,\eta\in\mathfrak{g}$ we have
\[\Omega^{\xi+\eta} = \Omega^{\xi} + \Omega^{\eta} + d^{\xi}\eta.\]

With this notation $\mathfrak{g}^\xi$ has curvature $\Omega+\Omega^\xi$ and predifferential $d+d^\xi$. We say that $\mathfrak{g}^\xi$ is the curved Lie algebra obtained by twisting $\mathfrak{g}$ by $\xi$.

\begin{proposition}
Let $\mathfrak{g}$ be a curved Lie algebra with curvature $\Omega$ and predifferential $d$. Let $\xi\in\mathfrak{g}$ be a degree one element. Then $\mathfrak{g}^\xi$ is a differential graded Lie algebra (and so has zero curvature) if and only if $\xi\in\MC(\mathfrak{g})$.
\end{proposition}

\begin{proof}
The curvature of $\mathfrak{g}^\xi$ vanishes if and only if $\Omega + \Omega^\xi = \Omega + d\xi + \frac{1}{2}[\xi,\xi] = 0$.
\end{proof}

\begin{proposition}
Let $\mathfrak{g}$ be a curved Lie algebra and let $\xi\in\mathfrak{g}$ be a degree one element. Then there is a bijection $\MC(\mathfrak{g}^\xi)\rightarrow\MC(\mathfrak{g})$ given by $\eta \mapsto \eta + \xi$.
\end{proposition}

\begin{proof}
We calculate that
\begin{align*}
\Omega + d(\eta + \xi) + \frac{1}{2}[\eta+\xi,\eta+\xi] &= \Omega + d\xi + \frac{1}{2}[\xi,\xi] + d\eta + +\frac{1}{2}[\eta,\xi] + \frac{1}{2}[\xi,\eta] +\frac{1}{2}[\eta,\eta]
\\
&= \Omega + \Omega^\xi + d\eta + d^\xi\eta + \frac{1}{2}[\eta,\eta]
\end{align*}
and so $\eta\in\MC(\mathfrak{g}^\xi)$ if and only if $\eta+\xi\in\MC(\mathfrak{g})$.
\end{proof}

\subsection{Maps of commutative algebras}
Maurer--Cartan elements correspond to maps between commutative differential graded algebras, which we shall now recall.

Given a Lie algebra $\mathfrak{g}$ define the ideals $[\mathfrak{g}]^n$ recursively by $[\mathfrak{g}]^1 = \mathfrak{g}$ and $[\mathfrak{g}]^{n} = [\mathfrak{g},\mathfrak{g}^{n-1}]$. Then $\mathfrak{g}$ is called \emph{nilpotent} if the \emph{descending central series} $[\mathfrak{g}]^1\supset [\mathfrak{g}]^2\supset [\mathfrak{g}]^3\supset\dots$ stabilises at $0$. Note that in the case $\mathfrak{g}$ is finite dimensional this is equivalent to the definition that for every $\xi\in\mathfrak{g}$, $\ad_\xi$ is nilpotent. A differential graded/curved Lie algebra is called nilpotent if the underlying Lie algebra is nilpotent.

By a \emph{formal} differential graded Lie algebra (or curved Lie algebra) we mean an inverse limit of \emph{finite dimensional} \emph{nilpotent} differential graded Lie algebras (or curved Lie algebras). Note that this is naturally a Lie algebra object in $\fdgvect$, although not all Lie algebra objects in $\fdgvect$ are formal.

Given a formal differential graded Lie algebra $\mathfrak{g}$ with continuous differential $d\co\Sigma^{-1}\mathfrak{g}\rightarrow\mathfrak{g}$ and continuous Lie bracket $m\co\mathfrak{g}\otimes\mathfrak{g}\rightarrow\mathfrak{g}$ we define the cobar construction of $\mathfrak{g}$ to be $\CE^{\bullet}(\mathfrak{g})$, the Chevalley--Eilenberg complex with trivial coefficients. More precisely set $\CE^{\bullet}(\mathfrak{g})$ to be the augmented differential graded commutative algebra $S\Sigma^{-1}\mathfrak{g}^*$ with differential $\delta$ defined on $\Sigma^{-1}\mathfrak{g}^*$ by $\delta = \Sigma^{-1}(d^* + m^*)$ and extended to $S\Sigma^{-1}\mathfrak{g}^*$ by the Leibniz rule.

Similarly, given a formal curved Lie algebra $\mathfrak{g}$ with curvature $\Omega\co \Sigma^{-2}k\rightarrow \mathfrak{g}$, predifferential $d\co\Sigma^{-1}\mathfrak{g}\rightarrow\mathfrak{g}$ and Lie bracket $m\co\mathfrak{g}\otimes\mathfrak{g}\rightarrow\mathfrak{g}$ we define the cobar construction $\CE^{\bullet}(\mathfrak{g})$ to be the unital (but not augmented) differential graded commutative algebra $S\Sigma^{-1}\mathfrak{g}^*$ with differential $\delta$ defined on $\Sigma^{-1}\mathfrak{g}^*$ by $\delta = \Sigma^{-1}(\Omega^* + d^* + m^*)$ and extended to $S\Sigma^{-1}\mathfrak{g}^*$ by the Leibniz rule.

Note that there is a functor from augmented differential graded commutative algebras to unital differential graded commutative algebras by forgetting the augmentation map $A\rightarrow k$ and the cobar construction of a given formal differential graded Lie algebra corresponds via this functor to the cobar construction of the corresponding formal curved Lie algebra with zero curvature.

\begin{definition}
\Needspace*{3\baselineskip}\mbox{}
\begin{itemize}
\item Let $\mathfrak{g}$ be a formal differential graded Lie algebra and $A$ be an augmented differential graded commutative algebra with augmentation ideal $A_+$. Noting that $\mathfrak{g}\otimes A_+$ is a differential graded Lie algebra, we write $\MC(\mathfrak{g},A)$ for the functor $(\mathfrak{g}, A)\mapsto \MC(\mathfrak{g}\otimes A_+)$.
\item Let $\mathfrak{g}$ be a formal curved Lie algebra and $A$ be a unital differential graded commutative algebra. Noting that $\mathfrak{g}\otimes A$ is a curved Lie algebra, we write $\MC(\mathfrak{g},A)$ for the functor $(\mathfrak{g}, A)\mapsto \MC(\mathfrak{g}\otimes A)$.
\end{itemize}
\end{definition}

\begin{remark}
Note in particular that $\MC(\mathfrak{g},k) = \MC(\mathfrak{g})$.
\end{remark}

Let $\mathfrak{g}$ be a formal differential graded Lie algebra and $A$ be an augmented differential graded commutative algebra. Note that a degree one element $\xi\in \mathfrak{g}\otimes A_+$ is a degree zero element in $\Sigma\mathfrak{g}\otimes A_+$ which determines and is determined by a map $\xi\co\Sigma^{-1}\mathfrak{g}^*\rightarrow A_+$. In turn this determines and is determined by a map of augmented graded commutative algebras $S\Sigma^{-1} \mathfrak{g}^*\rightarrow A$. The condition that $\xi$ determines a map of augmented \emph{differential} graded commutative algebras $\CE^{\bullet}(\mathfrak{g})\rightarrow A$ is precisely the condition that $\xi\in\MC(\mathfrak{g},A)$.

Similarly let $\mathfrak{g}$ be a formal curved Lie algebra and $A$ be a unital differential graded commutative algebra. A degree one element $\xi\in \mathfrak{g}\otimes A$ determines and is determined by a map $\xi\co\Sigma^{-1}\mathfrak{g}^*\rightarrow A$. In turn this determines and is determined by a map of unital graded commutative algebras $S\Sigma^{-1} g^*\rightarrow A$. The condition that $\xi$ determines a map of unital \emph{differential} graded commutative algebras $\CE^{\bullet}(\mathfrak{g})\rightarrow A$ is precisely the condition that $\xi\in\MC(\mathfrak{g},A)$.

The following proposition now follows from this discussion.

\begin{proposition}\label{prop:mcnatiso}
The functors $(\mathfrak{g},A)\mapsto\Hom(\CE^{\bullet}(\mathfrak{g}),A)$ and $(\mathfrak{g},A)\mapsto\MC(\mathfrak{g},A)$ (whether considered as functors from formal differential graded Lie algebras and augmented differential graded commutative algebras, or from formal curved Lie algebras and unital differential graded commutative algebras) are naturally isomorphic.
\qed
\end{proposition}

\section{The Maurer--Cartan moduli set}
\autoref{prop:mcnatiso} motivates the natural notion of homotopy between Maurer--Cartan elements. In this section the pair $\mathfrak{g}$ and $A$ are either a formal differential graded Lie algebra and an augmented differential graded commutative algebra (whose augmentation ideal we denote by $A_+$), or a formal curved Lie algebra and a unital differential graded commutative algebra.

Denote by $k[z,dz]$ the free unital differential graded commutative algebra on the generators $z$ and $dz$ with $\degree{z} = 0$, $\degree{dz} = 1$ and $d(z) = dz$. Denote by $A[z,dz]$ the augmented or unital differential graded commutative algebra given by $A\otimes k[z,dz]$. We denote the quotient maps given by setting $z$ to $0$ or $1$ by $|_0,|_1\co A[z,dz] \rightarrow A$.

\begin{definition}
Two elements $\xi,\eta\in\MC(\mathfrak{g},A)$ are called \emph{homotopic} if there is an element $h\in\MC(\mathfrak{g}, A[z,dz])$ with $h|_{0} = \xi$ and $h|_{1} = \eta$.
\end{definition}

\begin{remark}\label{rem:homotopyissullivanhomotopy}
By \autoref{prop:mcnatiso} we see that $h\in\MC(\mathfrak{g},A[z,dz])$ corresponds to a map $h\co \CE^{\bullet}(\mathfrak{g}) \rightarrow A[z,dz]$ restricting to the maps corresponding to $\xi,\eta\in\MC(\mathfrak{g},A)$ at $z=0$ and $z=1$. Therefore a homotopy of Maurer--Cartan elements is precisely a Sullivan homotopy (right homotopy with $A[z,dz]$ a path object for $A$) between the corresponding maps of augmented/unital differential graded commutative algebras.
\end{remark}

Homotopy of Maurer--Cartan elements defines a relation that may not be transitive (unless $\CE^{\bullet}(\mathfrak{g})$ is cofibrant) so we will consider the transitive closure.

\begin{definition}
We denote by $\MCmoduli(\mathfrak{g},A)$ the set of equivalence classes under the transitive closure of the homotopy relation. We call this \emph{the Maurer--Cartan moduli set of $\mathfrak{g}$ with coefficients in $A$}.
\end{definition}

\subsection{Gauge equivalence}
Let $\mathfrak{g}$ be a \emph{pronilpotent} differential graded/curved Lie algebra, by which we mean an inverse limit of nilpotent algebras, but which may not necessarily be finite dimensional. Recall that for every such Lie algebra there is an associative product $\bullet\co\mathfrak{g}\times\mathfrak{g}\rightarrow \mathfrak{g}$ given by the Baker--Campbell--Hausdorff formula which is functorial (given $f\co\mathfrak{g}\rightarrow\mathfrak{h}$ then $f(x\bullet y) = f(x)\bullet f(y)$) and for any unital associative algebra $A$ with pronilpotent ideal $I$ it holds for any $a,b\in I$ that $e^a e^b = e^{a\bullet b}$ where $e^a = \sum_{n\geq 0} \frac{a^n}{n!}\in A$ and $\bullet$ is taken with respect to the commutator Lie bracket on $A$. A property of $\bullet$ is that for any $x,y\in\mathfrak{g}$ if $[x,y]=0$ then $x\bullet y = x + y$.

Define the group $\exp(\mathfrak{g}) = \{e^x : x\in\mathfrak{g} \}$ with product defined as $e^x\cdot e^y = e^{x\bullet y}$. The identity is $1=e^0$ and $e^x\cdot e^{-x} = e^{-x}\cdot e^{x} = 1$. It follows from the pronilpotency of $\mathfrak{g}$ and the above properties of $\bullet$ that the adjoint representation $y\mapsto \ad_y$ exponentiates to an action of $\exp(\mathfrak{g})$ on $\mathfrak{g}$ given by $e^y\mapsto e^{\ad_y}$.

 Let $\xi\in\MC(\mathfrak{g})$ and $y\in\mathfrak{g}^0$. Define the \emph{gauge action} by
\[
e^y\cdot \xi = e^{\ad_y}\xi + (de^{\ad_y})y = \xi + \sum_{n=1}^{\infty} \frac{1}{n!}(\ad_y)^{n-1}(\ad_y\xi-dy).
\]
Then this indeed gives an action of $\exp(\mathfrak{g}^0)$ on $\MC(\mathfrak{g})$.

\begin{proposition}\label{prop:stabcycles}
Let $\mathfrak{g}$ be a pronilpotent curved Lie algebra. Given $\xi\in\MC(\mathfrak{g})$ then $\exp(\mathfrak{g}^0)_\xi = \{ e^y : (d+d^\xi)y = 0 \}$ where $\exp(\mathfrak{g}^0)_\xi$ is the stabiliser of $\xi$ by the gauge action.
\end{proposition}

\begin{proof}
Let $y\in \mathfrak{g}^0$. Since $(d+d^{\xi})y=-(\ad_y\xi - dy)$ then if $(d+d^{\xi})y=0$ it is clear that $e^y\cdot \xi = \xi$. Conversely since $\mathfrak{g}=\lim_{\leftarrow}\mathfrak{g}_i$ it is sufficient to show that $\ad_y\xi - dy=0$ under the image of every $\mathfrak{g}\rightarrow\mathfrak{g_i}$. The $\mathfrak{g}_i$ are nilpotent so for each $\mathfrak{g}_i$ there exists some least $N$ such that $(\ad_y)^{N}(\ad_y\xi-dy)=0$. Then $(\ad_y)^{N-1}(e^y\cdot\xi - \xi) = (\ad_y)^{N-1}(\ad_y\xi - dy) = 0$ so $\ad_y\xi-dy=0$ as required.
\end{proof}

Now let $\mathfrak{g}$ be a \emph{formal} curved/differential graded Lie algebra and let $A$ be a unital/augmented differential graded commutative algebra. Note that in this case $\mathfrak{g}\otimes A$ is pronilpotent.

\begin{definition}
Two Maurer--Cartan elements $\xi,\eta\in\MC(\mathfrak{g},A)$ are called \emph{gauge equivalent} if there is an element $y\in(\mathfrak{g}\otimes A)^0$ such that $e^y\cdot\xi = \eta$.
\end{definition}

It is natural to also consider the quotient of $\MC(\mathfrak{g},A)$ by the gauge action. In fact the following important result, due originally to Schlessinger--Stasheff \cite{schlessingerstasheff}, tells us the quotient is precisely $\MCmoduli(\mathfrak{g},A)$:

\begin{theorem}[Schlessinger--Stasheff theorem]
Two Maurer--Cartan elements are gauge equivalent if and only if they are homotopic.
\end{theorem}

\begin{proof}
If $\eta = e^y\cdot\xi$ then $e^{yz}\cdot\xi$ is a homotopy from $\xi$ to $\eta$. The converse is less straightforward and omitted here. Instead see, for example, \cite{chuanglazarev2}.
\end{proof}

Gauge equivalence is often more convenient to work with than homotopy equivalence.

\section{Lifts of Maurer--Cartan elements}
Given a map $f\co\mathfrak{g}\rightarrow \mathfrak{h}$ and an element $\xi\in\MC(\mathfrak{h})$ we wish to examine if $\xi$ lifts to a Maurer--Cartan element in $\mathfrak{g}$ and if so, in how many ways. This can be thought of as a general version of certain deformation theory problems and the associated obstruction theory.

The following proposition allows us to understand the space of lifts as the Maurer--Cartan set of a curved Lie algebra.

\begin{proposition}\label{prop:mcfibre}
Let $f\co\mathfrak{g}\rightarrow \mathfrak{h}$ be a map of curved Lie algebras. Let $\mathfrak{k}\subset\mathfrak{g}$ be the kernel of $f$ (regarded as a map of graded Lie algebras). Let $\xi_0\in\MC(\mathfrak{h})$ and $\xi\in\mathfrak{g}$ with $f(\xi)=\xi_0$.
\begin{enumerate}
\item The space $\mathfrak{k}$ is a curved Lie subalgebra of $\mathfrak{g}^{\xi}$, which we will denote by $\mathfrak{k}^{\xi}$.
\item The fibre over $\xi_0\in\MC(\mathfrak{h})$ of $\MC(\mathfrak{g})\rightarrow\MC(\mathfrak{h})$ is isomorphic to $\MC(\mathfrak{k}^\xi)$ by the map $\MC(\mathfrak{k}^{\xi})\subset\MC(\mathfrak{g}^{\xi})\xrightarrow{\eta\mapsto\eta+\xi}\MC(\mathfrak{g})$. In particular it is independent of the choice of $\xi$.
\end{enumerate}
\end{proposition}

\begin{proof}
To show $\mathfrak{k}$ is a curved Lie subalgebra of $\mathfrak{g}^{\xi}$ we must show $\Omega_{\mathfrak{g}} + \Omega^{\xi}\in\mathfrak{k}$ and for $\omega\in\mathfrak{k}$ we have $d^{\xi}\omega\in\mathfrak{k}$. But $f(\Omega_{\mathfrak{g}} + \Omega^{\xi}) = \Omega_{\mathfrak{h}} + \Omega^{\xi_0}= 0$ since $\xi_0$ is a Maurer--Cartan element in $\mathfrak{h}$. Also $f(d^{\xi}\omega) = [\xi_0,f(\omega)] = 0$.

That $\MC(\mathfrak{k}^{\xi})$ is the fibre over $\xi_0$ follows since $\eta\in\MC(\mathfrak{k}^{\xi})$ if and only if $\eta + \xi\in\MC(\mathfrak{g})$ and $f(\eta+\xi) = f(\eta) + \xi_0 = \xi_0$.
\end{proof}

\begin{remark}
Note that, despite what the notation may suggest, $\mathfrak{k}^{\xi}$ is \emph{not} necessarily obtained by twisting $\mathfrak{k}$ by some element of $\mathfrak{k}$.
\end{remark}

\begin{proposition}\label{prop:gaugeequivoffibres}
Let $f\co\mathfrak{g}\twoheadrightarrow \mathfrak{h}$ be a surjective map of pronilpotent curved Lie algebras. The fibres over any two gauge equivalent (and hence homotopy equivalent) elements $\xi_0, \xi_0'\in\MC(\mathfrak{h})$ are isomorphic.
\end{proposition}

\begin{proof}
There is a degree zero element $y_0\in\mathfrak{h}$ such that $e^{y_0}\cdot \xi_0 = \xi_0'$. Choose elements $\xi,y\in\mathfrak{g}$ with $f(\xi)=\xi_0$ and $f(y)=y_0$ so that $\MC(\mathfrak{k}^{\xi})$ is the fibre over $\xi_0$ and $\MC(\mathfrak{k}^{e^{y}\cdot\xi})$ is the fibre over $\xi_o'$ (since $f(e^y\cdot\xi) = e^{y_0}\cdot\xi_0 = \xi_0'$). Given $\eta\in\MC(\mathfrak{k}^{\xi})$ set $g(\eta) = e^{y}\cdot \eta$. Then $g(\eta)\in \MC(\mathfrak{k}^{e^{y}\cdot\xi})$ since $f(e^y\cdot\eta + e^y\cdot\xi)= \xi_0'$. This gives an isomorphism $g\co\MC(\mathfrak{k}^{\xi})\rightarrow\MC(\mathfrak{k}^{e^y\cdot\xi})$ as required.
\end{proof}

We now wish to understand the space of lifts \emph{up to homotopy}, or equivalently the fibre in $\MCmoduli(\mathfrak{g})$ over $\xi_0\in\MCmoduli(\mathfrak{h})$. From now it will be assumed that $f\co\mathfrak{g}\twoheadrightarrow\mathfrak{h}$ is a surjective map between curved Lie algebras. 

There are two natural ways of speaking about equivalence of lifts of $\xi_0\in\MC(\mathfrak{h})$. Given $\xi\in\mathfrak{g}$ such that $f(\xi)=\xi_0$ we could of course say that $\eta,\eta'\in\MC(\mathfrak{k}^\xi)$ are equivalent if they are equivalent as Maurer--Cartan elements in $\mathfrak{k}^\xi$. In terms of Sullivan homotopy this means that there is an element $h\in\MC(\mathfrak{k}^\xi[z,dz])$ with $h|_0 = \eta$ and $h|_1 = \eta'$. In particular at every value for $z$ it is the case that $h$ is an element of $\MC(\mathfrak{k}^\xi)$. In other words this is a homotopy through lifts of $\xi_0$. The space of lifts of $\xi_0$ up to homotopy in this sense is then just $\MCmoduli(\mathfrak{k}^\xi)$ and does not depend on the choice of $\xi$. Note that if $\xi_0'=e^{y_0}\cdot\xi_0$ then given $y\in\mathfrak{g}$ with $f(y)=y_0$ and a homotopy $h\in\MC(\mathfrak{k}^{\xi}[z,dz])$ with $h|_0 = \eta$ and $h|_1 = \eta'$ then $e^y\cdot h\in\MC(\mathfrak{k}^{e^y\cdot\xi}[z,dz])$ is a homotopy from $e^y\cdot\eta$ to $e^y\cdot\eta'$ so by \autoref{prop:gaugeequivoffibres} gauge equivalent elements have the same space of lifts up to homotopy in this sense. Therefore the space of lifts of up to homotopy in this sense is well defined for a homotopy class $\xi_0\in\MCmoduli(\mathfrak{h})$.

Alternatively we could say that $\eta,\eta'\in\MC(\mathfrak{k}^\xi)$ are equivalent if they are equivalent as Maurer--Cartan elements in $\mathfrak{g}$. That is, the elements $\eta+\xi, \eta'+\xi\in\MC(\mathfrak{g})$ are equivalent as Maurer--Cartan elements. In general this will not be the same since two elements may now be homotopic via a homotopy not necessarily through lifts of $\xi_0$. The space of lifts up to homotopy in this sense is just the fibre in $\MCmoduli(\mathfrak{g})$ over $\xi_0\in\MCmoduli(\mathfrak{h})$. These two different notions are related as follows.

\begin{theorem}\label{thm:fibration}
Let $f\co\mathfrak{g}\twoheadrightarrow \mathfrak{h}$ be a surjective map of pronilpotent curved Lie algebras. Then $\exp(H^0(\mathfrak{h}^{\xi_0}))$ acts on $\MCmoduli(\mathfrak{k}^\xi)$ and $\MCmoduli(\mathfrak{k}^\xi)/\exp(H^0(\mathfrak{h}^{\xi_0}))$ is isomorphic to the fibre in $\MCmoduli(\mathfrak{g})$ over $\xi_0\in\MCmoduli(\mathfrak{h})$.
\end{theorem}

\begin{proof}
Given $y_0\in\mathfrak{h}^0$ such that $(d+d^{\xi_0})y_0=0$ and $y,y'\in\mathfrak{g}^0$ with $f(y)=f(y')=y_0$ then for any $\eta\in\MC(\mathfrak{k}^\xi)$ let
\[h=e^{y'z}\cdot e^{-yz}\cdot e^{y}\cdot(\eta+\xi)\]
so that $h\in\MC(\mathfrak{g}[z,dz])$ is a homotopy from $e^y\cdot(\eta+\xi)$ to $e^{y'}\cdot(\eta+\xi)$. Then $f(h)=e^{y_0}\cdot \xi_0 = \xi_0$ by \autoref{prop:stabcycles} so $h$ is a homotopy through lifts of $\xi_0$ and so $e^y\cdot(\eta+\xi) - \xi$ and $e^{y'}\cdot(\eta + \xi) - \xi$ are homotopy equivalent as elements in $\MC(\mathfrak{k}^\xi)$. Therefore this gives a well defined action $e^{y_0}\star\eta = e^y\cdot(\eta+\xi)-\xi$ on $\MCmoduli(\mathfrak{k}^\xi)$ for cycles in $(\mathfrak{h}^{\xi_0})^0$. Furthermore given $x_0\in\mathfrak{h}^0$ such that $x_0 = y_0 + (d+d^{\xi_0})b_0$ for some $b_0\in\mathfrak{h}^{-1}$ and $b\in\mathfrak{g}^{-1}$ with $f(b)=b_0$ set $x=y+(d+d^{\xi})b$ so that $f(x)=x_0$. Set
\[h = e^{y + (d + d^\xi)(bz)}\cdot (\eta+\xi)\]
for $\eta\in\MC(\mathfrak{k}^\xi)$ and then $h\in\MC(\mathfrak{g}[z,dz])$ so $h$ gives a homotopy from $e^y\cdot(\eta+\xi)$ to $e^x\cdot(\eta+\xi)$. Since $y_0+(d+d^{\xi_0})(b_0z)$ is a cycle in $\mathfrak{h}^{\xi_0}[z,dz]$ then again by \autoref{prop:stabcycles} $f(h) = \xi_0$ so $h$ is a homotopy through lifts of $\xi_0$ and hence $e^{y_0}\star\eta$ and $e^{x_0}\star\eta$ are equivalent as elements in $\MC(\mathfrak{k}^\xi)$. Therefore $\star$ descends to a well defined action of $\exp(H^0(h^{\xi_0}))$ on $\MCmoduli(\mathfrak{k}^\xi)$.

That $\MCmoduli(\mathfrak{k}^\xi)/\exp(H^0(\mathfrak{h}^{\xi_0}))$ is isomorphic to the fibre in $\MCmoduli(\mathfrak{g})$ over $\xi_0\in\MCmoduli(\mathfrak{h})$ follows from how the action of $\exp(H^0(\mathfrak{h}^{\xi_0}))$ was defined above together with \autoref{prop:stabcycles}.
\end{proof}

\begin{remark}
Note that, like all these results, \autoref{thm:fibration} continues to hold in the supergraded setting, replacing $H^0(\mathfrak{h}^{\xi_0})$ with $H^{\mathrm{even}}(\mathfrak{h}^{\xi_0})$.
\end{remark}

\begin{remark}\label{rem:conceptualfibration}
There is perhaps a more conceptual, albeit less elementary, way of proving \autoref{thm:fibration}. The category of unital differential graded commutative algebras is almost a simplicial model category, for example using the results of \cite{hinich}. In particular it satisfies the corner axiom: Given a cofibration $i\co A\rightarrow B$ and a fibration $p\co X\rightarrow Y$ the induced map 
\[(i^*,p_*)\co \Hom(B,X)_\bullet \rightarrow \Hom(A,X)_\bullet\times_{\Hom(A,Y)_\bullet}\Hom(B,Y)_\bullet\]
is a fibration of simplicial sets. In particular for formal\footnote{Requiring formal as opposed to just pronilpotent here seems weaker, but in practice all the pronilpotent Lie algebras one encounters normally arise as the tensor product of a formal Lie algebra and discrete commutative algebra. Therefore considering $\MCmoduli(\mathfrak{g},A)$ instead of just $\MCmoduli(\mathfrak{g},k)$ is usually sufficient.} curved Lie algebras $\mathfrak{h}$ and $\mathfrak{g}$ with a map $f\co\mathfrak{g}\twoheadrightarrow\mathfrak{h}$, setting $A=\CE^{\bullet}(\mathfrak{h})$ and $B=\CE^{\bullet}(\mathfrak{g})$ then the map $f^*\co A\rightarrow B$ is a cofibration (see, for example, the characterisation of cofibrations given in \cite{hinich}). Furthermore $\MC_\bullet(\mathfrak{g},k)\simeq\Hom(B,k)_\bullet$ and $\MC_\bullet(\mathfrak{h},k)\simeq\Hom(A,k)_\bullet$ where $\MC_\bullet$ is the \emph{Maurer--Cartan simplicial set}. So a surjective map of formal curved Lie algebras yields a fibration of Maurer--Cartan simplicial sets and \autoref{thm:fibration} can be obtained by considering the long exact sequence in homotopy together with the facts that for any pronilpotent curved Lie algebra $\pi_0\MC_\bullet(\mathfrak{g})=\MCmoduli(\mathfrak{g})$ and $H^0(\mathfrak{g}^\xi)$ is $\pi_1$ of the connected component of $\MC_\bullet(\mathfrak{g})$ containing $\xi$. In particular, from this point of view the standard picture in \autoref{fig:mcfibre} now becomes quite enlightening to keep in mind.
\end{remark}

\begin{figure}[ht!]
\centering
\[{\footnotesize
\begin{xy}
*\xybox{+(0,7),{\ellipse va(210),_,va(-30){-}};
+(0,14),{\ellipse va(210),_,va(-30){-}};
+(0,14),{\ellipse va(210),_,va(-30){-}};}
!UC+(0,5)*\xybox{\ellipse(8,3){-}}="centery"
-(0,45)*\xybox{\ellipse(8,3){-}}+(0,5);p+(0,8);**\dir{-}*\dir2{>}
-(4,10)*{\xi_0}="xi",
+(0,2.4)*\dir{*};
p+(0,19)*\dir{*}="1",
+(0,7)*\dir{*}="2",
+(0,7)*\dir{*}="3",
+(0,12)*\dir{*}="4",
+(0,7),
**\dir{.},
"xi"+(24,5)*{\MC_\bullet(\mathfrak{h})}
+(0,30)*{\MC_\bullet(\mathfrak{g})}
-(40,0)*{\MC_\bullet(\mathfrak{k^{\xi}})}+(6,0)="k";
"1"-(1.5,0.5)**\crv{~**\dir{.} "k"-(0,15)}?(1)*\dir{>},
"2"-(3.5,-0.5)**\crv{~**\dir{.} "k"-(-2,5)}?(1)*\dir{>},
"3"-(3.5,-0.7)**\crv{~**\dir{.} "k"+(2,2)}?(1)*\dir{>},
"4"-(1.5,0.5)**\crv{~**\dir{.} "k"+(0,10)}?(1)*\dir{>},
\end{xy}
}\]
\caption{In this standard picture of a fibration $\MCmoduli(\mathfrak{h})$ has one element, $\MCmoduli(\mathfrak{g})$ has two elements and $\MCmoduli(\mathfrak{k^{\xi}})$ has four elements, although the fibre in $\MCmoduli(\mathfrak{g})$ over $\xi_0$ has two elements.}
\label{fig:mcfibre}
\end{figure}
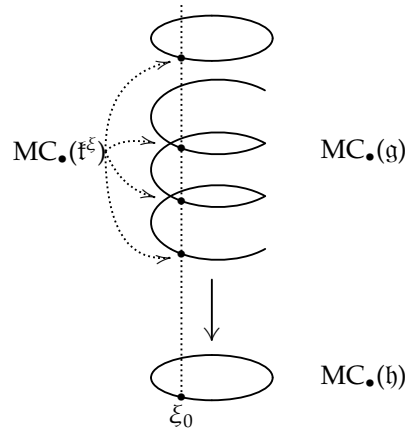

In practice we often have more structure than this general setup, normally in the form of a filtration. By a filtration of a curved Lie algebra $\mathfrak{g}$ we mean a descending filtration $\mathfrak{g}=F_0 \supset F_1 \supset F_2 \dots$ of subspaces such that the bracket and differential on $\mathfrak{g}$ preserves the filtration degree:
\[
[F_p,F_q] \subset F_{p+q}\qquad d(F_n)\subset F_n
\]
Furthermore we will require filtrations to be Hausdorff, so that $\bigcap F_i = 0$. Recall that such a filtration is \emph{complete} if $\mathfrak{g} = \lim_{\leftarrow} \mathfrak{g}/F_i$ (complete filtrations will always be assumed Hausdorff). If $\mathfrak{g}$ is a curved Lie algebra the \emph{canonical filtration} is given by $F_i = \ker(\mathfrak{g}\rightarrow \mathfrak{g}/[\mathfrak{g}]^i)$. This is a descending filtration which is Hausdorff if $\mathfrak{g}$ is pronilpotent.

\begin{proposition}\label{prop:liftuptower}
Let $\mathfrak{g}$ be a curved Lie algebra with a descending Hausdorff filtration $\mathfrak{g}=F_0 \supset F_1 \supset F_2 \dots$. Then $\xi\in\MC(\mathfrak{g})$ if and only if for all $n$, $\xi\in\MC(\mathfrak{g}/F_n)$. Moreover, if this filtration is complete then $\MC(\mathfrak{g}) = \lim_{\leftarrow} \MC(\mathfrak{g}/F_n)$.
\end{proposition}

\begin{proof}
Clearly if $\xi\in\MC(\mathfrak{g})$ then $\xi\in\MC(\mathfrak{g}/F_n)$. Conversely assume $\xi\notin\MC(\mathfrak{g})$. Then $\Omega + d\xi + \frac{1}{2}[\xi,\xi] = \alpha \neq 0$. Since the filtration is Hausdorff there is some $n > 0$ such that $\alpha\notin F_n$ so $\alpha \neq 0 \mod F_n$. If the filtration is complete then any compatible sequence $\xi_n\in\MC(\mathfrak{g}/F_n)$ assembles to an element $\xi\in\mathfrak{g}$ that is then clearly Maurer--Cartan.
\end{proof}

It is commonly the case when one is interested in lifting Maurer--Cartan elements along a surjective map $\mathfrak{g}\rightarrow\mathfrak{h}$ that $\mathfrak{g}$ has a complete filtration with $\mathfrak{g}/F_1\cong\mathfrak{h}$. Therefore from now on let $\mathfrak{k}$ be a curved Lie algebra with curvature $\Omega$ and predifferential $d$ and a descending complete Hausdorff filtration $\mathfrak{k}=F_0 \supset F_1 \supset F_2 \dots$ with the additional property $\mathfrak{k}=F_1$. Note that $\mathfrak{k}$ is then pronilpotent.

Finding lifts is now reduced to the general problem of finding $\xi\in\MC(\mathfrak{k})$. \autoref{prop:liftuptower} allows us to reinterpret the problem as finding a sequence of lifts up the tower of curved Lie algebras:
\[
0\cong\mathfrak{k}/F_1 \twoheadleftarrow \mathfrak{k}/F_2 \twoheadleftarrow \dots \twoheadleftarrow \mathfrak{k}/F_n \twoheadleftarrow \dots
\]
This is a standard picture in deformation theory and the problem is governed by an explicit obstruction theory, which we shall now unwrap. 

Let $\xi_n\in\MC(\mathfrak{k}/F_{n+1})$. We regard $\xi_n$ as an element $\xi_n\in\mathfrak{k}/F_{n+2}$ using the decomposition of vector spaces $\mathfrak{k}/F_{n+2}\cong \mathfrak{k}/F_{n+1} \oplus F_{n+1}/F_{n+2}$. As in \autoref{prop:mcfibre} the space $F_{n+1}/F_{n+2}$ is a curved Lie subalgebra of $(\mathfrak{k}/F_{n+2})^{\xi_n}$ with curvature $\Omega + \Omega^{\xi_n} \in F_{n+1}/F_{n+2}$ and predifferential $d + d^{\xi_n}$. However, since the bracket preserves filtration degree then $d^{\xi_n}$ only depends on $\xi_n \mod F_1$ which is just zero, so the predifferential is in fact just $d$. Furthermore, since $\Omega \in F_1$ then for $\eta\in F_{n+1}$ we have that $d^2(\eta) = [\Omega, \eta] \in F_{n+2}$ so $d^2 = 0 \mod F_{n+2}$. It now also follows that $\Omega + \Omega^{\xi_n}$ is a cocycle with respect to $d$. Moreover $\eta \in F_{n+1}/F_{n+2}$ is Maurer--Cartan if and only if $\Omega + \Omega^{\xi_n} + d\eta +\frac{1}{2}[\eta,\eta] = \Omega + \Omega^{\xi_n} + d\eta = 0 \mod F_{n+2}$, or in other words if and only if $\Omega + \Omega^{\xi_n}$ is the coboundary of $\eta$. We call the curvature $\Omega + \Omega^{\xi_n}$ the \emph{obstruction at level $n+1$}.

Denote by $C^{\bullet}_n(\mathfrak{k})$ the differential graded vector space $F_{n}/F_{n+1}$, with differential $d$ and the corresponding cohomology by $H^\bullet_n(\mathfrak{k})$. This discussion now leads to the following theorem, which is a generalisation of standard results having a certain flavour familiar from deformation theory.

\begin{theorem}\label{thm:liftobs}
An element $\xi_n\in\MC(\mathfrak{k}/F_{n+1})$ at level $n$ lifts to an element $\xi_{n+1}\in\MC(\mathfrak{k}/F_{n+2})$ at level $n+1$ if and only if the obstruction $\Omega + \Omega^{\xi_n}$ at level $n+1$, which is a cocycle in $C^2_{n+1}(\mathfrak{k})$, vanishes as a cohomology class in $H^2_{n+1}(\mathfrak{k})$.
\qed
\end{theorem}

Let $\xi_{n+1}=\xi_n + \eta$ and $\xi_{n+1}' = \xi_n + \eta'$ be two lifts at level $n+1$ of $\xi_n$. Then they are equivalent as lifts of $\xi_n$ if and only if there is some degree zero $y\in F_{n+1}/F_{n+2}$ such that $\eta = e^y\cdot \eta' = \eta' - dy$, so in fact if and only if $\eta$ and $\eta'$ differ by a coboundary in $C^1_{n+1}(\mathfrak{k})$. Clearly adding a cocycle in $C^1_{n+1}(\mathfrak{k})$ to a level $n+1$ lift gives another level $n+1$ lift and any two lifts differ by such a cocycle. Therefore one obtains another general version of familiar deformation theory results.

\begin{theorem}\label{thm:affinespaceoflifts}
The cohomology $H^1_{n+1}(\mathfrak{k})$ acts freely and transitively on the moduli space (up to homotopy preserving the element at level $n$) of level $n+1$ lifts.
\qed
\end{theorem}

Maps which preserve the cohomology associated to the obstruction theory are of particular interest.

\begin{definition}
Let $\mathfrak{k}$ and $\mathfrak{k'}$ be curved Lie algebras with descending complete filtrations $\mathfrak{k}=F_1\supset F_2\dots$ and $\mathfrak{k}'=F_1\supset F_2\dots$. A \emph{filtered quasi-isomorphism} $\theta\co\mathfrak{k}\rightarrow\mathfrak{k}'$ is a map of curved Lie algebras preserving the filtration degree so that $\theta(F_n)\subset \theta(F_n')$ and where the induced maps $\theta^*\co C_n^{\bullet}(\mathfrak{k})\rightarrow C_n^\bullet(\mathfrak{k}')$ are quasi-isomorphisms.
\end{definition}

The following result is a curved variation on a well established theme concerning the homotopy invariance of the Maurer--Cartan moduli set, see for example \cite{kontsevich3,lazarev}. From the obstruction theory perspective it can be understood as saying that the cohomology associated to the obstruction theory effectively determines the problem. The proof given here keeps this perspective in mind.

\begin{theorem}\label{thm:quasiisomc}
Let $\theta\co\mathfrak{k}\rightarrow\mathfrak{k}'$ be a filtered quasi-isomorphism of curved Lie algebras. Then $\theta$ induces an isomorphism $\MCmoduli(\mathfrak{k}) \cong \MCmoduli(\mathfrak{k}')$.
\end{theorem}

\begin{proof}
Let $\xi'\in\MC(\mathfrak{k}')$ and write $\xi_n'$ for the projection to $\MC(\mathfrak{k}'/F_{n+1}')$. Assume for some $n$ there is $\xi_n\in\MC(\mathfrak{k}/F_{n+1})$ with $\theta(\xi_n)$ equivalent to $\xi_n'$ as elements in $\MC(\mathfrak{k}'/F_{n+1}')$. By applying some gauge equivalence to $\xi'$ we may assume that $\xi_n'=\theta(\xi_n) \mod F_{n+1}$. Then since the obstruction at level $n+1$ in $C_{n+1}^2(\mathfrak{k'})$ vanishes in cohomology so does the obstruction in $C_{n+1}^2(\mathfrak{k})$. Let $\eta' = \xi_{n+1}' - \xi_n'\in C_{n+1}^1(\mathfrak{k})$. Since $\theta^*$ is a quasi-isomorphism choose some lift $\eta\in C_{n+1}^1(\mathfrak{k})$ of $\xi_n$ such that $\theta^*(\eta)-\eta'$ is a coboundary (this is possible by \autoref{thm:affinespaceoflifts}). Then $\theta(\xi_{n}+\eta)$ is equivalent to $\xi_{n+1}'$ as elements in $\MC(\mathfrak{k}'/F_{n+2}')$. Therefore by induction and \autoref{prop:liftuptower} $\theta$ induces a surjection of Maurer--Cartan moduli spaces.

Given $\xi_n\in\MC(\mathfrak{k}/F_{n+1})$ consider the map $(\mathfrak{k}/F_{n+1})^{\xi_n}\rightarrow (\mathfrak{k'}/F_{n+1}')^{\theta(\xi_n)}$ induced by $\theta$. This is a quasi-isomorphism on the associated graded complexes so this map is a quasi-isomorphism.

Now given $\xi, \zeta\in\MC(\mathfrak{k})$ assume $\theta(\xi)$ and $\theta(\zeta)$ are equivalent. Assume further that $\xi_n = \zeta_n$ for some $n$. Then $\theta(\xi_{n+1} - \xi_n)$ and $\theta(\zeta_{n+1} -\zeta_n)$ are, up to a coboundary, in the same orbit under the action of $H^0((\mathfrak{k}'/F_{n+1}')^{\theta(\xi_n)})$ by \autoref{thm:fibration}. But this is isomorphic to $H^0((\mathfrak{k}/F_{n+1})^{\xi_n})$ so the same holds true for $\xi_{n+1}-\xi_n$ and $\zeta_{n+1}-\zeta_n$ so by applying some gauge equivalence we may assume $\xi_{n+1} = \zeta_{n+1}$. Therefore by induction $\xi$ and $\zeta$ are equivalent and $\theta$ induces an injection and hence an isomorphism of Maurer--Cartan moduli spaces.
\end{proof}

\begin{remark}
As in \autoref{rem:conceptualfibration} there is again a more conceptual way of approaching \autoref{thm:quasiisomc} via the model category structure on unital differential graded commutative algebras. Take a map $\theta\co\mathfrak{k}\rightarrow\mathfrak{k}'$ of \emph{formal} curved Lie algebras such that $\theta^*\co\CE^\bullet(\mathfrak{k}')\rightarrow \CE^{\bullet}(\mathfrak{k})$ is a weak equivalence of unital differential graded commutative algebras ($\theta$ being a \emph{filtered} quasi-isomorphism is sufficient to ensure this). As noted in \autoref{rem:homotopyissullivanhomotopy} the Maurer--Cartan moduli set $\MCmoduli(\mathfrak{k},A)$ is the set of homotopy classes of maps $[\CE^{\bullet}(\mathfrak{k}),A]$ so it now follows that $\MCmoduli(\mathfrak{k},A)\cong\MCmoduli(\mathfrak{k}',B)$ for any $A$ and $B$ which are weakly equivalent.
\end{remark}

\chapter{Quantum homotopy algebras}\label{chap:quantum}

In this chapter we begin a study of quantum homotopy algebras and quantum lifts.

We will begin by reviewing the theory of hyperoperads which extends the basic theory of modular operads reviewed in \autoref{chap:operads}. For simplicity in this chapter we will \emph{not} consider modular operads to be extended modular operads in the sense of \autoref{chap:operads} so, for example, we will only need to consider stable graphs and stable $S$--modules.

We will then develop the theory of Maurer--Cartan elements in a modular operad and prove some of the basic results we will need concerning maps from the Feynman transform of a modular operad. After introducing the definitions of quantum and semi-quantum homotopy algebras we examine in some detail the problem of quantum lifting and provide a simple solution to the problem of lifting strict algebras.

We finish the chapter by considering applications to manifold invariants arising in the quantisation of Chern--Simons field theory.

\section{Hyperoperads}
The cobar construction of an operad is again an operad, in comparison to how the cobar construction of an associative algebra is an associative algebra. However, the Feynman transform, which is the corresponding notion of the cobar construction for modular operads, of a modular operad is not a modular operad but rather a \emph{modular $\mathfrak{K}$--operad}, similar to how the cobar construction of a commutative algebra is not a commutative algebra but rather a Lie algebra. Therefore it is somewhat necessary to consider \emph{hyperoperads}, the objects which govern algebraic structures such as modular operads.

Caveat lector: The technical issues surrounding hyperoperads are not the primary concern here and consequently only the pertinent details will be recalled here and the reader should consult the original paper of Getzler--Kapranov \cite{getzlerkapranov} for a more complete reference. Additionally, experts should note that some slight modifications and generalisations to the setup and notation commonly found in the literature have been made.

\subsection{Modular $\mathfrak{D}$--operads}
Denote by $\tilde{\Gamma}((g,n))$ the category of unlabelled stable graphs obtained from $\Gamma((g,n))$ by forgetting the labelling of legs.

A \emph{hyperoperad} $\mathfrak{D}$ consists of a collection of functors from each of the categories $\iso \tilde{\Gamma}((g,n))$ to $\dgvect$ such that the object associated to the graph with $1$ vertex and no edges is $k$ and for every morphism $f\co G\rightarrow G'$ in $\tilde{\Gamma}((g,n))$ there is an assigned morphism $\nu_f\co\mathfrak{D}(G')\otimes \bigotimes_{v\in\vertices(G')}\mathfrak{D}(f^{-1}(v))\rightarrow \mathfrak{D}(G)$, natural with respect to isomorphisms and satisfying certain unit and associativity conditions. If additionally each $\mathfrak{D}(G)$ is one dimensional and each map $\nu_f$ is an isomorphism we call $\mathfrak{D}$ a \emph{cocycle}. More succinctly a cocycle is a hyperoperad that is invertible in the symmetric monoidal category of hyperoperads. Denote its inverse by $\mathfrak{D}^{-1}$.

\begin{definition}
Let $\mathfrak{D}$ be a hyperoperad. Define the functor $\mathbb{M}_\mathfrak{D}$ by
\[
\mathbb{M}_{\mathfrak{D}}W((g,n))=\operatorname*{colim}_{G\in\iso\Gamma((g,n))}\mathfrak{D}(G)\otimes W((G)).
\]
Then this is an endofunctor on the category of stable $S$--modules and has the structure of a monad (triple) as shown in \cite{getzlerkapranov}. An algebra over this monad is called a \emph{modular $\mathfrak{D}$--operad}.
\end{definition}

Since $\mathbb{M}_\mathfrak{D}$ is a monad then it of course follows that $\mathbb{M}_{\mathfrak{D}}W$ itself has the natural structure of a modular $\mathfrak{D}$--operad, the free modular $\mathfrak{D}$--operad generated by $W$.

\begin{convention}
Unless otherwise stated, it will normally be assumed that a hyperoperad is in fact a cocycle.
\end{convention}

Let $\mathfrak{l}$ be a stable $S$--module with each $\mathfrak{l}((g,n))$ invertible (that is, one dimensional). Define an associated cocycle by
\[
\mathfrak{D}_{\mathfrak{l}}(G) = \mathfrak{l}((g,n))\otimes\bigotimes_{v\in\vertices(G)} \mathfrak{l}((g,n))^{-1}
\]
and call this cocycle the \emph{coboundary of $\mathfrak{l}$}. Denote the functor of tensoring $S$--modules with $\mathfrak{l}$ by $W\mapsto \mathfrak{l}W$. This induces an equivalence between the categories of modular $\mathfrak{D}$--operads and modular $\mathfrak{D}\otimes\mathfrak{D}_{\mathfrak{l}}$--operads via the natural isomorphism of monads $\mathbb{M}_{\mathfrak{D}\otimes\mathfrak{D}_\mathfrak{l}} \cong \mathfrak{l}\circ\mathbb{M}_{\mathfrak{D}}\circ\mathfrak{l}^{-1}$.

In practice it will only be necessary to consider a few cocycles. The trivial cocycle is simply $\mathbf{1}(G)=k$ and modular $\mathbf{1}$--operads are just modular operads. Of particular importance is the \emph{dualising cocycle} $\mathfrak{K}(G)=\Sigma^{-n}\bigwedge^n (k^{\edges(G)})$. This is not the coboundary of any stable $S$--module. For any hyperoperad $\mathfrak{D}$ define the dual hyperoperad by $\mathfrak{D}^{\vee}=\mathfrak{K}\otimes\mathfrak{D}^{-1}$. Note that $\mathbf{1}^{\vee}\cong\mathfrak{K}$.

Denote by $\mathfrak{s}$ and $\Sigma$ the $S$--modules
\begin{gather*}
\mathfrak{s}((g,n)) = \Sigma^{2-2g-n}\mathrm{sgn}_n
\\
\Sigma((g,n)) = \Sigma k
\end{gather*}
where $\mathrm{sgn}_n$ is the sign representation of $S_n$. Observe that $\mathfrak{D}_{\mathfrak{s}}^{\otimes 2}\cong \mathbf{1}$. The cocycle $\mathfrak{D}_{\Sigma}\otimes\mathfrak{D}_{\mathfrak{s}}$ coincides with $\mathfrak{K}$ on the category of trees.

\subsection{Cyclic $\mathfrak{D}$--operads}
It will also be convenient to consider cyclic operads over a cocycle. This is not entirely necessary since the cocycle $\mathfrak{K}$ is a coboundary when restricted to trees so cyclic $\mathfrak{K}$--operads will be equivalent to normal cyclic operads. However by considering cyclic $\mathfrak{K}$--operads separately the statement and proof of certain results will be more succinct and avoid excessively dense and complex notation.

Since the categories $T((n))$ and  $\Gamma((0,n))$ are isomorphic, it is clear what the corresponding monad will be:
\begin{definition}
Let $\mathfrak{D}$ be a hyperoperad. Define the functor $\mathbb{T}_{\mathfrak{D}}$ by
\[
\mathbb{T}_{\mathfrak{D}}U((n)) = \operatorname*{colim}_{T\in\iso T((n))}\mathfrak{D}(T)\otimes U((T)).
\]
Then this is an endofunctor on the category of cyclic $S$--modules and has the structure of a monad. An algebra over this monad is called a \emph{cyclic $\mathfrak{D}$--operad}.
\end{definition}

Once again $\mathbb{T}_{\mathfrak{D}}U$ itself has the structure of a cyclic $\mathfrak{D}$--operad, the free cyclic $\mathfrak{D}$--operad generated by $U$.

\begin{definition}
Given a modular $\mathfrak{D}$--operad $\mathcal{O}$ then the genus zero part of $\mathcal{O}$ is a cyclic $\mathfrak{D}$--operad. Denote this functor by $\mathcal{O}\mapsto \cyc{\mathcal{O}}$. If $\mathcal{Q}$ is a cyclic $\mathfrak{D}$--operad then the modular $\mathfrak{D}$--closure is the left adjoint to this functor and will be denoted by $\mathcal{Q}\mapsto\modc{\mathcal{Q}}$. The na\"ive $\mathfrak{D}$--closure is the right adjoint and will be denoted by $\mathcal{Q}\mapsto\modn{\mathcal{Q}}$.
\end{definition}

Since $\mathfrak{D}_{\Sigma}\otimes\mathfrak{D}_{\mathfrak{s}}$ coincides with $\mathfrak{K}$ on trees then for $d\in\mathbb{Z}$ the functors $\mathcal{Q}\mapsto (\Sigma\mathfrak{s})^{d} \mathcal{Q}$ give equivalences of cyclic $\mathfrak{D}$--operads and cyclic $\mathfrak{D}\otimes\mathfrak{K}^{\otimes d}$--operads. Since $\mathfrak{D}_{\mathfrak{s}}^2\cong\mathbf{1}$ this is also true for the functors $\mathcal{Q}\mapsto (\Sigma\mathfrak{s}^{-1})^d \mathcal{Q}$. Denote this latter functor by $\mathcal{Q}^d=(\Sigma\mathfrak{s}^{-1})^d\mathcal{Q}$. Note that if $d$ is even then $\mathfrak{D}_{\Sigma^d}\otimes\mathfrak{D}_{\mathfrak{s}^{-d}}\cong\mathfrak{K}^{\otimes d}$ on all graphs so then the functor $\mathcal{O}\mapsto\mathcal{O}^d$ takes modular $\mathfrak{D}$--operads to modular $\mathfrak{D}\otimes\mathfrak{K}^{d}$--operads. It now follows that if $\mathcal{Q}$ is a cyclic operad and $d$ is even then $\modc{Q^d} \cong \modc{Q}^d$.

\subsection{Algebras over modular and cyclic operads}
Let $V\in\dgvect$ have a \emph{symmetric} non-degenerate bilinear form of degree $d$, where by degree $d$ it is meant a map $\langle -, - \rangle\co V\otimes V \rightarrow \Sigma^d k$. Then $\mathcal{E}[V]((g,n))=V^{\otimes n}$ has the natural structure of a modular $\mathfrak{K}^{\otimes d}$--operad by the contraction of tensors with the bilinear form. The genus zero part is a cyclic $\mathfrak{K}^{\otimes d}$--operad with $\mathcal{E}[V]((n))=V^{\otimes n}$.

Let $W\in\dgvect$ have an \emph{anti-symmetric} non-degenerate bilinear form of degree $d$. Then $\mathcal{E}[W]((g,n)) = V^{\otimes n}$ has the natural structure of a modular $\mathfrak{K}^{\otimes d}\otimes\mathfrak{D}_{\mathfrak{s}}^{-1}$--operad by the contraction of tensors with the bilinear form. The genus zero part is a cyclic $\mathfrak{K}^{\otimes d}\otimes\mathfrak{D}_{\mathfrak{s}}^{-1}$--operad with $\mathcal{E}[W]((n))=W^{\otimes n}$.

The non-degeneracy of the bilinear form on $V$ means that there is an isomorphism $V\rightarrow \Sigma^dV^*$ given by $v\mapsto \langle v, - \rangle$. This yields a degree $-d$ non-degenerate bilinear form $\langle -, - \rangle^{-1} \co V^* \otimes V^* \rightarrow \Sigma^{-d}k$ which is symmetric if $d$ is even and anti-symmetric if $d$ is odd. Similarly $W^*$ has a non-degenerate bilinear form which is symmetric if $d$ is odd and anti-symmetric if $d$ is even.

Furthermore $\Sigma^{-1}V$ is equipped with an \emph{anti-symmetric} non-degenerate bilinear form of degree $d-2$ and $\Sigma^{-1}W$ is equipped with a symmetric non-degenerate bilinear form of degree $d-2$.

Consider the coboundary $\mathfrak{D}_{\tilde{\mathfrak{s}}}$ associated to the $S$--module $\tilde{\mathfrak{s}}((g,n))=\Sigma^{-n}\mathrm{sgn}_n$ and observe that $\mathfrak{D}_{\mathfrak{s}}\cong \mathfrak{D}_{\tilde{\mathfrak{s}}}\otimes\mathfrak{K}^{\otimes 2}$. Then the following identities hold: $\tilde{\mathfrak{s}}\mathcal{E}[V]\cong\mathcal{E}[\Sigma^{-1}V]$, $\tilde{\mathfrak{s}}\mathcal{E}[W]\cong\mathcal{E}[\Sigma^{-1}W]$, $\tilde{\mathfrak{s}}^d\mathcal{E}[V]\cong\mathcal{E}[V^*]$ and $\tilde{\mathfrak{s}}^d\mathcal{E}[W]\cong\mathcal{E}[W^*]$. Also $\mathcal{E}[V]\cong\mathcal{E}[\Sigma^{d}V^*]$ and $\mathcal{E}[W]\cong\mathcal{E}[\Sigma^dW^*]$.

Furthermore $\mathcal{E}[V]^{-d}$ and $\tilde{\mathfrak{s}}\mathcal{E}[W]^{-d}$ are cyclic operads and the bilinear form induces isomorphisms of (non-cyclic) operads $\End[V]\cong\mathcal{E}[V]^{-d}$, and $\End[\Sigma^{-1}W]\cong\mathfrak{s}\mathcal{E}[W]^{-d}$ where $\End[V]$ is the usual endomorphism operad given by $\End[V](n) = \IHom(V^{\otimes n}, V)$.

\begin{definition}\label{def:modalg}
Let $\mathcal{O}$ be a cyclic/modular $\mathfrak{K}^{\otimes d}$--operad. An algebra over $\mathcal{O}$ is $V\in\dgvect$ equipped with a symmetric non-degenerate bilinear form of degree $d$ and a map $\mathcal{O}\rightarrow\mathcal{E}[V]$ of cyclic/modular $\mathfrak{K}^{\otimes d}$--operads.
Let $\mathcal{O}'$ be a cyclic/modular $\mathfrak{K}^{\otimes d}\otimes\mathfrak{D}_{\mathfrak{s}}^{-1}$--operad. An algebra over $\mathcal{O}'$ is $W\in\dgvect$ equipped with an anti-symmetric non-degenerate bilinear form of degree $d$ and a map $\mathcal{O}'\rightarrow\mathcal{E}[W]$ of cyclic/modular $\mathfrak{K}^{\otimes d}\otimes\mathfrak{D}_{\mathfrak{s}}^{-1}$--operads.
\end{definition}

\begin{remark}
Since $\tilde{\mathfrak{s}}\mathcal{E}[W]\cong\mathcal{E}[\Sigma^{-1}W]$ then by simply replacing $W$ with $\Sigma^{-1} W$ it will be sufficient to consider only symmetric non-degenerate bilinear forms and hence only algebras over cyclic/modular $\mathfrak{K}^{\otimes d}$--operads.
\end{remark}

It can be shown (see \cite{getzlerkapranov}) that for any $d$, by tensoring with coboundaries, $\mathfrak{K}^{\otimes d}$ can be reduced to either the trivial cocycle if $d$ is even or $\mathfrak{K}$ if $d$ is odd. Therefore for most purposes it will be sufficient to consider just the two categories of modular operads and modular $\mathfrak{K}$--operads. Obviously in the supergraded setting this is true for more straightforward reasons.

Let $\mathcal{Q}$ be a cyclic operad and let $V\in\dgvect$ have a degree $d$ symmetric non-degenerate bilinear form. Given a map of (non-cyclic) operads $\mathcal{Q}\rightarrow \End[V]$ such that the composition $Q((n))\rightarrow\End[V](n-1)\rightarrow \Sigma^d\mathcal{E}[V^*]((n))$, given by applying the map $V\rightarrow \Sigma^d V^*$, yields a map of \emph{cyclic} $S$--modules then this extends via the isomorphism of operads $\End[V]\cong\mathcal{E}[V]^{-d}$ to a map of \emph{cyclic} operads $\mathcal{Q}\rightarrow\mathcal{E}[V]^{-d}$. In particular this is then equivalent to a map of cyclic $\mathfrak{K}^{\otimes d}$--operads $\mathcal{Q}^d\rightarrow\mathcal{E}[V]$, in other words $\mathcal{Q}^d$ controls $\mathcal{Q}$--algebras with a cyclically invariant degree $d$ symmetric non-degenerate bilinear form.

Note that the condition that the composition $Q((n))\rightarrow\End[V](n-1)\rightarrow \Sigma^d \mathcal{E}[V^*]((n))$ gives a map of cyclic $S$--modules does not require the bilinear form to be non-degenerate (although note that the cyclic $S$--module $\mathcal{E}[V^*]$ does not then have the structure of a cyclic operad) and so it makes sense for infinite dimensional $V\in\dgvect$, therefore it is possible to extend the notion of an algebra over a cyclic $\mathfrak{K}^{\otimes d}$--operad to $V\in\dgvect$ with a not necessarily non-degenerate bilinear form.

\begin{definition}
Let $\mathcal{Q}$ be a cyclic $\mathfrak{K}^{\otimes d}$--operad. A cyclic $\mathcal{Q}$--algebra is $V\in\dgvect$ equipped with a symmetric bilinear form of degree $d$ and a map of (non-cyclic) operads $\mathcal{Q}^{-d}\rightarrow\End[V]$ such that the composition $Q^{-d}((n))\rightarrow\End[V](n-1)\rightarrow \Sigma^d \mathcal{E}[V^*]((n))$ defines a map of cyclic $S$--modules. If the bilinear form is non-degenerate this is equivalent to \autoref{def:modalg}.
\end{definition}

Obviously it is a requirement that the bilinear form be non-degenerate when considering algebras over a \emph{modular} $\mathfrak{K}^{\otimes d}$--operad.

\begin{example}\label{ex:derhamalgebra}
Let $M$ be a closed oriented manifold of dimension $d$. Then the de Rham algebra $\Omega^{\bullet}(M)$ is a non-negatively graded commutative algebra with a \emph{degenerate} cyclically invariant bilinear form $\langle \omega, \eta \rangle = \int_M \omega\wedge \eta$ of degree $-d$. Therefore $\Omega^\bullet (M)$ is a cyclic $\com^{-d}$--algebra and Poincar\'e duality implies that $H^\bullet(M)$ is a $\modc{\com^{-d}}$--algebra with a \emph{non-degenerate} bilinear form.
\end{example}

\begin{convention}
To avoid confusion when discussing algebras over cyclic $\mathfrak{K}^{\otimes d}$--operads it will always be assumed that the bilinear form is in fact non-degenerate unless explicitly stated otherwise.
\end{convention}

\subsection{The Feynman transform of a modular $\mathfrak{D}$--operad}
Given a cocycle $\mathfrak{D}$ and a modular $\mathfrak{D}$--operad $\mathcal{O}$ denote by $\gamma_2$ the structures map of $\mathcal{O}$ restricted to the subspace
\[
\colim_{\substack{g\in\iso\Gamma((g,n))\\|\edges(G)|=1}} \mathfrak{D}(G)\otimes\mathcal{O}((G))
\]
of $\mathbb{M}_{\mathfrak{D}}\mathcal{O}$ corresponding to graphs with one internal edge. If $G$ is a graph with precisely one internal edge $e$ then $\mathfrak{K}^{-1}(G)\cong\Sigma k$ and so $\Sigma\gamma_2$ is a map of $S$--modules from the corresponding subspace of $\mathbb{M}_{\mathfrak{K}^{-1}\otimes\mathfrak{D}}\mathcal{O}$ of graphs with one internal edge to the $S$--module $\Sigma\mathcal{O}$. Denote the dual to this map by $\delta_2\co\Sigma^{-1}\mathcal{O}^*\rightarrow \mathbb{M}_{\mathfrak{D}^{\vee}}\mathcal{O}^*$. Then this extends by the modular operad version of the Leibniz rule to a degree zero map $\delta\co\Sigma^{-1}\mathbb{M}_{\mathfrak{D}^{\vee}}\mathcal{O}^*\rightarrow\mathbb{M}_{\mathfrak{D}^{\vee}}\mathcal{O}^*$, in other words $\delta$ is a differential on the free modular $\mathfrak{D}^{\vee}$--operad generated by $\mathcal{O}^*$. Taking the total differential $d_{\feyn} = d_{\mathcal{O}^*} + \delta$, where $d_{\mathcal{O}^*}$ is the internal differential on the $S$--module $\mathcal{O}^*$, we obtain a differential graded modular $\mathfrak{D}^{\vee}$--operad that is denoted by $\feyn_{\mathfrak{D}}\mathcal{O}$ and called the \emph{Feynman transform of $\mathcal{O}$}.

\begin{convention}
To avoid notation excessively decorated by symbols $\feyn\mathcal{O}$ will from now on be used instead of $\feyn_\mathfrak{D}\mathcal{O}$ when it is clear from context that $\mathcal{O}$ is a modular $\mathfrak{D}$--operad.
\end{convention}

\subsection{The cobar construction of a cyclic $\mathfrak{D}$--operad}
Let $\mathcal{Q}$ be a cyclic $\mathfrak{D}$--operad. Denote by $\cobar\mathcal{Q}$ the cyclic $\mathfrak{D}^{\vee}$--operad $\cyc{\feyn\modn{\mathcal{Q}}}$. In the case that $\mathcal{Q}$ is a cyclic $\mathfrak{K}^{\otimes d}$--operad then with this notation $\cobar\mathcal{Q} = (\Sigma\mathfrak{s})^{1-d}\dual(\Sigma\mathfrak{s})^{-d}\mathcal{Q}$ where $\dual$ is the usual cyclic cobar operad functor, see \cite{ginzburgkapranov,getzlerkapranov2}. From now on, the term cobar construction will normally refer to $\cobar$ and not $\dual$. Now observe that if $\mathcal{Q}$ is a cyclic $\mathfrak{D}$--operad then $\feyn\modn{\mathcal{Q}} \cong \modc{\cobar\mathcal{Q}}$.

One further generalisation of the usual notation will be necessary. If $\mathcal{Q}$ is a cyclic $\mathfrak{K}^{\otimes d}$--operad then it will be called a \emph{cyclic quadratic $\mathfrak{K}^{\otimes d}$--operad} if ${Q}^{-d}$ is a cyclic quadratic operad. In this case denote by $\mathcal{Q}^\twistshriek$ the \emph{dual cyclic quadratic $\mathcal{K}^{\otimes 1-d}$--operad} $\mathcal{Q}^\twistshriek = ((\mathfrak{s}^{-2}\mathcal{Q}^{-d})^{!})^{1-d}=(\mathfrak{s}^{2}(\mathcal{Q}^{-d})^{!})^{1-d}$. With this notation there is a canonical map $\cobar\mathcal{Q}^\twistshriek\rightarrow \mathcal{Q}$ and $\mathcal{Q}$ will be called Koszul with Koszul dual $\mathcal{Q}^\twistshriek$ if this is a quasi-isomorphism of cyclic $\mathfrak{K}^{\otimes 1-d}$--operads. Then $\mathcal{Q}$ is Koszul if and only if $\mathcal{Q}^{-d}$ is Koszul.

Let $\mathcal{Q}$ be a cyclic Koszul operad, then there is an isomorphism $\cobar(\mathcal{Q}^d)^\twistshriek\cong (\dual\mathcal{Q}^!)^{d}$. Therefore this means that $\cobar(\mathcal{Q}^d)^\twistshriek$ governs $\dual\mathcal{Q}^!$--algebras with a degree $d$ cyclically invariant bilinear form, in other words homotopy $\mathcal{Q}$--algebras with a degree $d$ cyclically invariant bilinear form. Also note that $(\mathcal{Q}^d)^\twistshriek \cong \mathfrak{s}^2(\mathcal{Q}^!)^{1-d}$ which means $\mathfrak{s}^{-2}(\mathcal{Q}^d)^\twistshriek$ governs $\mathcal{Q}^!$--algebras with a degree $1-d$ cyclically invariant bilinear form. Of particular importance is the supergraded setting, where $\mathfrak{s}^{-2}=\id$. Then $(\mathcal{Q}^d)^\twistshriek \cong (\mathcal{Q}^{!})^{1-d}$ so in particular it now holds that $\com^\twistshriek\cong\lie^1$, $\lie^\twistshriek\cong\com^1$ and $\ass^\twistshriek\cong \ass^1$.

\section{Maps from the Feynman transform of a modular operad}
In \autoref{chap:mc} we saw that maps from differential graded commutative algebras of the form $\CE^{\bullet}(\mathfrak{g})$ for some differential graded Lie algebra $\mathfrak{g}$ can be represented as Maurer--Cartan elements in some Lie algebra. In this section we will recall the corresponding facts about maps from certain operads being equivalent to Maurer--Cartan elements in a operad. However we will avoid talking about Maurer--Cartan elements in an operad directly by instead considering Maurer--Cartan elements in an associated Lie algebra. We will also develop some of the general results that we will use.

We will work with modular operads since they are the most complex objects we will deal with, but all of the results in this section will hold for other types of operads, in particular operads and cyclic operads, by replacing the Feynman transform $\feyn$ with the appropriate version of the cobar construction.

\subsection{Maurer--Cartan elements in an operad}
Recall the following standard construction associating a differential graded Lie algebra to a (modular) operad.

\begin{proposition}
Let $\mathcal{O}$ be a modular $\mathfrak{K}$--operad. Let
\[
\mathbf{L}(\mathcal{O}) = \Sigma^{-1} \prod_{g,n} \mathcal{O}((g,n))_{S_n}.
\]
Then $\mathbf{L}(\mathcal{O})$ has the structure of a differential graded Lie algebra. Furthermore this defines a functor from modular $\mathfrak{K}$--operads to formal differential graded Lie algebras.
\end{proposition}

\begin{proof}
The differential graded Lie algebra structure arises as follows: The Lie bracket is the sum over all possible ways of composing two elements and the differential is given by the sum of the internal differential on $\mathcal{O}$ and the differential obtained by summing over all possible ways of contracting a given element.

More precisely let $T$ be the stable graph with $1$ edge, $n+m-2$ legs and $2$ vertices $v$ and $v'$ of genus $g$ and $g'$ respectively with $n-1$ legs attached to $v$ and $m-1$ legs attached to $v'$. Then $T$ induces the composition map
\[
\circ \co \Sigma^{-1} [\mathcal{O}((g,n)) \otimes \mathcal{O}((g',m))] \longrightarrow \mathcal{O}((g+g', n+m-2))
\]
noting that $\mathfrak{K}(T) \cong \Sigma^{-1} k$.

Then the Lie bracket on $\mathbf{L}(\mathcal{O})$ is given by the composition:
\[
\xymatrix{
\Sigma^{-1} \mathcal{O}((g,n))_{S_n} \otimes \Sigma^{-1} \mathcal{O}((g',m))_{S_m} \ar[d]^{r \otimes r}
\\
\Sigma^{-1} \mathcal{O}((g,n)) \otimes \Sigma^{-1} \mathcal{O}((g',m)) \ar[d]^\cong
\\
\Sigma^{-2} [\mathcal{O}((g,n)) \otimes \mathcal{O}((g',m))] \ar[r]^{\Sigma^{-1}\circ} & \Sigma^{-1} \mathcal{O}((g+g', n+m-2))\ar[d]^q
\\
& \Sigma^{-1} \mathcal{O}((g+g', n+m-2))_{S_{n+m-2}}
}
\]
Note that it does not matter how we labelled the tree $T$ or how we identify $\mathcal{O}((g,n))\otimes \mathcal{O}((g',m))\cong \mathcal{O}((T))$ in the definition of $\circ$. The maps $q$ are $r$ are defined as follows. For $V$ a vector space with an action by a finite group $G$ we denote by $q\co V^G\rightarrow V_G$ the projection of invariants onto coinvariants. We denote its inverse by $r\co V_G\rightarrow V^G$ with $r(v)=\frac{1}{|G|}\sum_{\sigma \in G}\sigma(v)$. For convenience we also denote by $r$ the composition $V_G\rightarrow V^G \hookrightarrow V$.

Similarly let $G$ be the stable graph of genus $g+1$ with $1$ vertex of genus $g$, $1$ edge and $n$ legs. Then $G$ induces the contraction map
\[
\xi \co \Sigma^{-1} \mathcal{O}((g,n)) \longrightarrow \mathcal{O}((g+1,n-2))
\]
noting that $\mathfrak{K}(G)\cong \Sigma^{-1}k$.

The differential is then given by the composition:
\[
\xymatrix{
\Sigma^{-1} \mathcal{O}((g,n))_{S_n} \ar[d]^{r}
\\
\Sigma^{-1} \mathcal{O}((g,n))\ar[r]^{\xi} & \mathcal{O}((g+1, n-2))\ar[d]^q
\\
& \mathcal{O}((g+1, n-2))_{S_{n-2}}
}
\]
If $\mathcal{O}$ has an internal differential we of course take the sum of this and the differential just defined.

We leave it to the reader to check that this does indeed give a differential graded Lie algebra.

For $N\geq 1$ Set
\[
\mathbf{L}_{\geq N}(\mathcal{O}) = \Sigma^{-1}\prod_{2g + n-2\geq N} \mathcal{O}((g,n))_{S_n}.
\]
This gives a complete descending filtration of $\mathbf{L}(\mathcal{O})$. Then $\mathbf{L}(\mathcal{O})/\mathbf{L}_{\geq N}(\mathcal{O})$ is a finite dimensional nilpotent Lie algebra so $\mathbf{L}(\mathcal{O})$ is formal. The functoriality is easy to see.
\end{proof}

\begin{definition}
Let $\mathcal{O}$ be a modular $\mathfrak{D}$--operad and let $\mathcal{P}$ be a modular $\mathfrak{D}^\vee$--operad. Noting that $\mathcal{O}\otimes\mathcal{P}$ is a modular $\mathfrak{K}$--operad, we write $\MC(\mathcal{O},\mathcal{P})$ for the functor $(\mathcal{O},\mathcal{P})\mapsto \MC(\mathbf{L}(\mathcal{O}\otimes\mathcal{P}))$.
\end{definition}

We recall the following general theorem of Barannikov \cite{barannikov}.

\begin{theorem}\label{thm:barannikov}
Let $\mathcal{O}$ be a modular $\mathfrak{D}$--operad and let $\mathcal{P}$ be a modular $\mathfrak{D}^{\vee}$--operad. The maps $\feyn \mathcal{O} \rightarrow \mathcal{P}$ are in one-to-one correspondence with Maurer--Cartan elements of $\mathbf{L}(\mathcal{O} \otimes \mathcal{P})$.
\end{theorem}

We will now reword this theorem to take into account functoriality.

\begin{theorem}\label{thm:barannikovfunctor}
Let $\mathcal{O}$ be a modular $\mathfrak{D}$--operad and let $\mathcal{P}$ be a modular $\mathfrak{D}^{\vee}$--operad. The functors $(\mathcal{O},\mathcal{P}) \mapsto \Hom(\feyn \mathcal{O}, \mathcal{P})$ and $(\mathcal{O},\mathcal{P}) \mapsto \MC(\mathcal{O}, \mathcal{P})$ are naturally isomorphic.
\end{theorem}

\begin{proof}
The only addition to \autoref{thm:barannikov} is the naturality. The underlying vector space $\mathbf{L}(\mathcal{O}\otimes\mathcal{P})$ is naturally isomorphic to the vector space of maps of stable $S$--modules $\mathcal{O}^*\rightarrow\mathcal{P}$. This is sufficient since in this way the Maurer--Cartan condition is just equivalent to a map of $S$--modules $\mathcal{O}^*\rightarrow\mathcal{P}$ extending to a map $\feyn\mathcal{O}\rightarrow \mathcal{P}$ that is compatible with the differential.
\end{proof}

\begin{remark}
The differential graded Lie algebra $\mathbf{L}(\mathcal{O}\otimes\mathcal{P})$ is formal so the notions of gauge equivalence and homotopy of Maurer--Cartan elements coincide. As one might expect, a homotopy between Maurer--Cartan elements corresponds to a Sullivan homotopy between maps of operads, parallel to \autoref{rem:homotopyissullivanhomotopy}. In particular we have a good notion of a homotopy between two $\feyn\mathcal{O}$--structures on a vector space $V$.
\end{remark}

\subsection{Maps from $\feyn\modc{\com}$}
\Needspace*{4\baselineskip}

\begin{theorem}\label{thm:natiso}
Let $\mathcal{O}$ be a modular $\mathfrak{D}$--operad, $\mathcal{Q}$ be a modular $\mathfrak{E}$--operad and $\mathcal{P}$ be a modular $\mathfrak{D}^{\vee}\otimes\mathfrak{E}^{\vee}$--operad. Then there is a natural isomorphism
\[
\Hom(\feyn(\mathcal{O}\otimes \mathcal{P}), \mathcal{Q}) \cong \Hom(\feyn\mathcal{O}, \mathcal{P}\otimes \mathcal{Q}).
\]
\end{theorem}

\begin{proof}
Note that both sides are simply the set of Maurer--Cartan elements in the same differential graded Lie algebra $\mathbf{L}(\mathcal{O}\otimes\mathcal{P}\otimes\mathcal{Q})$. The result now follows by applying the Maurer--Cartan functor to the following diagram
\[
\xymatrix{
\mathbf{L}(\mathcal{O}\otimes\mathcal{P}\otimes\mathcal{Q})\ar[r]^{\id}\ar[d]^{\mathbf{L}f} & \mathbf{L}(\mathcal{O}\otimes\mathcal{P}\otimes\mathcal{Q})\ar[d]^{\mathbf{L}f}
\\
\mathbf{L}(\mathcal{O'}\otimes\mathcal{P'}\otimes\mathcal{Q'})\ar[r]^{\id} & \mathbf{L}(\mathcal{O'}\otimes\mathcal{P'}\otimes\mathcal{Q'})
}
\]
and using \autoref{thm:barannikovfunctor}
\end{proof}

We have the following result which is a modular operad version of \cite[Proposition 3.2.18]{ginzburgkapranov}.

\begin{theorem}\label{thm:feyncommap}
Let $\mathcal{O}$ be a modular $\mathfrak{D}$--operad. There is a natural map of modular $\mathfrak{K}$--operads $\feyn\modc{\com}\rightarrow \feyn\mathcal{O}\otimes \mathcal{O}$.
\end{theorem}

\begin{proof}
Since $\modc\com\otimes\mathcal{O}\cong\mathcal{O}$ the map is given by the image of the identity map $\id\in\Hom(\feyn \mathcal{O}, \feyn \mathcal{O})$ under the natural isomorphism 
\[\Hom(\feyn \mathcal{O}, \feyn \mathcal{O})\longrightarrow \Hom(\feyn\modc{\com}, \mathcal{O}\otimes\feyn\mathcal{O})\cong\Hom(\feyn\modc{\com}, \feyn\mathcal{O}\otimes\mathcal{O})\]
of \autoref{thm:natiso}.
\end{proof}

\begin{corollary}\label{cor:feynscommap}
Let $\mathcal{O}$ be a modular $\mathfrak{D}$--operad. There is a natural map of modular $\mathfrak{K}$--operads $\feyn\mathfrak{s}^{2}\modc{\com}\rightarrow \feyn\mathcal{O}\otimes \mathfrak{s}^{-2}\mathcal{O}$.
\end{corollary}

\begin{proof}
Observe that $\feyn\mathfrak{s}^{2}\modc{\com}\cong\mathfrak{s}^{-2}\feyn\modc{\com}$ and apply \autoref{thm:feyncommap}.
\end{proof}

\begin{proposition}\label{prop:feynassmap}
There is a map $\feyn\modc{\ass}\rightarrow\feyn\modc{\ass}\otimes\modc{\ass}$.
\end{proposition}

\begin{proof}
The modular operad $\modc{\ass}$ is the linearisation of the modular operad in the category of sets consisting of equivalence classes of ribbon graphs. Therefore there is a map $\modc{\ass}\rightarrow\modc{\ass}\otimes\modc{\ass}$ given by the linearisation of the diagonal map. This gives a map $\feyn(\modc{\ass}\otimes\modc{\ass})\rightarrow\feyn\modc{\ass}$ and we now use that $\Hom(\feyn(\modc{\ass}\otimes\modc{\ass}),\feyn\modc{\ass})\cong\Hom(\feyn\modc{\ass},\feyn\modc{\ass}\otimes\modc{\ass})$.
\end{proof}

\begin{remark}
It is worth noting that \autoref{prop:feynassmap} could also be approached by considering a suitable theory of \emph{ribbon modular operads}, modelled not on stable graphs but on stable ribbon graphs. We could then prove a corresponding version of \autoref{thm:natiso} and would then obtain a version of \autoref{thm:feyncommap} replacing $\modc{\com}$ with $\modc{\ass}$ since this would now be the identity object for ribbon modular operads.
\end{remark}

\begin{remark}
It is perhaps useful to describe the map of \autoref{thm:feyncommap} in more detail. The underlying stable $S$--module of $\feyn\modc{\com}$ is given by
\[\feyn\modc{\com}((g,n)) = \colim_{G\in\iso((g,n))}\modc{\com}^*((G))\otimes\mathfrak{K}(G)\cong\colim_{G\in\iso((g,n))}\mathfrak{K}(G)\]
and the underlying $S$--module of $\feyn\mathcal{O}\otimes \mathcal{O}$ is given by
\[[\feyn\mathcal{O}\otimes \mathcal{O}]((g,n))= \colim_{G\in\iso((g,n))} \mathfrak{K}(G)\otimes\mathfrak{D}^{-1}(G)\otimes\mathcal{O}^*((G))\otimes\mathcal{O}((g,n)).\]

For $G$ a stable graph of genus $g$ with $n$ legs denote by $\gamma_G:\mathfrak{D}(G)\otimes\mathcal{O}((G))\rightarrow \mathcal{O}((g,n))$ the structure maps of $\mathcal{O}$. Denote by $\hat{\gamma}_G$ the corresponding elements in
\[\mathfrak{D}^{-1}(G)\otimes\mathcal{O}^*((G))\otimes\mathcal{O}((g,n))\cong\Hom(\mathfrak{D}(G)\otimes\mathcal{O}((G)),\mathcal{O}((g,n))).\]
We define maps on the summands
\[\mathfrak{K}(G)\rightarrow \mathfrak{K}(G)\otimes\mathfrak{D}^{-1}(G)\otimes\mathcal{O}^*((G))\otimes\mathcal{O}((g,n))\]
by $\id_{\mathfrak{K}(G)}\otimes \hat{\gamma}_G$. These maps assemble to give a map of stable $S$--modules which is in fact the map of operads in \autoref{thm:feyncommap}.
\end{remark}

Let $\feyn\mathcal{O}\rightarrow\mathcal{P}$ be a map of operads. Given a map $f\co\mathfrak{s}^{-2}\mathcal{O}\rightarrow \mathcal{Q}$ we obtain a map $\feyn\mathcal{O}\otimes\mathfrak{s}^{-2}\mathcal{O}\rightarrow\mathcal{P}\otimes\mathcal{Q}$. Applying \autoref{cor:feynscommap} we can then obtain a map $\feyn\mathfrak{s}^2\modc{\com}\rightarrow\mathcal{P}\otimes\mathcal{Q}$. This procedure gives a map $\theta_f\co\MC(\mathcal{O},\mathcal{P})\rightarrow \MC(\mathfrak{s}^2\modc{\com},\mathcal{P}\otimes\mathcal{Q})$. It is natural to ask whether this map of Maurer--Cartan elements is induced by a map of differential graded Lie algebras. There is an obvious candidate: the map $f$ induces a map $\mathbf{L}(\mathfrak{s}^2f\otimes\id)\co\mathbf{L}(\mathcal{O}\otimes\mathcal{P})\rightarrow\mathbf{L}(\mathfrak{s}^2\mathcal{Q}\otimes\mathcal{P})\cong\mathbf{L}(\mathfrak{s}^2\modc{\com}\otimes\mathcal{P}\otimes\mathcal{Q})$. This does indeed induce the same map of Maurer--Cartan elements:

\begin{proposition}
The map on Maurer--Cartan elements induced by $\mathbf{L}(\mathfrak{s}^2f\otimes\id)$ is $\theta_f$.
\end{proposition}

\begin{proof}
 Let $\xi\in\MC(\mathcal{O},\mathcal{P})$ be a map $\xi\co\feyn\mathcal{O}\rightarrow\mathcal{P}$. By \autoref{thm:natiso} we have the following commutative diagram
\[
\xymatrix{
\Hom(\feyn\mathcal{O},\feyn\mathcal{O})\ar[r]\ar[d]^{\xi\co\feyn\mathcal{O}\rightarrow\mathcal{P}} & \Hom(\feyn\modc{\com},\feyn\mathcal{O}\otimes\mathcal{O})\ar[dd]^{\xi\otimes \mathfrak{s}^2f \co \feyn\mathcal{O}\otimes\mathcal{O}\rightarrow\mathcal{P}\otimes\mathfrak{s}^2\mathcal{Q}}
\\
\Hom(\feyn\mathcal{O},\mathcal{P})\ar[d]^{\mathfrak{s}^2f\co\mathcal{O}\rightarrow\mathfrak{s}^2\mathcal{Q}}
\\
\Hom(\feyn\mathfrak{s}^2\mathcal{Q},\mathcal{P})\ar[r] & \Hom(\feyn\modc{\com},\mathcal{P}\otimes\mathfrak{s}^2\mathcal{Q})\cong\MC(\mathfrak{s}^2\modc{\com},\mathcal{P}\otimes\mathcal{Q})
}
\]
where the horizontal maps are isomorphisms and the vertical maps are induced by the maps of operads shown. Then the image of $\id\in\Hom(\feyn\mathcal{O},\feyn\mathcal{O})$ via the clockwise route is the image of $\xi$ under $\theta_f$. The image of $\id$ via the anticlockwise route is the image of $\xi$ under the map of Maurer--Cartan elements induced by $\mathbf{L}(\mathfrak{s}^2f\otimes\id)$.
\end{proof}

We will later be particularly interested in the case that $\mathcal{P}$ and $\mathcal{Q}$ are endomorphism operads.

\section{Quantum homotopy algebras and the lifting problem}
In this section we will recall the definitions of cyclic homotopy algebras (homotopy algebras with an inner product) and define quantum and semi-quantum homotopy algebras.

\subsection{Unimodular $\mathcal{Q}$--algebras}
Let $\mathcal{Q}$ be an operad and $A$ a $\mathcal{Q}$--algebra. Recall that a module $V$ over $A$ is given by maps
\[\mu_{n,i}\co \mathcal{Q}(n)\otimes A^{\otimes i-1}\otimes V \otimes A^{\otimes n-i}\rightarrow V\]
for $1\leq i\leq n$ satisfying the natural unit, associativity and equivariance conditions (see Ginzburg and Kapranov \cite{ginzburgkapranov}).

\begin{remark}\label{rem:moduledefinition2}
Note that the equivariance condition in fact means that a module $V$ over $A$ is completely determined by the maps $\mu_{n,n}$ and so we could also view an $A$--module as a set of maps 
\[\mu_{n}\co \mathcal{Q}(n)\otimes A^{\otimes n-1}\otimes V \rightarrow V\]
satisfying the unit and associativity conditions and equivariance with respect to the $S_{n-1}$ action permuting the factors of $A$.
\end{remark}

\begin{definition}
Let $A$ be a $\mathcal{Q}$--algebra. Then $A$ is an $A$--module in the natural way by 
\[\mu_{n,i}(\phi,a_1,\dots,a_i,\dots, a_n) = \phi(a_1,\dots,a_i,\dots, a_n).\]
This is the \emph{adjoint representation} and we define $\ad^{\phi}_{a_1,\dots,a_{n-1}}\co A \rightarrow A$ by
\[\ad^{\phi}_{a_1,\dots,a_{n-1}}(a) = \mu_{n,n}(\phi,a_1,\dots,a_{n-1},a).\]
\end{definition}

The notion of a unimodular algebra is standard and the definition can be extended to any algebra over a quadratic operad.

\begin{definition}
Let $\mathcal{Q}$ be a quadratic operad and $V$ a finite dimensional $\mathcal{Q}$--algebra. If for all $v\in V$ and $\phi\in\mathcal{Q}(2)$ we have $\tr(\ad^\phi_v)=0$ then we call $V$ a \emph{unimodular $\mathcal{Q}$--algebra}.
\end{definition}

\begin{example}\label{ex:2dassunimod}
Let $A$ be the two dimensional unital associative superalgebra generated by an element $a$ of odd degree with $a^2=1$ and with an odd scalar product given by $\langle a, 1 \rangle = 1$. Direct calculation shows the (super)traces of the left and right multiplications by $1$ and $a$ are zero and therefore $A$ is unimodular.
\end{example}

\begin{proposition}\label{prop:unimodtensor}
Let $\mathcal{Q},\mathcal{Q}'$ be quadratic operads, $V$ a $\mathcal{Q}$--algebra and $W$ a $\mathcal{Q}'$--algebra. Then the $\mathcal{Q}\otimes\mathcal{Q}'$--algebra $V\otimes W$ is unimodular if and only if at least one of $V$ or $W$ is.
\end{proposition}

\begin{proof}
This follows from the straightforward calculation that for all $\phi\in\mathcal{Q}(2)$, $\psi\in\mathcal{Q}'(2)$ and $v\in V$, $w\in W$ then $\tr (\ad^{\phi\otimes\psi}_{v\otimes w}) = (-1)^{\degree{\phi} \degree{\psi}}\tr (\ad^\phi_v)\tr(\ad^{\psi}_w)$.
\end{proof}

\begin{definition}
Let $\mathcal{P}$ be a cyclic quadratic $\mathfrak{K}^{\otimes d}$--operad and $V$ a cyclic $\mathcal{P}$--algebra. Then $V$ is a \emph{cyclic unimodular $\mathcal{P}$--algebra} if the underlying algebra over the (non-cyclic) operad $\mathcal{P}^{-d}$ is unimodular.
\end{definition}

It will be useful to have an operad governing cyclic unimodular $\mathcal{P}$--algebras.

\begin{definition}
Let $\mathcal{P}$ be a cyclic quadratic $\mathfrak{K}^{\otimes d}$--operad. Then we define $\unimod\mathcal{P} = \modc{\mathcal{P}}/\langle \modc{\mathcal{P}}((1,1)) \rangle$.
\end{definition}

\begin{proposition}\label{prop:unimodisunimod}
A cyclic algebra over $\mathcal{P}$ lifts to an algebra over $\unimod\mathcal{P}$ if and only if it is unimodular.
\end{proposition}

\begin{proof}
Given a map of modular $\mathfrak{K}^{\otimes d}$--operads $\Gamma\co\modc{\mathcal{P}}\rightarrow\mathcal{E}[V]$ then $\Gamma$ vanishes on $\modc{\mathcal{P}}((1,1))$ if and only if for all $\phi\in\mathcal{P}(2)$ we have $\Gamma(\xi_{23}(\phi))=\xi_{23}(\Gamma(\phi))=0$. If $\Gamma(\phi)=\sum a_i\otimes b_i\otimes c_i$ then $\Gamma(\xi_{23}(\phi)) = \sum (-1)^{\epsilon}\langle b_i, c_i \rangle  a_i$. For $v,w\in V$ then $\ad^\phi_v(w)=\sum (-1)^\epsilon \langle a_i, v\rangle \langle b_i, w \rangle c_i$ and so $\tr(\ad^\phi_v)=0$ for all $v \in V$ if and only if $\Gamma(\xi_{23}(\phi))=0$.
\end{proof}

\begin{example}\label{ex:oddcomunimod}
By \autoref{prop:unimodisunimod} it follows that for $d$ even any cyclic $\lie^d$--algebra is unimodular since $\unimod\lie^d \cong \modc{\lie^d}$. Similarly for $d$ odd $\modn{\com^{d}} \cong \modc{\com^{d}}$ and so any Frobenius algebra with a scalar product of odd degree is unimodular.
\end{example}

\begin{remark}
\autoref{prop:unimodisunimod} tells us that the modular $\mathfrak{K}^{\otimes d}$--operad $\unimod\mathcal{P}$ governs the cyclic unimodular algebras over $\mathcal{P}$. Given a (not necessarily cyclic) quadratic operad $\mathcal{Q}$ there is also a corresponding structure, a wheeled operad, governing unimodular $\mathcal{Q}$--algebras. See \cite{granaker} for example.
\end{remark}

\subsection{Cyclic, semi-quantum and quantum homotopy $\mathcal{P}$--algebras}
Let $\mathcal{P}$ be a cyclic Koszul $\mathfrak{K}^{\otimes d}$--operad with Koszul dual $\mathcal{P}^\twistshriek$.

\begin{definition}
\Needspace*{3\baselineskip}\mbox{}
\begin{itemize}
\item We call an algebra over $\feyn\modn{\mathcal{P}^\twistshriek}$ a \emph{cyclic homotopy $\mathcal{P}$--algebra}.
\item We call an algebra over $\feyn{\unimod\mathcal{P}^\twistshriek}$ a \emph{semi-quantum homotopy $\mathcal{P}$--algebra}.
\item We call an algebra over $\feyn\modc{\mathcal{P}^\twistshriek}$ a \emph{quantum homotopy $\mathcal{P}$--algebra}.
\end{itemize}
\end{definition}

\begin{notation}
When $\mathcal{P}$ is one of $\ass^{d}$, $\com^{d}$ or $\lie^{d}$ we replace the words `homotopy $\mathcal{P}$' with $A_\infty^d$, $C_\infty^d$ or $L_\infty^d$ respectively.
\end{notation}

\begin{remark}\label{rem:evenlinftytrivial}
Since for $d$ even $\modn{\com^{1-d}}\cong\modc{\com^{1-d}}$ it follows that quantum, semi-quantum and cyclic homotopy $L_\infty^d$--algebras are all the same for even $d$. Similarly for $d$ odd $\unimod\lie^{1-d} \cong \modc{\lie^{1-d}}$ and so quantum and semi-quantum $C_\infty^d$--algebras are the same for odd $d$.
\end{remark}

\begin{remark}
It is worth noting that the notions of semi-quantum and quantum homotopy $\mathcal{P}$--algebras are not well-defined up to homotopy. By this we mean that if $\mathcal{P}$ is quasi-isomorphic to $\mathcal{Q}$ it is not necessarily the case that semi-quantum and quantum homotopy $\mathcal{Q}$--algebras are equivalent to semi-quantum and quantum homotopy $\mathcal{P}$--algebras. In particular the general notion of a homotopy $\mathcal{P}$--algebra as an algebra over some cofibrant replacement for $\mathcal{P}$ does not require $\mathcal{P}$ to be Koszul, however the notion of a quantum homotopy $\mathcal{P}$--algebras does since it depends on the choice of a dual operad to $\mathcal{P}$. While it is tempting to call algebras over $\feyn\modc{\cobar\mathcal{P}}$ a quantum homotopy $\mathcal{P}$--algebra for a general operad $\mathcal{P}$ it would not necessarily coincide with the definition we have given for when $\mathcal{P}$ is Koszul. Indeed $\feyn\modc{\cobar\mathcal{P}}\cong\feyn\feyn\modn{\mathcal{P}}$ so is a cofibrant replacement for $\modn{\mathcal{P}}$ which in general is not the case for $\feyn\modc{\mathcal{P}^\twistshriek}$.
\end{remark}

\begin{theorem}
Let $\mathcal{P}$ be a cyclic Koszul $\mathfrak{K}^{\otimes d}$--operad with Koszul dual $\mathcal{P}^\twistshriek$.
\begin{enumerate}
\item The tensor product of a quantum homotopy $\mathcal{P}$--algebra with an $\mathfrak{s}^{-2}\mathcal{P}^\twistshriek$--algebra has the natural structure of a quantum $L_\infty^1$--algebra.
\item The tensor product of a semi-quantum homotopy $\mathcal{P}$--algebra with a unimodular $\mathfrak{s}^{-2}\mathcal{P}^\twistshriek$--algebra has the natural structure of a quantum $L_\infty^1$--algebra.
\item The tensor product of a cyclic homotopy $\mathcal{P}$--algebra with an $\underline{\mathfrak{s}^{-2}\mathcal{P}^\twistshriek}$--algebra has the natural structure of a quantum $L_\infty^1$--algebra.
\item The tensor product of a cyclic homotopy $\mathcal{P}$--algebra with a cyclic homotopy $\mathfrak{s}^{-2}\mathcal{P}^\twistshriek$--algebra has the natural structure of a cyclic $L_\infty^1$--algebra. 
\end{enumerate}
\end{theorem}

\begin{proof}
Observe first that $(\lie^1)^\twistshriek\cong\mathfrak{s}^2\com$ and recall that $\feyn\mathfrak{s}^2\mathcal{O}\cong\mathfrak{s}^{-2}\feyn\mathcal{O}$.
\begin{enumerate}
\item By \autoref{cor:feynscommap} there is a natural map $\feyn\modc{\mathfrak{s}^{2}\com}\rightarrow \feyn \modc{\mathcal{P}^\twistshriek}\otimes\modc{\mathfrak{s}^{-2}\mathcal{P}^\twistshriek}$.
\item By \autoref{cor:feynscommap} there is a natural map $\feyn\modc{\mathfrak{s}^2\com}\rightarrow \feyn \unimod\mathcal{P}^\twistshriek\otimes\unimod\mathfrak{s}^{-2}\mathcal{P}^\twistshriek$.
\item By \autoref{cor:feynscommap} there is a natural map $\feyn\modc{\mathfrak{s}^2\com}\rightarrow \feyn\modn{\mathcal{P}^\twistshriek}\otimes \modn{\mathfrak{s}^{-2}\mathcal{P}^\twistshriek}$.
\item Since $(\mathfrak{s}^2\mathcal{P})^\twistshriek\cong\mathfrak{s}^{-2}\mathcal{P}^\twistshriek$ then the cyclic operad version of \autoref{cor:feynscommap} implies that there is a natural map $\cobar\mathfrak{s}^2\com \rightarrow \cobar \cobar \mathcal{P} \otimes \cobar\mathfrak{s}^2\mathcal{P}$.\qedhere
\end{enumerate}
\end{proof}

\begin{remark}
Recall that for $\mathcal{Q}$ a cyclic operad $\mathfrak{s}^{-2}(\mathcal{Q}^d)^\twistshriek$ governs $\mathcal{Q}^!$--algebras with a degree $1-d$ cyclically invariant bilinear form.
\end{remark}

\begin{example}\label{ex:comlinftytensor}
Let $A$ be an algebra over $\com^{-d}$, in other words a Frobenius algebra with a degree $-d$ scalar product. Let $\mathfrak{g}$ be a cyclic $L_\infty^{1+d}$--algebra (for example it could just be a cyclic Lie algebra with a degree $1+d$ scalar product). Since $\mathfrak{s}^{-2}(\lie^{1+d})^\twistshriek\cong\com^{-d}$ then $A\otimes \mathfrak{g}$ has a degree $1$ bilinear form and the natural structure of a cyclic $L_\infty^1$--algebra. If $d$ is odd then $\modc{\com^{-d}}\cong\modn{\com^{-d}}$ so this means that $A\otimes\mathfrak{g}$ is in fact naturally a quantum $L_\infty^1$--algebra.
\end{example}

\subsection{Quantum lifting}
The maps $\modc{\mathcal{P}^\twistshriek}\rightarrow \unimod\mathcal{P}^\twistshriek \rightarrow \modn{\mathcal{P}^\twistshriek}$ induce maps $\feyn\modn{\mathcal{P}^\twistshriek}\rightarrow \feyn\unimod\mathcal{P}^\twistshriek \rightarrow \feyn\modc{\mathcal{P}^\twistshriek}$. This leads naturally to the question of lifting, considered by Hamilton in \cite{hamilton2} for the case $\mathcal{P}=\ass$ in supergraded vector spaces.

More precisely let $V$ have the structure of a cyclic homotopy $\mathcal{P}$--algebra. Then one can ask if there exists a map $g$ or a map $h$ lifting this structure:

\[
\xymatrix{
\feyn\modn{\mathcal{P}^\twistshriek}\ar[r]\ar[d] & \feyn\unimod\mathcal{P}^\twistshriek\ar@{.>}[ld]_{g}\ar[r] & \feyn\modc{\mathcal{P}^\twistshriek}\ar@{.>}[lld]^{h}
\\
\mathcal{E}(V)
}
\]

Furthermore one can study how many lifts there are up to homotopy. With the theory developed previously, this problem could now be translated into the problem of lifting Maurer--Cartan elements in the associated differential graded Lie algebras.

For $\mathcal{O}$ be a modular $\mathfrak{D}$--operad denote by $g\mathcal{O}\subset \mathcal{O}$ the suboperad of positive genus, that is the kernel of the map $\mathcal{O} \twoheadrightarrow \modn{\cyc{\mathcal{O}}}$. The following follows from \autoref{thm:barannikovfunctor} and \autoref{prop:mcfibre}.

\begin{theorem}\label{thm:quantumlifting}
Let $V$ be a cyclic homotopy $\mathcal{P}$--algebra which is represented by a Maurer--Cartan element $\xi\in\MC(\mathbf{L}(\modn{P^\twistshriek}\otimes \mathcal{E}[V]))$.
\begin{itemize}
\item The set of lifts to a quantum homotopy $\mathcal{P}$--algebra is given by $\MC(\mathbf{L}(g\modc{\mathcal{P}^\twistshriek}\otimes \mathcal{E}[V])^{\xi})$, the set of Maurer--Cartan elements in the curved Lie algebra $\mathbf{L}(\modc{g\mathcal{P}^\twistshriek}\otimes \mathcal{E}[V])^{\xi}$. The corresponding Maurer--Cartan moduli set is the space of lifts up to homotopy fixing the cyclic homotopy $\mathcal{P}$--algebra structure.
\item The set of lifts to a semi-quantum homotopy $\mathcal{P}$--algebra is given by $\MC(\mathbf{L}(g\unimod\mathcal{P}^\twistshriek\otimes \mathcal{E}[V])^\xi)$, the set of Maurer--Cartan elements in the curved Lie algebra $\mathbf{L}(g\unimod\mathcal{P}^\twistshriek\otimes \mathcal{E}[V])^{\xi}$. The corresponding Maurer--Cartan moduli set is the space of lifts up to homotopy fixing the cyclic homotopy $\mathcal{P}$--algebra structure.
\end{itemize}
\end{theorem}

Note that the curved Lie algebras $\mathbf{L}(g\modc{\mathcal{P}^\twistshriek}\otimes \mathcal{E}[V])^\xi$ and $\mathbf{L}(g\unimod\mathcal{P}^\twistshriek\otimes \mathcal{E}[V])^\xi$ admit natural complete filtrations by genus. It follows from our general Maurer--Cartan theory that the spaces of lifts of two homotopy equivalent cyclic homotopy $\mathcal{P}$--structures coincide.

The following lemma will be helpful later.

\begin{lemma}\label{lem:liesimplify}
Let $V\in\dgvect$ have a degree $d$ non-degenerate symmetric bilinear form and let $\mathcal{P}$ be a Koszul cyclic operad. Then
\[
(\mathcal{P}^d)^\twistshriek\otimes\mathcal{E}[V]\cong\Sigma^{d+3}\mathcal{P}^!\otimes\mathcal{E}[\Sigma^{-1}V^*]
\]
\end{lemma}

\begin{proof}
Since $(\mathcal{P}^d)^\twistshriek\cong\mathfrak{s}^2(\mathcal{P}^!)^{1-d}\cong\Sigma^{1-d}\mathfrak{s}^{d+1}\mathcal{P}^{!}$ and for genus zero $\mathfrak{s}^{d+1}\cong\Sigma^{2d+2}\tilde{\mathfrak{s}}^{d+1}$ then $(\mathcal{P}^d)^\twistshriek\otimes\mathcal{E}[V]\cong\Sigma^{d+3}\mathcal{P}^!\otimes\tilde{\mathfrak{s}}^{d+1}\mathcal{E}[V]$. But $\tilde{\mathfrak{s}}^{d+1}\mathcal{E}[V]\cong\tilde{\mathfrak{s}}\mathcal{E}[V^*]\cong\mathcal{E}[\Sigma^{-1}V^*]$.
\end{proof}

\section{Quantum lifts of strict algebras}
The problem of lifting a $\mathcal{P}$--algebra to a (semi-)quantum homotopy $\mathcal{P}$--algebra admits a simple solution. In this section $\mathcal{P}$ will be a cyclic Koszul $\mathfrak{K}^{\otimes d}$--operad with Koszul dual $\mathcal{P}^\twistshriek$.

Recall there is a commutative diagram
\[
\xymatrix{
\feyn\modn{\mathcal{P}^\twistshriek}\ar[r]\ar[d]_\Phi &  \feyn\modc{\mathcal{P}^\twistshriek}\ar[d]_\Psi\\
\modc{\mathcal{P}}\ar[r]&\modn{\mathcal{P}}
}
\]
where the map $\Psi$ is the extension of $\Phi$ by zero. More precisely, given a stable graph $G$ we define $\Psi$ on the space of $\modc{\mathcal{P}^\twistshriek}^*$-decorations of $G$ to be $\Phi$ if all the vertices of $G$ have zero genus and zero otherwise. 

We can extend this diagram with the following theorem.

\begin{theorem}\label{thm:unimoddiagram}
There are maps $f$ and $g$ making the following diagram commute:
\[
\xymatrix{
\feyn\modn{\mathcal{P}^\twistshriek}\ar[r]\ar[d]_\Phi & \feyn\unimod\mathcal{P}^\twistshriek\ar[r]\ar@{.>}[ld]_f & \feyn\modc{\mathcal{P}^\twistshriek}\ar[d]\ar@{.>}[ld]_g\\
\modc{\mathcal{P}}\ar[r] & \unimod\mathcal{P}\ar[r] & \modn{\mathcal{P}}
}
\]
\end{theorem}

\begin{proof}
The maps $f$ and $g$ are again the extensions of $\Phi$ by zero. To verify these are maps of operads we need to verify that the differentials coincide. Given an element $x$ in the space of $\modc{\mathcal{P}^\twistshriek}^*$--decorations of a graph $G$ with at least one vertex having non-zero genus then $d(x)$ is a sum of decorations of graphs with either a vertex with non-zero genus or a simple loop at a vertex and so $g(d(x))$ is zero in $\unimod\mathcal{P}$. Similarly given a non-zero element $x$ in the space of ${\unimod\mathcal{P}^\twistshriek}^*$--decorations of a graph $G$ with at least one vertex having non-zero genus then $d(x)$ is a sum of decorations of graphs with either a vertex having non-zero genus or a simple loop at a vertex. In the latter case the vertex cannot also be trivalent and so $f(d(x))=0$.
\end{proof}

\begin{corollary}
Let $V$ be a cyclic $\mathcal{P}$--algebra so that $V$ is naturally a homotopy $\mathcal{P}$--algebra. Then $V$ lifts to a semi-quantum homotopy $\mathcal{P}$--algebra.
\qed
\end{corollary}

\begin{corollary}
Let $V$ be a cyclic unimodular $\mathcal{P}$--algebra so that $V$ is naturally a homotopy $\mathcal{P}$--algebra. Then $V$ lifts to a quantum homotopy $\mathcal{P}$--algebra.
\qed
\end{corollary}

\begin{proposition}\label{prop:unimodnecessary}
Let $\mathcal{O}$ be a modular $\mathfrak{K}^d$--operad with zero differential and a map $h\co\modc{\mathcal{P}}\rightarrow\mathcal{O}$. There is a map $h'$ completing the diagram
\[
\xymatrix{
\feyn\modn{\mathcal{P}^\twistshriek}\ar[r]\ar[d] &  \feyn\modc{\mathcal{P}^\twistshriek}\ar@{.>}[d]^{h'}\\
\modc{\mathcal{P}}\ar[r]^h&\mathcal{O}
}
\]
if and only if $h$ factors through $\modc{\mathcal{P}}\rightarrow\unimod\mathcal{P}$.
\end{proposition}

\begin{proof}
If $h$ factors as described then $h'$ can arise by composition with the map $g$ in \autoref{thm:unimoddiagram}. Assume $h'$ exists, then we must show $h$ vanishes on $\mathcal{P}((1,1))$. Let $G$ be the graph with $1$  genus $0$ vertex with a simple loop and $1$ leg. Then the image of $\mathcal{P}((1,1))$ under $h$ is the same as the image under $h'$ of $\modc{\mathcal{P}^\twistshriek}^*$--decorations on $G$. But the differential of $\feyn\modc{\mathcal{P}^\twistshriek}$ is surjective onto this space, so $h'$ must map it to zero since $\mathcal{O}$ has zero differential.
\end{proof}

\begin{corollary}
Let $V$ be a cyclic $\mathcal{P}$--algebra with zero differential. Then $V$ lifts to a quantum homotopy $\mathcal{P}$--algebra if and only if $V$ is unimodular.
\end{corollary}

\begin{proof}
Apply \autoref{prop:unimodnecessary} with $\mathcal{O}=\mathcal{E}[V]$.
\end{proof}

\begin{remark}
Note that in the case of $V$ being unimodular it can always be lifted trivially (so that all higher operations are zero).
\end{remark}

\Needspace*{3\baselineskip}
\begin{examples}\label{ex:strictexamples}
The following examples all take place in the category of supergraded vector spaces for reasons of clarity.
\begin{enumerate}
\item \label{ex:evenlielifts}Recall that $\klie \cong \modc{\lie}$ and $\modc{\com^1}\cong\modn{\com^1}$. In this case \autoref{thm:unimoddiagram} gives the following diagram:
\[
\xymatrix{
\feyn\modn{\lie}\ar[r]\ar[d] & \feyn\unimod\lie\ar@{=}[r]\ar[ld] & \feyn\modc{\lie}\ar[d]\ar[ld]\\
\modc{\com^1}\ar@{=}[r] & \unimod\com^1\ar@{=}[r] & \modn{\com^1}
}
\]
Therefore all cyclic $\com^1$--algebras, that is Frobenius algebras with an odd scalar product, lift to quantum $C_\infty^1$--algebras. By applying $\feyn$ to this diagram we also see that any cyclic $\lie$--algebra, that is a cyclic differential graded Lie algebra with an even scalar product, lifts to a quantum $L_\infty$--algebra, but these are just the same as cyclic $L_\infty$--algebras.
\item Let $A$ be the two dimensional unital associative superalgebra generated by an element $a$ of odd degree with $a^2=1$ and with an odd scalar product given by $\langle a, 1 \rangle = 1$. As noted in \autoref{ex:2dassunimod} $A$ is unimodular and therefore $A$ is unimodular and is hence a quantum $A_\infty^1$--algebra.
\item Any cyclic differential graded associative algebra is a semi-quantum $A_\infty$--algebra.
\item Let $B$ be any cyclic differential graded associative algebra with even scalar product. Then $B\otimes A$ is an algebra over $\modc{\ass}\otimes\feyn\modc{\ass}$ and therefore is a trivial quantum $A_\infty^1$--algebra via the map $\feyn\modc{\ass}\rightarrow\modc{\ass}\otimes\feyn\modc{\ass}$. In particular $\widetilde{\Mat_n(V)}=\Mat_n(V)\otimes A$ is a quantum $A_\infty$--algebra for $V$ any cyclic $\ass$--algebra.
\end{enumerate}
\end{examples}

\section{Quantum lifts of cyclic $L_\infty$--algebras}
In this section the obstruction theory associated to lifting cyclic $L_\infty^d$--algebras will be studied more concretely and related to Chevalley--Eilenberg cohomology.

\begin{convention}
As observed in \autoref{rem:evenlinftytrivial} only the case when $d$ is odd will be of interest, so throughout this section $d$ will be assumed to be odd.
\end{convention}

Let $V$ be a (not necessarily cyclic) $L_\infty$--algebra, in other words and algebra over the operad $\dual\com$. By the operad version of \autoref{thm:barannikov} this is equivalent to a Maurer--Cartan element in the Lie algebra
\[
\Sigma^{-1}\prod_n(\mathfrak{s}^2\com^1\otimes\End[V])(n)_{S_n}
\]
and this space is isomorphic to
\[
\prod_{n=2}^{\infty}\IHom(\Sigma^{-1}V^*,(\Sigma^{-1}V^*)^{\otimes n}_{S_n}).
\]
It is now straightforward to see that this can be identified with the space of derivations of quadratic order and higher on the differential graded commutative algebra $\widehat{S}\Sigma^{-1}V^*$ which we denote by $\Der(\widehat{S}\Sigma^{-1}V^*)_+$. Furthermore this is an isomorphism of differential graded Lie algebras. The structure of an $L_\infty$--algebra on $V$ is then equivalent to a degree one derivation $m\in\Der(\widehat{S}\Sigma^{-1}V^*)_+$ with $m^2=0$. Denote the order $n$ part of $m$ by $m_n$ so that $m=m_2 + m_3 + \dots$ on the subspace $\Sigma^{-1}V^*$.

\begin{definition}\label{def:linftycecomplex}
Let $(V,m)$ be a (not necessarily cyclic) $L_\infty$--algebra. Then the Chevalley--Eilenberg cohomology complex with trivial coefficients $\CE^\bullet(V)$ is the differential graded vector space $\widehat{S} \Sigma^{-1} V^*$ with differential $m + d_V$ where $d_V$ is the internal differential on $\widehat{S}\Sigma^{-1}V^*$ arising from the internal differential on $V$.

Denote $\HCE^\bullet(V)$ the cohomology of this complex. Denote by $\CE^{\bullet}_{\geq i}(V)=\widehat{S}_{\geq i}\Sigma^{-1}V^*$ the subspace generated by tensors of order at least $i$. Denote by $\CE^{\bullet}(V)_+=\CE^{\bullet}_{\geq 3}(V)$.
\end{definition}

\begin{remark}
\autoref{def:linftycecomplex} generalises the Chevalley--Eilenberg complex of a differential graded Lie algebra.
\end{remark}

\begin{definition}\label{def:cycliclie}
Let $V\in\dgvect$ have a degree $d$ symmetric non-degenerate bilinear form. Denote by $\langle -, - \rangle^{-1}$ the degree $-d-2$ symmetric non-degenerate bilinear form on $\Sigma^{-1}V^*$. Set $\frak{g}(V)=\Sigma^{d+2} \widehat{S} \Sigma^{-1}V^*$. Define a differential graded Lie algebra structure on this space by the formula
\[\{ a_1\dots a_n, b_1\dots b_m \} = \sum_{i=1}^{n}\sum_{j=1}^{m}(-1)^{\epsilon}\langle a_i, b_j \rangle^{-1}  a_{i+1}\dots a_1\dots a_{i-1} b_{j+1}\dots b_1\dots b_{j-1}\]
where $(-1)^\epsilon$ is the sign arising from the Koszul sign rule. The differential, denoted $d_V$, is just that arising from the internal differential on $V$.

Denote by $\mathfrak{g}(V)_+=\Sigma^{d+2}\widehat{S}_{\geq 3}\Sigma^{-1}V^*$ the subalgebra generated by elements of order at least $3$.
\end{definition}

The Maurer--Cartan equation for $\mathfrak{g}(V)_+$ is sometimes called \emph{the classical master equation}.

\begin{remark}
Note that if we regard $V$ as having the zero $L_\infty$--structure then $\frak{g}(V)=\Sigma^{d+2}\CE^\bullet(V)$ as differential graded vector spaces.
\end{remark}

The well known fact that $\CE^{\bullet}(V)$ is the complex governing the deformation theory of cyclic $L_\infty^d$--algebras is the content of the following proposition.

\begin{proposition}
Let $V$ be a cyclic $L_\infty^d$--algebra.
\begin{itemize}
\item There is an isomorphism of differential graded Lie algebras $\mathbf{L}((\lie^d)^\twistshriek\otimes\mathcal{E}[V])\cong\mathfrak{g}(V)_+$.
\item If $\xi^0\in\MC(\mathfrak{g}(V))$ is the Maurer--Cartan element representing the $L_\infty^d$--algebra structure on $V$ then as differential graded vector spaces $\mathfrak{g}(V)_+^{\xi^0}\cong\Sigma^{d+2}\CE^{\bullet}(V)_+$.
\end{itemize}
\end{proposition}

\begin{proof}
By \autoref{lem:liesimplify} $\mathbf{L}((\lie^d)^\twistshriek\otimes\mathcal{E}[V])\cong\mathbf{L}\Sigma^{d+3}\mathcal{E}[\Sigma^{-1}V^*]$. Clearly as differential graded vector spaces $\mathbf{L}\Sigma^{d+3}\mathcal{E}[\Sigma^{-1}V^*]\cong\mathfrak{g}(V)_+$. The first part now follows by a straightforward check that the Lie brackets coincide.

Denote by $m\in\Der(\widehat{S}\Sigma^{-1}V^*)_+$ the derivation corresponding to the underlying $L_\infty$--structure on $V$. For the second part let $\xi_n^0$ be the order $n$ component of $\xi$, then it is sufficient to show that for any $a\in\Sigma^{-1}V^*$ then $[\xi^0_{n+1},a] = m_n(a)$. But this follows from the fact that the composition
\[\IHom(\Sigma^{-1}V^*,(\Sigma^{-1}V^*)^{\otimes n}_{S_n})\cong\End[\Sigma V](n)_{S_n}\rightarrow\Sigma^{d+2}\mathcal{E}[\Sigma^{-1}V^*](n)_{S_n}\twoheadrightarrow\Sigma^{d+2}\widehat{S}_{n+1}\Sigma^{-1}V^*
\]
induced by $V\mapsto\Sigma^d V^*$ takes $m_n$ to $\xi^0_{n+1}$.
\end{proof}

\begin{corollary}
If $V$ is a cyclic $L_\infty^d$--algebra then $\Sigma^{d+2}\CE^{\bullet}(V)_+$ has the structure of a differential graded Lie algebra. Maurer--Cartan elements correspond to cyclic $L_\infty^d$--algebra structures on $V$.
\qed
\end{corollary}
 
\begin{definition}\label{def:quantummaster}
Let $V\in\dgvect$ have a degree $d$ symmetric non-degenerate bilinear form. Define $\Delta\co\mathfrak{g}(V)\rightarrow\mathfrak{g}(V)$ by
\[\Delta(a_1\dots a_n) = \sum_{i < j} (-1)^{\epsilon}\langle a_i, a_j \rangle^{-1} a_1\dots a_{i-1}a_{i+1}\dots a_{j-1}a_{j+1} \dots a_n\]
where $(-1)^{\epsilon}$ is the sign arising from the Koszul sign rule. Define the Lie algebra $\modc{\mathfrak{g}}(V) = \mathfrak{g}(V)\otimes k[[\hbar]]$ where $\hbar$ has degree $2d+2$ so as vector spaces $k[[\hbar]]\cong \prod_{n=0}^\infty\Sigma^{-(2d+2)n}k$. with Lie bracket given by extending $\hbar$--linearly the Lie bracket on $\mathfrak{g}(V)$ and differential given by $\hbar\Delta + d_V$.

For $a\in\mathfrak{g}(V)$ with order $n$ define the order of the element $a\hbar^g\in\modc{\mathfrak{g}}(V)$ to be $2g+n$. Denote by $\modc{\mathfrak{g}}(V)_+$ the subalgebra of $\modc{\mathfrak{g}}(V)$ generated by elements of order at least $3$.
\end{definition}

The Maurer--Cartan equation for $\modc{\mathfrak{g}}(V)_+$ is sometimes called \emph{the quantum master equation}.

\begin{proposition}
Let $V\in\dgvect$ have an odd degree $d$ non-degenerate symmetric bilinear form.
There is an isomorphism of differential graded Lie algebras $\mathbf{L}(\modc{(\lie^d)^\twistshriek}\otimes\mathcal{E}[V])\cong\modc{\mathfrak{g}}(V)_+$.
\end{proposition}

\begin{proof}
Since $1-d$ is even then $\modc{\mathfrak{s}^2\com^{1-d}}\cong\mathfrak{s}^2\modc{\com}^{1-d}$ and arguing as in the proof of \autoref{lem:liesimplify}, observing that as stable $S$--modules $\mathfrak{s}^{d+1}\cong\tilde{\Sigma}^{2d+2}\tilde{\mathfrak{s}}^{d+1}$ where $\tilde{\Sigma}((g,n))=\Sigma^{1-g}k$, it follows that $\mathbf{L}(\modc{(\lie^d)^\twistshriek}\otimes\mathcal{E}[V])\cong\mathbf{L}\Sigma^{1-d}\tilde{\Sigma}^{2d+2}\mathcal{E}[\Sigma^{-1}V^*]$.

Since $\Sigma^{1-d}\tilde{\Sigma}^{2d+2}((g,n))\cong\Sigma^{d+3}\Sigma^{-(2d+2)g}k$ then clearly as graded vector spaces $\mathbf{L}\Sigma^{1-d}\tilde{\Sigma}^{2d+2}\mathcal{E}[\Sigma^{-1}V^*]\cong\modc{\mathfrak{g}}(V)_+$. The proposition now follows by a straightforward check that the Lie brackets and differentials coincide.
\end{proof}

\begin{theorem}
Let $V$ be a cyclic $L_\infty^d$--algebra which is represented by a Maurer--Cartan element $\xi^0\in\mathfrak{g}(V)_+$. The space of lifts to a quantum $L_\infty^d$--algebra is the fibre over $\xi^0$ of
\[
\MC(\modc{\mathfrak{g}}(V)_+)\rightarrow \MC(\mathfrak{g}(V)_+)
\]
which is $\MC(\hbar\modc{\mathfrak{g}}(V)_+^{\xi^0})$. The differential graded Lie algebra $\modc{\mathfrak{g}}(V)_+$ has a complete filtration $\modc{\mathfrak{g}}(V)=F_0\supset F_2\dots$ by powers of $\hbar$.

Furthermore a lift $\xi^0+\xi^1\hbar+\dots + \xi^n\hbar^n$ to level $n$ lifts to level $n+1$ if and only if the obstruction
\[
\Delta\xi^n + \sum_{\substack{1\leq i \leq j \leq n\\i+j=n+1}}[\xi^i,\xi^j]
\]
which is a cocycle in $\CE^{2-d-(2d+2)n}(V)$, vanishes as a cohomology class. In addition the cohomology $\HCE^{1-d-(2d+2)n}(V)$ (or $\HCE^{1-d-(2d+2)n}_{\geq 1}(V)$ if $n=0$) acts freely and transitively on the moduli space (up to homotopy preserving $\xi^0+\dots+\xi^n\hbar^n$) of level $n+1$ lifts.
\end{theorem}

\begin{proof}
This is obtained by \autoref{thm:quantumlifting}, \autoref{thm:liftobs} and \autoref{thm:affinespaceoflifts}.
\end{proof}

\begin{corollary}
If $\HCE^{\mathrm{odd}}(V)=0$ then $V$ admits a quantum $L_\infty^d$--lift. If $\HCE^{\mathrm{even}}(V)=0$ then any two lifts are homotopic. If $\HCE^{\bullet}(V)=0$ then $V$ admits a unique (up to homotopy) quantum $L_\infty^d$--lift.
\qed
\end{corollary}

\section{Application to Chern--Simons field theory}
In \cite{costello4} Costello considers the quantisation of Chern--Simons field theory via an infinite dimensional Batalin--Vilkovisky formalism. We will not describe Chern--Simons field theory here since that would take us too far afield.

However, one of the main results of the approach to quantisation of Chern--Simons field theory in \cite{costello4} is the construction of a certain quantum $L_\infty$--algebra structure associated to a closed oriented manifold, canonical up to homotopy. In this section we wish to interpret these ideas in terms of minimal models and quantum lifts.

\subsection{Minimal models for algebras over cyclic and modular operads}
We will first briefly recall some of the pertinent definitions and theorems concerning the construction of minimal models over cyclic and modular operads, taken from Chuang--Lazarev \cite{chuanglazarev2,chuanglazarev3}.

For simplicity in this section we will only consider supergraded vector spaces. It should be emphasised that this is by no means necessary but will be helpful due to the variety of different conventions concerning degrees in the literature.

\begin{definition}
Let $V$ be a differential supergraded vector space with an odd or even symmetric bilinear form. A \emph{Hodge decomposition} of $V$ is a choice of a maps $s\co V\rightarrow \Pi V$ and $t\co V\rightarrow V$ such that:
\begin{itemize}
\item $s^2 = 0$
\item $t^2 = t$
\item $d\circ t = t\circ d$
\item $d\circ s + s \circ d = \id - t$
\item $s\circ t = t\circ s = 0$
\item $\langle s(x), y \rangle = (-1)^{\degree{x}}\langle x, s(y) \rangle$
\item $\langle t(x), y \rangle = \langle x, t(y) \rangle$
\end{itemize}
A Hodge decomposition is called \emph{harmonious} if $d\circ t = 0$.
\end{definition}

This definition of a Hodge decomposition can be reformulated in geometric terms \cite[Proposition 2.5]{chuanglazarev3}.

\begin{proposition}
Let $V$ be a differential supergraded vector space with an odd or even symmetric bilinear form. A harmonious Hodge decomposition on $V$ is equivalent to a decomposition of $V$ into a direct sum of three subspaces
\[
V = W \oplus \im d \oplus U
\]
with $W\cong H^\bullet(V)$ and such that $W^{\perp} = \im d \oplus U$ and $U$ is an isotropic subspace of $V$.
\end{proposition}

When $V$ is finite dimensional a harmonious Hodge decomposition always exists.

The main construction from \cite{chuanglazarev3} is summarised in the following theorem.

\begin{theorem}
Let $\mathcal{P}$ be a cyclic operad (or cyclic $\mathfrak{K}$--operad). Then any choice of a harmonious Hodge decomposition $V = W \oplus \im d \oplus U$ on a $\cobar\mathcal{P}$--algebra $V$ gives rise to a $\cobar\mathcal{P}$--structure on $W\cong H^\bullet(V)$ called a \emph{cyclic minimal model} for $V$. Furthermore any two such minimal models are homotopy equivalent as $\cobar\mathcal{P}$--algebra structures on $H^\bullet(V)$.
\end{theorem}

\begin{remark}
Note also that Chuang--Lazarev \cite{chuanglazarev2} in fact prove this result for the case when the bilinear form on $V$ is not necessarily non-degenerate.
\end{remark}

In \cite{chuanglazarev2} a version for modular operads is also given.

\begin{theorem}\label{thm:modminmod}
Let $\mathcal{O}$ be a modular operad (or modular $\mathfrak{K}$--operad). Then any choice of a harmonious Hodge decomposition $V = W \oplus \im d \oplus U$ on a $\feyn\mathcal{O}$--algebra $V$ gives rise to a $\feyn\mathcal{O}$--structure on $W\cong H^\bullet(V)$ called a \emph{modular minimal model} for $V$. Furthermore any two such minimal models are homotopy equivalent as $\feyn\mathcal{O}$--algebra structures on $H^\bullet(V)$.
\end{theorem}

\begin{remark}
Note that these minimal model structures are homotopy equivalent to the original structure in the following sense: regarding the contractible subspace $\im d \oplus U$ as having the trivial $\feyn\mathcal{O}$--algebra structure then the sum of $\feyn\mathcal{O}$--algebra structures on $W \oplus \im d \oplus U$ is homotopy equivalent to the original $\feyn\mathcal{O}$--algebra structure on $V$. Also note that in this case the modular minimal model is a lift of the cyclic minimal model associated to the structure on $V$ of an algebra over the genus $0$ part of $\feyn\mathcal{O}$.
\end{remark}

\begin{remark}
By disregarding any mention of bilinear forms one also obtains analogous results for normal operads, see \cite{chuanglazarev2}, which is familiar from the classical context of homological perturbation theory.
\end{remark}

\subsection{The Cattaneo--Mn\"ev finite dimensional Chern--Simons model}
Cattaneo and Mn\"ev consider a `toy model' of Chern--Simons field theory in \cite{cattaneomnev}. In particular they consider a finite dimensional differential graded Frobenius algebra $A$ (which we will consider as supergraded) with a odd degree bilinear form (to be thought of as modelling the de Rham algebra of a closed oriented $3$--manifold, cf \autoref{ex:derhamalgebra}) and a finite dimensional cyclic Lie algebra $\mathfrak{g}$ with zero differential and an even degree non-degenerate bilinear form. Then starting with the data of a Hodge decomposition of $A$ they construct an `effective action' on the cohomology of $A\otimes \mathfrak{g}$. This is, in essence, a particular solution to the quantum master equation in the differential supergraded Lie algebra $\Pi\widehat{S}\Pi(H^\bullet(A)\otimes \mathfrak{g})^*[[h]]$ from \autoref{def:quantummaster}, in other words a quantum $L_\infty^{1}$--structure on $H^\bullet(A)\otimes \mathfrak{g}$.

In our language, the algebra $A$ is an algebra over $\modc{\com^1}$. But as noted in \autoref{ex:oddcomunimod} $\modc{\com^1}\cong\modn{\com^1}$. Then, regarding $\mathfrak{g}$ as a cyclic $L_\infty$--algebra, $A\otimes\mathfrak{g}$ is a $\modn{\com^1} \otimes\feyn\modn{\com^1}$--algebra and hence an algebra over $\feyn\modcom$ (which in the supergraded setting is a quantum $L_\infty^1$--algebra, cf \autoref{ex:comlinftytensor}). From this perspective the construction of Cattaneo--Mn\"ev follows immediately from the general theory of minimal models.

\begin{theorem}\label{thm:cattaneomnev}
The quantum $L_\infty^1$--algebra modular minimal model of $A\otimes \mathfrak{g}$ is equivalent to the quantum $L_\infty^1$--algebra obtained from the construction of the effective action of Cattaneo--Mn\"ev \cite{cattaneomnev}.
\end{theorem}

\begin{proof}
This follows from a straightforward comparison of the explicit formulae used in \cite{cattaneomnev} to the formulae in \cite{chuanglazarev2}. We omit the details.
\end{proof}

This perspective may be significantly more conceptual for those with certain mathematical tastes.

We now immediately obtain from the general theory the result that this construction does not depend, up to homotopy, on the choice of Hodge decomposition. Furthermore this argument immediately works for the case when $\mathfrak{g}$ is more generally a cyclic $L_\infty$--algebra with not necessarily zero differential in which case we get a quantum $L_\infty^1$--algebra structure on $H^\bullet(A)\otimes H^\bullet(\mathfrak{g})$.

Additionally, since $A$ is an algebra over $\modn{\com^1}$ it is in particular unimodular. Therefore $A$ itself lifts trivially to a quantum $C_\infty^1$--algebra and so the cyclic $C_\infty^1$--minimal model structure on $H^\bullet(A)$ lifts to a quantum $C_\infty^1$--minimal model. For any cyclic Lie algebra $\mathfrak{g}$, the quantum $L_\infty^1$--algebra structure on $H^\bullet(A)\otimes\mathfrak{g}$ is simply obtained from tensoring an $\feyn\modc{\lie}$--algebra with a $\modc{\lie}$--algebra.

From the perspective of \autoref{thm:cattaneomnev} the effective action is a particular quantum lift of the usual cyclic $L_\infty$--minimal model of $A\otimes\mathfrak{g}$. The reason such a lift exists is because $A$ is equipped with an \emph{odd} bilinear form.

Let us now consider the case when $A$ is instead equipped with an even bilinear form, so that $A$ is an algebra over $\modc{\com}$. Then $\mathfrak{g}$ should be a cyclic Lie algebra with an odd bilinear form. But now we have no guarantee that $A\otimes\mathfrak{g}$ necessarily lifts to a quantum $L_\infty^1$--algebra so it is only possible to apply the general cyclic minimal model construction, but not necessarily the modular minimal model construction. If $\mathfrak{g}$ is unimodular then it lifts trivially to a quantum $L_\infty^1$--algebra and so $A\otimes \mathfrak{g}$ is a $\modc{\com}\otimes\feyn\modc{\com}$--algebra and hence a quantum $L_\infty^1$--algebra. We can now proceed as above. 

In general there do exist cyclic Lie algebras with odd bilinear forms which are not unimodular. For example pick any finite dimensional Lie algebra (concentrated in degree $0$ say) $\mathfrak{h}$ which is not unimodular then it is not too difficult to see that the odd double extension $\mathfrak{g}=\mathfrak{h}\oplus\Pi\mathfrak{h}^*$ is also not unimodular. Here the Lie bracket on $\mathfrak{g}$ is given by $[x\oplus f, y\oplus g] = [x,y] \oplus(f\circ\ad_y - g\circ\ad_x)$ and the bilinear form is given by $\langle x\oplus f, y\oplus g) = f(y) + g(x)$.

\subsection{Costello's quantisation of Chern--Simons field theory}
The perspective of the preceding section is enlightening and the language of algebras over modular operads gives a new more conceptual perspective.

However `real' Chern--Simons field theory is more complicated since, of course, the de Rham algebra of a closed oriented manifold is in general infinite dimensional. However, in \cite{costello4} Costello succeeds in constructing a canonical (up to homotopy) quantum $L_\infty^1$--lift of the classical cyclic minimal model.

Let $M$ be a closed oriented manifold of odd dimension and let $\mathfrak{g}$ be a cyclic supergraded Lie algebra with an even non-degenerate bilinear form. Then $\Omega^\bullet(M)\otimes\mathfrak{g}$ is a supergraded $\lie^1$--algebra with a degenerate bilinear form. The existence of a harmonious Hodge decomposition on $\Omega^\bullet(M)$ follows from choosing a Riemannian metric on $M$ and applying Hodge's theorem, see \cite[Example 2.9 (2)]{chuanglazarev3}. Therefore one can ask when the cyclic minimal model $H^\bullet(M)\otimes\mathfrak{g}$ admits a quantum lift.

The difficulty in determining whether a lift exists arises since $\Omega^\bullet(M)$ is now infinite dimensional so we are not able to apply the general theory we have developed. However, recall the following nice theorem due to Lambrechts--Stanley \cite{lambrechtsstanley}, which we restate in our current language.

\begin{theorem}[Lambrechts--Stanley theorem]
Let $A$ be a differential $\mathbb{Z}$--graded unital cyclic commutative algebra with a possibly degenerate bilinear form of degree $d$ and which is \emph{simply-connected}, meaning $H^0(A)=k$ and $H^1(A)=0$. Furthermore assume that the form on $H^\bullet(A)$ is non-degenerate. Then there exists a differential $\mathbb{Z}$--graded cyclic commutative algebra $A'$ with a \emph{non-degenerate} bilinear form of degree $d$ which is \emph{weakly equivalent} to $A$, meaning it can be connected to $A$ by a zig-zag of morphisms of differential $\mathbb{Z}$--graded commutative algebras inducing isomorphisms of cyclic algebras on homology.
\end{theorem}

Since weakly equivalent algebras have $\infty$--isomorphic minimal models then if $M$ is simply connected we can replace $\Omega^\bullet(M)$ by a weakly equivalent differential graded cyclic commutative algebra $A$ with a non-degenerate odd bilinear form such that the cyclic minimal models of $A$ and $\Omega^\bullet(M)$ are homotopy equivalent cyclic $C_\infty^1$--structures on $H^\bullet(A)\cong H^\bullet(M)$. Now by the above arguments the cyclic minimal model of $A$ admits a quantum $C_\infty^1$--lift and we have now obtained in a rather straightforward manner the following result.

\begin{theorem}
Let $M$ be a simply connected manifold of odd dimension and let $\mathfrak{g}$ be a cyclic Lie algebra with an even non-degenerate bilinear form. Then the cyclic $C_\infty^1$--minimal model structure on $H^\bullet(M)$ admits a quantum $C_\infty^1$--lift and the cyclic $L_\infty^1$--minimal model structure on $H^\bullet(M)\otimes \mathfrak{g}$ admits a quantum $L_\infty^1$--lift.
\qed
\end{theorem}

Of course this theorem also follows from the substantially stronger, although much more involved, result of Costello \cite{costello4} and indeed without the requirement $M$ be simply connected. Moreover Costello constructs a \emph{canonical} quantum lift whereas the above argument only shows the existence of a quantum lift. However the operadic approach here is of a rather different flavour and hence of independent interest.

\begin{remark}\label{rem:whatinvariant}
Note that, in the context of the previous section, Costello's construction of a quantum lift could be interpreted as an extension of the modular minimal model construction to the infinite dimensional case. Of course this is a rather imprecise statement but an ambitious project would be to try to make this more precise and attempt to understand what sort of invariant Costello's quantum lift is. In particular, unlike the cyclic minimal model, there is no obvious reason it should necessarily be a homotopy invariant of $M$.
\end{remark}

We finish by considering the case that $M$ is even dimensional. Note that $H^\bullet(M)$ has a unit. But since the bilinear form on $H^\bullet(M)$ is even it follows that the (super)trace of the identity map on $H^\bullet(M)$ is non-zero so $H^{\bullet}(M)$ is not unimodular. If there exists any quantum $L_\infty^1$--algebra structure on $H^\bullet(M)\otimes\mathfrak{g}$ then by a straightforward calculation it can be seen that the underlying $\lie^1$--algebra is necessarily unimodular since the differential is zero. So it is a necessary condition that $\mathfrak{g}$ be unimodular by \autoref{prop:unimodtensor}. It now follows as before that this is also sufficient (at least for $M$ simply connected). The requirement that $\mathfrak{g}$ be unimodular was not included in \cite{costello4}.

\chapter{Epilogue}
In this final closing chapter we'll briefly consider a couple of possible directions for future research triggered by the work in this thesis.

\section{Closed KTCFT structure on the involutive Hochschild complex}
The results of \autoref{chap:moduli} and \autoref{chap:dihedral} leads one naturally to wonder whether there is an analogue of the theorem of Costello \cite{costello3} which equips the Hochschild chain complex of a cyclic $A_\infty$--algebra (which is equivalent to an open topological conformal field theory) with the natural structure of a \emph{closed} topological conformal field theory. \autoref{chap:dihedral} suggests such an analogue should exist, equipping the involutive Hochschild chain complex of a cyclic involutive $A_\infty$--algebra (which is equivalent to an open Klein topological conformal field theory) with the natural structure of a \emph{closed} Klein topological conformal field theory. If such a structure does exist, how does it relate to the open topological conformal field on the usual Hochschild chain complex in terms of the splitting of Hochschild cohomology as involutive and skew involutive Hochschild cohomology?

The main difficulty with this question is in understanding what a `closed Klein topological conformal field theory' should be. Even before considering the conformal version there is an added subtlety. One might guess that a closed Klein topological field theory is equivalent to a commutative Frobenius algebra with an involution, in comparison with the oriented case. However, as we have seen, this is not true since there is additional structure due to the projective plane not being obtained by gluing genus zero surfaces. This subtlety will no doubt reappear in the conformal version.

\section{Invariance of quantum minimal models}
The following question motivated by the last section of \autoref{chap:quantum} is clearly important: what sort of invariant of the manifold $M$ is the quantum $L_\infty^1$--structure on $H^{\bullet}(M)\otimes\mathfrak{g}$ constructed by Costello? It is known that the usual cyclic $L_\infty^1$--minimal model structure encodes much of the rational homotopy type of $M$ for varying $\mathfrak{g}$ and in particular it is homotopy invariant.

However, as noted in \autoref{rem:whatinvariant} there is no clear reason to believe the quantum $L_\infty^1$--structure is homotopy invariant. The powerful feature of the usual cyclic minimal models of spaces is the ability to algebraically encode rational homotopy information about a space. From this point of view the quantum $L_\infty^1$--structure can be thought of as a `quantisation' of the rational homotopy type. Therefore, understanding what (if any) additional information is algebraically encoded by this extra structure is of particular interest. However, it seems as though this could be a rather challenging problem.

\appendix
\def\chapterautorefname{Appendix}

\backmatter

\bibliography{references}
\bibliographystyle{alphaurl}

\end{document}